%% file: Invariance.tex
\documentclass[11pt,reqno]{amsart}

\newlength{\myhmargin}  
\usepackage[textheight=574pt, textwidth=445pt, marginparwidth=50pt, centering]{geometry}
\usepackage{amsmath,amssymb,amsthm,amsfonts,amscd,flafter,
graphicx,verbatim,pinlabel,mathrsfs,setspace}
\usepackage{upgreek} 
\usepackage[toc,title]{appendix}
\usepackage[none]{hyphenat}
\usepackage[all]{xy}
\usepackage{epstopdf}
\usepackage[colorlinks=true,citecolor=blue,linkcolor=blue]{hyperref}
\usepackage[all]{hypcap}
\usepackage[usenames, dvipsnames]{color}
\usepackage{todonotes}

\title{Sutured ECH is a natural invariant}

\author[{\c C}a{\u g}atay Kutluhan]{{\c C}a{\u g}atay Kutluhan}
\address{Department of Mathematics \\ University at Buffalo}
\email{kutluhan@buffalo.edu}

\author[Steven Sivek]{Steven Sivek}
\address{Department of Mathematics \\ Imperial College London}
\email{s.sivek@imperial.ac.uk}

\makeatletter
\let\@wraptoccontribs\wraptoccontribs
\makeatother
\contrib[with an appendix by]{C.~H.~Taubes}
\address{Department of Mathematics \\ Harvard University}
\email{chtaubes@math.harvard.edu}

\thanks{{\c C}a{\u g}atay Kutluhan was supported by NSF grant DMS-1360293. Steven Sivek was supported by NSF postdoctoral fellowship DMS-1204387. C.~H.~Taubes was supported by NSF grant DMS-1401192.}

\DeclareMathAlphabet{\mathpzc}{OT1}{pzc}{m}{it}
\def\R{{\mathbb{R}}}
\def\C{{\mathbb{C}}}

\def\Z{{\mathbb{Z}}}
\def\T{{\mathbb{T}}}
\newcommand\hf{\widehat{HF}}

\newcommand\zz{\mathbb{Z}}

\newcommand\Sc{\text{Spin}^c}
\newcommand\spc{\mathfrak{s}}

\newcommand\ssm{\smallsetminus}

\newcommand\ecc{\mathit{ECC}}
\newcommand\ech{\mathit{ECH}}
\newcommand\isomto{\xrightarrow{\sim}}
\newcommand\ft{\mathfrak{t}}
\newcommand\HMfrom{\widehat{\mathit{HM}}}
\newcommand\CMfrom{\widehat{\mathit{CM}}}
\newcommand\calh{\mathcal{H}}
\newcommand\calp{\mathcal{P}}
\newcommand\ynt{{Y_n}^T}

\newcommand\ynth{({Y_n}^T)'}
\newcommand\ant{{\alpha_n}^T}

\newcommand\anth{({\alpha_n}^T)'}
\newcommand\antz{\alpha_n^{0,T}}
\newcommand\antzp{(\alpha_n^{0,T})'}
\newcommand\anto{\alpha_n^{1,T}}
\newcommand\antop{(\alpha_n^{1,T})'}
\newcommand\jntz{J_n^{0,T}}
\newcommand\jntzp{(J_n^{0,T})'}
\newcommand\jnto{J_n^{1,T}}
\newcommand\jntop{(J_n^{1,T})'}
\newcommand\jnt{{J_n}^T}
\newcommand\jpnt{{J'_n}^T}

\newcommand\spinc{\mathrm{Spin}^c}
\newcommand\spb{\mathbb{S}}

\newcommand\si{\mathbb{S}_{{\rm I}}}
\newcommand\psii{\psi_{{\rm I}}}
\newcommand\nabi{\nabla_{\mathrm{I}}}

\newcommand\diracop{\mathcal{D}}

\newcommand\aA{\hat{\mathrm{a}}_A}
\newcommand\Ai{\uptheta_0}
\newcommand\ahat{\hat{\mathrm{a}}}
\newcommand\calB{\mathcal{B}_{(A,\psi)}}
\newcommand\cale{\mathcal{E}}

\newcommand\en{{\textsc e}}
\newcommand\bbh{\mathbb{H}}
\newcommand\clm{\mathfrak{cl}}
\newcommand\fra{\mathfrak{a}}
\newcommand\frb{\mathfrak{b}}
\newcommand\frc{\mathfrak{c}}
\newcommand\frcs{\mathfrak{cs}}
\newcommand\frd{\mathfrak{d}}
\newcommand\dis{\operatorname{dist}}
\newcommand\fre{\mathfrak{e}}
\newcommand\frE{\mathfrak{E}}
\newcommand\frf{\mathfrak{f}}

\newcommand\frg{\mathfrak{g}}
\newcommand\frh{\mathfrak{h}}
\newcommand\frhi{\frh_{{\rm I}}}
\newcommand\fri{\mathfrak{i}}
\newcommand\frp{\mathfrak{p}}
\newcommand\frP{\mathfrak{P}}
\newcommand\frq{\mathfrak{q}}
\newcommand\rmD{\mathrm{D}}
\newcommand\LL{\mathcal{L}}
\newcommand\LLf{\mathfrak{L}}
\newcommand\Li{\mathcal{L}_{{\rm I}}}
\newcommand\Lif{\mathfrak{L}_{{\rm I}}}
\newcommand\sca{\textsc{a}} 
\newcommand\sch{{\textsc h}}
\newcommand\scL{{\textsc l}}
\newcommand\uscL{\underline{{\textsc l}}}
\newcommand\scm{\textsc{m}}
\newcommand\sco{\textsc{o}}  
\newcommand\usco{\underline{\textsc{o}}} 
\newcommand\scp{\textsc{p}}
\newcommand\scq{\textsc{q}}
\newcommand\scu{{\textsc u}}
\newcommand\hi{h_{{\rm I}}}
\newcommand\rmn{\mathrm{N}}
\newcommand\rmx{\mathrm{x}}
\newcommand\frw{\mathfrak{w}}
\newcommand\pzee{\mathpzc{e}}
\newcommand\pzz{\mathpzc{z}}
\newcommand\frz{\mathfrak{z}}
\newcommand\rmu{\mathrm{U}}
\newcommand\pze{\mathpzc{E}}
\newcommand\upze{\underline{\mathpzc{E}}}

\newcommand\pzq{\mathpzc{q}}
\newcommand\hu{\hat{u}}
\newcommand\hatom{\hat{\upomega}_0}
\newcommand\ind{\upiota_\frd}
\newcommand\he{\hat{e}}
\newcommand\frQ{\mathfrak{Q}}
\newcommand\frS{\mathfrak{S}}
\newcommand\frT{\mathfrak{T}}
\newcommand\frv{\mathfrak{v}}
\newcommand\calv{\mathcal{V}}

\DeclareFontFamily{U}{mathx}{\hyphenchar\font45}
\DeclareFontShape{U}{mathx}{m}{n}{
      <5> <6> <7> <8> <9> <10>
      <10.95> <12> <14.4> <17.28> <20.74> <24.88>
      mathx10
      }{}
\DeclareSymbolFont{mathx}{U}{mathx}{m}{n}
\DeclareFontSubstitution{U}{mathx}{m}{n}
\DeclareMathAccent{\widecheck}{0}{mathx}{"71}
\newcommand{\HMto}{\widecheck{\mathit{HM}}}

\newtheorem{theorem}{Theorem}[section]

\newtheorem{lemma}[theorem]{Lemma}

\newtheorem{conjecture}[theorem]{Conjecture}
\newtheorem{corollary}[theorem]{Corollary}
\newtheorem{proposition}[theorem]{Proposition}

\theoremstyle{definition}
\newtheorem{definition}[theorem]{Definition}
\newtheorem{remark}[theorem]{Remark}
\newtheorem*{warning}{Warning}

\newtheorem{pt}{Part}
\newtheorem{pt2}{Part}
\newtheorem{pt3}{Part}
\newtheorem{prt}{Part}
\newtheorem{step}{Step}
\newtheorem{sp2}{Step}
\newtheorem{sp3}{Step}
\newtheorem{sp4}{Step}
\newtheorem{sp5}{Step}
\newtheorem{sp6}{Step}
\makeatletter
\let\@@pmod\pmod
\DeclareRobustCommand{\pmod}{\@ifstar\@pmods\@@pmod}
\def\@pmods#1{\mkern4mu({\operator@font mod}\mkern 6mu#1)}
\makeatother

\makeatletter
\newtheorem*{rep@thm}{\rep@title}
\newcommand{\newreptheorem}[2]{%
\newenvironment{rep#1}[1][0,0]{%
\def\rep@title{#2##1}%
\begin{rep@thm}}%
{\end{rep@thm}}}
\makeatother
\newreptheorem{theorem}{}
\numberwithin{equation}{section}
\begin{document}
\sloppy
\begin{abstract} 
We show that sutured embedded contact homology is a natural invariant of sutured contact $3$-manifolds which can potentially detect some of the topology of the space of contact structures on a $3$-manifold with boundary. The appendix, by C.~H.~Taubes, proves a compactness result for the completion of a sutured contact $3$-manifold in the context of Seiberg--Witten Floer homology, which enables us to complete the proof of naturality.
\end{abstract}

\maketitle

\section{Introduction}
\label{sec:intro}
\input{intro.tex}

\subsection{Acknowledgements}

The first two authors thank Russell Avdek, Michael Hutchings, and Chris Wendl for helpful conversations. We are also greatly indebted to the third author for providing the crucial correspondence results that make up the appendix, leading to a proof of naturality.

\section{Sutured ECH and some related constructions}
\label{sec:constructions}

\subsection{Sutured ECH}
\label{ssec:sutured-ech}
\input{sutured-ech.tex}

\subsection{Contact 1-handles}
\label{ssec:handles}
\input{handles.tex}

\subsection{Liouville forms on the boundary}
\label{ssec:liouville-boundary}
\input{flexibility.tex}

\subsection{Closed manifolds and continuation maps}
\label{ssec:continuation}
\input{closed-embedding.tex}

\section{Independence of the almost complex structure}
\label{sec:j-independence}
\input{j-independence.tex}

\section{Independence of the contact form}
\label{sec:alpha-independence}

\subsection{An isomorphism for isotopic contact forms}
\label{ssec:alpha-independence}
\input{alpha-independence.tex}

\subsection{Isomorphisms and embedding data}
\label{ssec:isomorphism-embedding}
\input{isomorphism-embedding.tex}

\subsection{A natural version of Theorem \ref{thm:alpha-independence}}
\label{ssec:alpha-naturality}
\input{alpha-naturality.tex}

\section{Some properties of the contact class}
\label{sec:contact-class}
\input{contact-class.tex}

\section{Invariance under gluing 1-handles}
\label{sec:handle-invariance}
\input{handle-invariance.tex}

\section{Stabilization and a canonical version of sutured ECH}
\label{sec:stabilize}
\input{stabilize.tex}

\newpage
\begin{appendix}
\numberwithin{theorem}{subsection}
\numberwithin{equation}{subsection}
\section{(by C. H. Taubes)}
\label{sec:appendix}

\subsection{Setting the stage}
\label{ssec:setup}
\input{a_setup.tex}

\subsection{Monopoles}
\label{ssec:monopoles}
\input{a_monopoles.tex}

\subsection{Instantons}
\label{ssec:instantons}
\input{a_instantons.tex}

\subsection{A priori bounds for instantons}
\label{ssec:instanton-bounds}
\input{a_instanton-bounds.tex}

\subsection{Local convergence to pseudo-holomorphic curves}
\label{ssec:local-convergence}
\input{a_local-convergence.tex}

\subsection{Convergence of instantons as $T\to\infty$}
\label{ssec:instanton-convergence}
\input{a_instanton-convergence.tex}

\subsection{Comparing instantons on $Y_\infty$ and $Y_T$}
\label{ssec:instanton-comparison}
\input{a_instanton-comparison.tex}
\end{appendix}
\newpage
\bibliographystyle{hplain}
\bibliography{References}

\end{document}

%% file: intro.tex

Embedded contact homology, defined by Hutchings, is an invariant of closed 3-manifolds $Y$ equipped with nondegenerate contact forms $\lambda$.  Given a homology class $\Gamma \in H_1(Y)$ and a generic symplectization-admissible almost complex structure $J$ on $\R \times Y$, one considers the free Abelian group $\ecc(Y,\lambda,\Gamma,J)$ generated by \emph{admissible orbit sets} $\Theta = \{(\Theta_i, m_i)\}$ with the following properties:
\begin{itemize}\leftskip-0.35in
\item $\{\Theta_i\}$ is a finite collection of distinct embedded Reeb orbits in $(Y,\lambda)$, and the $m_i$ are positive integers;
\item $m_i = 1$ if $\Theta_i$ is a hyperbolic Reeb orbit;
\item $[\Theta] := \sum_i m_i [\Theta_i]$ is equal to $\Gamma$.
\end{itemize}
This group admits a differential $\partial$ defined by its action on the set of generators via a suitable count of $J$-holomorphic curves in $\R\times Y$ with {\it ECH index}, denoted $I$, equal to 1. To be more explicit, given two generators $\Theta = \{(\Theta_i, m_i)\}$ and $\Theta' = \{(\Theta'_i, m'_i)\}$, the coefficient $\langle \partial\Theta,\Theta'\rangle$ is a signed count, modulo $\R$-translation, of $I=1$ $J$-holomorphic curves asymptotic to $\R\times\Theta_i$ with multiplicity $m_i$ at $+\infty$ and to $\R\times\Theta'_i$ with multiplicity $m'_i$ at $-\infty$.  The resulting homology is denoted $\ech(Y,\lambda,\Gamma,J)$, while the direct sum over all $\Gamma$ is denoted $\ech(Y,\lambda,J)$.  Embedded contact homology is equipped with a natural $\zz[U]$-module structure, and there exists a contact class $c(\lambda) = [\emptyset]\in\ech(Y,\lambda,0,J)$ corresponding to the empty set of Reeb orbits.  For more information, see the detailed expositions by Hutchings \cite{hutchings-icm, hutchings-budapest}.

It was conjectured that $\ech(Y,\lambda,J)$ should in fact be a topological invariant of $Y$, independent of $J$ and of the contact form $\lambda$, but direct proofs of this have been elusive.  The independence was finally established by Taubes, who showed the following in a long series of papers \cite{taubes1,taubes2,taubes3,taubes4,taubes5}:
\begin{theorem}[Taubes]
\label{thm:taubes}
There is a canonical isomorphism of relatively graded $\zz[U]$-modules
\[ \ech_*(Y,\lambda,\Gamma,J) \cong \widehat{HM}^{-*}(Y,\spc_{\xi}+PD(\Gamma)), \]
where $\widehat{HM}^{*}$ denotes a particular version of Seiberg--Witten Floer cohomology \cite{kmbook} and $\spc_\xi$ is the canonical $\Sc$ structure associated to $\xi = \ker(\lambda)$.  This isomorphism sends the ECH contact class $c(\lambda)$ to the contact invariant \cite{km-contact} in Seiberg--Witten Floer cohomology.
\end{theorem}

More recently, following an extension of Heegaard Floer homology to balanced sutured manifolds by Juh{\'a}sz \cite{juhasz-sutured} and subsequent analogues in Seiberg--Witten and instanton Floer homologies by Kronheimer and Mrowka \cite{km-excision}, Colin, Ghiggini, Honda, and Hutchings \cite{cghh} defined a generalization of embedded contact homology to sutured contact manifolds.  To any sutured contact $3$-manifold $(M,\Gamma)$ with contact form $\alpha$ and generic almost complex structure $J$ {\it tailored} to $(M,\alpha)$, they define the \emph{sutured embedded contact homology} group
\[ \ech(M,\Gamma,\alpha,J), \]
with a direct sum decomposition 
\[\ech(M,\Gamma,\alpha,J)=\bigoplus_{h\in H_1(M)}\ech(M,\Gamma,\alpha,J,h),\] 
just as in the case of closed manifolds, and conjecture \cite[Conjecture 1.2]{cghh} that this is a topological invariant of $(M,\Gamma)$.

In this paper we prove that sutured ECH does not depend on the contact form $\alpha$ or the almost complex structure $J$. 
\begin{theorem}
\label{thm:intro-alpha-isomorphism}
Let $(M,\Gamma)$ be a sutured manifold, and let $\alpha_0,\alpha_1$ be adapted, nondegenerate contact forms on $(M,\Gamma)$ such that $\ker(\alpha_0)$ and $\ker(\alpha_1)$ are isotopic rel a neighborhood of $\Gamma$.  For any two generic almost complex structures $J_0$ and $J_1$ tailored to $(M,\alpha_0)$ and $(M,\alpha_1)$, respectively, there is an isomorphism
\[ \ech(M,\Gamma,\alpha_0,J_0) \cong \ech(M,\Gamma,\alpha_1,J_1) \]
which respects direct sum decompositions over $H_1(M)$ and sends the contact class $c(\alpha_0)$ to the contact class $c(\alpha_1)$.
\end{theorem}
\noindent We hope to address the question of topological invariance in future work. 

Theorem \ref{thm:intro-alpha-isomorphism} was proved simultaneously by Colin, Ghiggini, and Honda \cite[Theorem 10.2.2]{cgh-openbook}, using a construction which is virtually identical to ours. Our strategy of proof is to embed sutured contact manifolds $(M,\Gamma,\alpha)$ inside certain families $(Y_n,\alpha_n)$ of closed contact manifolds whose construction is reminiscent of an open book decomposition.  The contact forms $\alpha_n$ are such that given a constant $L>0$, the ECH generators with symplectic action less than $L$ associated to $(M,\Gamma,\alpha,J)$ for a generic tailored almost complex structure $J$ and to $(Y_n,\alpha_n,J_n)$ for an appropriate extension $J_n$ of $J$ coincide canonically when $n$ is sufficiently large, as well as the moduli spaces of pseudo-holomorphic curves which define the respective differentials.  Hutchings and Taubes \cite{ht2} constructed canonical isomorphisms between filtered ECH groups for the closed contact manifolds $(Y_n,\alpha_n)$ with different choices of $J_n$, and we can transfer these back to the filtered ECH groups for $(M,\Gamma,\alpha)$.  We then turn isotopies of contact forms on $(M,\Gamma)$ into exact symplectic cobordisms between the $Y_n$ and use the cobordism maps defined in \cite{ht2} for filtered ECH to relate the ECH groups for pairs of isotopic contact forms.

By construction, the isomorphism of Theorem \ref{thm:intro-alpha-isomorphism} potentially depends on a choice of what we call \emph{embedding data} (see Definition \ref{def:embedding-data}) for $(M,\Gamma,\alpha)$, appearing in both our work and in that of Colin--Ghiggini--Honda \cite{cgh-openbook}. However, one might expect that the embedding data (essentially the triple $(Y_n,\alpha_n, J_n)$) should play a fairly minor role, since the filtered ECH group which coincides with the appropriate filtered ECH group for $(M,\Gamma,\alpha,J)$ is concentrated near the submanifold $M \subset Y_n$.   (See Section \ref{ssec:sutured-ech} for the precise definition of filtered ECH.)  We can make this precise using the correspondence results proved in the appendix to this paper concerning Seiberg--Witten theory for completions of $(M,\Gamma,\alpha,J)$, allowing us to prove an analogue of \cite[Theorem~1.3]{ht2} for sutured ECH.

\begin{theorem}
\label{thm:filtered-sutured-ech}
Let $(M,\Gamma,\alpha)$ be a sutured contact manifold.
\begin{enumerate}\leftskip-0.25in
\item If $\alpha$ is an $L$-nondegenerate contact form, then there is a canonically defined group $\ech^L(M,\Gamma,\alpha)$ which is isomorphic to $\ech^L(M,\Gamma,\alpha,J)$ for any $J$.
\item If $L < L'$ and $\alpha$ is $L'$-nondegenerate, then there is a canonical map 
\[i^{L,L'}: \ech^L(M,\Gamma,\alpha) \to \ech^{L'}(M,\Gamma,\alpha)\] induced by the maps $i^{L,L'}_J$ of \eqref{eq:i-LL-J} and satisfying $i^{L,L''} = i^{L',L''} \circ i^{L,L'}$ for $L<L'<L''$.
\item The direct limit $\ech(M,\Gamma,\alpha)$ of the system $(\{\ech^L(M,\Gamma,\alpha)\}_L, \{i^{L,L'}\}_{L,L'})$ is canonically isomorphic to $\ech(M,\Gamma,\alpha,J)$ for any $J$.
\end{enumerate}
\end{theorem}

The same results also produce a stronger version of Theorem~\ref{thm:intro-alpha-isomorphism}:

\begin{theorem}
\label{thm:intro-alpha-naturality}
Fix a contact form $\alpha$ which is adapted to the sutured manifold $(M,\Gamma)$, and let $\Xi(M,\Gamma,\alpha)$ denote the space of cooriented contact structures on $(M,\Gamma)$ which agree with $\ker(\alpha)$ on a neighborhood of $\partial M$.  Then sutured ECH canonically defines a local system on $\Xi(M,\Gamma,\alpha)$, i.e.\ a functor
\[ \ech: \Pi_1(\Xi(M,\Gamma,\alpha)) \to \textsc{AbGroup}, \]
from the fundamental groupoid of $\Xi(M,\Gamma,\alpha)$ to the category of Abelian groups.  It canonically assigns a group $\ech(M,\Gamma,\xi,\alpha|_{\partial M})$ to any $\xi \in \Xi(M,\Gamma,\alpha)$, depending only on $\xi$ and on the restriction of $\alpha$ to $\partial M$, and an isomorphism
\[ F_{\xi_s}: \ech(M,\Gamma,\xi_0,\alpha|_{\partial M}) \isomto \ech(M,\Gamma,\xi_1,\alpha|_{\partial M}) \]
to any homotopy class of paths $\xi_s \subset \Xi(M,\Gamma,\alpha)$.  These groups and isomorphisms decompose naturally with respect to $H_1(M)$, and there is a contact class $c(\xi) \in \ech(M,\Gamma,\xi,\alpha|_{\partial M})$ satisfying $F_{\xi_s}(c(\xi_0)) = c(\xi_1)$ for all paths $\xi_s$.
\end{theorem}

\begin{remark}
\label{rem:signs}
All of the theorems stated above are true with integer coefficients.  In Theorems \ref{thm:intro-alpha-isomorphism} and \ref{thm:intro-alpha-naturality} this requires some care, because we use the ECH cobordism maps defined by Hutchings and Taubes \cite[Theorem 1.9]{ht2} (most notably throughout Section \ref{sec:alpha-independence}), and these are only proved to work over $\Z/2\Z$. The issue is that one must choose a homology orientation to avoid a sign ambiguity in the corresponding Seiberg--Witten Floer cobordism map.  In our case, we work with topologically product cobordisms, and these have canonical homology orientations, so the cobordism maps on ECH exist and have the desired properties over $\Z$.
\end{remark}

We do not know whether the local system of Theorem \ref{thm:intro-alpha-naturality} has nontrivial monodromy: in other words, whether there exists a closed loop $\xi_s: (S^1,*) \to \Xi(M,\Gamma,\alpha)$ for some $(M,\Gamma,\alpha)$ which induces a nontrivial automorphism of $\ech(M,\Gamma,\xi_*,\alpha|_{\partial M})$.  This could potentially be used to detect interesting topology in the space of contact structures on $(M,\Gamma)$.  In general little is known about such spaces, though some cases are understood due to work of Eliashberg \cite{eliashberg-ot, eliashberg-twenty}, Ding--Geiges \cite{ding-geiges}, and Geiges--Klukas \cite{geiges-klukas}, among others.

The fact that the isomorphisms of Theorem \ref{thm:intro-alpha-naturality} may depend on the paths $\xi_s$ rather than just their endpoints is not terribly surprising, however, because similar phenomena occur in other sutured homology theories.  In these other homology theories, one associates a natural invariant to a sutured manifold and a canonical isomorphism to any diffeomorphism of sutured manifolds; this naturality was proved for sutured Floer homology by Juh{\'a}sz and Thurston \cite{juhasz-thurston} and for sutured monopole homology by Baldwin and the second author \cite{bs1}, and there is no guarantee in either case that a sequence of diffeomorphisms whose composition is the identity map on $(M,\Gamma)$ will produce the identity on the sutured homology. For example, there is a natural identification $SFH(Y(p)) = \hf(Y,p)$, where $p$ is a point in the closed $3$-manifold $Y$ and $Y(p)$ is the complement of a ball around $p$ with a single suture.  Moving $p$ along a closed loop in $Y$ produces an automorphism of $\hf$ and thus an action of $\pi_1(Y,p)$ on $\hf(Y,p)$ which is expected to be nontrivial, and similarly for $\widetilde{HM}(Y,p) := \textbf{\textup{\underline{SHM}}}(Y(p))$.  In contrast, the other variants $\HMto$, $\HMfrom$, and $\overline{HM}$ of Seiberg--Witten Floer homology do not use a basepoint or have a nontrivial $\pi_1$-action \cite{kmbook}, hence neither does $\ech$ for closed manifolds, and the same is expected for $HF^+$, $HF^-$, and $HF^\infty$ \cite{juhasz-thurston}.

\subsection{Organization}
The organization of this paper is as follows.  Section \ref{sec:constructions} reviews the definition of sutured ECH and gathers some topological constructions involving contact manifolds which will be used in upcoming parts of this paper. In Section \ref{sec:j-independence} we show that both $\ech(M,\Gamma,\alpha,J)$ and the filtered groups $\ech^L(M,\Gamma,\alpha,J)$ are canonically independent of the almost complex structure $J$, up to a choice of embedding data.  In Section \ref{sec:alpha-independence}, we use ECH cobordism maps to show that $\ech(M,\Gamma,\alpha)$ is also independent of $\alpha$ up to isotopy, and we prove a weaker version of Theorem \ref{thm:intro-alpha-naturality}, showing that the claimed local system exists but may also depend on embedding data.  In Section \ref{sec:contact-class} we use this to prove some properties of the contact class in sutured ECH, and in Section \ref{sec:handle-invariance} we show that isomorphisms corresponding to attaching contact $1$-handles can be made natural for certain compatible choices of embedding data.  In Section \ref{sec:stabilize} we prove that all of the above constructions are independent of the embedding data and hence canonical, completing the proof of Theorems \ref{thm:filtered-sutured-ech} and \ref{thm:intro-alpha-naturality}.  This relies on the technical results in Appendix \ref{sec:appendix}, which relates the Seiberg-Witten Floer homology of certain closed manifolds (the embedding data) to that of an open manifold built by completing the sutured contact manifold $(M,\Gamma,\alpha)$.

%% file: sutured-ech.tex

In this section we will recall the definition of sutured contact manifolds and of sutured ECH from \cite{cghh}.

\begin{definition}[{\cite[Definition 2.8]{cghh}}]
\label{def:sutured-contact-manifold}
A \emph{sutured contact manifold} is a triple $(M,\Gamma,\alpha)$, where $M$ is an oriented $3$-manifold with corners, $\Gamma \subset \partial M$ is a closed, embedded, oriented multicurve, and $\alpha$ is a contact form.  We require the following to hold:
\begin{itemize}\leftskip-0.35in
\item There is a neighborhood 
\[ U(\Gamma) = [-1,0]_\tau \times [-1,1]_t \times \Gamma \]
of $\Gamma = \{(0,0)\}\times \Gamma$ such that $U(\Gamma) \cap \partial M$ is the closure of $\partial U(\Gamma) \ssm \{-1\}\times[-1,1]\times \Gamma$.
\item The corners of $M$ are $\{(0,\pm 1)\} \times \Gamma$.
\item The closure of $\partial M \ssm \{0\}\times [-1,1]\times \Gamma$ is a disjoint union of oriented surfaces
\[ R_+(\Gamma) \sqcup R_-(\Gamma), \]
where $R_\pm(\Gamma)$ has oriented boundary $\{(0,\pm 1)\} \times \Gamma$.  (In particular, $R_-(\Gamma)$ has orientation opposite to the boundary orientation of $\partial M$.)
\item The contact form $\alpha$ restricts to Liouville forms $\beta_\pm$ on $R_\pm(\Gamma)$.
\item On $U(\Gamma)$ we have $\alpha = Cdt + e^\tau \beta_0$ for some constant $C>0$ and volume form $\beta_0$ on $\Gamma$.
\end{itemize}
We will often refer to the neighborhood $U(\Gamma)$ using the above coordinates $\tau$ and $t$; in these coordinates the Reeb vector field on $U(\Gamma)$ is $\frac{1}{C}\partial_t$.
\end{definition}

We remark that the sutured manifold $(M,\Gamma)$ underlying a sutured contact manifold is always \emph{balanced}, meaning that $\chi(R_+(\Gamma)) = \chi(R_-(\Gamma))$.  Indeed, up to rounding corners the boundary $\partial M$ is a convex surface (in the sense of Giroux \cite{giroux-convex}) with respect to $\xi=\ker(\alpha)$, having positive region $R_+(\Gamma)$ and negative region $R_-(\Gamma)$, and so $\langle e(\xi), \partial M\rangle = \chi(R_+(\Gamma))-\chi(R_-(\Gamma))$.  The left side is zero since $[\partial M] = 0$ in $H_2(M)$, from which the claim follows.

Note that the last condition in Definition~\ref{def:sutured-contact-manifold} implies that the Liouville forms on $R_\pm(\Gamma) \cap U(\Gamma) = [-1,0]_\tau \times \{\pm 1\}_t \times \Gamma$ are identical, namely they are equal to $e^\tau \beta_0$ and have Liouville vector field $\partial_\tau$.  Thus we can extend $(R_{\pm}(\Gamma),\beta_\pm)$ to complete Liouville manifolds $(\widehat{R_\pm(\Gamma)}, \widehat{\beta}_\pm)$ by gluing on ends of the form $([0,\infty)_\tau \times \{\pm 1\}_t \times \Gamma, e^{\tau}\beta_0)$.  Moreover, the Reeb vector field $R_\alpha$ is everywhere transverse to $R_\pm(\Gamma)$, otherwise $d\beta_\pm(R_\alpha,\cdot) = d\alpha(R_\alpha,\cdot) = 0$ would imply that $R_\alpha=0$ at some point by the nondegeneracy of $d\beta_\pm$, which is a contradiction. Thus we can extend the coordinate $t$ on $U(\Gamma)$ to collar neighborhoods $(1-\epsilon,1]_t\times R_+(\Gamma)$ and $[-1,-1+\epsilon)_t \times R_-(\Gamma)$ of $R_\pm(\Gamma)$ by declaring $\partial_t = C\cdot R_\alpha$, and hence write $\alpha=Cdt+\beta_\pm$ there.

\begin{definition}[{\cite[Section 2.4]{cghh}}]
\label{def:completion}
Let $(M,\Gamma,\alpha)$ be a sutured contact manifold.  We define the vertical completion $(M_v,\alpha_v)$ of $(M,\alpha)$ to be the manifold
\[ M_v =\big( (-\infty,-1] \times R_-(\Gamma)\big) \cup_{R_-(\Gamma)} M \cup_{R_+(\Gamma)} \big([1,\infty) \times R_+(\Gamma)\big), \]
with contact form $\alpha_v$ equal to $\alpha$ on $M$ and extended as $Cdt + \beta_\pm$ to the other pieces.  This manifold has boundary $\{0\} \times \R \times \Gamma$, with contact form $Cdt+e^\tau \beta_0$ on a collar neighborhood $[-1,0]_\tau \times [-1,1]_t \times \Gamma$.

The \emph{completion} $(M^*,\alpha^*)$ of $(M,\Gamma,\alpha)$ is then the open manifold
\[ M^* = M_v \cup_{\{0\}\times \R \times \Gamma} [0,\infty) \times \R \times \Gamma, \]
with $\alpha^*|_{M_v} = \alpha_v$ and $\alpha^* = Cdt+e^\tau\beta_0$ on the rest of $M^*$.
\end{definition}

We remark that the $t$ coordinate is well-defined on all of $M^* \ssm \mathrm{int}(M)$ and on a collar neighborhood of $\partial M$, with the Reeb vector field equal to $\frac{1}{C}\partial_t$ throughout these regions, hence all closed Reeb orbits of $\alpha^*$ lie entirely in $M$.  There is also a well-defined region $[-1,\infty)_\tau \times \R \times \Gamma$ with associated $\tau$ coordinate, where we have $\alpha = Cdt+e^\tau\beta_0$.

Given a contact manifold $(Y,\alpha)$ with $\rm{ker}(\alpha)=\xi$ oriented by $d\alpha$, we say an almost complex structure $J$ on the symplectization $\R_s \times Y$ is \emph{$\alpha$-adapted} if it is $s$-invariant, preserves $\xi$, and satisfies $J(\partial_s) = R_\alpha$ and $d\alpha(v,Jv) > 0$ for all nonzero $v\in\xi$.

\begin{definition}[{\cite[Section 3.1]{cghh}}]
\label{def:tailored}
An almost complex structure $J$ on $\R\times M^*$ is \emph{tailored} to the completion $(M^*,\alpha^*)$ if it is $\alpha^*$-adapted and $\partial_t$-invariant on a neighborhood of $\R_s \times (M^*\ssm\mathrm{int}(M))$, and if moreover its projection $J_0$ to the completed Liouville manifolds $(\widehat{R_\pm(\Gamma)}, \widehat{\beta}_\pm)$ is $\widehat{\beta}_\pm$-adapted, meaning that
\begin{itemize}
\item On $R_\pm(\Gamma)$, we have $d\beta_{\pm}(v,J_0v) > 0$ for all nonzero tangent vectors $v$;
\item On the ends $[0,\infty)_\tau \times \{\pm 1\} \times \Gamma$, $J_0$ is $\tau$-invariant and sends $\partial_\tau$ to the unique tangent vector $R_{\beta_0}$ to $\{(\tau,\pm 1)\} \times \Gamma$ satisfying $\beta_0(R_{\beta_0}) = 1$.
\end{itemize}
\end{definition}

\begin{definition}[{\cite[Section 6.3]{cghh}}]
\label{def:sutured-ech}
Let $(M,\Gamma,\alpha)$ be a sutured contact manifold with a non-degenerate contact form $
\alpha$ and completion $(M^*,\alpha^*)$ and generic (in the sense of \cite{ht0}) tailored almost complex structure $J$ on $\R\times M^*$.  Then we define the chain complex
\[ \ecc(M,\Gamma,\alpha,J) \]
to be generated over $\zz$ by orbit sets of $(M^*,\alpha^*)$, with differential $\partial$ such that the coefficient $\langle \partial\Theta_+, \Theta_-\rangle$ is a signed count of elements in the moduli space
\[ \mathcal{M}_{I=1}(\R\times M^*, J; \Theta_+, \Theta_-) \]
of ECH index 1 J-holomorphic curves from $\Theta_+$ to $\Theta_-$.  The resulting \emph{sutured embedded contact homology} group is denoted $\ech(M,\Gamma,\alpha,J)$.
\end{definition}

The definition of sutured ECH is identical to the definition of ECH for closed manifolds, as is the proof that $\partial^2 = 0$ modulo some extra analysis (see \cite[Section 6.1]{cghh}).  It is helpful to note that given two orbit sets $\Theta_+$ and $\Theta_-$, which we will sometimes call {\it ECH generators}, no curve $u\in\mathcal{M}_{I=1}(\R\times M^*, J; \Theta_+, \Theta_-)$ can enter the region of $\R\times M^*$ where $\tau > 0$ by \cite[Lemma 5.5]{cghh}, and that for fixed $\Theta_+$ and $\Theta_-$, the $t$-coordinates on any curve in this moduli space are uniformly bounded by \cite[Proposition 5.20]{cghh}.

We will also make extensive use of the symplectic action filtration on ECH, whose definition and important properties are borrowed verbatim from the analogue for closed contact manifolds.

\begin{definition}[{\cite[Section 1.2]{ht2}}]
\label{def:filtration}
Let $(M,\Gamma,\alpha)$ be a sutured contact manifold.  The \emph{symplectic action} of an ECH generator $\Theta = \{(\Theta_i,m_i)\}$ is defined as
\[ \mathcal{A}_\alpha(\Theta) := \sum_i m_i \int_{\Theta_i} \alpha. \]
Let $(M,\Gamma,\alpha)$ have no Reeb orbits of action equal to $L$, and all Reeb orbits of action less than $L$ be nondegenerate. Fix a tailored almost complex structure $J$ on $\R\times M^*$ for which the genericity condition from \cite{ht0} holds only for those orbit sets with action less than $L$. Such almost complex structures are called {\it $\textit{ECH}^L$-generic}. Then define the \emph{filtered ECH} subcomplex
\[ \ecc^L(M,\Gamma,\alpha,J) \subset \ecc(M,\Gamma,\alpha,J) \]
to be generated by all admissible orbit sets of action strictly less than $L$.  Its homology is denoted $\ech^L(M,\Gamma,\alpha,J)$. 
\end{definition}

The fact that $\ecc^L$ is a subcomplex follows from the fact that if there is a $J$-holomorphic curve other than a product cylinder from $\Theta_+$ to $\Theta_-$, then $\mathcal{A}_\alpha(\Theta_+) \geq \mathcal{A}_\alpha(\Theta_-)$ by Stokes's theorem, hence $\partial$ lowers the symplectic action.

Given any $L < L'$, there are also natural maps
\begin{equation}
\label{eq:i-LL-J}
i^{L,L'}_J: \ech^L(M,\Gamma,\alpha,J) \to \ech^{L'}(M,\Gamma,\alpha,J)
\end{equation}
induced by the inclusion $\ecc^L \hookrightarrow \ecc^{L'}$, whose direct limit as $L\to\infty$ is the unfiltered homology group $\ech(M,\Gamma,\alpha,J)$.  Hutchings and Taubes \cite[Theorem 1.3]{ht2} showed in the closed case that for any $J$ and $J'$ there are canonical isomorphisms
\[ \ech^L(Y,\lambda,J) \isomto \ech^L(Y,\lambda,J') \]
which compose naturally and commute with the various $i^{L,L'}_J$ and $i^{L,L'}_{J'}$.  Thus for closed contact 3-manifolds, filtered ECH defines a canonical group $\ech^L(Y,\lambda)$ for any $L$, together with canonical maps $i^{L,L'}: \ech^L(Y,\lambda) \to \ech^{L'}(Y,\lambda)$; these are independent of $J$ but depend on $\lambda$ since it is what we use to define the action functional $\mathcal{A}_\lambda$.  In Section \ref{sec:j-independence} we will prove similar results for filtered sutured ECH.

%% file: handles.tex

In this section we will discuss how to attach contact $1$-handles to a sutured contact manifold. A contact 1-handle is a 3-manifold with corners of the form $H\times [-1,1]$ where $H$ is topologically a disk, equipped with a contact form that is invariant under translations along the $[-1,1]$ factor, and it is attached to $M$ along neighborhoods of a pair of points in $\Gamma$ suitably to form a new sutured contact manifold.

\begin{theorem}
\label{thm:1-handle}
Let $(M,\Gamma, \alpha)$ be a sutured contact manifold, and choose distinct points $p,q \in \Gamma$.  Then it is possible to attach a contact 1-handle $H \times [-1,1]_t$ to $M$ along $p$ and $q$ to produce a new sutured contact manifold $(M',\Gamma',\alpha')$ such that
\begin{enumerate}\leftskip-0.25in
\item $M$ is a submanifold of $M'$, and $\alpha'|_M = \alpha$.
\item The contact form $\alpha'$ has the form $Cdt + \tilde{\beta}$ on $H\times [-1,1]$, where $\tilde{\beta}$ is a 2-form on $H$.
\end{enumerate}
In particular, the Reeb vector field is $\frac{1}{C}\partial_t$ on $H\times [-1,1]$, and so the Reeb orbits of $(M',\alpha')$ are canonically identified with those of $(M,\alpha)$.
\end{theorem}

\begin{proof}
We start by choosing a coordinate $\theta$ on each component of $\Gamma$ so that $p$ and $q$ have disjoint embedded neighborhoods of the form $N_{p},N_{q}\simeq[-1,1]_\theta\subset\Gamma$, with $p$ and $q$ at coordinate $0$ in their respective neighborhoods and such that $\beta_{0}|_{N_{p}\cup N_{q}}=Kd\theta$ for some constant $K$. Thus along $[-1,0]_\tau \times [-1,1]_t\times N_{p}$ and $[-1,0]_\tau \times [-1,1]_t\times N_{q}$ we can write 
\[ \alpha=Cdt+Ke^{\tau}d\theta. \]
Since the points $p,q$ have neighborhoods contained entirely within $U(\Gamma)$, where the contact structure is $t$-invariant, and since $(R_{+}(\Gamma),Ke^{\tau}d\theta)$ is a Liouville domain, it will suffice to describe how to attach a Weinstein $1$-handle $H$ \cite{weinstein} to $R_{+}(\Gamma)$ along the points $p_{+}=(0,1,p)$ and $q_{+}=(0,1,q)$ so that the resulting domain $(R_{+}(\Gamma)\cup H,\tilde{\beta})$ satisfies 
\[ \tilde{\beta}|_{R_{+}(\Gamma)}=\alpha|_{R_{+}(\Gamma)}=Ke^{\tau}d\theta. \]
Then the restriction of the new contact form $\alpha'$ to the 3-dimensional 1-handle $H\times[-1,1]_{t}$ will read $Cdt + \tilde{\beta}$, as promised.

We now describe a model of our $1$-handle. Fix a constant $L>\sqrt{K^{2}+2}$, let $A=\frac{L^{2}}{2}-1$, and consider the compact region $\Omega \subset \R^2$ defined by 
\begin{eqnarray*}
\frac{x^{2}}{2}-y^{2}\leq A & \mathrm{and} & -1\leq y\leq1
\end{eqnarray*}
with area form $\omega=dx\wedge dy$. This region has ``vertical'' boundary components defined by arcs of the hyperbola $\frac{x^{2}}{2}-y^{2}=A$ between $y=-1$ and $y=1$, and its ``horizontal'' boundary components are the line segments satisfying $y=\pm1$ and $-L\leq x\leq L$. The area form has a primitive $\beta=-2ydx-xdy$, for which the vector field $Y=-x\partial_{x}+2y\partial_{y}$ satisfies $\iota_{Y}d\beta=\iota_{Y}\omega=\beta$. In fact, this vector field is normal to the vertical part of $\partial \Omega$, where it points into $\Omega$, and it is transverse to the horizontal part of $\partial \Omega$ pointing outward. See Figure \ref{fig:model-handle}.

\begin{figure}[ht]
\labellist
\small \hair 2pt
\pinlabel $a$ at 30 50
\tiny
\pinlabel $x$ at 133 65
\pinlabel $y$ [B] at 70 127
\endlabellist
\centering
\includegraphics{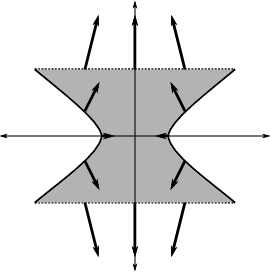}
\caption{A model 1-handle $\Omega$, bounded by $\frac{x^{2}}{2}-y^{2}\leq A$ and $-1\leq y\leq1$.}
\label{fig:model-handle}
\end{figure}

Next, we wish to glue the neighborhood $\{0\}_{\tau}\times\{1\}_t\times[-1,1]_{\theta}\subset \partial R_{+}(\Gamma)$ of $p_{+}$ to the leftmost vertical arc $a$ of $\partial \Omega$. We will use the diffeomorphism $f:[-1,1]\to a$ defined by 
\[ f(\theta)=(-\sqrt{2(\theta^{2}+A)},\theta), \]
which satisfies 
\[ f^{*}(\beta|_a)=f^{*}(-2ydx-xdy)=\frac{6\theta^{2}+2A}{\sqrt{2(\theta^{2}+A)}}d\theta. \]
Letting $g(\theta)=\frac{6\theta^{2}+2A}{\sqrt{2(\theta^{2}+A)}}$, we see that $g(\theta)\geq\sqrt{2(\theta^{2}+A)}>K$ for all $\theta$ as long as $A>\frac{K^{2}}{2}$, or equivalently $L>\sqrt{K^{2}+2}$.

In the symplectization $([0,\infty)_\tau \times\{1\}_t\times [-1,1]_\theta, Ke^\tau d\tau \wedge d\theta)$, the graph $\Gamma$ of $\ln(g(\theta)/K)$ is a contact-type hypersurface with contact form equal to $f^*(\beta|_a)$.  Thus it has a small collar neighborhood defined by 
\[ \left|\tau - \ln\left(\frac{g(\theta)}{K}\right)\right| < \epsilon \]
which is canonically identified with the symplectization $((-\epsilon, \epsilon)_\tau \times a, d(e^\tau\cdot \beta|_a))$.  On the other hand, the region of the plane defined by the time $< \epsilon$ flow of $Y$ from $a$ is also symplectomorphic to $([0,\epsilon)_\tau \times a, d(e^\tau \cdot \beta|_a))$, identifying $Y$ with $\partial_\tau$ since $Y$ is a symplectic vector field transverse to $a$, and so we can glue the region $S$ of $[0,\infty)_\tau \times\{1\}_t\times [-1,1]_\theta$ defined by 
\[ 0 \leq \tau \leq \ln\left(\frac{g(\theta)}{K}\right) \]
to $\Omega$ along their boundaries, identifying $\Gamma$ with $a$ and $\partial_\tau$ with $Y$.  This produces the glued-up symplectic manifold and 1-form
\[ (R_{+}(\Gamma)\cup_{\tau=0}S\cup_{\tau=\ln(g/K)}\Omega,\tilde{\beta}) \]
as seen in Figure \ref{fig:1-handle-gluing}. By construction, the $1$-form $\tilde{\beta}$ agrees with $Ke^{\tau}d\theta$ and $\beta$ on $R_{+}(\Gamma)$ and $\Omega$, respectively.

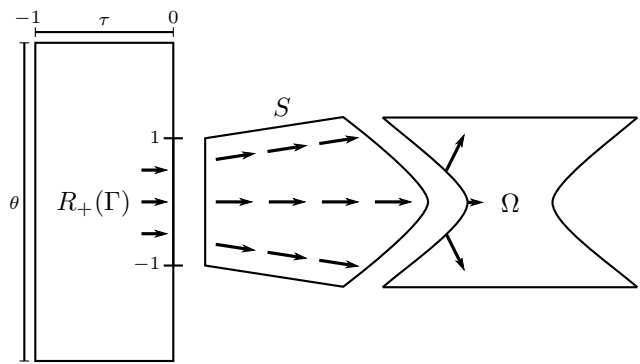
\begin{figure}[ht]
\labellist
\small \hair 2pt
\pinlabel $R_+(\Gamma)$ at 28 60
\pinlabel $S$ [B] at 99 93
\pinlabel $\Omega$ at 185 60
\tiny
\pinlabel $\theta$ [r] at 2 60
\pinlabel $-1$ [r] at 55 37
\pinlabel $1$ [r] at 55 85
\pinlabel $\tau$ [B] at 32 127
\pinlabel $-1$ [B] at 3 128
\pinlabel $0$ [B] at 58 128
\endlabellist
\centering
\includegraphics{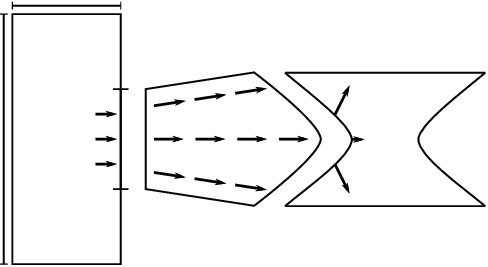}
\caption{The surfaces $R_{+}(\Gamma)$ and $\Omega$ are glued together by inserting a piece $S$ of the symplectization $[0,\infty)_{\tau}\times[-1,1]_{\theta}$.}
\label{fig:1-handle-gluing}
\end{figure}

If we repeat the same procedure to attach a neighborhood of $q_{+}$ to the rightmost vertical arc of $\partial \Omega$, the result is an ordered pair $(R',\tilde{\beta})$, where $R'$ is topologically the union of $R_{+}(\Gamma)$ and a 1-handle, and $\tilde{\beta}$ is a primitive of the area form constructed on $R'$. The vector field $\tilde{Y}$ satisfying $\iota_{\tilde{Y}}d\tilde{\beta}=\tilde{\beta}$ is equal to $\partial_{\tau}$ on $\partial R'\cap\partial R_{+}(\Gamma)$, points outward along $\mathrm{int}(\partial R'\cap\partial \Omega)$, and is tangent to $\partial R'$ along $\partial R'\cap\partial S$. In order to produce a Liouville domain, however, we need $\partial R'$ to be smooth and $\tilde{Y}$ to point out along $\partial R'$, so we achieve this by carving out the actual 1-handle from $R'$. Specifically, we take $R'\cap \Omega$ to consist of the region $|y|\leq\frac{1}{2}$. We then extend this boundary curve $y=\frac{1}{2}$ smoothly along the $S$ regions as the graph of a monotonic function of $\tau$ so that it is always transverse to $\tilde{Y}=\partial_{\tau}$ and extends smoothly to the segment $\{0\}_{\tau}\times\{1\}_t\times[1,1+\epsilon]_{\theta}$ of $\partial R_{+}(\Gamma)$. We reflect this curve across the lines $y=0$ on $\Omega$ and $\theta=0$ on $S$ to get an analogous curve containing the line $y=-\frac{1}{2}$, and then we declare these two curves to be the boundary of the 1-handle $H$; see Figure \ref{fig:1-handle-glued}.

\begin{figure}[ht]
\labellist
\small \hair 2pt
\pinlabel $R_+(\Gamma)$ at 27 60
\pinlabel $R_+(\Gamma)$ at 281 60
\pinlabel $H$ at 154 60
\endlabellist
\centering
\includegraphics{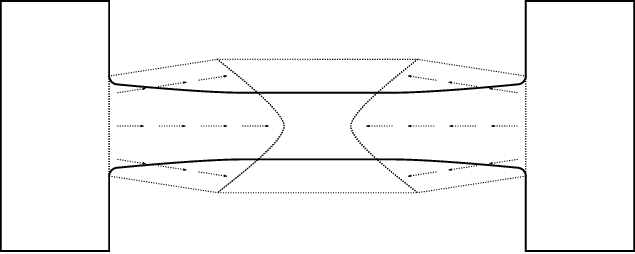}
\caption{The Weinstein 1-handle $H$ glued to $R_{+}(\Gamma)$.}
\label{fig:1-handle-glued}
\end{figure}

Finally, now that we have glued a Weinstein 1-handle to $R_{+}(\Gamma)$ to get a new Liouville domain $(R_{+}(\Gamma'),\tilde{\beta})$ with Liouville vector field $\tilde{Y}$, we need to check that we can find a collar neighborhood $[-1,0]_{\tau}\times\partial R_{+}(\Gamma')$ of the boundary on which $\partial_{\tau}=\tilde{Y}$. This is automatically satisfied at all points except possibly those which belonged to the model $\Omega$, since $\tilde{Y}$ was equal to the original vector field $\partial_{\tau}$ everyhere else. On $\Omega$, however, we have $\tilde{Y}=-x\partial_{x}+2y\partial_{y}$, and so the time-$t$ flow of $-\tilde{Y}$ from a point $(x,\pm\frac{1}{2})$ is $(xe^{t},\pm\frac{1}{2}e^{-2t})$. It is easy to check that these flows are all disjoint for $0\leq t\leq1$, as desired. Note that this collar neighborhood is contained entirely within the union of $H$ and $[-1,0]_{\tau}\times\partial R_{+}(\Gamma)$.

We can now define the tuple $(M',\Gamma',\alpha')$. We let $M'$ be constructed by attaching the handle 
\[ H \times [-1,1]_t =\overline{R_{+}(\Gamma')\smallsetminus R_{+}(\Gamma)}\times[-1,1]_t \]
to $M$ in the obvious way; its contact form $\alpha'$ is taken to be $\alpha$ on $M$, hence $Cdt + e^\tau \beta_0$ on $U(\Gamma)$, and so it extends over $H \times [-1,1]$ as $Cdt + \tilde{\beta}$.  This is shown in Figure \ref{fig:1-handle-mprime}, with the new sutures $\Gamma'$ in red.
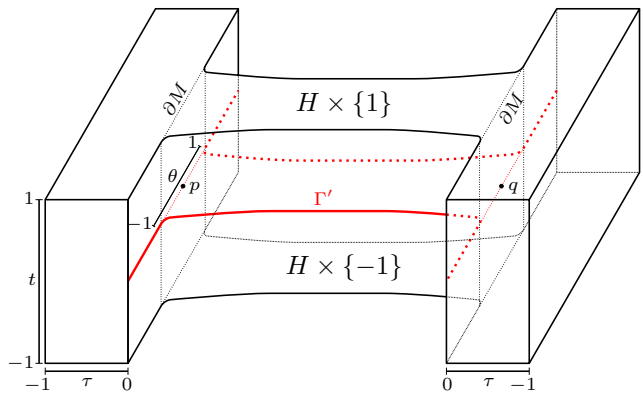
\begin{figure}[ht]
\labellist
\small \hair 2pt
\pinlabel $H\times\{1\}$ at 155 136
\pinlabel $H\times\{-1\}$ at 155 54
\tiny
\pinlabel $\rotatebox{61}{$\partial{M}$}$ at 70 139
\pinlabel $\rotatebox{61}{$\partial{M}$}$ at 239 134
\pinlabel $t$ [r] at 3 48
\pinlabel $1$ [r] at 2 89
\pinlabel $-1$ [r] at 2 6
\pinlabel $\tau$ [t] at 26 1
\pinlabel $-1$ [t] at 1 1
\pinlabel $0$ [t] at 46 1
\pinlabel $p$ at 80 94
\pinlabel $q$ at 240 94
\pinlabel $\theta$ at 69 100
\pinlabel $-1$ at 52 76
\pinlabel $1$ at 79 117
\pinlabel $\tau$ [t] at 228 1
\pinlabel $0$ [t] at 207 1
\pinlabel $-1$ [t] at 245 1
\color{red}
\pinlabel $\Gamma'$ at 145 89
\color{black}
\endlabellist
\centering
\includegraphics[scale=0.75]{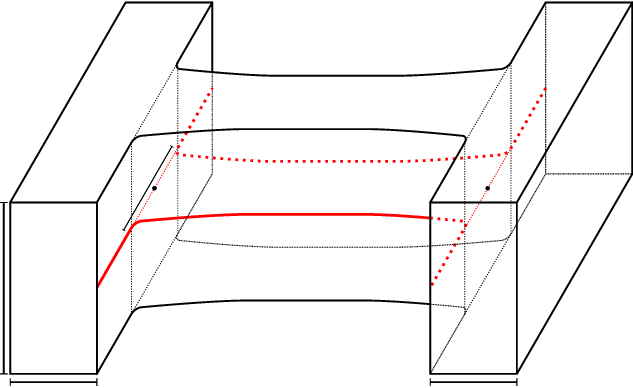}
\caption{The new sutured manifold $(M',\Gamma')$. The dotted lines on the interior show the boundary and sutures of the original $(M,\Gamma)$.}
\label{fig:1-handle-mprime}
\end{figure}
We define $U(\Gamma') = [-1,0]_\tau \times [-1,1]_t \times \Gamma'$ by identifying $\tau=0$ with the closure of $\partial M' \ssm (R_+(\Gamma') \cup R_-(\Gamma'))$, and then extending this to the rest of $-1 \leq \tau \leq 0$ by the flow of $\tilde{Y}$.  (Although $\tilde{Y}$ was only defined on $R_{+}(\Gamma')$, the neighborhood $U(\Gamma')$ is still well-defined because $U(\Gamma')$ is contained within $U(\Gamma)\cup (H\times [-1,1])$, which is by construction identified as a region of $R_{+}(\Gamma')$ times $[-1,1]_{t}$.) It is straightforward to check that $(M',\Gamma',\alpha')$ satisfies the hypotheses of Theorem \ref{thm:1-handle}, so the proof is complete.
\end{proof}

%% file: flexibility.tex

Suppose that $(M,\Gamma,\alpha)$ is a sutured contact manifold whose horizontal boundary regions $R_+(\Gamma)$ and $R_-(\Gamma)$ are diffeomorphic rel boundary.  We may wish to glue $R_+(\Gamma)$ to $R_-(\Gamma)$ in a way which preserves not only the contact structure but also the contact form, and so it would be useful to know that the Liouville forms $\beta_\pm = \alpha|_{R_\pm(\Gamma)}$ are equal under some diffeomorphism.  We will show in this subsection that this can be done without changing the $\ech$ chain complex or even the symplectic action functional.

First, we note that $\alpha$ has the form $Cdt + \beta_\pm$ in collar neighborhoods of the surfaces $R_\pm(\Gamma)$, and likewise $Cdt + \beta$ on the neighborhood $U(\Gamma)$ of $\Gamma$, and it may be convenient to replace $C$ with a larger constant.  We now explain how to do this.

By definition, there exists a parametrization $\psi:(-1,0]_\tau\times[-1,1]_{t}\times\Gamma\rightarrow U(\Gamma)$, and neighborhoods of $R_+(\Gamma)$ and $R_-(\Gamma)$ respectively diffeomorphic to $(1-\epsilon_+,1]\times R_+(\Gamma)$ and $[-1,-1+\epsilon_-)\times R_-(\Gamma)$ for some $\epsilon_-,\epsilon_+ \ll 1$. The contact form $\alpha$ restricts to $U(\Gamma)$ and each of the latter neighborhoods respectively as $Cdt+\beta$, $Cdt+\beta_+$, and $Cdt+\beta_-$ for some $C$. Now, let $t_0>0$ and define 
\[ M_{t_0}=M\cup_{\{1\}\times R_+(\Gamma)}[1,1+2t_0]_t\times R_+(\Gamma). \]
The latter is a sutured contact manifold with sutures $\Gamma_{t_0}=\{t_0\}\times\Gamma$ in a neighborhood
\[ U(\Gamma_{t_0})=U(\Gamma)\bigcup_{(-1,0]_\tau\times\{1\}\times\Gamma}(-1,0]_\tau\times[1,1+2t_0]_t\times\Gamma, \]
and adapted contact form $\alpha_{t_0}$ extending $\alpha$ over $[1,1+t_0]_t\times R_+(\Gamma)$ as $Cdt+\beta_+$. In these coordinates we have $R_+(\Gamma_{t_0})=\{1+2t_0\}\times R_+(\Gamma)$ and $R_-(\Gamma_{t_0})=R_-(\Gamma)$. We may parametrize $U(\Gamma_{t_0})$ via 
\[ \Psi:(-1,0]_\tau\times[-1,1]_{t'}\times\Gamma\rightarrow U(\Gamma_{t_0}) \]
such that $\Psi(\tau,t',\theta)=(\tau,(1+t_0)t'+t_0,\theta)$. With respect to this parametrization, we compute
\[ \alpha_{t_0}|_{U(\Gamma_{t_0})}=C(1+t_0)dt'+\beta, \]
and there exist collar neighborhoods of $R_+(\Gamma_{t_0})$ and $R_-(\Gamma_{t_0})$ respectively diffeomorphic to $(1-\frac{\epsilon_+ + 2t_0}{1+t_0},1] \times R_+(\Gamma_{t_0})$ and $[-1,-1+\frac{\epsilon_-}{1+t_0}) \times R_-(\Gamma_{t_0})$ in this parametrization. Note that $\epsilon_+<\frac{\epsilon_+ + 2t_0}{1+t_0}$ for any $t_0>0$. On these neighborhoods, $\alpha_{t_0}$ reads $C(1+t_0)dt'+\beta_+$ and $C(1+t_0)dt'+\beta_-$, respectively. 

Now consider the completions $(M^\ast,\alpha^\ast)$ and $(M^\ast_{t_0},\alpha^\ast_{t_0})$ of the sutured contact 3-manifolds $(M,\Gamma,\alpha)$ and $(M_{t_0},\Gamma_{t_0},\alpha_{t_0})$, respectively, as in \cite[Section 2.4]{cghh}. The completions $(M^\ast,\alpha^\ast)$ and $(M^\ast_{t_0},\alpha^\ast_{t_0})$ are identical. Therefore, an almost complex structure $J$ tailored to $(M^\ast,\alpha^\ast)$ is also tailored to $(M^\ast_{t_0},\alpha^\ast_{t_0})$ and vice versa, and so we have proved the following.
\begin{lemma}
\label{lem:scaling}
There exists a canonical isomorphism $\ech(M,\Gamma,\alpha,J) \cong \ech(M_{t_0},\Gamma_{t_0},\alpha_{t_0},J)$.  Moreover, since any closed Reeb orbits in $(M_{t_0},\alpha_{t_0})$ actually lie in $(M,\alpha = \alpha_{t_0}|_M)$, the action functionals $\mathcal{A}_\alpha$ and $\mathcal{A}_{\alpha_{t_0}}$ are identical, and so this canonical isomorphism preserves the action filtration on each homology group.
\end{lemma}

\begin{remark}
Note that the above construction and lemma would hold if we replaced $R_+(\Gamma)$ with $R_-(\Gamma)$.  Also, this construction does not require that $R_+(\Gamma)$ be diffeomorphic to $R_-(\Gamma)$, but rather works for any sutured contact manifold.
\end{remark}

Thus we can replace the constant $C$ appearing in $\alpha$ with any $C'>C$ by taking $t_0 = \frac{C'-C}{C}$.  With the above understood, we now turn to the question of whether we can find a diffeomorphism between the two Liouville manifolds $(R_+(\Gamma),\beta_+)$ and $(R_-(\Gamma),\beta_-)$.
 
\begin{lemma}
\label{lem:glue-liouville}
Let $(M,\Gamma,\alpha)$ be a sutured contact manifold with completion $(M^*, \alpha^*)$ and tailored almost complex structure $J$ on $\R \times M^*$, and suppose that there is a diffeomorphism $\psi: R_+(\Gamma) \to R_-(\Gamma)$ which fixes $U(\Gamma) = [-1,0]_\tau \times [-1,1]_t \times \Gamma$ in the sense that
\[ \psi(\tau,1,x) = (\tau,-1,x). \]
Then there is a sutured contact manifold $(M',\Gamma',\alpha')$ and tailored almost complex structure $J'$ on $\R\times (M')^*$ such that:
\begin{enumerate}\leftskip-0.25in
\item \label{item:glue-liouville-cond1} There exists a diffeomorphism $(M,\Gamma, U(\Gamma)) \xrightarrow{\sim} (M',\Gamma',U(\Gamma'))$ such that $\alpha$ and $\alpha'$ agree away from collar neighborhoods of $\partial M$ and $\partial M'$.  In particular, the closed Reeb orbits of $(M,\Gamma)$ are the same with respect to either contact form, and the action functionals $\mathcal{A}_\alpha$ and $\mathcal{A}_{\alpha'}$ are equal.
\item \label{item:glue-liouville-cond2} There is a diffeomorphism $f: R_+(\Gamma') \to R_-(\Gamma')$ such that $f^*(\beta'_-) = \beta'_+$, and $f$ is isotopic to $\psi$ under the identifications $R_+(\Gamma') \cong R_+(\Gamma)$ and $R_-(\Gamma') \cong R_-(\Gamma)$ (which are well-defined up to isotopy).
\item There is a canonical isomorphism $\ecc(M,\Gamma,\alpha,J) \cong \ecc(M',\Gamma',\alpha',J')$, defined by sending an orbit set $\Theta \in \ecc(M,\Gamma,\alpha,J)$ to itself.
\end{enumerate}
\end{lemma}

\begin{proof}
We start by identifying small neighborhoods of $R_+(\Gamma)$ and $R_-(\Gamma)$ with $[1-\epsilon,1]\times R_+(\Gamma)$ and $[-1,-1+\epsilon] \times R_-(\Gamma)$, respectively.  On these neighborhoods, $\alpha$ has the form
\[ \alpha = Cdt + \beta_\pm \]
where $\beta_\pm$ is a Liouville form on $R_\pm(\Gamma)$ which restricts to $e^\tau \beta_0$ on $R_\pm(\Gamma)\cap U(\Gamma)$.  We then glue the contact manifolds
\begin{eqnarray*}
([1,\infty) \times R_+(\Gamma), Cdt + \beta_+) & \mathrm{and} & ((-\infty,-1] \times R_-(\Gamma), Cdt + \beta_-)
\end{eqnarray*}
to $M$ along $R_+(\Gamma)$ and $R_-(\Gamma)$ respectively to form the vertical completion $(M_v,\alpha_v)$ as in Definition \ref{def:completion}.

Given the diffeomorphism $\psi: R_+(\Gamma) \to R_-(\Gamma)$, we know that $\psi^*\beta_-$ and $\beta_+$ are identical along $R_+(\Gamma) \cap U(\Gamma)$. We will construct an isotopy $\phi_\lambda: R_+(\Gamma) \hookrightarrow (1-\frac{\epsilon}{2}, 1+\frac{\epsilon}{2}) \times R_+(\Gamma)$ such that 
\begin{itemize}\leftskip-0.35in
\item $\phi_0$ is the embedding $x \mapsto (1,x)$.
\item Each $\phi_\lambda(R_+(\Gamma))$ is transverse to the Reeb vector field $\frac{1}{C}\partial_t$.
\item For all $x \in R_+(\Gamma) \cap U(\Gamma)$ and all $\lambda\in[0,1]$, we have $\phi_\lambda(x) = (1,x)$.
\item $\psi^\ast\beta_-=\phi_1^\ast\alpha_v$.
\end{itemize}

Take the contact 1-forms $\alpha_v=Cdt+\beta_+$ and $\alpha_v'=Cdt+\psi^\ast\beta_-$ on $(1-\frac{\epsilon}{2}, 1+\frac{\epsilon}{2}) \times R_+(\Gamma)$. For any $\lambda\in[0,1]$ the 1-form $\alpha_\lambda=\lambda\alpha_v+(1-\lambda)\alpha_v'=Cdt+[\lambda\beta_++(1-\lambda)\psi^\ast\beta_-]$ is contact on $(1-\frac{\epsilon}{2}, 1+\frac{\epsilon}{2}) \times R_+(\Gamma)$ since each $\beta_\lambda=\lambda\beta_++(1-\lambda)\psi^\ast\beta_-$ is a Liouville form on $R_+(\Gamma)$. Now we will apply the Moser trick to the family of contact 1-forms $\alpha_\lambda$ to construct a vector field $v_\lambda=h\partial_t+u_\lambda$ such that 
\begin{equation*}
\mathcal{L}_{v_\lambda}\alpha_\lambda+\frac{d\alpha_\lambda}{d\lambda}=0.
\end{equation*}
If $\tilde{\phi}_\lambda$ is the flow of $v_\lambda$, then it will follow from the identity $\frac{d}{d\lambda}\left(\tilde{\phi}^{\ast}_\lambda(\alpha_\lambda)\right) = \tilde{\phi}^{\ast}_\lambda(\mathcal{L}_{v_\lambda} \alpha_\lambda + \dot{\alpha}_\lambda)$ that $\tilde{\phi}_\lambda^*(\alpha_\lambda)$ is constant, and hence $\tilde{\phi}_1^*(\alpha_v) = \tilde{\phi}_1^*(\alpha_1)$ is equal to $\tilde{\phi}_0^*(\alpha_0) = \alpha'_v$.  Note that
\begin{eqnarray*}
\mathcal{L}_{v_\lambda}\alpha_\lambda&=&\iota_{v_\lambda}d\alpha_\lambda+d\iota_{v_\lambda}\alpha_\lambda\\
&=&\iota_{u_\lambda}d\beta_\lambda+d(Ch+\iota_{u_\lambda}\beta_\lambda)\\
\frac{d\alpha_\lambda}{d\lambda}&=&\beta_+-\psi^\ast\beta_-.
\end{eqnarray*}
Hence, solving the equation
\[ \iota_{u_\lambda}d\beta_\lambda = \psi^\ast\beta_--\beta_+ \]
for the vector field $u_\lambda$, we obtain the desired vector field as $v_\lambda=-\frac{\iota_{u_\lambda}\beta_\lambda}{C}\partial_t+u_\lambda$, which is $t$-independent. In light of Lemma \ref{lem:scaling}, assume that $C$ is sufficiently large so that this vector field induces an isotopy 
\[ \tilde{\phi}_\lambda:\left(1-\frac{\epsilon}{4}, 1+\frac{\epsilon}{4}\right) \times R_+(\Gamma)\hookrightarrow \left(1-\frac{\epsilon}{2}, 1+\frac{\epsilon}{2}\right) \times R_+(\Gamma) \]
satisfying $\tilde{\phi}_1^\ast\alpha_v=\alpha_v'$. Note also that since $\alpha_v=\alpha_v'$ on $(1-\frac{\epsilon}{4}, 1+\frac{\epsilon}{4}) \times (R_+(\Gamma)\cap U(\Gamma))$, the isotopy is the identity over this region. Then $\phi_\lambda:=\tilde{\phi}_\lambda|_{\{1\}\times R_+(\Gamma)}$ satisfies the properties listed above.

By the isotopy extension theorem, we can extend $\phi_\lambda$ to an isotopy of all of $M_v$ which is supported on $(1-\epsilon,1+\epsilon) \times R_+(\Gamma)$ and fixes $U(\Gamma)$ pointwise. Now let $M'$ denote the submanifold $\phi_1(M)$ of $M_v$. This is a sutured contact manifold with sutures $\Gamma' = \Gamma$, $R_-(\Gamma')=R_-(\Gamma)$ and $R_+(\Gamma')=\phi_1(R_+(\Gamma))$, and contact form $\alpha' = \alpha_v|_{M'}$. On $M\ssm((1-\epsilon, 1] \times R_+(\Gamma))$ we have $\phi_1^*(\alpha_v) = \alpha$, thus
\[ \phi_1|_M: (M,\Gamma,U(\Gamma)) \to (M',\Gamma',U(\Gamma')) \]
is a diffeomorphism satisfying condition \eqref{item:glue-liouville-cond1}. 
Regarding $\phi_1|_{R_+(\Gamma)}$ as a diffeomorphism from $R_+(\Gamma)$ to $R_+(\Gamma')$, the diffeomorphism
$$f = \psi \circ (\phi_1|_{R_+(\Gamma)})^{-1}: R_+(\Gamma') \to R_-(\Gamma)$$
satisfies $f^\ast\beta_-=\alpha_v|_{R_+(\Gamma')}=:\beta'_+$, and it is isotopic to $\psi$ by construction. This proves condition \eqref{item:glue-liouville-cond2}.

Finally, if we complete $(M',\Gamma',\alpha')$ vertically then it is clear that the resulting manifold is canonically $(M_v, \alpha_v)$.  Thus when we extend horizontally as well, the resulting completed manifold $(M'^*,\alpha'^*)$ is identical to the completion $(M^*,\alpha^*)$ of $(M,\Gamma,\alpha)$.  Since $R_+(\Gamma')$ is a smooth perturbation of $R_+(\Gamma)$ in $M_v$, it is easy to check that $J' = J$ is tailored to $(M^*,\alpha^*)$ if we take $\epsilon$, and hence the perturbation, sufficiently small.  We conclude that
\[ \ech(M,\Gamma,\alpha, J) = \ech(M^*,\alpha^*, J) = \ech((M')^*,(\alpha')^*, J') = \ech(M',\Gamma',\alpha',J'), \]
as desired.
\end{proof}

%% file: closed-embedding.tex

We wish to imitate the neck-stretching arguments of \cite{cghh} to show that sutured ECH does not depend on a choice of almost complex structure or contact form, and that gluing contact 1-handles to $(M,\Gamma,\alpha)$ preserves ECH up to isomorphism.  In order to do so, we will need to use filtered ECH and define suitable continuation maps (cf.\ \cite{ht2}), and the latter will require us to work with closed contact manifolds.

\begin{lemma}
\label{lem:glue-pm-filtered}
Let $(M',\Gamma',\alpha')$ be a sutured contact manifold with connected suture $\Gamma'$ such that $\alpha' = Cdt + \beta'_\pm$ on neighborhoods of $R_\pm(\Gamma')$.  Suppose that there is a diffeomorphism
\[ f: R_+(\Gamma') \to R_-(\Gamma') \]
sending $[-1,0]_\tau \times \{1\}_t \times \Gamma' \subset R_+(\Gamma') \cap U(\Gamma')$ to $[-1,0]_\tau \times \{-1\}_t \times \Gamma' \subset R_-(\Gamma') \cap U(\Gamma')$ by the map $(\tau,1,x) \mapsto (\tau,-1,x)$, and that $f^*(\beta'_-) = \beta'_+$.  Then for each $N>0$ and $n>\frac{N}{2C}$, there is a contact manifold $(M'_n,\alpha'_n)$ and an embedding
\[ (M',\alpha') \hookrightarrow (M'_n,\alpha'_n) \]
such that $\alpha'_n|_{M'} = \alpha'$ and
\begin{enumerate}\leftskip-0.25in
\item \label{cond:glue-pm-filtered-1} $\partial M'_n$ is a pre-Lagrangian torus foliated by Reeb orbits of action at least $N$;
\item \label{cond:glue-pm-filtered-2} $M' \cap \partial M'_n = \{0\} \times [-1,1] \times \Gamma' \subset U(\Gamma')$;
\item \label{cond:glue-pm-filtered-3} if $\gamma \subset M'_n$ is a closed Reeb orbit of action less than $N$ for $\alpha'_n$, then $\gamma \subset M'$.
\end{enumerate}
\end{lemma}

\begin{proof}
Let $M'_n$ be the 3-manifold
\[ \left( M' \cup ([-n,n]_t \times R_+(\Gamma')) \right) / {\sim}, \]
in which we glue $R_+(\Gamma') \subset M'$ to $\{-n\} \times R_+(\Gamma')$ by the identity map and $\{n\} \times R_+(\Gamma')$ to $R_-(\Gamma') \subset M'$ by $f$.  This admits a contact form $\alpha'_n$ which restricts to $M'$ and to $[-n,n]\times R_+(\Gamma')$ as $\alpha'$ and $Cdt + \beta'_+$, respectively.  The Reeb vector field on $[-n,n]\times R_+(\Gamma')$ is $\frac{1}{C}\partial_t$, so any closed Reeb orbit $\gamma$ of $(M'_n,\alpha'_n)$ which passes through some point $(t_0, x)$ of $[-n,n]\times R_+(\Gamma')$ must contain all of $[-n,n] \times \{x\}$, and in particular
\[ \int_{\gamma} \alpha'_n \geq \int_{[-n,n]\times\{x\}} Cdt + \beta'_+ = 2Cn. \]
If we take $n > \frac{N}{2C}$ then it follows that $\mathcal{A}_{\alpha'_n}(\gamma) > N$.

The boundary of the contact manifold $(M'_n,\alpha'_n)$ is now a pre-Lagrangian torus diffeomorphic to $\{0\}_\tau \times S^1_t \times \Gamma'$, where $S^1_t$ is formed by identifying the endpoints of the $[-1,1]$-factor of $U(\Gamma')$ with the endpoints of the $[-n,n]$-factor of $[-n,n] \times R_+(\Gamma')$.  The contact form $\alpha'_n$ is equal to $Cdt + e^\tau \beta_0$ on the collar neighborhood $[-1,0]_\tau \times S^1_t \times \Gamma'$ of $\partial M'_n$, where $\beta_0$ is a volume form on $\Gamma'$.  The boundary of $M'_n$ is foliated by closed orbits of the Reeb vector field $\frac{1}{C}\partial_t$ with action at least $N$, and any Reeb orbit of action at most $N$ lies on the interior of $M' \subset M'_n$, so $(M'_n,\alpha'_n)$ satisfies all of the desired conditions.
\end{proof}

\begin{proposition}
\label{prop:embed-filtered}
Let $(M'',\alpha'')$ be a contact manifold whose boundary is a pre-Lagrangian torus with collar neighborhood diffeomorphic to
\[ ([-1,0]_\tau \times S^1_t \times \Gamma, Cdt + e^\tau \beta_0), \]
where $\beta_0$ is a volume form on $\Gamma \cong S^1$, and $\partial M''$ is identified with the $\{\tau=0\}$ locus. If each Reeb orbit on the boundary has action greater than $2N$, then there is a closed contact manifold $(Y,\lambda)$ and an embedding
\[ (M'',\alpha'') \hookrightarrow (Y,\lambda) \]
such that $\lambda|_{M''} = \alpha''$ and every closed Reeb orbit $\gamma \subset (Y,\lambda)$ with $A_{\lambda}(\gamma) < N$ lies in $M''$.
\end{proposition}

\begin{proof}
We will form $(Y,\lambda)$ by gluing a solid torus to $\partial M''$ with an appropriate contact form, following Thurston--Winkelnkemper \cite{tw}.  Suppose for convenience that we have picked a coordinate $\theta$ on $\Gamma$ so that $\beta_0 = Kd\theta$ for some constant $K > 0$, and so $\alpha'' = Cdt+Ke^\tau d\theta$ on a neighborhood of $\partial M'$.  Let $l_t > \frac{2N}{C}$ denote the length of $S^1_t$ with respect to the $t$ coordinate, so that for example if $(M'',\alpha'')$ was constructed in the proof of Lemma \ref{lem:glue-pm-filtered} as $(M'_n,\alpha'_n)$ then we would have $l_t = 2n+2$, and likewise let $l_\theta$ denote the length of $\Gamma$ in the $\theta$ coordinate.

Consider the solid torus $\mathbb{T} = D^2_{(r,t)} \times S^1_\theta$, where $D^2$ has radius $\rho = \frac{l_t}{2\pi}$, $t$ is an angular coordinate measuring arclength along $\partial D^2$, and $S^1_\theta$ is identified with $\Gamma$.  We will glue $\partial \mathbb{T}$ to $\partial M''$ by the map
\[ (r=\rho, t, \theta) \mapsto (\tau=0, t, \theta), \]
identifying the coordinate $r$ on $\mathbb{T}$ with $\rho-\tau$ on $\partial M''$.  We thus wish to find a contact form
\[ \lambda_{\mathbb{T}} = f(r)d\theta + g(r)dt \]
on $\mathbb{T}$, satisfying the contact condition $\frac{fg' - gf'}{r} > 0$, such that 
\[ (f,g) = \begin{cases} (Ke^{\rho-r}, C) & \mathrm{for\ } r\geq\frac{3\rho}{4} \\ (A,r^2) & \mathrm{for\ } r\leq\frac{2\rho}{3} \end{cases} \]
for some constant $A > \frac{N}{l_\theta}$ which also satisfies $A > Ke^\rho$.
The Reeb vector field on $\mathbb{T}$ will be
\[ R_{\lambda_{\mathbb{T}}} = \frac{g' \partial_\theta - f'\partial_t}{fg' - gf'}. \]
Thus on the region where $(f,g) = (A,r^2)$, near $r=0$, we will have $R_{\lambda_{\mathbb{T}}} = \frac{1}{A}\partial_\theta$ and so the Reeb orbits will have the form $\gamma_{r_0,t_0} = (r_0,t_0) \times S^1_\theta$, with $A_{\lambda_\mathbb{T}}(\gamma_{r_0,t_0}) = Al_\theta > N$.  Similarly, on the region where $(f,g) = ((Ke^\rho)e^{-r},C)$, near $r=\rho$, we will have $R_{\lambda_{\mathbb{T}}} = \frac{1}{C}\partial_t$, so the Reeb orbits will have the form $\gamma_{r_0,\theta_0} = \{r = r_0\} \times \{\theta_0\}$ and action $2\pi r_0 \cdot C$.  By assumption we know that $2\pi\rho \cdot C > 2N$, or equivalently that $\pi\rho \cdot C > N$.

We now define $(f(r),g(r))$ as a smooth path satisfying $\frac{fg'-gf'}{r} > 0$ and the above boundary conditions near $r=0$ and $r=\rho$, and such that
\begin{itemize}\leftskip-0.35in
\item $f'(r) \leq 0$ and $g'(r) \geq 0$ for all $r$, and $f'(r) < 0$ whenever $r < A$.
\item $f'(r)$ and $g'(r)$ are never simultaneously $0$.
\item $f(r) = A$ for all $r \leq \frac{2\rho}{3}$, and $\frac{3C}{4} < g(\frac{2\rho}{3}) < C$.
\end{itemize}
See Figure \ref{fig:torus-contact} for a graph of $(f,g)$.
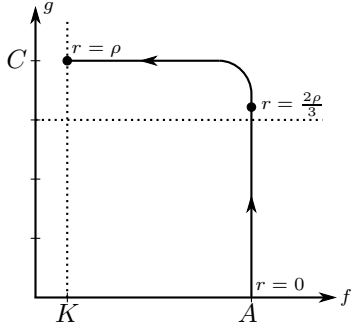
\begin{figure}[ht]
\labellist
\small \hair 2pt
\pinlabel $C$ at -5 92
\pinlabel $A$ at 82 -4
\pinlabel $K$ at 14 -4
\tiny
\pinlabel $r=0$ [l] at 83 7
\pinlabel $r=\frac{2\rho}{3}$ [l] at 85 75
\pinlabel $r=\rho$ at 25 95
\pinlabel $f$ at 119 2
\pinlabel $g$ at 7 111
\endlabellist
\centering
\includegraphics{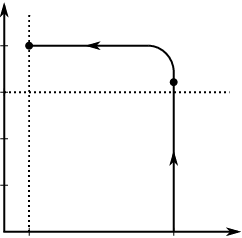}
\caption{A graph of the functions $(f(r),g(r))$ used to construct $\lambda$, $0 \leq r \leq \rho$.}
\label{fig:torus-contact}
\end{figure}
Each torus $r = r_0$ is foliated by Reeb orbits, and these are closed when $\frac{-f'}{g'} = \frac{2\pi r_0 m}{l_\theta n}$ for some relatively prime, nonnegative integers $m$ and $n$.  In this case each embedded Reeb orbit $\gamma$ has action
\[ \mathcal{A}_{\lambda_{\mathbb{T}}}(\gamma) = \int_\gamma f(r_0)d\theta + g(r_0)dt = f(r_0)\cdot l_\theta n + g(r_0)\cdot 2\pi r_0 m. \]
If $m \geq 1$ then $-f'(r_0) > 0$, so $r_0 > \frac{2\rho}{3}$ and $g(r_0) \geq g(\frac{2\rho}{3}) > \frac{3C}{4}$ and thus
\[ \mathcal{A}_{\lambda_{\mathbb{T}}}(\gamma) \geq g(r_0)\cdot 2\pi r_0 > \frac{3C}{4} \cdot 2\pi\left(\frac{2\rho}{3}\right) = \pi \rho \cdot C > N. \]
Otherwise $m=0$, hence $f'(r_0)=0$ and $n=1$ and so $\mathcal{A}_{\lambda_{\mathbb{T}}}(\gamma) \geq f(r_0)\cdot l_\theta = A l_\theta > N$.  We conclude that all closed Reeb orbits of $(\mathbb{T},\lambda_{\mathbb{T}})$ have action greater than $N$, and so
\[ (Y,\lambda) = (M'',\alpha'') \cup (\mathbb{T},\lambda_{\mathbb{T}}) \]
is the desired closed contact manifold.
\end{proof}

\begin{remark}
Although it will not be needed, we can strengthen Proposition \ref{prop:embed-filtered} by only requiring that the Reeb orbits along $\partial M''$ have action at least $(1+\epsilon)N$ for any fixed $\epsilon > 0$.  The use of $2N$ comes from the construction of the pair $(f(r), g(r))$, where we chose $f(r)=A$ for all $r \leq \frac{2}{3}\rho$ and $g\left(\frac{2}{3}\rho\right) > \frac{3}{4}C$, but if we replace the fractions $\frac{2}{3}$ and $\frac{3}{4}$ with $(1+\epsilon)^{-1/2}$ then the case $m \geq 1$ in the proof would become $\mathcal{A}_{\lambda_{\mathbb{T}}}(\gamma) > \frac{2\pi \rho C}{1+\epsilon}$, and $2\pi\rho C$ is exactly the action of those boundary orbits.
\end{remark}

Suppose we have a sutured contact manifold $(M',\Gamma',\alpha')$ with connected suture, and we choose a constant $L > 0$.  Then we can use Lemma \ref{lem:glue-pm-filtered} to embed $(M',\alpha')$ inside some $(M'',\alpha'') = (M'_n,\alpha'_n)$ such that $\partial M''$ is pre-Lagrangian and any Reeb orbit of $(M'',\alpha'')$ with action at most $2L$ is contained in $M'$.  By Proposition \ref{prop:embed-filtered}, since the Reeb orbits on $\partial M''$ have action greater than $2L$, we can then glue on a solid torus to get a closed contact manifold $(Y,\lambda)$ such that
\begin{itemize}\leftskip-0.35in
\item There is an embedding $M' \hookrightarrow Y$ such that $\lambda|_{M'} = \alpha'$;
\item Every Reeb orbit of $(Y,\lambda)$ of action at most $L$ lies inside $M'$.
\end{itemize}
We will make repeated use of this construction in Sections \ref{sec:j-independence} and \ref{sec:alpha-independence}.

%% file: j-independence.tex

Let $(M,\Gamma,\alpha)$ be a sutured contact manifold.  In this section we will prove that the embedded contact homology of $(M,\Gamma,\alpha)$ does not depend on a choice of almost complex structure.  More precisely, given generic tailored almost complex structures $J$ and $J'$ on $\R \times M^*$, where $M^*$ is the completion of $M$, we will construct isomorphisms on filtered $\ech$,
\[ \Phi^L_{J,J'}: \ech^L(M,\Gamma,\alpha,J) \isomto \ech^L(M,\Gamma,\alpha,J'), \]
which commute with the natural inclusion-induced maps $i^{L,L'}: \ech^L \to \ech^{L'}$ on either side and thus induce an isomorphism in the direct limit,
\[ \Phi_{J,J'}: \ech(M,\Gamma,\alpha,J) \isomto \ech(M,\Gamma,\alpha,J'). \]

We will assume throughout that there are no $\ech$ generators of action exactly equal to $L$ or $L'$ whenever these variables appear.  We also remark that we can treat $J$ and $J'$ as being defined only on $\R\times M_v$ where $M_v$ is the vertical completion of Definition \ref{def:completion}, since by \cite[Lemma 5.5]{cghh} no holomorphic curve appearing in the construction of $\ech$ will enter the rest of $\R \times M^*$.  

We begin by forming a new sutured contact manifold $(M',\Gamma',\alpha')$ from $(M,\Gamma,\alpha)$ by choosing pairs of points $(p_i,q_i) \in \Gamma$ and attaching a contact $1$-handle $H_i$ along each pair as in Theorem \ref{thm:1-handle}.  We insist on adding these $H_i$ in such a way that $\Gamma'$ is connected, and hence $R_+(\Gamma')$ is connected and diffeomorphic to $R_-(\Gamma')$.  For example, if $\Gamma$ has components $\Gamma_0, \dots, \Gamma_k$ for some $k \geq 0$, then we can pick distinct points $p_i \in \Gamma_0$ and $q_i \in \Gamma_i$ and attach $H_i$ to the pair $(p_i, q_i)$ for each $i=1,\dots,k$.  (If $\Gamma$ is connected then we do not need to attach any handles at all.)  Note that this procedure preserves both the collection of closed orbits and the action functional.  By Lemma \ref{lem:glue-liouville}, we can arrange without changing the completion $((M')^*, (\alpha')^*)$ that there is a diffeomorphism $f: R_+(\Gamma') \to R_-(\Gamma')$ such that $f^*(\beta'_-) = \beta'_+$, where $\beta'_\pm$ are the Liouville forms on $R_\pm(\Gamma')$.  Let $\calh$ denote the collection of 1-handles used to construct $M'$ from $M$ and $J_\calh$ denote any extension of $J$ to $\R \times (M')^*$.

\begin{lemma}
\label{lem:tau-stretching}
Let $\Theta_+$ and $\Theta_-$ be generators of $\ecc(M,\Gamma,\alpha,J)$.  Then
\[ \mathcal{M}_{I=1}(\R \times M^*, J; \Theta_+,\Theta_-) = \mathcal{M}_{I=1}(\R \times (M')^*, J_\calh; \Theta_+, \Theta_-). \]
\end{lemma}
\begin{proof}
Any curve $\mathcal{C} \in \mathcal{M}_{I=1}(\R\times M^*, J; \Theta_+,\Theta_-)$ corresponds to a unique $J_\calh$-holomorphic curve in $\R \times (M')^*$ in an obvious way, since as mentioned above we know from \cite[Lemma 5.5]{cghh} that no point of $\mathcal{C}$ projects into the region $\{\tau \geq 0\}$ of $M^*$.  Conversely, any curve in $\mathcal{M}_{I=1}(\R\times (M')^*, J_\calh; \Theta_+,\Theta_-)$ whose projection to $(M')^*$ avoids the handles $H_i$ corresponds to a unique $J$- holomorphic curve in $\R \times M^*$, so our goal is to show that these curves do not enter the handles.

Let $\mathcal{C} \in \mathcal{M}_{I=1}(\R \times (M')^*, J_\calh; \Theta_+, \Theta_-)$ be a curve whose projection enters one of the handles $H_i = H \times [-1,1]_t$, where the contact form is $\alpha' = Cdt+\beta$ for some $2$-form $\beta$ on $H$.  If we let $(\tau,t,\theta)$ denote the coordinates on $H_i \cap U(\Gamma')$ induced from $U(\Gamma') = [-1,0]_\tau \times [-1,1]_t \times \Gamma'$, then we can identify a surface $H^*$ by attaching a cylindrical end of the form $[0,\infty)_\tau \times (\Gamma' \cap \partial H)$ to $H$ along $\{\tau = 0\}$; the associated region $H^* \times \R_t \subset (M')^*$ is fibered by Reeb trajectories of the form $\{x\} \times \R_t$.  Fix a point $p_0 = (x_0,t_0) \in H \times [-1,1]$ which lies in the projection of $\mathcal{C}$ to $(M')^*$, and let $\gamma_0 = \{x_0\} \times \R$ be the Reeb trajectory through that point.  Then we can take any point $x_1 \in H^* \ssm H$ with positive $\tau$-coordinate, on the Reeb trajectory $\gamma_1 = \{x_1\} \times \R$, and let $P$ be a path from $p_0$ to $p_1$ inside $H^*$.  This determines a path $P \times \{t_0\} \subset (M')^*$ which is disjoint from $M$.

Since the asymptotic Reeb orbits of $\mathcal{C}$ all lie on the interior of $M \subset M'$, we can find constants $S,T > 0$ such that the $s$-- and $t$-coordinates of $\mathcal{C}$ on $\R_s \times ((M')^* \ssm M)$ have absolute values everywhere less than $S$ and $T$ respectively.  Let $A$ denote the $3$-chain $[-S,S] \times [-T,T] \times P \subset \R_s \times (M')^*$.  Then $\partial A$ is nullhomologous, so the intersection number $\mathcal{C}\cdot \partial A$ is zero.  However, $\mathcal{C}$ does not intersect $\partial A$ along $\partial([-S,S]\times[-T,T]) \times P$, so all of the points of intersection occur along $[-S,S] \times [-T,T] \times \partial P \subset \R_s \times \R_t \times \partial P$, and thus
\[ 0 = \mathcal{C} \cdot \partial A = \mathcal{C} \cdot (\R\times\gamma_1) - \mathcal{C} \cdot (\R\times\gamma_0). \]
Again we know that $\mathcal{C} \cdot (\R \times \gamma_1) = 0$ by \cite[Lemma 5.5]{cghh}, so it follows that $\mathcal{C} \cdot (\R \times \gamma_0) = 0$ as well.  But $\mathcal{C}$ and $\R \times \gamma_0$ are $J_\calh$-holomorphic curves, so all of their points of intersection are positive \cite{mcduff, micallef-white}, and since their intersection is nonempty we have a contradiction.
\end{proof}

Next, for each sufficiently large $n>0$, we embed $(M',\alpha')$ in a closed contact manifold $(Y_n,\alpha_n)$. Following Lemma \ref{lem:glue-pm-filtered}, let
\begin{equation}\label{eq:with-t-neck}
M'_n = \left(M' \cup \left([-n,n]_{t} \times R_+(\Gamma')\right)\right) / {\sim},
\end{equation}
where we identify $\{-n\} \times R_+(\Gamma')$ with $R_+(\Gamma') = \{1\} \times R_+(\Gamma') \subset M'$ by the identity map and $\{n\} \times R_+(\Gamma')$ with $R_-(\Gamma') = \{-1\} \times R_-(\Gamma')$ by the above diffeomorphism $f$.  We can then define a contact form $\alpha'_n$ on $M'_n$ by $\alpha'_n|_{M'} = \alpha'$ and $\alpha'_n|_{[-n,n]\times R_+(\Gamma')} = Cdt+\beta'_+$, where $\alpha' = Cdt+\beta'_+$ on a neighborhood $(1-\epsilon,1] \times R_+(\Gamma')$ of $R_+(\Gamma')$ in $M'$.  The condition $f^*(\beta'_-) = \beta'_+$ ensures that this is a well-defined contact form.  Let $J'_n$ be an almost complex structure on $\R \times M'_n$ which agrees with $J_\calh$ on $\R \times M'$ and is $t$-invariant on $[-n,-\epsilon')\times R_+(\Gamma')$ and on $(\epsilon',n]\times R_+(\Gamma')$ for some fixed $\epsilon' > 0$.

The manifold $(M'_n,\alpha'_n)$ admits a completion $(M'_{n,h},\alpha'_{n,h})$ in the sense of \cite[Section 8.1]{cghh}, obtained by attaching a semi-infinite horizontal end to $M'_n$ of the form $[0,\infty)_\tau\times\partial M'_n$.  In what follows, we will use $(M'_n,J'_n)$ when we actually mean $M'_{n,h}$ and some tailored extension of $J'_n$ over $\R \times M'_{n,h}$, but for the sake of discussing moduli spaces of holomorphic curves there is no difference, since by \cite[Lemma 5.5]{cghh} such curves can never enter the horizontal end.

\begin{lemma}
\label{lem:t-stretching}
Let $\Theta_+$ and $\Theta_-$ be generators of $\ecc(M',\Gamma',\alpha',J_\calh)$. Then for sufficiently large $n>0$, depending only on $\alpha'$, $J_\calh$, and the orbit sets $\Theta_\pm$, we have
$$\mathcal{M}_{I=1}(\mathbb{R}\times (M')^\ast,J_\calh;\Theta_+,\Theta_-) = \mathcal{M}_{I=1}(\mathbb{R}\times M'_n,J'_{n};\Theta_+,\Theta_-).$$
\end{lemma}
\begin{proof}
Our goal is to show that for sufficiently large $n>0$ there are no ECH index-1 $J'_{n}$-holomorphic curves in $\mathbb{R}\times M'_n$ with positive ends at $\Theta_+$ and negative ends at $\Theta_-$ that has non-empty intersection with $\mathbb{R}\times[-\epsilon',\epsilon']\times R_+(\Gamma')$.  We know from \cite[Proposition 5.20]{cghh} that there are uniform upper bounds on the absolute value of the $t$-coordinates of any curve in $\mathcal{M}_{I=1}(\R\times(M')^*,J_\calh;\Theta_+,\Theta_-)$, so it will then follow that there is a natural correspondence between ECH index-1 curves in $\R\times(M')^*$ and in $\R \times M'_n$ for $n$ sufficiently large both in the above sense and also compared to these uniform bounds.

Suppose to the contrary that there exists a strictly increasing sequence $\{n_i\}_{i\in\mathbb{N}}$ and $C_{n_i}\in\mathcal{M}_{I=1}(\mathbb{R}\times M'_{n_i}, J'_{n_i};\Theta_+,\Theta_-)$ that has non-empty intersection with $\mathbb{R}\times[-\epsilon',\epsilon']\times R_+(\Gamma')$. Without loss of generality, we may assume that either the projection of each $C_{n_i}$ onto $[-n_i,-\epsilon')$ is surjective or the projection of each $C_{n_i}$ onto $(\epsilon',n_i]$ is surjective. Since there is no essential difference between the two cases, we will consider the former.

With the preceding understood, let $C'_{n_i}=C_{n_i}\cap\mathbb{R}\times[-n_i,n_i]\times R_+(\Gamma')$. Since the $d(\alpha')^\ast$-energy of $C_{n_i}$ is uniformly bounded, so is the $d(\alpha')^\ast$-energy of $C'_{n_i}$. Then, by passing to a subsequence if necessary, there exist sequences of numbers $a_i\in\mathbb{R}$ and $b_i\in(-n_i,n_i)$, and unbounded sequences of strictly increasing positive numbers $s_i$ and $t_i$ such that 
\begin{itemize}\leftskip-0.35in
\item $[b_i-t_i,b_i+t_i]\subset(-n_i,-\epsilon')$,
\item $C'_{n_i}\cap[a_i-s_i,a_i+s_i]\times[b_i-t_i,b_i+t_i]\times R_+(\Gamma')\neq\emptyset$, and 
$$\int_{C'_{n_i}\cap[a_i-s_i,a_i+s_i]\times[b_i-t_i,b_i+t_i]\times R_+(\Gamma')}d(\alpha')^\ast\to0$$
\end{itemize}
as $i\to\infty$. The fact that $C_{n_i}$ surjects onto $[-n_i,-\epsilon')$ is what allows us to find suitable $t_i \to \infty$: given a value of $t$, we can find some interval of length $2t$ in which an arbitrarily small but nonzero fraction of the $d(\alpha')^*$-energy of $C_{n_i}$ is concentrated as long as $n_i$ is sufficiently large.  In particular, we can choose the $s_i \to \infty$ arbitrarily and the above properties will still hold. Now translate each $C'_{n_i}\cap[a_i-s_i,a_i+s_i]\times[b_i-t_i,b_i+t_i]\times R_+(\Gamma')$ so as to get curves $C''_{n_i}\subset[-s_i,s_i]\times[-t_i,t_i]\times R_+(\Gamma')$.  Note that these curves are holomorphic with respect to a single almost complex structure on $\R \times \R \times R_+(\Gamma')$, since before the translation the $C''_{n_i}$ all lived in regions $\R_s \times (-n_i,-\epsilon')_{t} \times \R_+(\Gamma')$ where the corresponding $J'_{n_i}$ is $s$- and $t$-invariant.

Since the curves $C''_{n_i}$ have uniformly bounded Hofer energy by \cite[Proposition 5.13]{behwz}, an application of Gromov compactness for holomorphic currents by Taubes \cite[Proposition 3.3]{taubes-compendium} shows that $\{C''_{n_i}\}_{i\in\mathbb{N}}$ admits a subsequence which converges weakly to a proper pseudo-holomorphic curve $\mathcal{C}$ in $\mathbb{R}_s\times\mathbb{R}_{t}\times R_+(\Gamma')$ on which $d(\alpha')^\ast$ vanishes. To be more explicit, let $\{K_j\}_{j\in\mathbb{N}}$ be an exhausting sequence of compact subsets of $\mathbb{R}\times\mathbb{R}\times R_+(\Gamma')$. Then for each $j\in\mathbb{N}$, $\{C''_{n_i}\cap K_j\}_{i\in\mathbb{N}}$ has a subsequence that converges to a pseudo-holomorphic curve in $K_j$, which in turn gives a subsequence that converges to a pseudo-holomorphic curve in $\mathbb{R}\times\mathbb{R}\times R_+(\Gamma')$ as $j\to\infty$. Since $\mathcal{C}$ has zero $d(\alpha')^*$-energy, it must be supported on $\R_s \times \gamma$ for some Reeb orbit $\gamma$; then we have $\mathcal{C}=\mathbb{R}\times\gamma$ by properness and the fact that holomorphic maps are open. Moreover, $\gamma$ cannot be closed since the Reeb vector field on $\mathbb{R}_{t}\times R_+(\Gamma')$ is given by $\frac{1}{C}\frac{\partial}{\partial t}$. But then $\mathbb{R}\times\gamma$ does not have finite Hofer energy, which is a contradiction. 
\end{proof}

We can now finish the construction of $(Y_n,\alpha_n)$ using Proposition \ref{prop:embed-filtered} to glue a solid torus to $(M'_n, \alpha'_n)$ in such a way that all closed Reeb orbits of action at most $n$ lie in the interior of $(M',\alpha')$ for $n$ sufficiently large.  We will take $J_{n}$ to be an almost complex structure on $\R \times Y_n$ which agrees with $J'_{n}$ on $\R \times M'_n$: it is identical to $J_\calh$ on $\R \times M'$ and is $t$-invariant on $[-n,-\epsilon) \times R_+(\Gamma')$ and on $(\epsilon,n] \times R_+(\Gamma')$ for some fixed $\epsilon > 0$.  This defines a tuple $(Y_n, \alpha_n, J_{n})$ with $Y_n$ closed, and from the construction it is clear that this tuple can be taken to vary smoothly with $n$.  Since $(M,\alpha)$ embeds in $(Y_n, \alpha_n)$, we can canonically identify Reeb orbits in $(M,\alpha)$ with their images in $(Y_n,\alpha_n)$.

\begin{definition}
\label{def:embedding-data}
The collection of $1$-handles $\calh$ and diffeomorphism $f: R_+(\Gamma') \isomto R_-(\Gamma)$ used to construct each $(Y_n,\alpha_n)$ from $(M,\Gamma,\alpha)$ will be referred to as \emph{embedding data} for $(M,\Gamma,\alpha)$.
\end{definition}

\begin{lemma}
\label{lem:embedding}
Let $\Theta_+$ and $\Theta_-$ be generators of $\ecc(M,\Gamma,\alpha,J)$. Then for sufficiently large $n>0$, depending only on $\alpha'$, embedding data, $J_\calh$, and the orbit sets $\Theta_\pm$, we have
\[ \mathcal{M}_{I=1}(\R \times M'_n, J'_{n}; \Theta_+, \Theta_-) = \mathcal{M}_{I=1}(\R \times Y_n, J_{n}; \Theta_+,\Theta_-). \]
\end{lemma}

\begin{proof}
We will show that no pseudo-holomorphic curve in $\mathcal{M}_{I=1}(\R \times Y_n, J_{n}; \Theta_+, \Theta_-)$ has non-empty projection onto the solid torus $V_n = \overline{Y_n \ssm M'_n}$. In order to prove this, suppose to the contrary that there exists a pseudo-holomorphic curve $\mathcal{C}\in\mathcal{M}_{I=1}(\R \times Y_n, J_{n}; \Theta_+,\Theta_-)$ whose projection onto the the solid torus $V_n$ is non-empty. Then either $\mathcal{C}$ has non-empty intersection with $\mathbb{R}\times[-1,0]_\tau\times\partial M'_n$, or $\mathcal{C}$ decomposes into connected components some of which lie in $\mathbb{R}\times M'_n$ and the others in $\mathbb{R}\times V_n$. In the latter case, a connected component of $\mathcal{C}$ must have ends asymptotic to closed Reeb orbits in $V_n$. But neither $\Theta_+$ nor $\Theta_-$ contains such closed Reeb orbits. Therefore, suppose that $\mathcal{C}$ has non-empty intersection with $\mathbb{R}\times[-1,0]_\tau\times\partial M'_n$.

Since the restriction of the contact form $\alpha'_n$ to $[-1,0]_\tau\times\partial M'_n \cong [-1,0]_\tau \times S^1_{t} \times \Gamma'$ reads $Cdt+e^\tau\beta'_0$, the Reeb vector field there is $\frac{1}{C}\partial_t$ and so it is foliated by closed Reeb orbits, which are all null-homologous in $Y_n$. Hence, $\mathcal{C}$ has non-empty intersection with $\mathbb{R}\times \gamma$ for a closed Reeb orbit $\gamma$ in $[-1,0]_\tau\times\partial M'_n$. Note that the simple Reeb orbits that appear in the collections $\Theta_+$ and $\Theta_-$ are disjoint and sufficiently far from any closed Reeb orbit in this region. Hence the algebraic intersection number of $\mathcal{C}$ with $\mathbb{R}\times\gamma$ for any closed Reeb orbit $\gamma$ in $[-1,0]_\tau\times\partial M'_n$ is well-defined and positive since local intersection numbers for two pseudo-holomorphic curves are positive \cite{mcduff, micallef-white}. Furthermore, the algebraic intersection number of $\mathcal{C}$ with $\mathbb{R}\times\gamma$ is equal to the intersection of the homology classes $[\mathcal{C}]\in H_2(Y_n,\Theta_+,\Theta_-)$ and $[\gamma]\in H_1(Y_n)$, which is zero since the latter homology class is zero. This is a contradiction.\end{proof}

\begin{lemma}
\label{lem:canonical-cor}
Let $\Theta_+$ and $\Theta_-$ be generators of $\ecc(M,\Gamma,\alpha,J)$. Then for sufficiently large $n>0$, depending only on $\alpha'$, the embedding data, $J_\calh$, and the orbit sets $\Theta_\pm$, we have
\[ \mathcal{M}_{I=1}(\R \times M^\ast, J; \Theta_+, \Theta_-) = \mathcal{M}_{I=1}(\R \times Y_n, J_{n}; \Theta_+,\Theta_-). \]
\end{lemma}
\begin{proof}
This follows immediately from Lemmas \ref{lem:tau-stretching}, \ref{lem:t-stretching}, and \ref{lem:embedding}.
\end{proof}

Given a constant $L > 0$ and an $\textit{ECH}^L$-generic tailored almost complex structure $J$ on $\R\times M^\ast$, if we take $n$ sufficiently large with respect to $L$ as in Lemma \ref{lem:glue-pm-filtered}, and depending on $J_\calh$ in the sense of Lemma \ref{lem:t-stretching}, we have now shown that there are canonical identifications
\begin{equation}
\label{eq:canid}
\ecc^L(M,\Gamma,\alpha,J) \cong \ecc^L(M',\Gamma',\alpha',J_\calh) \cong \ecc^L(Y_n, \alpha_n, J_{n}), 
\end{equation}
since we can canonically identify the finitely many generators of all three complexes and the moduli spaces defining the differential on each complex are all identical.  We define an isomorphism
\[\Phi^{L,J}_{n}: \ech^L(M,\Gamma,\alpha,J) \isomto \ech^L(Y_n,\alpha_n)\]
as the composition of the isomorphism $\tilde{\Phi}^{L,J}_n: \ech^L(M,\Gamma,\alpha,J) \isomto \ech^L(Y_n,\alpha_n, J_{n})$ corresponding to this identification with the canonical isomorphisms
\begin{equation}
\label{eq:huttab}
\ech^L(Y_n, \alpha_n, J_{n}) \to \ech^L(Y_n,\alpha_n)
\end{equation}
between any member of the transitive system $\{\ech^L(Y_n,\alpha_n,J_n)\}_{J_n}$ and the group $\ech^L(Y_n,\alpha_n)$ canonically associated to the transitive system (see \cite[Theorem 1.3]{ht2}).

The following lemma will be useful in understanding how to relate the various maps $\Phi^{L,J}_{n}$.

\begin{lemma}
\label{lem:j-independence-homotopy}
Fix $L > 0$ and let $Y$ be a closed $3$-manifold with a 1-parameter family of pairs $\{(\lambda_t, J_t)\mid 0 \leq t \leq 1\}$, where each $\lambda_t$ is an $L$-nondegenerate contact form and $J_t$ is an $\ech^L$-generic, symplectization-admissible almost contact structure for $\lambda_t$.  Suppose that
\begin{enumerate}\leftskip-0.25in
\item \label{item:j-independence-homotopy-1} The family $\{(\lambda_t, J_t)\}$ is constant on some neighborhood of each closed Reeb orbit of $(Y,\lambda_0)$ with $\lambda_0$-action at most $L$.
\item \label{item:j-independence-homotopy-2} For each $t$, every closed Reeb orbit $\gamma$ of $(Y,\lambda_t)$ with $A_{\lambda_t}(\gamma) \leq L$ coincides with a Reeb orbit of $(Y,\lambda_0)$ with $A_{\lambda_0}(\gamma) = A_{\lambda_t}(\gamma)$ from item \eqref{item:j-independence-homotopy-1}.
\end{enumerate}
Then the canonical isomorphism 
\[ \ech^L(Y,\lambda_0,J_0) \isomto \ech^L(Y,\lambda_1,J_1) \]
of \cite[Theorem 1.3]{ht2} is induced by the isomorphism of chain complexes which sends a generator $\Theta \in \ecc^L(Y,\lambda_0,J_0)$ to its image $\Theta \in \ecc^L(Y,\lambda_1,J_1)$.
\end{lemma}

\begin{proof}
Let $(\lambda^1_0, J^1_0)$ be a preferred $L$-flat approximation to $(\lambda^0_0,J^0_0) : = (\lambda_0,J_0)$, as defined in \cite[Appendix B]{taubes1} and used in \cite[Lemma 3.6]{ht2}, with preferred homotopy $\{(\lambda^s_0,J^s_0) \mid 0 \leq s \leq 1\}$.  We can assume for any fixed $\epsilon > 0$ that each $(\lambda^s_0, J^s_0)$ agrees with $(\lambda_0,J_0)$ outside an $\epsilon$-neighborhood $N_\epsilon$ of the Reeb orbits of $(Y,\lambda_0)$ of action at most $L$ (this is stated as \cite[Lemma 3.6{c}]{ht2}), so we will take $\epsilon$ small enough to ensure that the family $(\lambda_t,J_t)$ is constant on a $2\epsilon$-neighborhood of these same orbits.

Since the condition of being a preferred $L$-flat approximation depends only on neighborhoods of these orbits, we can now define a preferred $L$-flat approximation for any $(\lambda^0_t, J^0_t) : = (\lambda_t,J_t)$ with $t > 0$ using the preferred homotopy
\[ (\lambda^s_t, J^s_t) = \begin{cases}
(\lambda^s_0, J^s_0) & \mathrm{on\ }N_\epsilon \\
(\lambda^0_t, J^0_t) & \mathrm{on\ }Y \ssm N_\epsilon.
\end{cases} \]
The desired isomorphism is now a composition of isomorphisms
\[ \ech^L(Y,\lambda_0,J_0) \isomto \ech^L(Y,\lambda^1_0,J^1_0) \isomto \ech^L(Y,\lambda^1_1,J^1_1) \isomto \ech^L(Y,\lambda_1,J_1) \]
in which the first and third isomorphism are induced by chain maps $\Theta \mapsto \Theta$ by the discussion following \cite[Definition 3.2]{ht2} using the homotopies $\{(\lambda^s_0,J^s_0) \mid 0 \leq s \leq 1\}$ and $\{(\lambda^s_1,J^s_1) \mid 0 \leq s \leq 1\}$ respectively, and the second isomorphism is as well by \cite[Lemma 3.4(d)]{ht2} using the homotopy $\{(\lambda^1_t,J^1_t) \mid 0 \leq t \leq 1\}$.
\end{proof}

We can now show that the isomorphisms $\Phi^{L,J}_n$ do not depend on how we chose our extension $J_\calh$ of an $\textit{ECH}^L$-generic tailored almost complex structure $J$ on $\R\times M^\ast$ to $\R \times (M')^*$.

\begin{lemma}
\label{lem:j-independence-kappa}
Suppose that we have fixed embedding data for which $J_{\calh,0}$ and $J_{\calh,1}$ are almost complex structures on $\R \times (M')^*$ as constructed above, and $n$ is large enough with respect to both $J_{\calh,0}$ and $J_{\calh,1}$ to define the corresponding isomorphisms $\Phi^{L,J}_{n,0}$ and $\Phi^{L,J}_{n,1}$.  Then these isomorphisms are equal for all sufficiently large $n$ with respect to the embedding data, $L$, and both $J_{\calh,0}$ and $J_{\calh,1}$.
\end{lemma}

\begin{proof}
Under the hypotheses of this lemma, each of the complexes $\ecc^L(Y_n,\alpha_n,J_{n,0})$ and $\ecc^L(Y_n,\alpha_n,J_{n,1})$ is canonically identical to $\ecc^L(M,\Gamma,\alpha,J)$.  Thus for all large enough $n$ we have a diagram
\label{}\[ \xymatrix{
\ech^L(M,\Gamma,\alpha,J) \ar@{=}[d] \ar[r]^-{\sim} &
\ech^L(Y_n,\alpha_n,J_{n,0}) \ar[r]^-{\sim} & 
\ech^L(Y_n,\alpha_n) \ar@{=}[d] \\
\ech^L(M,\Gamma,\alpha,J) \ar[r]^-{\sim} &
\ech^L(Y_n,\alpha_n,J_{n,1}) \ar[r]^-{\sim} &
\ech^L(Y_n,\alpha_n)
} \]
in which the indicated isomorphisms are the canonical ones and the compositions of the maps across each row are equal to $\Phi^{L,J}_{n,0}$ and $\Phi^{L,J}_{n,1}$, respectively.  This diagram commutes, and hence $\Phi^{L,J}_{n,0} = \Phi^{L,J}_{n,1}$, if and only if the composition
\[ \ech^L(Y_n,\alpha_n,J_{n,0}) \isomto \ech^L(Y_n,\alpha_n) \isomto \ech^L(Y_n,\alpha_n,J_{n,1}) \]
is induced by a chain map $\Theta \mapsto \Theta$ sending each collection of Reeb orbits to itself. Here, the rightmost isomorphism is the inverse of the rightmost isomorphism in the bottom row of the diagram above. Note that this composition is exactly the canonical isomorphism
\[ \ech^L(Y_n,\alpha_n,J_{n,0}) \isomto \ech^L(Y_n,\alpha_n,J_{n,1}) \]
of \cite[Theorem 1.3]{ht2}.

Let $\{J_{n,s} \mid 0 \leq s \leq 1\}$ be any continuous, symplectization-admissible family which restricts to $\R \times M$ as $J$.  Lemmas \ref{lem:tau-stretching} and \ref{lem:t-stretching} ensure that each $J_{n,s}$ is $\ech^L$-generic for $n$ sufficiently large.  More precisely, fix generators $\Theta_+$ and $\Theta_-$ of $\ecc^L(Y_n,\alpha_n,J_{n,s})$; then we claim that for all sufficiently large $n$, the moduli space $\mathcal{M}_{I=1}(\R\times Y_n, J_{n,s}; \Theta_+, \Theta_-)$ is independent of $s$ and canonically identified with $\mathcal{M}_{I=1}(\R\times M^*, J; \Theta_+, \Theta_-)$.  If not, then we can find a sequence $(n_i, s_i)$ with $0\leq s_i \leq 1$ and $n_i \to \infty$, and $J_{n_i,s_i}$-holomorphic curves $\mathcal{C}_i$ which intersect the neck $\R \times [-\epsilon',\epsilon'] \times R_+(\Gamma')$ in $\R\times Y_{n_i}$ as in the proof of Lemma \ref{lem:t-stretching}, and then applying the same Gromov compactness argument as before to the sequence $\mathcal{C}_i$ yields a contradiction.  For $n$ which is large enough in this sense, an application of Lemma \ref{lem:j-independence-homotopy} to the family
\[ (\lambda_t, J_t) = \left(\alpha_n, J_{n,s}\right) \]
on $Y = Y_n$ now completes the proof.
\end{proof}

Suppose that $L' > L$, $J$ is an $\textit{ECH}^{L'}$-generic tailored almost complex structure on $\R\times M^\ast$, and that $n$ is chosen to be large with respect to $L'$. Then we have a commutative diagram
\[ \xymatrix{
\ech^L(M,\Gamma,\alpha,J) \ar[r]^-{\sim}_{\eqref{eq:canid}} \ar[d]_{i^{L,L'}_J} &
\ech^L(Y_n,\alpha_n,J_{n}) \ar[d]^{i^{L,L'}_{J_{n}}} \ar[r]^-{\sim}_-{\eqref{eq:huttab}} &
\ech^L(Y_n, \alpha_n) \ar[d]^{i^{L,L'}} \\
\ech^{L'}(M,\Gamma,\alpha,J) \ar[r]^-{\sim}_{\eqref{eq:canid}} &
\ech^{L'}(Y_n,\alpha_n,J_{n}) \ar[r]^-{\sim}_-{\eqref{eq:huttab}}&
\ech^{L'}(Y_n, \alpha_n),
} \]
in which the leftmost square commutes by the above discussion, and the existence and commutativity of the rightmost square are essentially the content of \cite[Theorem 1.3(b)]{ht2}.  It follows that the maps $\Phi^{L,J}_{n}$ and $\Phi^{L',J}_{n}$ commute with the $i^{L,L'}$ for all such large $n$.  We would also like to show that these maps do not depend on $n$ in the following sense.

\begin{proposition}
\label{prop:j-independence-n-iso}
Having fixed $L$ and embedding data, there exists $N>0$ depending on $L$ and the embedding data such that for all  $n, n'$ with $N<n < n'$, there are isomorphisms 
\[ \Psi^L_{n,n'}: \ech^L(Y_n,\alpha_n) \isomto \ech^L(Y_{n'},\alpha_{n'}) \]
such that:
\begin{enumerate}\leftskip-0.25in
\item $\Psi^L_{n,n''} = \Psi^L_{n',n''} \circ \Psi^L_{n,n'}$ whenever $n<n'<n''$;
\item The diagram
\[ \xymatrix{
\ech^L(M,\Gamma,\alpha,J) \ar[r]^{\Phi^{L,J}_{n}} \ar[dr]_{\Phi^{L,J}_{n'}}
& \ech^L(Y_n, \alpha_n) \ar[d]^{\Psi^L_{n,n'}} \\
& \ech^L(Y_{n'}, \alpha_{n'})
} \]
commutes whenever $n$ and $n'$ are large enough so that the maps $\Phi^{L,J}_{n}$ and $\Phi^{L,J}_{n'}$ are defined for $J$.
\end{enumerate}
\end{proposition}

\begin{proof}
We will use the almost complex structures $J_{n}$ and $J_{n'}$ constructed above on $\R \times Y_n$ and $\R \times Y_{n'}$.  Let $f_t: [-n,n] \hookrightarrow [-n',n']$ be a family of smooth, injective maps such that $f_0(x) = x$ and $f_t(x) = x \pm t(n'-n)$ near $x=\pm n$, so that $f_1$ has image $[-n',n']$.  Then we can define a continuous family of diffeomorphisms
\[ \varphi_t: Y_n \to Y_{(1-t)n + tn'} \]
by letting $\varphi_t|_{M'} = \mathrm{Id}_{M'}$ and $\varphi_s|_{[-n,n]\times R_+(\Gamma')} = f_s \times \mathrm{Id}_{R_+(\Gamma')}$, and by extending this smoothly over the solid torus $V_n = \overline{Y_n \ssm (M' \cup [-n,n]\times R_+(\Gamma'))}$ by a diffeomorphism $V_n \to V _{(1-t)n+tn'}$ which depends continuously on $t$ and is the identity when $t=0$.  Note that $\varphi_0: Y_n \to Y_n$ is the identity map.

We now apply Lemma \ref{lem:j-independence-homotopy} to the family of contact forms and almost complex structures
\[ (\lambda_t, J_t) = \left(\varphi_t^*(\alpha_{(1-t)n+tn'}), \varphi_t^*(J_{(1-t)n+tn'})\right) \]
on $Y_n$ to see that the canonical isomorphism
\[ \ech^L(Y_n,\alpha_n,J_{n}) \isomto \ech^L(Y_n,\varphi_1^*(\alpha_{n'}), \varphi_1^*(J_{n'})) \]
is induced by a chain map of the form $\Theta \mapsto \Theta$.  Since $\ech^L(Y_n,\varphi_1^*(\alpha_{n'}),\varphi_1^*(J_{n'}))$ is canonically identified with $\ech^L(Y_{n'}, \alpha_{n'}, J_{n'})$ in the same way, we now have an isomorphism
\[ \psi_{n,n'}: \ech^L(Y_n, \alpha_n, J_{n}) \isomto \ech^L(Y_{n'}, \alpha_{n'}, J_{n'}) \]
which is defined at the chain level by $\Theta \mapsto \Theta$.  This map fits into a commutative triangle with the isomorphisms $\ech^L(M,\Gamma,\alpha,J) \isomto \ech^L(Y_s, \alpha_s, J_{s})$ for $s = n$ and $s = n'$, hence the isomorphism
\[ \Psi^L_{n,n'}: \ech^L(Y_n, \alpha_n) \isomto \ech^L(Y_{n'}, \alpha_{n'}) \]
which $\psi_{n,n'}$ canonically induces must also commute with $\Phi^{L,J}_{n}$ and $\Phi^{L,J}_{n'}$, as desired.  It is also clear from the construction that $\psi_{n,n''} = \psi_{n',n''} \circ \psi_{n,n'}$, and hence $\Psi^L_{n,n''} = \Psi^L_{n',n''} \circ \Psi^L_{n,n'}$ follows immediately.
\end{proof}

Given an $\textit{ECH}^L$-generic tailored almost complex structure $J'$ on $\R \times M^*$, we can likewise construct isomorphisms
\[ \Phi^{L,J'}_{n}: \ech^L(M,\Gamma,\alpha, J') \isomto \ech^L(Y_n,\alpha_n). \]
This gives us an isomorphism
\[ \Phi^L_{J,J',n} = \left(\Phi^{L,J'}_{n}\right)^{-1} \circ \Phi^{L,J}_{n}: \ech^L(M,\Gamma,\alpha,J) \isomto \ech^L(M,\Gamma,\alpha,J'). \]

\begin{lemma}
\label{lem:j-independence-n}
Having fixed $L$ and embedding data, the isomorphism $\Phi^L_{J,J',n}$ does not depend on the choice of $n$ that is sufficiently large with respect to $L$, the embedding data, $J$ and $J'$.
\end{lemma}

\begin{proof}
We have shown in Proposition \ref{prop:j-independence-n-iso} that the diagram
\[ \xymatrix{
\ech^L(M,\Gamma,\alpha,J) \ar[r]^-{\Phi^{L,J}_{n}} \ar[dr]_{\Phi^{L,J}_{n'}}
& \ech^L(Y_n, \alpha_n) \ar[d]^{\Psi^L_{n,n'}}
& \ech^L(M,\Gamma,\alpha,J') \ar[l]_-{\Phi^{L,J'}_{n}} \ar[dl]^{\Phi^{L,J'}_{n'}} \\
& \ech^L(Y_{n'}, \alpha_{n'}) &
} \]
commutes whenever $n < n'$ are large enough for these maps to be defined, and so we have
\begin{align*}
\Phi^L_{J,J',n'} &= \left(\Phi^{L,J'}_{n'}\right)^{-1} \circ \Phi^{L,J}_{n'} \\
&= \left(\Psi_{n,n'} \circ \Phi^{L,J'}_{n}\right)^{-1} \circ \left(\Psi_{n,n'} \circ \Phi^{L,J}_{n}\right) \\
&= \left(\Phi^{L,J'}_{n}\right)^{-1} \circ \Phi^{L,J}_{n} \\
&= \Phi^L_{J,J',n}.
\qedhere
\end{align*}
\end{proof}

Since $i^{L,L'} \circ \Phi^{L,J}_{n} = \Phi^{L',J}_{n} \circ i^{L,L'}$, we thus have commutative diagrams
\[ \xymatrix{
\ech^L(M,\Gamma,\alpha,J) \ar[r]^-{\Phi^L_{J,J'}}_{\sim} \ar[d]_{i^{L,L'}_{J}} &
\ech^L(M,\Gamma,\alpha,J') \ar[d]^{i^{L,L'}_{J'}} \\
\ech^{L'}(M,\Gamma,\alpha,J) \ar[r]^-{\Phi^{L'}_{J,J'}}_{\sim} &
\ech^{L'}(M,\Gamma,\alpha,J')
} \]
where $\Phi^L_{J,J'} = \Phi^L_{J,J',n}$ for $n$ large and likewise for $\Phi^{L'}_{J,J'}$; although both maps require a choice of $n$, we may take $n$ to be sufficiently large for both maps and according to Lemma \ref{lem:j-independence-n} the maps do not depend on our choice.

\begin{theorem}
\label{thm:j-independence}
Let $J$ and $J'$ be generic tailored almost complex structures on $\R \times M^*$, and fix a choice of embedding data for $(M,\Gamma,\alpha)$ as in Definition \ref{def:embedding-data}.  Then for all $L>0$ there are natural isomorphisms
\[ \Phi^L_{J,J'}: \ech^L(M,\Gamma,\alpha,J) \isomto \ech^L(M,\Gamma,\alpha,J'), \]
depending only on $L$, $J$, and $J'$, and satisfying
\begin{enumerate}\leftskip-0.25in
\item \label{item:thm-j-indep-1} $i^{L,L'}_{J'} \circ \Phi^L_{J,J'} = \Phi^{L'}_{J,J'} \circ i^{L,L'}_{J}$ for all $L' > L$.
\item \label{item:thm-j-indep-2} $\Phi^L_{J,J} = \mathrm{Id}$, and $\Phi^L_{J,J''} = \Phi^L_{J',J''} \circ \Phi^L_{J,J'}$ for any $J,J',J''$.
\end{enumerate}
The limit of these maps as $L \to \infty$ is a natural isomorphism
\[ \Phi_{J,J'}: \ech(M,\Gamma, \alpha, J) \isomto \ech(M,\Gamma, \alpha, J') \]
which also satisfies $\Phi_{J,J''} = \Phi_{J',J''} \circ \Phi_{J,J'}$.
\end{theorem}

\begin{proof}
We have already seen that the maps $\Phi^L_{J,J'} = \Phi^L_{J,J',n}$ for $n$ large with respect to $L$ satisfy property \eqref{item:thm-j-indep-1}, and property \eqref{item:thm-j-indep-2} can be easily checked by looking at the maps for a single choice of $n$.  The maps $\Phi^L_{J,J'}$ specify an isomorphism from the directed system
\[ (\{\ech^L(M,\Gamma,\alpha,J)\}_L, \{i^{L,L'}_{J}\}_{L,L'}) \]
to the directed system
\[ (\{\ech^L(M,\Gamma,\alpha,J')\}_L, \{i^{L,L'}_{J'}\}_{L,L'}), \]
and hence induce an isomorphism $\Phi_{J,J'}$ on their direct limits, since the direct limit is an exact functor.  But since homology commutes with taking direct limits, and
\[ \ecc(M,\Gamma,\alpha,J) = \lim_{L\to\infty} \ecc^L(M,\Gamma,\alpha,J) \]
and likewise for $J'$, the direct limits of each system are canonically isomorphic to the associated unfiltered homology groups.  Thus $\Phi_{J,J'}$ is the desired isomorphism
\[ \ech(M,\Gamma,\alpha,J) \isomto \ech(M,\Gamma,\alpha,J'). \qedhere \]
\end{proof}

In other words, a choice of embedding data uniquely determines transitive systems of groups $(\{\ech^L(M,\Gamma,\alpha,J)\}_J,\{\Phi^L_{J,J'}\}_{J,J'})$, and hence a canonical group
\begin{equation}
\label{eq:ech-l-canonical}
\ech^L(M,\Gamma,\alpha) \subset \prod_J \ech^L(M,\Gamma,\alpha,J)
\end{equation}
as the subgroup consisting of tuples $(x_J)_{J}$ such that $\Phi^L_{J,J'}(x_J) = x_{J'}$ for all pairs $J,J'$. For fixed $L>0$, and $J$, there exist canonical isomorphisms 
\begin{equation}
\label{eq:P-L-J}
P^{L,J}:\ech^L(M,\Gamma,\alpha,J)\to\ech^L(M,\Gamma,\alpha)
\end{equation}
depending only on the embedding data. The inclusion maps $i^{L,L'}_{J}$ induce a well-defined map $i^{L,L'}: \ech^L(M,\Gamma,\alpha) \to \ech^{L'}(M,\Gamma,\alpha)$, and the direct limit as $L\to\infty$ is a well-defined group
\[ \ech(M,\Gamma,\alpha) \]
which is also determined by the transitive system $(\{\ech(M,\Gamma,\alpha,J)\}_J, \{\Phi_{J,J'}\}_{J,J'})$ as in \eqref{eq:ech-l-canonical}.  Thus we can rephrase the last part of Theorem \ref{thm:j-independence} as saying that there is a group $\ech(M,\Gamma,\alpha)$ equipped with a canonical isomorphism
\[ P^J: \ech(M,\Gamma,\alpha,J) \isomto \ech(M,\Gamma,\alpha) \]
for every $J$.

Each group $\ech^L(M,\Gamma,\alpha,J)$ contains a class $[\emptyset]$ corresponding to the empty set of Reeb orbits, which is a cycle because there are no non-empty $J$-holomorphic curves with empty positive end by Stokes's Theorem since a $J$-holomorphic curve has positive $d\alpha$-energy. We claim that this induces a well-defined element of $\ech^L(M,\Gamma,\alpha)$, and hence of $\ech(M,\Gamma,\alpha)$ as well.

\begin{proposition}
\label{prop:contact-class-j}
The isomorphisms $\Phi^L_{J,J'}: \ech^L(M,\Gamma,\alpha,J) \isomto \ech^L(M,\Gamma,\alpha,J')$ carry $[\emptyset]$ to $[\emptyset]$, and hence define a unique element $[\emptyset] \in \ech^L(M,\Gamma,\alpha)$ such that $i^{L,L'}([\emptyset]) = [\emptyset]$ for all $L' > L$.
\end{proposition}

\begin{proof}
We construct a closed manifold $(Y_n,\alpha_n)$, together with almost complex structures $J_n$ and $J'_n$ on $\R\times Y_n$ corresponding to $J$ and $J'$, for sufficiently large $n$ as above.  Unraveling the definition of $\Phi^L_{J,J'}$, we wish to show that the canonical isomorphism
\begin{equation}
\label{eq:yn-change-of-j}
\ech^L(Y_n,\alpha_n,J_n) \isomto \ech^L(Y_n,\alpha_n,J'_n)
\end{equation}
of \cite[Theorem 1.3]{ht2} carries $[\emptyset]$ to $[\emptyset]$.  This claim will follow from the fact that the ECH cobordism maps induced by exact symplectic cobordisms do not depend on almost complex structures.

Fix $\epsilon > 0$ small enough to ensure that $(M,\Gamma,\alpha)$ has no ECH generators whose $\alpha_n$--action lies in the interval $[L,e^{\epsilon}L]$, and consider the product cobordism $(X,\lambda) = ([-\epsilon,0] \times Y_n, e^s \alpha_n)$.  The fact that the cobordism maps $\Phi^L(X,\lambda)$ of \cite[Theorem 1.9]{ht2} do not depend on an almost complex structure means that we have a commutative diagram
\[ \xymatrix{
\ech^L(Y_n,\alpha_n, J_n) \ar[rr]^-{\Phi^L(X,\lambda,J)} \ar[d] &&
\ech^L(Y_n,e^{-\epsilon}\alpha_n, J_n) \ar@{=}[d] \\
\ech^L(Y_n,\alpha_n,J'_n) \ar[rr]^-{\Phi^L(X,\lambda,J')} &&
\ech^L(Y_n,e^{-\epsilon}\alpha_n, J_n)
} \]
where the leftmost vertical map is the isomorphism \eqref{eq:yn-change-of-j} and each $\Phi^L(X,\lambda)$ is determined by equipping $(X,\lambda)$ with an appropriate cobordism-admissible almost complex structure $J$ or $J'$.  Now it follows from \cite[Corollary 5.8]{ht2} that the top map $\Phi^L(X,\lambda,J)$ is a composition
\[ \ech^L(Y_n,\alpha_n,J_n) \xrightarrow{s} \ech^{e^{-\epsilon}L}(Y_n,e^{-\epsilon}\alpha_n,J_n) \xrightarrow{i^{e^{-\epsilon}L,L}_{J_n}} \ech^L(Y_n,e^{-\epsilon}\alpha_n,J_n) \]
of scaling and inclusion-induced maps, and the latter is an isomorphism since by assumption there are no ECH generators with $(e^{-\epsilon}\alpha_n)$--action in the interval $[e^{-\epsilon}L,L]$, so $\Phi^L(X,\lambda,J)$ is an isomorphism and it carries $[\emptyset]$ to $[\emptyset]$.

The holomorphic curves axiom of \cite[Theorem 1.9]{ht2} says that $\Phi^L(X,\lambda,J')$ is induced by a chain map
\[ \hat{\Phi}: \ecc^L(Y_n, \alpha_n, J'_n) \to \ecc^L(Y_n, e^{-\epsilon}\alpha_n, J_n) \]
such that for any ECH generator $\Theta \neq \emptyset$, we have $\langle \hat{\Phi}(\emptyset), \Theta\rangle = 0$ since there are no broken $J'$-holomorphic curves in the completion $\overline{X}$ from $\emptyset$ to $\Theta$; and $\langle \hat{\Phi}(\emptyset), \emptyset\rangle = 1$ since the unique broken $J'$-holomorphic curve from $\emptyset$ to $\emptyset$ in $\overline{X}$ is a union of covers of product cylinders, which is empty. In other words, $\hat{\Phi}(\emptyset) = \emptyset$, and so we have $\Phi^L(X,\lambda,J')([\emptyset])=[\emptyset]$; but now the commutativity of the above diagram implies that \eqref{eq:yn-change-of-j} sends $[\emptyset]$ to $[\emptyset]$ as well.
\end{proof}

Proposition \ref{prop:contact-class-j} tells us that there is a natural element
\[ c(\alpha) := [\emptyset] \in \ech(M,\Gamma,\alpha). \]
We will show that this contact class $c(\alpha)$ is an invariant of the underlying contact structure in Theorem \ref{thm:alpha-independence-natural}, just as Taubes \cite{taubes5} showed it to be for closed contact $3$-manifolds.

\begin{remark}
\label{rem:canonical-transitive-systems}
The canonical isomorphisms $\Phi^L_{J,J'}$ are ``canonical'' in the sense that they do not depend on the choices of $J_\calh$ and $n$ and they compose nicely, but in order to define them we had to fix an initial choice of embedding data.  In particular, the transitive systems of groups we build out of the $\Phi^L_{J,J'}$ are not necessarily canonical, since different embedding data might yield a different set of ``canonical'' isomorphisms and hence a different transitive system.  However, the groups $\ech^L(M,\Gamma,\alpha)$ and $\ech(M,\Gamma,\alpha)$ we build for any two choices of embedding data are still isomorphic to each other, since each of them is isomorphic to any individual $\ech^L(M,\Gamma,\alpha,J)$ or $\ech(M,\Gamma,\alpha,J)$.
\end{remark}

%% file: alpha-independence.tex

In this section we will prove that $\ech(M,\Gamma,\alpha)$ is independent of the contact form $\alpha$ up to isotopy.  We will assume from the beginning that we have fixed a choice of embedding data for $(M,\Gamma,\alpha)$ as in Definition \ref{def:embedding-data}. 

\begin{warning}
It should be understood that the ECH homology groups and homomorphisms between them that appear in this section depend \emph{a priori} on a choice of embedding data.
\end{warning}

Suppose that $(M,\Gamma)$ is a sutured contact manifold with respect to two different contact forms $\alpha_0$ and $\alpha_1$, and that the contact forms agree on a neighborhood of $\partial M$: there is a neighborhood of $R_\pm(\Gamma)$ on which the contact forms both satisfy $\alpha_i = C dt + \beta_\pm$ for some constant $C$ and Liouville forms $\beta_\pm$ on $R_\pm(\Gamma)$, and a neighborhood $U(\Gamma)$ of the sutures $\Gamma$ where $\alpha_i = Cdt + e^\tau \beta_0$ for a volume form $\beta_0$ on $\Gamma$. Suppose furthermore that there is a 1-parameter family of contact forms $\alpha_s$ interpolating between $\alpha_0$ and $\alpha_1$, and that this family is constant on that neighborhood of $\partial M$.

Given $n > 0$, we can embed $(M,\alpha_0)$ into a closed contact manifold $(Y_n,\alpha^0_n)$ following the procedure of Section \ref{sec:j-independence}.  If we replace $\alpha^0_n|_M = \alpha_0$ with $\alpha_s$, then we get another contact form $\alpha^s_n$ on $Y_n$ such that $\alpha^0_n = \alpha^s_n$ except potentially on the interior of $M$.  Then the obvious embedding $i: M \hookrightarrow Y_n$ satisfies $i^*(\alpha^s_n) = \alpha_s$ for all $s$, and for any fixed $L$ we recall that by construction any Reeb orbits in $(Y_n,\alpha^s_n)$ of action less than $L$ will lie inside $(M,\alpha_s)$ whenever $n$ is large with respect to $L$.  If $\alpha_s$ depends smoothly on $s$ then this construction can be made smooth with respect to both $s$ and $n$.

\begin{lemma}
\label{lem:cobordism-family}
Let $\alpha_s$ be a family of contact forms on $Y$, $0 \leq s \leq 1$.  Then for $R$ sufficiently large there are exact symplectic cobordisms
\[ ([0,R] \times Y, \lambda_{01}) \mathrm{\ and\ } ([R,2R]\times Y, \lambda_{10}) \]
from $(Y,\alpha_0)$ to $(Y,e^R\alpha_1)$ and from $(Y,e^R\alpha_1)$ to $(Y,e^{2R}\alpha_0)$ respectively\footnote{Our description is the opposite of the convention in \cite{ht2}, in which a 4-manifold $X$ with $\partial X = Y_2 - Y_1$ is said to be a cobordism from $Y_2$ to $Y_1$.}, so that the composite cobordism
\[ ([0,2R] \times Y, \lambda_{01} \cup \lambda_{10}) \]
is homotopic through exact symplectic cobordisms to the symplectization $([0,2R]\times Y, e^s \alpha_0)$.
\end{lemma}

\begin{proof}
Take $R > 0$ and let $\psi: [0,R] \to [0,1]$ be smooth, nondecreasing function with $\psi(s)=0$ on a neighborhood of $s=0$ and $\psi(s)=1$ on a neighborhood of $s=R$, and consider the $1$-form
\[ \lambda_{01} = e^s \alpha_{\psi(s)} \]
on $[0,R]_s \times Y$.  We compute
\[ d\lambda_{01} = e^s \left[ds\wedge \left(\alpha_{\psi(s)} + \psi'(s)\dot{\alpha}_{\psi(s)}\right) + d\alpha_{\psi(s)}\right] \]
and so
\[ d\lambda_{01} \wedge d\lambda_{01} = 2e^{2s}\left[ ds \wedge \alpha_{\psi(s)} \wedge d\alpha_{\psi(s)} + \psi'(s)ds \wedge \dot{\alpha}_{\psi(s)} \wedge d\alpha_{\psi(s)} \right]. \]
The first term in brackets is a volume form on $[0,R]\times Y$, and by taking $R$ large we can make $\psi'(s)$ small enough to ensure that the sum is a volume form.  Hence for large $R$ we can find $\psi$ such that $([0,R] \times Y, \lambda_{01})$ is an exact symplectic cobordism from $(Y, \alpha_0)$ to $(Y, e^R\alpha_1)$.
We can extend $\psi$ to a function $\psi: [0,2R] \to [0,1]$ by symmetry, setting $\psi(R+t) = \psi(R-t)$ for $0 \leq t \leq R$, and  then let $\lambda_{10} = e^s \alpha_{\psi(s)}$ on $[R,2R] \times Y$.  If $R$ is large enough then $([R,2R] \times Y, \lambda_{10})$ will also be an exact symplectic cobordism from $e^R\alpha_1$ to $e^{2R} \alpha_0$.

If we have taken $\psi'(s)$ to be very small, then every element of the family of 1-forms $\lambda_t = e^s\alpha_{t\psi(s)}$ on $[0,2R]\times Y$ will be a Liouville form as well, and so we have a homotopy from $\lambda_0 = e^s\alpha_0$ to $\lambda_1 = \lambda_{01}\cup\lambda_{10}$.  Thus $([0,2R], \lambda_{01}\cup\lambda_{01})$ is homotopic as an exact symplectic cobordism to the symplectization $([0,2R],e^s\alpha_0)$.
\end{proof}

We can use Lemma \ref{lem:cobordism-family} to find $R>0$ and exact symplectic cobordisms
\[ X^n_{01} = ([0,R] \times Y_n, \lambda^{01}_n) \mathrm{\ and\ } X^n_{10} = ([R,2R] \times Y_n, \lambda^{10}_n) \]
from $(Y_n, \alpha^0_n)$ to $(Y_n, e^R\alpha^1_n)$ and from $(Y_n,e^R\alpha^1_n)$ to $(Y_n,e^{2R}\alpha^0_n)$ respectively, whose composition is homotopic to the symplectization $([0,2R]\times Y_n, e^s\alpha^0_n)$.  Since the families of contact forms $\alpha^s_n$ are independent of $s$ on $Y_n \ssm M$, the Liouville forms agree with $e^s\alpha^0_n$ on $[0,2R] \times \overline{Y_n \ssm M}$ and so the choice of $\psi(s)$ in the proof only matters on the interior of $M$.  This means that our choice of $R$ depends only on the family $(M,\alpha_s)$ and not on $n$, and so we can fix a universal choice of $R$ once and for all.

Now if $\alpha_0$, $e^R\alpha_1$, and $e^{2R}\alpha_0$ are $e^{2R}L$-nondegenerate, then according to \cite[Theorem 1.9]{ht2}, we have well-defined cobordism maps
\begin{align*}
\Phi^{e^{2R}L}(X^n_{01})&: \ech^{e^{2R}L}(Y_n,e^R\alpha^1_n) \to \ech^{e^{2R}L}(Y_n,\alpha^0_n) \\
\Phi^{e^{2R}L}(X^n_{10})&: \ech^{e^{2R}L}(Y_n,e^{2R}\alpha^0_n) \to \ech^{e^{2R}L}(Y_n,e^R\alpha^1_n)
\end{align*}
when $n$ is sufficiently large so that each of the contact forms involved is $e^{2R}L$-nondegenerate (as the only Reeb orbits of action less than $e^{2R}L$ are contained in the interior of $M$). If $n$ is sufficiently large with respect to $e^{2R}L$, then these $\ech$ groups can be associated canonically with certain sutured $\ech$ groups for $M$, following the discussion in Section \ref{sec:j-independence}.  Namely, given appropriate almost complex structures $J^i_{n}$ on $\R\times Y_n$ corresponding to generic tailored almost complex structures $J_i$ on $\R \times M^*$, we have
\[ \ech^{e^{2R}L}(Y_n, \alpha^i_n) \cong \ech^{e^{2R}L}(M,\Gamma,\alpha_i,J_i) \]
and, using the scaling isomorphisms of \cite[Theorem 1.3]{ht2},
\begin{align*}
\ech^{e^{2R}L}(Y_n, e^R \alpha^1_n) \cong & \ech^{e^R L}(Y_n, \alpha^1_n) \cong \ech^{e^R L}(M,\Gamma,\alpha_1,J_1) \\
\ech^{e^{2R}L}(Y_n, e^{2R}\alpha^0_n) \cong & \ech^{L}(Y_n, \alpha^0_n) \cong \ech^{L}(M,\Gamma,\alpha_0,J_0).
\end{align*}
Each of these isomorphisms carries a collection of Reeb orbits $\Theta \subset Y_n$, which lies in $M$ if $n$ is large, to its image $\Theta \subset M$.  Composing them with the maps $\Phi^{e^{2R}L}(X^n_{01})$ and $\Phi^{e^{2R}L}(X^n_{10})$ gives us homomorphisms
\begin{align*}
\Psi^{R,L}_{n,01}&: \ech^{e^{R}L}(M,\Gamma,\alpha_1,J_1) \to \ech^{e^{2R}L}(M,\Gamma,\alpha_0,J_0) \\
\Psi^{R,L}_{n,10}&: \ech^{L}(M,\Gamma,\alpha_0,J_0) \to \ech^{e^{R}L}(M,\Gamma,\alpha_1,J_1). 
\end{align*}

\begin{lemma}
\label{lem:alpha-ind-composition}
For sufficiently large $n$ with respect to $e^{2R}L$, the composition
\[ \Psi^{R,L}_{n,01} \circ \Psi^{R,L}_{n,10}: \ech^L(M,\Gamma,\alpha_0,J_0) \to \ech^{e^{2R}L}(M,\Gamma,\alpha_0,J_0) \]
is equal to the map $i^{L,e^{2R}L}$ induced by the inclusion of chain complexes.
\end{lemma}

\begin{proof}
Using the canonical isomorphisms to various filtered $\ech$ complexes for $(Y_n,\alpha^i_n, J^i_{n})$ together with the composition and the scaling axioms of \cite[Theorem 1.9]{ht2} for $\ech$ cobordism maps, we can identify this composition with the cobordism map
\[ \Phi^{e^{2R}L}([0,2R]\times Y_n, \lambda^{01}_n\cup\lambda^{10}_n): \ech^{e^{2R}L}(Y_n,e^{2R}\alpha^0_n) \to \ech^{e^{2R}L}(Y_n,\alpha^0_n). \]
Since $\lambda^{01}_n \cup \lambda^{10}_n$ is homotopic to the symplectization $e^s \alpha^0_n$, the homotopy invariance axiom of \cite[Theorem 1.9]{ht2} says that this map is equal to 
\[ \Phi^{e^{2R}L}([0,2R]\times Y_n,  e^s\alpha^0_n): \ech^{e^{2R}L}(Y_n,e^{2R}\alpha^0_n) \to \ech^{e^{2R}L}(Y_n,\alpha^0_n) \]
which by \cite[Corollary 5.8]{ht2} is equal to the composition
\[ \ech^{e^{2R}L}(Y_n,e^{2R}\alpha^0_n) \xrightarrow{s} \ech^{L}(Y_n,\alpha^0_n) \xrightarrow{i^{L,e^{2R}L}} \ech^{e^{2R}L}(Y_n,\alpha^0_n) \]
of scaling and inclusion-induced maps.  Translating this back into a morphism from $\ech^L(M,\Gamma,\alpha_0,J_0)$ to $\ech^{e^{2R}L}(M,\Gamma,\alpha_0,J_0)$ now completes the proof.
\end{proof}

One corollary of Lemma \ref{lem:alpha-ind-composition} is that the composition $\Psi^{R,L}_{n,01} \circ \Psi^{R,L}_{n,10}$ is independent of $n$ for all sufficiently large $n$.  We claim that the individual morphisms are themselves independent of $n$:

\begin{proposition}
\label{prop:alpha-ind-psi-n}
Given $n<n'$, both of which are large with respect to $e^{2R}L$, we have
\[ \Psi^{R,L}_{n,01} = \Psi^{R,L}_{n',01}: \ech^{e^{R}L}(M,\Gamma,\alpha_1,J_1) \to \ech^{e^{2R}L}(M,\Gamma,\alpha_0,J_0) \]
and
\[ \Psi^{R,L}_{n,10} = \Psi^{R,L}_{n',10}: \ech^L(M,\Gamma,\alpha_0,J_0) \to \ech^{e^{R}L}(M,\Gamma,\alpha_1,J_1). \]
\end{proposition}

\begin{proof}
We focus on $\Psi^{R,L}_{n,01} = \Psi^{R,L}_{n',01}$, following the strategy of \cite[Lemma 6.5]{ht2}; the other case is proved identically.  Moreover, we can identify a smooth family of diffeomorphisms $Y_n \isomto Y_m$ for all $m \geq n$ by identifying the copies of $M$ inside each $Y_m$ and identifying the neck $[-n,n]_t \times R_+(\Gamma')$ in $Y_n$ with the corresponding necks in each $Y_m$ via a smooth family of diffeomorphisms $[-n,n] \isomto [-m,m]$; having done so, we will abuse notation and write $\alpha^s_m$ to mean the pullback of $\alpha^s_m$ to $Y_n$, and likewise for the corresponding Liouville forms $\lambda^{01}_m$.  Unraveling the definitions of the maps, we now wish to prove that the diagram
\[ \xymatrix{
\ech^{e^{2R}L}(Y_n, e^{R}\alpha^1_n) \ar[rr]^{\Phi^{e^{2R}L}(X^n_{01})} \ar[d]_{\Psi^{e^{2R}L}_{n,n'}} &&
\ech^{e^{2R}L}(Y_n,\alpha^0_n) \ar[d]^{\Psi^{e^{2R}L}_{n,n'}} \\
\ech^{e^{2R}L}(Y_n, e^{R}\alpha^1_{n'}) \ar[rr]^{\Phi^{e^{2R}L}(X^{n'}_{01})} &&
\ech^{e^{2R}L}(Y_n,\alpha^0_{n'})
} \]
commutes, where the maps $\Psi^{e^{2R}L}_{n,n'}$ are the isomorphisms defined in Section \ref{sec:j-independence} which identify a collection of Reeb orbits $\Theta \subset Y_n$ with its image in $Y_{n'}$.  Using the compactness of $[n,n']$, it suffices to prove this for all $n'$ in some open neighborhood of $n$.

Take $\epsilon > 0$ small enough that the Liouville form $\lambda^{01}_n$ is equal to $e^s\alpha^0_n$ on $[0,\epsilon] \times Y_n$ and to $e^s\alpha^1_n$ on $[R-\epsilon,R]\times Y_n$ (in which case the analogous statement holds for $\lambda^{01}_{n'}$).  Assume without loss of generality that $n < n'$.  If $n'$ is sufficiently close to $n$, then we may fix a smooth, nondecreasing function $\phi:[0,\epsilon] \to [n,n']$ with $\phi(s)=n$ near $s=0$ and $\phi(s)=n'$ near $s=\epsilon$ such that the 1-forms
\[ \begin{array}{ll}
\lambda_0 = e^s\alpha^0_{\phi(s)} &\mathrm{\ on\ } [0,\epsilon] \times Y_n \\
\lambda_1 = e^s\alpha^1_{\phi(R-s)} &\mathrm{\ on\ } [R-\epsilon, R] \times Y_n
\end{array} \]
are both Liouville forms.  The induced cobordism map
\[ \Phi^{e^{2R}L}([0,\epsilon]\times Y_n,\lambda_0) : \ech^{e^{2R}L}(Y_n, e^\epsilon\alpha^0_{n'}) \to \ech^{e^{2R}L}(Y_n, \alpha^0_n) \]
fits into a commutative diagram
\[ \xymatrix{
\ech^{e^{2R}L}(Y_n, e^{\epsilon}\alpha^0_{n'}) \ar[rr]^-{\Phi^{e^{2R}L}(\lambda_0)} \ar[d]_s &&
\ech^{e^{2R}L}(Y_n, \alpha^0_n) \ar[d]^{\Psi^{e^{2R}L}_{n,n'}} \\
\ech^{e^{2R-\epsilon}L}(Y_n, \alpha^0_{n'}) \ar[rr]^-{i^{e^{2R-\epsilon}L,e^{2R}L}} &&
\ech^{e^{2R}L}(Y_n, \alpha^0_{n'})
} \]
by applying \cite[Lemma 5.6]{ht2} as in the proof of \cite[Lemma 6.5]{ht2}.  We know that the composition $i^{e^{2R-\epsilon}L,e^{2R}L} \circ s$ is equal to $\Phi^{e^{2R}L}([0,\epsilon] \times Y_n, e^s \alpha^0_{n'})$ by \cite[Corollary 5.8]{ht2}, hence
\begin{equation}
\label{eq:alpha-ind-psi-n-1}
\Phi^{e^{2R}L}([0,\epsilon] \times Y_n, e^s \alpha^0_{n'}) = \Psi^{e^{2R}L}_{n,n'} \circ \Phi^{e^{2R}L}([0,\epsilon] \times Y_n, \lambda_0).
\end{equation}
Applying the same arguments to the cobordism map
\[ \Phi^{e^{2R}L}([R-\epsilon,R]\times Y_n,\lambda_1) : \ech^{e^{2R}L}(Y_n, e^R\alpha^1_n) \to \ech^{e^{2R}L}(Y_n, e^{R-\epsilon}\alpha^1_{n'}), \]
we conclude that the diagram
\[ \xymatrix{
\ech^{e^{2R}L}(Y_n, e^{R}\alpha^1_{n'}) \ar[rr]^-{(\Psi^{e^{2R}L}_{n,n'})^{-1}} \ar[d]_{s} &&
\ech^{e^{2R}L}(Y_n, e^{R}\alpha^1_n) \ar[d]^{\Phi^{e^{2R}L}(\lambda_1)} \\
\ech^{e^{2R-\epsilon}L}(Y_n, e^{R-\epsilon}\alpha^1_{n'}) \ar[rr]^-{i^{e^{2R-\epsilon}L,e^{2R}L}} &&
\ech^{e^{2R}L}(Y_n, e^{R-\epsilon}\alpha^1_{n'})
} \]
commutes, and hence that
\begin{equation}
\label{eq:alpha-ind-psi-n-2}
 \Phi^{e^{2R}L}([R-\epsilon,R]\times Y_n, \lambda_1) = \Phi^{e^{2R}L}([R-\epsilon,R]\times Y_n, e^s \alpha^1_{n'}) \circ \Psi^{e^{2R}L}_{n,n'}.
\end{equation}

Now the composition axiom of \cite[Theorem 1.9]{ht2} says that the composition
\[ \xymatrix @R=10pt {
\llap{$\ech^{e^{2R}L}(Y_n, e^R \alpha^1_n)$}
\ar[rrr]^-{\Phi^{e^{2R}L}(\lambda_1)} 
&&& \rlap{$\ech^{e^{2R}L}(Y_n, e^{R-\epsilon} \alpha^1_{n'})$} \\
\ar[rrr]^-{\Phi^{e^{2R}L}([\epsilon,R-\epsilon]\times Y_n, \lambda^{n'}_{10})}
&&& \rlap{$\ech^{e^{2R}L}(Y_n, e^\epsilon \alpha^0_{n'})$} \\
\ar[rrr]^-{\Phi^{e^{2R}L}(\lambda_0)}
&&& \rlap{$\ech^{e^{2R}L}(Y_n, \alpha^0_n)$}
} \]
is equal to the cobordism map
induced by the union $\lambda_0 \cup (\lambda^{n'}_{01}|_{[\epsilon,R-\epsilon]\times Y_n}) \cup \lambda_1$, which is homotopic to $\lambda^n_{01}$ by a homotopy stationary on $[0,R]\times M$ and on $\{0,R\}\times Y_n$.  Hence by the homotopy axiom of \cite[Theorem 1.9]{ht2}, the composition is in fact equal to the map $\Phi^{e^{2R}L}(X^n_{01})$.  Using equations \eqref{eq:alpha-ind-psi-n-1} and \eqref{eq:alpha-ind-psi-n-2}, we have a commutative diagram
\[ \xymatrix{
\ech^{e^{2R}L}(Y_n,e^R\alpha^1_n) \ar[d]_{\Phi^{e^{2R}L}(\lambda_1)} \ar[rr]^-{\Psi^{e^{2R}L}_{n,n'}} &&
\ech^{e^{2R}L}(Y_n, e^R\alpha^1_{n'}) \ar[d]^{\Phi^{e^{2R}L}(e^s\alpha^1_{n'})} \\
\ech^{e^{2R}L}(Y_n,e^{R-\epsilon}\alpha^1_{n'}) \ar[d]_{\Phi^{e^{2R}L}(\lambda_{01}^{n'})} \ar[rr]^{\mathrm{Id}} &&
\ech^{e^{2R}L}(Y_n, e^{R-\epsilon}\alpha^1_{n'}) \ar[d]^{\Phi^{e^{2R}L}(\lambda_{01}^{n'})} \\
\ech^{e^{2R}L}(Y_n,e^{\epsilon}\alpha^0_{n'}) \ar[d]_{\Phi^{e^{2R}L}(\lambda_0)} \ar[rr]^{\mathrm{Id}} &&
\ech^{e^{2R}L}(Y_n, e^{\epsilon}\alpha^0_{n'}) \ar[d]^{\Phi^{e^{2R}L}(e^s\alpha^0_{n'})} \\
\ech^{e^{2R}L}(Y_n,\alpha^0_n) \ar[rr]^-{\Psi^{e^{2R}L}_{n,n'}} &&
\ech^{e^{2R}L}(Y_n,\alpha^0_{n'})
} \]
such that the compositions along the left and right columns are $\Phi^{e^{2R}L}(X^n_{01})$ and $\Phi^{e^{2R}L}(X^{n'}_{01})$ respectively; the latter claim follows from the fact that the composition of the three cobordisms on the right is exactly the cobordism $([0,R]\times Y_n, \lambda^{n'}_{01})$.  The outermost square is then exactly the commutative diagram we wanted, so the proof is complete.
\end{proof}

Thus we can drop the $n$ from the subscript and speak of well-defined maps $\Psi^{R,L}_{10}$ and $\Psi^{R,L}_{01}$.  In fact, it now follows that these maps do not depend on $J_0$ and $J_1$: given another generic tailored almost complex structure $J'_0$ on $\R \times (M)^*$, and fixing $L$, we can choose $n$ sufficiently large for both $J_0$ and $J'_0$ so that the rows of the commutative diagram
\[ \xymatrix{
\ech^L(M,\Gamma,\alpha_0,J_0) \ar[r]^-{\Phi^L_{n}} \ar[d]_{\Phi^L_{J_0,J'_0}} &
\ech^L(Y_n,\alpha^0_n) \ar[r]^-\phi \ar@{=}[d] &
\ech^{e^{R}L}(M,\Gamma,\alpha_1,J_1) \ar@{=}[d] \\
\ech^L(M,\Gamma,\alpha_0,J'_0) \ar[r]^-{(\Phi')^L_{n}} &
\ech^L(Y_n,\alpha^0_n) \ar[r]^-\phi &
\ech^{e^{R}L}(M,\Gamma,\alpha_1,J_1),
} \]
in which $\phi$ denotes the composition
\[
\ech^L(Y_n,\alpha^0_n) \xrightarrow{s\circ\Phi^{e^{2R}L}(X^n_{10})\circ s}
\ech^{e^{R}L}(Y_n, \alpha^1_n) \xrightarrow{(\Phi^{e^{R}L}_{n})^{-1}}
\ech^{e^{R}L}(M,\Gamma,\alpha_1,J_1)
\]
for some appropriate scaling maps $s$, define the maps $\Psi^{R,L}_{10}$ for $J_0$ and $J'_0$ respectively.  Note that the commutativity of the leftmost square follows from the definition of $\Phi^L_{J_0,J'_0}$ in Section~\ref{sec:j-independence}.  By the same argument we can show that given $J_1$ and $J'_1$ for $(M,\Gamma,\alpha_1)$, the corresponding maps $\Psi^{R,L}_{10}$ differ by post-composing with $\Phi^{e^{R}L}_{J_1,J'_1}$.  In this sense the maps $\Psi^{R,L}_{10}$ do not depend on either $J_0$ or $J_1$, and so we have well-defined map
\[ \Psi^{R,L}_{10}: \ech^L(M,\Gamma,\alpha_0) \to \ech^{e^{R}L}(M,\Gamma,\alpha_1) \]
independently of the almost complex structures involved, and likewise for $\Psi^{R,L}_{01}$.

For any fixed $L' > L$, we can take $n$ to be large with respect to $e^{2R}L'$, and then the corresponding maps on filtered sutured ECH are induced by ECH cobordism maps, so the inclusion axiom of \cite[Theorem 1.9]{ht2} gives us a commutative diagram
\[ \xymatrix{
\ech^{L}(M,\Gamma,\alpha_0) \ar[r]^-{\Psi^{R,L}_{10}} \ar[d]_{i^{L,L'}} &
\ech^{e^{R}L}(M,\Gamma,\alpha_1) \ar[r]^-{\Psi^{R,L}_{01}} \ar[d]^{i^{e^{R}L,e^{R}L'}} &
\ech^{e^{2R}L}(M,\Gamma,\alpha_0) \ar[d]^{i^{e^{2R}L,e^{2R}L'}} \\
\ech^{L'}(M,\Gamma,\alpha_0) \ar[r]^-{\Psi^{R,L'}_{10}} &
\ech^{e^{R}L'}(M,\Gamma,\alpha_1) \ar[r]^-{\Psi^{R,L'}_{01}} &
\ech^{e^{2R}L'}(M,\Gamma,\alpha_0).
}\]
Taking the direct limit as $L \to \infty$, we get a pair of maps
\[ \ech(M,\Gamma,\alpha_0) \xrightarrow{F} \ech(M,\Gamma,\alpha_1) \xrightarrow{G} \ech(M,\Gamma,\alpha_0). \]

\begin{proposition}
\label{prop:alpha-ind-gf}
The composition $G\circ F$ is an isomorphism.
\end{proposition}

\begin{proof}
Lemma \ref{lem:alpha-ind-composition} tells us that $G\circ F$ is a direct limit of morphisms corresponding to the rows of the commutative diagram
\[ \xymatrix{
\ech^L(M,\Gamma,\alpha_0) \ar[rr]^-{i^{L,e^{2R}L}} \ar[d]_{i^{L,L'}} &&
\ech^{e^{2R}L}(M,\Gamma,\alpha_0) \ar[d]^{i^{e^{2R}L,e^{2R}L'}} \\
\ech^{L'}(M,\Gamma,\alpha_0) \ar[rr]^-{i^{L',e^{2R}L'}} &&
\ech^{e^{2R}L'}(M,\Gamma,\alpha_0) \\
} \]
as $L \to \infty$.  Since the maps $i^{L,e^{2R}L}$ are the same inclusion-induced maps as the ones used to define the directed system $(\{\ech^L(M,\Gamma,\alpha_0)\}_L, \{i^{L,L'}\}_{L,L'})$, it follows immediately that $G\circ F$ is in fact the identity map.
\end{proof}

\begin{theorem}
\label{thm:alpha-independence}
Let $(M,\Gamma)$ be a sutured contact manifold with respect to two different nondegenerate contact forms $\alpha_0$ and $\alpha_1$ which agree on a neighborhood of $\partial M$, and suppose that $\alpha_0$ is isotopic to $\alpha_1$ through a family of contact forms which is constant on a neighborhood of $\partial M$.  Then there is an isomorphism
\[ \ech(M,\Gamma,\alpha_0) \isomto \ech(M,\Gamma,\alpha_1) \]
which a priori depends on the choice of embedding data, and which carries $c(\alpha_0)$ to $c(\alpha_1)$.
\end{theorem}

\begin{proof}
The composition $\ech(M,\Gamma,\alpha_0) \xrightarrow{F} \ech(M,\Gamma,\alpha_1) \xrightarrow{G} \ech(M,\Gamma,\alpha_0)$, which was constructed using the isotopy $\alpha_s$ from $\alpha_0$ to $\alpha_s$, is an isomorphism by Proposition \ref{prop:alpha-ind-gf}.  On the other hand, we could repeat its construction using the isotopy $\alpha_{1-s}$ from $\alpha_1$ to $\alpha_0$ instead, and it is straightforward to see that the corresponding isomorphism we would construct is
\[ \ech(M,\Gamma,\alpha_1) \xrightarrow{G} \ech(M,\Gamma,\alpha_0) \xrightarrow{F} \ech(M,\Gamma,\alpha_1). \]
Therefore $F$ and $G$ are inverses, and $F$ is the desired isomorphism.

It remains to be shown that $F(c(\alpha_0)) = c(\alpha_1)$, where we recall that each $c(\alpha_i)$ is the homology class of the empty set of Reeb orbits.  We observe that each of the maps $\Psi^{R,L}_{10}: \ech^L(M,\Gamma,\alpha_0) \to \ech^{e^{R}L}(M,\Gamma,\alpha_1)$ carries $[\emptyset]$ to $[\emptyset]$, which is a consequence of the fact that each cobordism map $\Phi^{e^{2R}L}(X^n_{10})$ does by \cite[Remark 1.11]{ht2}.  Since $F$ is the direct limit of these maps, it follows that $F([\emptyset]) = [\emptyset]$ as desired.
\end{proof}

We can now complete the proof of Theorem \ref{thm:intro-alpha-isomorphism}.

\begin{corollary}
\label{cor:alpha-independence-general}
Let $(M,\Gamma)$ be a sutured contact manifold with respect to two different contact forms $\alpha_1$ and $\alpha_2$ such that $\xi_1 = \ker(\alpha_1)$ and $\xi_2 = \ker(\alpha_2)$ are isotopic rel $N(\Gamma) = \overline{\partial M\ssm(R_+(\Gamma)\cup R_-(\Gamma))}$.  Then $\ech(M,\Gamma,\alpha_1) \cong \ech(M,\Gamma,\alpha_2)$ by an isomorphism carrying $c(\alpha_1)$ to $c(\alpha_2)$.  In other words, sutured ECH depends up to isomorphism only on $(M,\Gamma)$ and the underlying contact structure up to isotopy rel $N(\Gamma)$.
\end{corollary}

\begin{proof}
Suppose the given contact forms can be written $\alpha_i = C_i dt + \beta^i_\pm$ near $R_\pm(\Gamma)$ and $\alpha_i = C_i dt + e^\tau \beta_0^i$ on $U(\Gamma)$.  Since both ECH and the contact structures $\xi_i$ are unchanged when we multiply either $\alpha_i$ by some positive constant, we may assume that $C_1 = C_2 = C$ for some $C > 0$.  Then the contact forms restrict to the region $\{\tau = 0\} \subset U(\Gamma)$ as $Cdt + \beta_0^i$, and since the contact structures are identical here we must have $\beta^1_0 = \beta^2_0$, hence $\alpha_1|_{U(\Gamma)} = \alpha_2|_{U(\Gamma)}$.  Finally, by following the proof of Lemma \ref{lem:glue-liouville} we may perturb $(M,\alpha_2)$ in a neighborhood of $R_\pm(\Gamma)$ without changing the completion $(M^*,\alpha_2^*)$ in order to assume that $\beta_\pm^1 = \beta_\pm^2$.

We have reduced the problem to the case where $\alpha_1$ and $\alpha_2$ agree near $\partial M$, and now we can realize the isotopy from $\xi_1$ to $\xi_2$ by a family of contact forms $\alpha_s$ which also agree with these at the boundary.  It follows from Theorem \ref{thm:alpha-independence} that $\ech(M,\Gamma,\alpha_1) \cong \ech(M,\Gamma,\alpha_2)$ and that this isomorphism identifies the contact classes.
\end{proof}

%% file: isomorphism-embedding.tex

Theorem \ref{thm:alpha-independence} relies on a choice of embedding data for $(M,\Gamma,\alpha)$ in several ways.  First of all, as explained in Remark \ref{rem:canonical-transitive-systems} we had to fix a choice of embedding data in order to uniquely define the groups $\ech(M,\Gamma,\alpha_i)$ independently of a choice of generic tailored almost complex structure $J_i$; this data was used to construct each of the closed manifolds $Y_n$.  Once we had this in hand, we used the same embedding data to construct the cobordisms which define each of the maps $\Psi^{R,L}_{n,10}$, whose direct limit over $L$ is the isomorphism
\[ F: \ech(M,\Gamma,\alpha_0) \isomto \ech(M,\Gamma,\alpha_1) \]
of Theorem \ref{thm:alpha-independence}.  Our goal in this subsection is to see that while the parts of these constructions which use Seiberg--Witten Floer homology are sensitive to the topology of $Y_n$ and hence to the embedding data, the parts which make use of holomorphic curves are not.

\begin{proposition}
\label{prop:gromov-compactness-cobordisms}
Let $\alpha_s$ be a path of contact forms on $(M,\Gamma)$ which is constant near $\partial M$.  Given a choice of embedding data for $(M,\Gamma,\alpha_i)$ resulting in the closed manifold $(Y_n,\alpha^i_n)$ for $i=0,1$, and generic tailored almost complex structures $J_i$ for $\alpha_i$ on $\R \times M^\ast$, let $(\R \times Y_n, \lambda_n^{10})$ be the exact symplectic cobordism used to construct the map
\[ \Psi^{R,L}_{n,10}: \ech^L(M,\Gamma,\alpha_0,J_0) \to \ech^{e^{R}L}(M,\Gamma,\alpha_1,J_1) \]
of Section \ref{ssec:alpha-independence} for some strongly cobordism-admissible $J$ on $(\R\times Y_n, \lambda_n^{10})$.  Suppose moreover that $J$ is $t$-invariant on $\R \times ([1,n-\epsilon'] \times R_+(\Gamma))$ and on $\R\times ([-n+\epsilon',-1] \times R_-(\Gamma))$, where we glue $\{\pm 1\} \times R_{\pm}(\Gamma)$ to $R_\pm(\Gamma) \subset M$ and $\epsilon'$ is as in Section \ref{sec:j-independence}.  

Given any fixed Reeb orbit sets $\Theta_+ \in \ecc^{e^{2R}L}(M,e^{2R}\alpha_0)$ and $\Theta_- \in \ecc^{e^{2R}L}(M,e^R \alpha_1)$, there exists $k$ depending on $\alpha_0$, $\alpha_1$, the embedding data, $J$, and the orbit sets $\Theta_\pm$ such that the following is true for all $n$ which are sufficiently large in a sense depending on $k$ and the data that $k$ depends on.  Let $M_k \subset Y_n$ denote the submanifold
\[ M_k = M \cup ([1,k] \times R_+(\Gamma) ) \cup ([-k,-1] \times R_-(\Gamma)). \]
Then any broken $J$-holomorphic curve from $\Theta_+$ to $\Theta_-$ in $\R \times Y_n$ is contained in $\R \times M_k$.
\end{proposition}

\begin{proof}
We observe that any Reeb orbit of action at most $e^{2R}L$ on either end of the cobordism $[R,2R]\times Y_n$ actually lies in the corresponding copy of $M \subset M_n$.  Moreover, $d\lambda^{10}_n$ restricts to $[R,2R] \times \overline{M_n \ssm M}$ as the symplectization of $\overline{M_n \ssm M}$, since $\alpha_0$ and $\alpha_1$ agree near $\partial M$.  Thus if we fix a pair of ECH generators $\Theta_+ \in \ecc^{e^{2R}L}(M,e^{2R}\alpha_0)$ and $\Theta_- \in \ecc^{e^{2R}L}(M, e^R\alpha_1)$, the arguments of Lemmas \ref{lem:tau-stretching}, \ref{lem:t-stretching}, and \ref{lem:embedding} show that for large $n$, all broken $J$-holomorphic curves from $\Theta_+$ to $\Theta_-$ are confined to some fixed $\R \times M_k$ in the completed cobordisms $\R \times Y_n$.

To be more precise, the proofs of Lemmas \ref{lem:tau-stretching} and \ref{lem:embedding} work exactly as before but we must be careful in repeating the Gromov compactness argument of Lemma \ref{lem:t-stretching}, which shows that for large $n$ no holomorphic curves from $\Theta_+$ to $\Theta_-$ can project onto the region denoted $[-\epsilon',\epsilon']_\ft \times R_+(\Gamma')$.  The potential problem is that we cannot guarantee that the curves $C_{n_i}''$ constructed in the proof are all holomorphic with respect to a single almost complex structure on $\R_s\times\R_\ft\times R_+(\Gamma')$, since our $J$ is only guaranteed to be $\R_s$-invariant on the regions $s<R$ and $s>2R$.  If we can pass to a subsequence such that either $a_i\to\infty$ or $a_i\to -\infty$, then we can modify our choices of $s_i$ as needed to make sure that $s_i < |a_i|-2R$ while still ensuring that $s_i \to \infty$; then the curves $C''_{n_i}$ all come from a region of $\R_s\times\R\times R_+(\Gamma')$ (either $s>2R$ or $s<R$) where $J$ is independent of $s$, and we can proceed as before.

In the remaining case, the numbers $a_i$ are all bounded, say $|a_i| < A$, and by passing to a subsequence we can assume that they converge to some value $A_0$.  Then we can replace the intervals $[a_i-s_i,a_i+s_i]$ with the larger $[-A-s_i,A+s_i]$ (in other words, replace $a_i$ and $s_i$ with $0$ and $s_i+A$).  Then the resulting curves $C''_{n_i}$ are still nonempty, with $d(\alpha')^*$-energy converging to zero, and since each $[-A-s_i,A+s_i]$ is centered at $s=0$ they are all holomorphic with respect to the same almost complex structure on $\R\times\R\times R_+(\Gamma')$; moreover, all but finitely many of them intersect any given neighborhood of $\{(A_0,0)\}\times R_+(\Gamma')$.  Now we can repeat the rest of the proof of Lemma \ref{lem:t-stretching} and conclude that all of the holomorphic curves from $\Theta_+$ to $\Theta_-$ are confined to $\R \times M_k$ in the completed cobordism, as desired.
\end{proof}

\begin{theorem}
\label{thm:h1-splitting}
There is a natural splitting
\[ \ech(M,\Gamma,\alpha) = \bigoplus_{h\in H_1(M)} \ech(M,\Gamma,\alpha;h) \]
such that $c(\alpha) \in \ech(M,\Gamma,\alpha; 0)$, and each isomorphism $F$ constructed as in Theorem \ref{thm:alpha-independence} restricts to an isomorphism
\[ F_{\alpha_s;h}: \ech(M,\Gamma,\alpha_0; h) \isomto \ech(M,\Gamma,\alpha_1; h) \]
for every $h \in H_1(M)$.
\end{theorem}

\begin{proof}
For $h\in H_1(M)$ we define the subcomplex $\ecc(M,\Gamma,\alpha,J;h) \subset \ecc(M,\Gamma,\alpha,J)$ to be generated by all Reeb orbit sets $\Theta = \{(\Theta_i,m_i)\}$ with total homology class $h = \sum_i m_i[\Theta_i]$.  We recall that this is in fact a subcomplex because if $\langle \partial\Theta_+, \Theta_- \rangle \neq 0$, then there is a $J$-holomorphic curve from $\Theta_+$ to $\Theta_-$ in $\R\times M^\ast$ whose projection to $M^\ast$ shows that $\Theta_+$ is homologous to $\Theta_-$ in $M^\ast$, hence in $M$.  This provides the desired splitting of $\ech(M,\Gamma,\alpha,J)$, and our goal is to show that the splitting is compatible with the canonical isomorphisms $\Phi_{J,J'}$ of Theorem \ref{thm:j-independence} and the isomorphisms $F$ of Theorem \ref{thm:alpha-independence}.

We now recall that the maps $\Psi^{R,L}_{10}$ are independent of the almost complex structures.  If we consider the path of contact structures $\alpha_s = \alpha$ for $0\leq s \leq 1$, $R=1$, and $\psi = \mathrm{Id}_{[0,1]}$, then this means that the map
\begin{equation}
\label{eq:psi-constant-alpha}
\Psi^{R,L}_{10}: \ech^L(M,\Gamma,\alpha,J) \to \ech^{e^R L}(M,\Gamma,\alpha,J')
\end{equation}
for an appropriate cobordism-admissible almost complex structure is equal to the composition
\[ \ech^L(M,\Gamma,\alpha,J) \xrightarrow{\Phi^L_{J,J'}} \ech^L(M,\Gamma,\alpha,J') \xrightarrow{\Psi^{R,L}_{10}} \ech^{e^{R}L}(M,\Gamma,\alpha,J') \]
in which the latter cobordism has a product almost complex structure, and so the latter map is equal to the inclusion-induced $i^{L,e^{R}L}_{J'}$, again by \cite[Corollary 5.8]{ht2}.  In particular, the limit of the maps $i^{L,e^{R}L}_{J'}$ is the identity map on $\ech(M,\Gamma,\alpha,J')$, hence the limit of \eqref{eq:psi-constant-alpha} as $L\to\infty$ is just the canonical isomorphism
\[ \Phi_{J,J'}: \ech(M,\Gamma,\alpha,J) \to \ech(M,\Gamma,\alpha,J'). \]

Now supposing that the matrix coefficient $\langle \Phi_{J,J'}(\Theta_+), \Theta_-\rangle$ is nonzero for some Reeb orbit sets $\Theta_+$ and $\Theta_-$, the corresponding coefficient must be nonzero in one of the maps \eqref{eq:psi-constant-alpha}.  The Holomorphic Curves axiom of \cite[Theorem 1.9]{ht2} asserts the existence of a broken holomorphic curve $C$ from $\Theta_+$ to $\Theta_-$ in the corresponding cobordism $\R \times Y_n$, and following the proof of Proposition \ref{prop:gromov-compactness-cobordisms} we see that if $n$ is sufficiently large then $C$ is confined to some $\R\times M_k$.  It follows that $[\Theta_+]=[\Theta_-]$ in $H_1(M_k)$, and since $M_k$ retracts onto $M$ this holds in $H_1(M)$ as well.  Thus $\Phi_{J,J'}$ restricts to a map $\Phi_{J,J';h}: \ech(M,\Gamma,\alpha,J;h) \to \ech(M,\Gamma,\alpha,J';h)$, with $\Phi_{J,J'} = \bigoplus_h \Phi_{J,J';h}$, and the resulting transitive systems of groups produce the $\ech(M,\Gamma,\alpha;h)$.

The splitting of the isomorphisms $F$ associated to a path $\{\alpha_s\}$ of contact forms follows by an identical argument: if $\langle F(\Theta_+), \Theta_-\rangle \neq 0$, then there is a broken holomorphic curve from $\Theta_+$ to $\Theta_-$ in the associated cobordism $\R \times Y_n$, hence in some $\R \times M_k$ for $n$ large enough, and this suffices to show that $[\Theta_+] = [\Theta_-]$ in $H_1(M)$.
\end{proof}

If the cobordism maps $\Psi^{R,L}_{10}$ were constructed by some count of broken holomorphic curves, then Proposition \ref{prop:gromov-compactness-cobordisms} could be used to show that these maps are independent of the original choice of embedding data.  Indeed, we can take the symplectic cobordisms $\R \times Y_n$ and $\R \times Y'_n$, with appropriate almost complex structures $J$ and $J'$ as in the proposition such that $J$ and $J'$ agree on $\R \times M_{n-\epsilon'}$, viewed canonically as a submanifold of both $\R \times Y_n$ and $\R\times Y'_n$.  Then for any fixed $\Theta_+$ and $\Theta_-$, and sufficiently large $n$, there is a canonical identification
\begin{equation}
\label{eq:embedding-data-same-curves}
\mathcal{M}(\R \times Y_n, J; \Theta_+,\Theta_-) = \mathcal{M}(\R\times Y'_n, J'; \Theta_+,\Theta_-)
\end{equation}
of the moduli spaces of holomorphic curves from $\Theta_+$ to $\Theta_-$ in either cobordism, since in both cases all of the curves involved lie in some appropriate $\R \times M_k$.

Unfortunately, the maps from \cite[Theorem 1.9]{ht2} which we used to construct the $\Psi^{R,L}_{10}$ are defined not by counting holomorphic curves, but by identifying ECH with Seiberg--Witten Floer cohomology and using the corresponding Seiberg--Witten cobordism maps, so \eqref{eq:embedding-data-same-curves} will not suffice to show that the two maps $\Psi^{R,L}_{10}$ are the same.  However, in light of \eqref{eq:embedding-data-same-curves} and the canonical isomorphisms between ECH and Seiberg--Witten Floer homology, the following conjecture would immediately imply that the maps $\Phi^{R,L}_{10}$ are in fact independent of the embedding data, and so is their direct limit $F: \ech(M,\Gamma,\alpha_0) \to \ech(M,\Gamma,\alpha_1)$.

\begin{conjecture}
\label{conj:hm-curves-diagram}
Fix $L>0$.  Let $(X_i,\lambda_i)$ be exact symplectic cobordisms from $(Y^i_-,\lambda^i_-)$ to $(Y^i_+,\lambda^i_+)$, with cobordism-admissible almost complex structures $J_i$, for $i=0,1$.  Suppose we have a 4-dimensional manifold $W$ with corners and a $1$-form $\lambda_W$, and disjoint embedded $3$-manifolds with corners $Z_\pm \subset \partial W$, satisfying the following:
\begin{itemize}\leftskip-0.35in
\item There are embeddings $f_i: W \hookrightarrow X_i$ for $i=0,1$ such that $f_i(Z_\pm) \subset Y^i_\pm$ and $\lambda_i|_W = \lambda_W$, so in particular $\lambda^i_\pm|_{Z_\pm} = \lambda_W|_{Z_\pm}$.
\item All Reeb orbits of action less than $L$ in $(Y^i_\pm,\lambda^i_\pm)$ lie in $f_i(Z_\pm)$.
\item All broken $J_i$-holomorphic curves between ECH generators of action less than $L$ in the completed cobordisms $\overline{X}_i$ have image in $\R\times\mathrm{int}(f_i(W))$.
\end{itemize}
Then there is a commutative diagram involving filtered Seiberg--Witten Floer cohomology:
\[ \xymatrix{
\HMfrom^*_{L}(Y^0_+; \lambda^0_+, J^0_+,r)
\ar[rrrr]^{\HMfrom^*_{L}(X_0;\lambda_0,J_0,r)} \ar[d]_-{\rotatebox{90}{$\sim$}} &&&&
\HMfrom^*_{L}(Y^0_-; \lambda^0_-, J^0_-,r) \ar[d]^-{\rotatebox{90}{$\sim$}}
\\
\HMfrom^*_{L}(Y^1_+; \lambda^1_+, J^1_+,r) \ar[rrrr]^{\HMfrom^*_{L}(X_1;\lambda_1,J_1,r)} &&&&
\HMfrom^*_{L}(Y^1_-; \lambda^1_-, J^1_-,r)
} \]
in which the horizontal maps are defined by the chain maps of \cite[Equation (4.17)]{ht2}, and the vertical isomorphisms are defined by the canonical identification of monopoles with ECH generators of action at most $L$ inside $Z_\pm$.
\end{conjecture}

Although this conjecture is open, Appendix \ref{sec:appendix} proves enough cases of it that we will be able to remove the potential dependence on the embedding data.  We will return to this in Section~\ref{sec:stabilize}.

%% file: alpha-naturality.tex

The construction of the isomorphism $F: \ech(M,\Gamma,\alpha_0) \isomto \ech(M,\Gamma,\alpha_1)$ of Theorem \ref{thm:alpha-independence} involved several choices beyond the embedding data, and it is reasonable to ask whether $F$ depends on any of them.  In this section we will show that many of these choices don't matter and thus determine to what extent we can speak of a well-defined group $\ech(M,\Gamma,\xi)$, rather than just $\ech(M,\Gamma,\alpha)$.

We observe that in addition to embedding data, the isomorphism $F$ potentially depends on the constant $R > 0$, the isotopy $\{\alpha_s \mid 0 \leq s \leq 1\}$ of contact forms, and the diffeomorphism $\psi: [0,R] \to [0,1]$ which we used to construct the Liouville form $\lambda_{01} = e^s\alpha_{\psi(s)}$ on $[0,R] \times Y_n$.  Thus for now we will write $F = F_{\alpha_s}^{R,\psi}$.  We will once again fix a choice of embedding data from the beginning.

Our first goal is to show that $F^{R,\psi}_{\alpha_s}$ does not depend on a sufficiently large choice of $R$.

\begin{lemma}
\label{lem:alpha-naturality-rescale}
For any $R' > R$, we have $F_{\alpha_s}^{R',\psi(Rs/R')} = F_{\alpha_s}^{R,\psi}$.
\end{lemma}

\begin{proof}
For convenience we will work with the analogous maps $G_{\alpha_s}^{R,\psi}$ instead and use the assertion of Theorem \ref{thm:alpha-independence} that these maps are the inverses of $F_{\alpha_s}^{R,\psi}$.  Form a 1-parameter family of Liouville forms $\lambda_t$ on $[0,R'] \times Y_n$ by the formula
\[ \lambda_t = \begin{cases}
e^s \alpha_{\psi\left(\frac{Rs}{(1-t)R+tR'}\right)}, & 0 \leq s \leq (1-t)R+tR' \\
e^s \alpha_1, & (1-t)R+tR' \leq s \leq R'.
\end{cases} \]
The Liouville forms provide exact symplectic cobordisms and hence cobordism maps
\begin{align*}
\Psi_{\lambda_0} &: \ech^{e^{R'}L}(Y_n,e^{R'}\alpha^1_n) \to \ech^{e^{R'}L}(Y_n,e^{R}\alpha^1_n) \to \ech^{e^{R'}L}(Y_n,\alpha^0_n) \\
\Psi_{\lambda_1} &: \ech^{e^{R'}L}(Y_n,e^{R'}\alpha^1_n) \to \ech^{e^{R'}L}(Y_n, \alpha^0_n)
\end{align*}
which are equal because $\lambda_0$ and $\lambda_1$ are homotopic.  If $n$ is big enough with respect to $e^{R'}L$ to provide isomorphisms $(\Phi_i)^{e^{R'}L}_{n}: \ech^{e^{R'}L}(M,\Gamma,\alpha_i,J_i) \isomto \ech^{e^{R'}L}(Y_n,\alpha^i_n)$ for some $J_i$, then the $\Psi_{\lambda_i}$ induce equal maps
\[ \xymatrix{
\ech^{L}(M,\Gamma,\alpha_1) \ar[rr]^-{i^{L,e^{R'-R}L}} &&
\ech^{e^{R'-R}L}(M,\Gamma,\alpha_1) \ar[r] &
\ech^{e^{R}(e^{R'-R}L)}(M,\Gamma,\alpha_0),
} \]
in which the left arrow is an inclusion-induced map by \cite[Corollary 5.8]{ht2} since it corresponds to the symplectization $([R,R'] \times Y_n, e^s\alpha_1)$, and 
\[ \ech^L(M,\Gamma,\alpha_1) \to \ech^{e^{R'}L}(M,\Gamma,\alpha_0), \]
respectively.  Since these maps are identical, so are their direct limits as $L\to\infty$; the direct limit of the first map is the composition
\[ \xymatrix{
\ech(M,\Gamma,\alpha_1) \ar[r]^{\mathrm{Id}} &
\ech(M,\Gamma,\alpha_1) \ar[r]^{G_{\alpha_s}^{R,\psi}} &
\ech(M,\Gamma,\alpha_0),
} \]
while the direct limit of the second map is $G_{\alpha_s}^{R',\psi(Rs/R')}$.  We conclude that $G_{\alpha_s}^{R,\psi} = G_{\alpha_s}^{R',\psi(Rs/R')}$, and thus $F_{\alpha_s}^{R,\psi} = F_{\alpha_s}^{R',\psi(Rs/R')}$ as well.
\end{proof}

\begin{proposition}
\label{prop:alpha-naturality-psi}
The isomorphisms $F_{\alpha_s}^{R,\psi}$ do not depend on $R$ or on $\psi: [0,R] \isomto [0,1]$.
\end{proposition}

\begin{proof}
Given $F_{\alpha_s}^{R_0,\psi_0}$ and $F_{\alpha_s}^{R_1,\psi_1}$, we may pick some $R \gg R_0,R_1$ and rescale $\psi_0$ and $\psi_1$ to $\psi_0(R_0s/R)$ and $\psi_1(R_1s/R)$ respectively.  By Lemma \ref{lem:alpha-naturality-rescale}, it suffices to show that $F_{\alpha_s}^{R,\psi_0(R_0s/R)} = F_{\alpha_s}^{R,\psi_1(R_1s/R)}$, so without loss of generality we can assume that $R_0$ and $R_1$ are both equal to $R$.  Furthermore, given any $\epsilon > 0$ we can choose $R$ sufficiently large so that $0 \leq \psi'_0(s), \psi'_1(s) < \epsilon$ for all $s$.

Now consider the straight line homotopy $\psi_t(s) = (1-t)\psi_0(s) + t\psi_1(s)$ from $\psi_0$ to $\psi_1$, and let $\lambda_t$ denote the $1$-form $e^s\alpha_{\psi_t(s)}$ on $[0,R] \times M$.  For all $t$, $0 \leq t \leq 1$, the map $\psi_t$ is a diffeomorphism $[0,R] \isomto [0,1]$, and given $\epsilon$ as above we have $0 \leq \psi'_t(s) < \epsilon$ as well.  We extend each $\psi_t: [0,R] \isomto [0,1]$ to a function $\psi_t: [0.2R] \to [0,1]$ by reflection across $s=R$, as in Lemma \ref{lem:cobordism-family}, and now we have $|\psi'_t(s)| < \epsilon$ instead.  Thus if $\epsilon$ is sufficiently small, for any $t$ the form $d\lambda_t \wedge d\lambda_t$ will be arbitrarily close to the volume form $2e^{2s}ds \wedge d\alpha_{\psi_t(s)} \wedge d\alpha_{\psi_t(s)}$ on $[0,2R] \times M$ and hence a volume form as well, so each $\lambda_t$ will be a Liouville form.

For each $n$ we now build the contact manifolds $(Y_n,\alpha^i_n)$ by extending $\alpha_0$ and $\alpha_1$ over $\overline{Y_n \ssm M}$ in the same way, and then we extend $\lambda_t$ over $[0,2R]\times Y_n$ by using the symplectization of $\alpha^0_n|_{\overline{Y_n \ssm M}}$.  Then the homotopy $\lambda_t$ extends to a homotopy of exact symplectic cobordisms
\[ ([R,2R] \times Y_n, \lambda_0) \simeq ([R,2R] \times Y_n, \lambda_1), \]
for each $n$, hence by \cite[Theorem 1.9]{ht2} the induced maps are equal.  But we use the $\lambda_0$ cobordisms to construct $F_{\alpha_s}^{R,\psi_0}$ and the $\lambda_1$ cobordisms to construct $F_{\alpha_s}^{R,\psi_1}$ in the same way, so we must have $F_{\alpha_s}^{R,\psi_0} = F_{\alpha_s}^{R,\psi_1}$ as desired.
\end{proof}

Since the isomorphisms $F_{\alpha_s}^{R,\psi}$ are independent of $R$ and $\psi$, we can drop them from the notation and simply write $F_{\alpha_s}$.

\begin{proposition}
\label{prop:F-alpha-homotopy-class}
The isomorphisms $F_{\alpha_s}: \ech(M,\Gamma,\alpha_0) \to \ech(M,\Gamma,\alpha_1)$ only depend on the homotopy class of the path $\alpha_s$ rel boundary.
\end{proposition}

\begin{proof}
Suppose we have two such paths $\alpha_s^0$ and $\alpha_s^1$ which are related by a homotopy $t \mapsto \alpha_s^t$.  Then just as in the proof of Proposition \ref{prop:alpha-naturality-psi} for sufficiently large $R$ we can find a $\psi:[0,R] \to [0,1]$, extended by symmetry to the domain $[0,2R]$, such that for all $t$, the $1$-form $e^s\alpha^t_{\psi(s)}$ is the primitive of a symplectic form on $[R,2R] \times M$ which extends to a Liouville form $\lambda_t$ on each $[R,2R] \times Y_n$.  Since $\lambda_t$ provides a homotopy between the Liouville forms $\lambda_0$ and $\lambda_1$, it follows that $F_{\alpha^0_s} = F_{\alpha^1_s}$ as before.
\end{proof}

Finally, we observe that if the path $\alpha_s$ is constant then $F_{\alpha_s} = \mathrm{Id}_{\ech(M,\Gamma,\alpha_0)}$, since the maps
\[ \Psi^{R,L}_{10}: \ech^L(M,\Gamma,\alpha_0) \to \ech^{e^{R}L}(M,\Gamma,\alpha_0) \]
whose direct limit as $L\to \infty$ defines $F_{\alpha_s}$ are the inclusion-induced maps $i^{L,e^{R}L}$.  It is also clear that given two paths $\alpha_s^0$ and $\alpha_s^1$ with $\alpha_0^1 = \alpha_1^0$, the map associated to the concatenation of these paths is simply the composition $F_{\alpha_s^1} \circ F_{\alpha_s^0}$.  

\begin{lemma}
\label{lem:f-alpha-independent}
Let $\Lambda(M,\Gamma,\alpha)$ denote the space of contact forms which agree with $\alpha$ on some neighborhood of $\partial M$ and have the same kernel, viewed as a cooriented contact structure.  Then there are canonical isomorphisms
\[ F_{\alpha_0,\alpha_1}: \ech(M,\Gamma,\alpha_0) \to \ech(M,\Gamma,\alpha_1) \]
for any $\alpha_0,\alpha_1\in \Lambda(M,\Gamma,\alpha)$.  Moreover, these isomorphisms compose naturally and hence canonically define a group
\[ \ech(M,\Gamma,\xi,\alpha|_{\partial M}). \]
These isomorphisms and the resulting group depend only on a choice of embedding data.
\end{lemma}

\begin{proof}
Any $\alpha_0, \alpha_1 \in \Lambda(M,\Gamma,\alpha)$ have the form $f_i\alpha$ for some positive functions $f_i: M\to \R$ which are equal to $1$ on a neighborhood of $\partial M$, so $\Lambda(M,\Gamma,\alpha)$ is convex and hence contractible.  In particular, any two paths $\alpha_s$ and $\alpha'_s$ from $\alpha_0$ to $\alpha_1$ inside $\Lambda(M,\Gamma,\alpha)$ are homotopic, so they define identical isomorphisms
\[ F_{\alpha_s} = F_{\alpha'_s}: \ech(M,\Gamma,\alpha_0) \isomto \ech(M,\Gamma,\alpha_1) \]
by Proposition \ref{prop:F-alpha-homotopy-class}.  We can thus use the straight line homotopy $\alpha_s = (1-s)\alpha_0 + s\alpha_1$ to define a canonical isomorphism which only depends on the endpoints, denoted $F_{\alpha_0,\alpha_1}$.

Given a third $\alpha_2 \in \Lambda(M,\Gamma,\alpha)$, it follows immediately from that
\[ F_{\alpha_0,\alpha_2} = F_{\alpha_1,\alpha_2} \circ F_{\alpha_0,\alpha_1}, \]
and so we have a transitive system of isomorphisms
\[ \left( \{ \ech(M,\Gamma,\alpha_0) \}_{\alpha\in \Lambda(M,\Gamma,\alpha)}, \{F_{\alpha_0,\alpha_1}\}_{\alpha_0,\alpha_1 \in \Lambda(M,\Gamma,\alpha)} \right). \]
This canonically defines a group $\ech(M,\Gamma,\xi,\alpha|_{\partial M})$ up to the choice of embedding data, as desired.
\end{proof}

We remark that the contact class defines an element $c(\xi) \in \ech(M,\Gamma,\xi,\alpha|_{\partial M})$, since the isomorphisms $F_{\alpha_0,\alpha_1}$ send $c(\alpha_0)$ to $c(\alpha_1)$.

\begin{lemma}
\label{lem:f-xi-independent}
Let $\Xi(M,\Gamma,\alpha|_{\partial M})$ be the space of all contact structures on $(M,\Gamma)$ admitting contact forms which agree with $\alpha$ on a neighborhood of $\partial M$.  For any isotopy $\xi_s$ of contact structures in $\Xi(M,\Gamma,\alpha|_{\partial M})$, there is a canonical isomorphism
\[ F_{\xi_s}: \ech(M,\Gamma,\xi_0,\alpha|_{\partial M}) \isomto \ech(M,\Gamma,\xi_1,\alpha|_{\partial M}) \]
depending only on the homotopy class of $\xi_s$ and sending $c(\xi_0)$ to $c(\xi_1)$.
\end{lemma}

\begin{proof}
Suppose we have two different contact forms $\alpha_0,\alpha'_0$ for $\xi_0$ and likewise $\alpha_1,\alpha'_1$ for $\xi_1$ such that all four contact forms agree with $\alpha$ near $\partial M$.  Let $\alpha_s$ and $\alpha'_s$ be families of contact forms in $\Lambda(M,\Gamma,\alpha)$ from $\alpha_0$ to $\alpha_1$ and $\alpha'_0$ to $\alpha'_1$, respectively, such that $\ker(\alpha_s) = \ker(\alpha'_s)$ for all $s$.  Then there are positive functions $g_s: M\to\R$ such that $g_s = 1$ on a neighborhood of $\partial M$ and $\alpha'_s = g_s\alpha_s$.
We wish to show that the diagram
\[ \xymatrix{
\ech(M,\Gamma,\alpha_0) \ar[r]^{F_{\alpha_s}} \ar[d]_{F_{\alpha_0,\alpha'_0}} &
\ech(M,\Gamma,\alpha_1) \ar[d]^{F_{\alpha_1,\alpha'_1}} \\
\ech(M,\Gamma,\alpha'_0) \ar[r]^{F_{\alpha'_s}} &
\ech(M,\Gamma,\alpha'_1)
} \]
commutes, where the vertical isomorphisms are the canonical ones from Lemma \ref{lem:f-alpha-independent}.

The composition of the top and right arrows is the isomorphism induced by a path from $\alpha_0$ to $\alpha'_1 = g_1 \alpha_1$; this path is defined by going from $\alpha_0$ to $\alpha_1$ via $\alpha_s$, and then from $\alpha_1$ to $g_1\alpha_1$ via a straight line homotopy.  Likewise, the composition of the left and bottom arrows is induced by a path from $\alpha_0$ to $\alpha'_0=g_0\alpha_0$ via a straight line homotopy, and then to $\alpha'_1$ via $\alpha'_s = g_s\alpha_s$.  There is a natural fibration $\Lambda(M,\Gamma,\alpha) \to \Xi(M,\Gamma,\alpha)$ whose fibers are the contractible spaces of positive functions $M\to\R$ which are $1$ near $\partial M$, and these compositions of paths in $\Lambda(M,\Gamma,\alpha)$ project to homotopic paths in the base, hence they are homotopic in $\Lambda(M,\Gamma,\alpha)$ as well.  Proposition \ref{prop:F-alpha-homotopy-class} says that they induce the same isomorphisms, so the diagram commutes as desired.
\end{proof}

Putting all of these results together finishes the proof of Theorem \ref{thm:intro-alpha-naturality}, up to showing that the construction is independent of the choice of embedding data:

\begin{theorem}
\label{thm:alpha-independence-natural}
Given a sutured contact manifold $(M,\Gamma,\alpha)$, fix a choice of embedding data and let $\Xi(M,\Gamma,\alpha)$ denote the space of all cooriented contact structures which agree with $\ker(\alpha)$ on a neighborhood of $\partial M$.  Then the assignment of $\ech(M,\Gamma,\xi,\alpha|_{\partial M})$ to each $\xi \in \Xi(M,\Gamma,\alpha)$,
together with the map
\[ \big(\xi_s: [0,1] \to \Xi(M,\Gamma,\alpha)\big) \mapsto \big(F_{\xi_s}: \ech(M,\Gamma,\xi_0,\alpha|_{\partial M}) \isomto \ech(M,\Gamma,\xi_1,\alpha|_{\partial M})\big), \]
defines a local system on $\Xi(M,\Gamma,\alpha)$ which satisfies $F_{\xi_s}(c(\xi_0)) = c(\xi_1)$ for every path $\xi_s$.  Moreover, it is independent of all choices except possibly the embedding data.
\end{theorem}

Once we have proved in Section \ref{sec:stabilize} that the maps $\Psi^{R,L}_{n,10}: \ech^L(M,\Gamma,\alpha_0,J_0) \to \ech^{e^{R}L}(M,\Gamma,\alpha_1,J_1)$ of Section \ref{ssec:alpha-independence} do not depend on embedding data, it should be possible to prove that the isomorphism of Corollary \ref{cor:alpha-independence-general} will be natural.  We will not prove this claim here.

%% file: contact-class.tex

Now that we know that the contact class $c(\xi) \in \ech(M,\Gamma,\xi,\alpha|_{\partial M})$ is well-defined, and that isotopies rel boundary of $\xi$ induce isomorphisms on ECH which preserve $c(\xi)$, we can investigate some of the properties of $c(\xi)$.  Although the sutured ECH contact class had not been explicitly written down before, the definitions and theorems in this section are straightforward adaptations of results which are known for closed 3-manifolds and which should not be surprising to experts. We remark that Theorem \ref{thm:alpha-independence} will suffice for these applications, and thus we can choose to omit $\alpha|_{\partial M}$ from the notation, because vanishing and nonvanishing results will only require us to understand ECH up to isomorphism.  We begin by proving that $c(\xi)$ detects tightness. 

\begin{theorem}
\label{thm:overtwisted}
Suppose that $(M,\Gamma,\xi)$ is a sutured contact manifold for which $\xi$ is overtwisted.  Then $c(\xi) = 0$.
\end{theorem}

\begin{proof}
Let $\Delta \subset \mathrm{int}(M)$ be an overtwisted disk and $U$ an open neighborhood of $\Delta$ which is disjoint from an open neighborhood of $\partial M$.  If we perform a full Lutz twist on a transverse knot in $U \ssm \Delta$, then the resulting contact structure $\xi'$ is homotopic to $\xi$ as a $2$-plane field by a homotopy supported in a neighborhood of that knot in $U\ssm \Delta$, and in particular this homotopy is constant on $M\ssm U$ and on a neighborhood of $\Delta$.  A theorem of Eliashberg \cite[Theorem 3.1.1]{eliashberg-ot} now says that $\xi$ and $\xi'$ are isotopic as contact structures by an isotopy supported in $U$.

The same argument works if we perform an $n$-fold Lutz twist, so for any given $n$ we may actually assume that $\xi'$ contains such a Lutz tube inside $U$.  Since $\xi$ and $\xi'$ are isotopic rel $M\ssm U$, we can find a family $\{\alpha_s\}$ of contact forms constant on a neighborhood of $\partial M$ such that $\xi = \ker(\alpha_0)$, $\xi' = \ker(\alpha_1)$, and $\alpha_1$ equals a smooth perturbation of $\cos(r)dz + \sin(nr)d\theta$ on the Lutz tube $S^1_z\times D^2_{r,\theta}$.  Eliashberg \cite[Appendix]{yau-overtwisted} shows that if $n$ is large, then we can replace $\alpha_1$ by a suitable Morse-Bott perturbation containing a nondegenerate Reeb orbit $\gamma$ such that $\partial \gamma = \emptyset$, and hence $c(\alpha_1) = [\emptyset]$ vanishes.  The isomorphism
\[ \ech(M,\Gamma,\xi,\alpha|_{\partial M}) \isomto \ech(M,\Gamma,\xi',\alpha|_{\partial M}) \]
of Theorem \ref{thm:alpha-independence-natural} constructed from $\{\xi_s=\ker(\alpha_s)\}$ sends $c(\xi)$ to $c(\xi')=0$, and thus $c(\xi)=0$ as well.
\end{proof}

In fact, overtwisted disks are just the simplest of a family of objects which force the contact class to vanish, namely the planar torsion domains studied by Wendl \cite{wendl-hierarchy}, to which we refer for full details.

\begin{definition}[{\cite[Definition 2.13]{wendl-hierarchy}}]
For any integer $k\geq 0$, a \emph{planar $k$-torsion domain} is a contact manifold $(M,\xi)$, possibly with boundary, supported by a blown-up summed open book which is not a symmetric summed open book and containing an irreducible subdomain $M^P \subset M$ such that $M^P$ does not touch $\partial M$, $\overline{M \ssm M^P}$ is nonempty, and $M^P$ contains planar pages with $k+1$ boundary components.

We say that $(M,\xi)$ has \emph{planar $k$-torsion} if there is a contact embedding of a planar $k$-torsion domain into $(M,\xi)$.
\end{definition}

Wendl showed that a closed contact $3$-manifold has planar $0$-torsion if and only if it is overtwisted, and that if a closed contact manifold has Giroux torsion then it also has planar $1$-torsion \cite[Theorem 3]{wendl-hierarchy}.  He also showed that if a closed contact manifold $(Y,\xi)$ contains planar $k$-torsion for any $k$, then its ECH contact invariant $c(\xi)$ vanishes \cite[Theorem 2]{wendl-hierarchy}; we claim that this applies to the contact class in sutured ECH as well.

\begin{theorem}
\label{thm:planar-torsion}
Let $(M,\Gamma,\xi)$ be a sutured contact manifold for which $(M,\xi)$ has planar torsion.  Then $c(\xi) = 0$.
\end{theorem}

\begin{proof}
The proof is exactly as in \cite[Section 4.2.2]{wendl-hierarchy}: let $M_0 \subset M$ be a planar $k$-torsion domain with planar piece $M_0^P$, choosing $k$ as small as possible.  Then there is a contact form $\alpha'$ with $\xi'=\ker(\alpha')$ isotopic to $\xi$, and a compatible almost complex structure $J$ on $\R \times M^*$ provided by \cite[Theorem 7]{wendl-hierarchy}, for which (after a Morse--Bott perturbation) one can explicitly find an ECH orbit set $\bf{\gamma}_0$ supported on $M_0^P$ such that $\partial \bf{\gamma}_0 = \emptyset$, hence $c(\alpha')=0$.  Since $M^P_0$ avoids $\partial M$, we can arrange $\xi'$ to agree with $\xi$ near $\partial M$ and apply Theorem \ref{thm:alpha-independence-natural} to conclude that $c(\xi) = 0$.
\end{proof}

We can also prove some nonvanishing results for the contact class:

\begin{theorem}
\label{thm:stein-fillable}
Let $Y$ be a closed $3$-manifold with Stein fillable contact structure $\xi$.  Then the contact invariant $c(\xi|_{Y(1)}) \in \ech(Y(1),\xi|_{Y(1)})$ is nonzero.
\end{theorem}

\begin{proof}
Write $(Y,\xi)$ as the boundary of a Stein domain $(X,J)$ with strictly plurisubharmonic exhaustion function $\phi: X \to \R$.  We can write $Y=\phi^{-1}(c)$ for some regular value $c$ of $\phi$, and if $\lambda = -d\phi \circ J$ then we can take the contact form on $Y$ to be $\alpha = \lambda|_Y$.  Without loss of generality we may assume that $\phi$ has isolated critical points and a unique global minimum at $p$, and fix $b > \phi(p)$ such that $p$ is the only critical point in $\phi^{-1}((-\infty,b])$, which is therefore a $4$-ball.  We will let $W = \phi^{-1}([b,c])$.

It now follows that $(W,\lambda)$ is an exact symplectic cobordism from $(S^3,\eta)$ to $(Y,\alpha)$, where $\eta = \lambda|_{S^3}$ and $\xi_{\mathrm{std}} = \ker(\eta)$ is the standard tight contact structure on $S^3$.  It has Liouville vector field $\nabla\phi$ with respect to the metric $d\lambda(\cdot,J\cdot)$.  Thus we can find a point $y \in Y$ such that the flow line of $-\nabla\phi$ through $y$ avoids all of the critical points of $\phi$ and hence reaches $S^3 \subset \partial W$ in time $c-b$.  Since this condition is open in $Y$, we can also find a sufficiently small Darboux ball $B\subset Y$ around $y$ such that the $-\nabla\phi$-flow through any point of $B$ reaches $S^3$ in the same time.  Thus we have an embedding
\[ ([b,c]_t \times B, e^{t-c}\alpha|_B) \hookrightarrow (W,\lambda) \]
which identifies the Liouville vector fields $\partial_t$ and $\nabla\phi$, and satisfies $\{b,c\} \times B \subset \partial W$.

Since the boundary of $(\overline{Y\ssm B}, \alpha)$ is not the boundary of a sutured contact manifold but is simply a convex surface, we must now modify $\overline{Y\ssm B}$ near its boundary as in \cite[Lemma 4.1]{cghh} to get the sutured contact manifold $Y(1)$.  Applying an identical modification to each $\{t\} \times B$ for $t \in [b,c]$, we turn $\tilde{W} = \overline{W \ssm ([b,c]\times B)}$ into a cobordism with corners from $S^3(1)$ to $Y(1)$.  The boundary of $\tilde{W}$ consists of $-S^3(1) \sqcup Y(1)$ together with a horizontal component $[b,c]_t \times \partial Y(1)$, along which the Liouville vector field $\nabla\phi$ is identified with $\partial_t$.

Now suppose that $c(\alpha|_{Y(1)})$ is zero, and fix a generic almost complex structure $J_Y$ tailored to $(Y(1),\alpha|_{Y(1)})$. Then there is a relation
\[ \sum c_i \partial \Theta_i = \emptyset \]
for some $\sum c_i\Theta_i\in\ecc(Y(1),\alpha|_{Y(1)}, J_Y)$, and if $L>0$ is greater than the symplectic action of any of the $\Theta_i$ then the same relation holds in the filtered complex $\ecc^L(Y(1),\alpha|_{Y(1)}, J_Y)$.  If we choose embedding data for $Y(1)$ as in Section \ref{sec:j-independence}, then for sufficiently large $n$ we have a tuple
\[ (Y_n, \alpha_n, J_n) \]
and an embedding $(Y(1),\alpha|_{Y(1)}) \hookrightarrow (Y_n,\alpha_n)$ for which there is a canonical identification
\[ \ech^L(Y(1),\alpha|_{Y(1)}, J_Y) \cong \ech^L(Y_n,\alpha_n, J_n). \]
But then if we let $Z = \overline{Y_n \ssm Y(1)}$, then we can glue $([b,c]_t \times Z, e^{t-c}\alpha_n|_Z)$ to $\tilde{W}$ along its horizontal boundary to get an exact symplectic cobordism from some closed manifold to $(Y_n,\alpha_n)$.  The concave end corresponds to a manifold $(Y'_{n},\alpha'_{n})$ determined by making the same choice of embedding data for $S^3(1)$ (recall that the definition of embedding data only involves the contact form near the boundary of a sutured manifold), and so if $n$ is also sufficiently large for $(Y'_n,\alpha'_n)$ and $L$ then we can also canonically identify
\[ \ech^L(S^3(1), \eta|_{S^3(1)}, J_{S^3}) \cong \ech^L(Y'_{n}, \alpha'_{n}, J'_{n}) \]
in the same way.

The symplectic cobordism we have constructed induces a map
\[ \ech^L(Y_n,\alpha_n,J_n) \to \ech^L(Y'_{n}, \alpha'_{n}, J'_{n}) \]
which sends $[\emptyset]$ to $[\emptyset]$ by Theorem 1.9 and Remark 1.11 of \cite{ht2}.  We know that the class $[\emptyset]$ in the source is zero since it vanishes in $\ech^L(Y(1),\alpha|_{Y(1)},J_Y)$, hence the target $[\emptyset]$ is zero as well.  But this means that $[\emptyset] = 0$ in $\ech^L(S^3(1),\eta|_{S^3(1)},J_{S^3})$, and the relation which forces it to vanish must persist in the unfiltered $\ech(S^3(1),\eta|_{S^3(1)},J_{S^3})$ as well.  We know that $c(\xi_{\mathrm{std}}|_{S^3(1)})$ is nonzero by \cite[Lemma 1.7]{cghh}, since $S^3(1)$ is a product sutured manifold, so this is a contradiction and we conclude that $c(\xi|_{Y(1)}) \neq 0$ as desired.
\end{proof}

The same ideas as in the proof of Theorem \ref{thm:stein-fillable} can be used to construct morphisms on sutured ECH corresponding to an appropriate notion of exact symplectic cobordisms $(X,\lambda)$ with corners, in which $\partial X$ consists of a ``vertical'' part comprised of two connected sutured contact manifolds transverse to the Liouville vector field $Y$ as well as a ``horizontal'' part tangent to $Y$ and foliated by its flow lines.  For example, the construction in the proof would yield a map
\[ \ech(Y(1),\xi|_{Y(1)},\alpha|_{\partial Y(1)}) \to \ech(S^3(1),\xi_{\mathrm{std}}|_{S^3(1)},\eta|_{\partial S^3(1)}) \]
sending $c(\xi|_{Y(1)})$ to $c(\xi_{\mathrm{std}}|_{S^3(1)})$ up to sign (cf.\ Remark \ref{rem:signs}).  We will not pursue this further, though the details are straightforward.

%% file: handle-invariance.tex

We will now use the results of Section \ref{sec:j-independence} to show that if we attach contact $1$-handles to a sutured contact manifold $(M,\Gamma,\alpha)$ to produce $(M',\Gamma',\alpha')$, that there is a canonical isomorphism
\[ \ech(M,\Gamma,\alpha) \isomto \ech(M',\Gamma',\alpha'). \]
We have already seen from Lemma \ref{lem:tau-stretching} that given a tailored almost complex structure $J$ on $\R\times M^*$, there is actually an isomorphism of chain complexes 
\[\ecc(M,\Gamma,\alpha,J) \cong \ecc(M',\Gamma',\alpha',J_H)\] for an appropriate $J_H$.  However, these isomorphisms are not automatically natural, meaning compatible with the canonical isomorphisms $\Phi_{J,J'}$ of Theorem \ref{thm:j-independence} which relate ECH for different choices of $J$ on the left and likewise for the $\Phi_{J_H,J'_H}$ on the right.  Thus we need to work to show that we have a canonical isomorphism of the $J$-independent sutured ECH groups.

To make this goal precise, suppose that $(M,\Gamma,\alpha)$ is a sutured contact manifold, and that the dividing set $\Gamma$ has standard neighborhood
\[ U(\Gamma) = [-1,0]_\tau \times [-1,1]_t \times \Gamma \]
on which $\alpha = Cdt + e^\tau \beta_0$ for some area form $\beta_0$ on $\Gamma$.  We will pick distinct points $p,q \in \Gamma$ and attach a contact $1$-handle $H\times[-1,1]_t$ along the points $p$ and $q$ as in Theorem \ref{thm:1-handle} to form a new contact manifold $(M',\Gamma',\alpha')$.  Then $M$ embeds into $M'$ in an obvious way, with $\alpha'|_M = \alpha$, and the closed Reeb orbits of $(M',\alpha')$ all lie in $M$ and are hence canonically identified with the closed Reeb orbits of $(M,\alpha)$.  Any choice of embedding data for $(M',\Gamma',\alpha')$ in which none of the $1$-handles are attached along $H\times[-1,1]_t$ also determines embedding data for $(M,\Gamma,\alpha)$: we simply let the set of $1$-handles attached to $M$ be the set of $1$-handles attached to $M'$ together with $H\times[-1,1]_t$, and then we use the same diffeomorphism $R_+ \isomto R_-$ for both $M'$ and $M$.  We will say that such simultaneous choices of embedding data for $M$ and $M'$ are \emph{$H$-compatible}.  We will now prove:

\begin{theorem}
\label{thm:1-handle-invariance}
Let $(M',\Gamma',\alpha')$ be constructed from $(M,\Gamma,\alpha)$ by attaching a contact 1-handle $H\times[-1,1]_t$ as described above, with a fixed choice of $H$-compatible embedding data for $(M,\Gamma,\alpha)$ and $(M',\Gamma',\alpha')$, and take $L>0$ such that no collection of Reeb orbits has action exactly $L$.  Then there is a canonical isomorphism $F^L: \ech^L(M,\Gamma,\alpha) \isomto \ech^L(M',\Gamma',\alpha')$ such that the diagram
\[ \xymatrix{
\ech^{L}(M,\Gamma,\alpha) \ar[r]^-{F^L}_-{\sim} \ar[d]_{i^{L,L'}} &
\ech^{L}(M',\Gamma',\alpha') \ar[d]^{i^{L,L'}} \\
\ech^{L'}(M,\Gamma,\alpha) \ar[r]^-{F^{L'}}_-{\sim} &
\ech^{L'}(M',\Gamma',\alpha')
} \]
commutes whenever $L < L'$.  Moreover, the induced isomorphism on the direct limits of each column,
\[ F: \ech(M,\Gamma,\alpha) \isomto \ech(M',\Gamma',\alpha'), \]
carries the contact class $c(\alpha)$ to the contact class $c(\alpha')$.
\end{theorem}

\begin{remark}
Although Theorem \ref{thm:1-handle-invariance} is stated for $(M',\Gamma',\alpha')$ obtained by attaching a single handle, we remark that this construction may be considered a special case of an \emph{interval-fibered extension} \cite[Example 2.10]{cghh} where the Liouville manifold $W$ is $[-\epsilon,0]\times\Gamma_0\cup H$ for some $\epsilon\ll 1$ with $\Gamma_0\subset\Gamma$ being a collection of connected components of $\Gamma$ along which the contact 1-handle is attached. We state the theorem in this ``special case'' for convenience, but in fact the proof of Theorem \ref{thm:1-handle-invariance} will work equally well for arbitrary interval-fibered extensions with compatible embedding data. The analogous statement for sutured contact homology is \cite[Theorem 9.1]{cghh}. The main difference between the methods used in the proof of Theorem \ref{thm:1-handle-invariance} and in the proof of \cite[Theorem 9.1]{cghh} is that continuation maps for varying almost complex structures in the context of sutured ECH are currently available to us only after taking closures since continuation maps in ECH are defined through cobordism maps in Seiberg--Witten Floer cohomology, whereas in sutured contact homology the authors of \cite{cghh} can use holomorphic curve techniques to define continuation maps by hand.
\end{remark}

\begin{proof}[Proof of Theorem \ref{thm:1-handle-invariance}]
Fix generic tailored almost complex structures $J$ and $J'$ on $\R \times M^*$ and $\R \times (M')^*$, where $M^*$ and $(M')^*$ are the completions of $M$ and $M'$.  Let $J_H$ be another tailored almost complex structure on $\R \times (M')^*$ which restricts to the vertical completion $M_v \subset M^*$ of $M$ as $J$. Then Lemma \ref{lem:tau-stretching} applies in this situation, so that given any two generators $\Theta_+$ and $\Theta_-$ of $\ecc(M,\Gamma,\alpha,J)$ we have 
\[ \mathcal{M}_{I=1}(\R \times M^*, J; \Theta_+, \Theta_-) = \mathcal{M}_{I=1}(\R \times (M')^*, J_H; \Theta_+, \Theta_-). \]
Hence for any $L > 0$ such that there are no $\ech$ generators of action exactly $L$, the complexes $\ecc^L(M,\Gamma,\alpha,J)$ and $\ecc^L(M',\Gamma',\alpha',J_H)$ are canonically equal, by the map which identifies an orbit set $\Theta$ in $M$ with its image in $M'$.

Using this canonical identification, we can now define an isomorphism
\[ F^L_{J,J',J_H}: \ech^L(M,\Gamma,\alpha,J) \isomto \ech^L(M',\Gamma',\alpha',J') \]
as the composition
\[ \ech^L(M,\Gamma,\alpha,J) \isomto \ech^L(M',\Gamma',\alpha',J_H) \xrightarrow{\Phi^L_{J_H,J'}} \ech^L(M',\Gamma',\alpha',J'). \]
Given any two choices $J_{H,0}$ and $J_{H,1}$ for $J_H$, it follows essentially from Lemma \ref{lem:j-independence-kappa} that the isomorphism
\[ \Phi^L_{J_{H,0},J_{H,1}}: \ech^L(M',\Gamma',\alpha',J_{H,0}) \isomto \ech^L(M',\Gamma',\alpha',J_{H,1}) \]
is induced by the chain map which sends $\Theta \mapsto \Theta$ for each collection $\Theta$ of Reeb orbits, and so the leftmost square of the diagram
\[ \xymatrix{
\ech^L(M,\Gamma,\alpha,J) \ar[r]^-{\sim} \ar@{=}[d] &
\ech^L(M',\Gamma',\alpha',J_{H,0}) \ar[r]^-{\Phi^L_{J_{H,0},J'}} \ar[d]^{\Phi^L_{J_{H,0},J_{H,1}}} &
\ech^L(M',\Gamma',\alpha',J') \ar@{=}[d] \\
\ech^L(M,\Gamma,\alpha,J) \ar[r]^-{\sim} &
\ech^L(M',\Gamma',\alpha',J_{H,1}) \ar[r]^-{\Phi^L_{J_{H,1},J'}} &
\ech^L(M',\Gamma',\alpha',J')
} \]
commutes, while the rightmost square commutes by Theorem \ref{thm:j-independence}.  The compositions of the maps along the top and bottom rows are $F^L_{J,J',J_{H,0}}$ and $F^L_{J,J',J_{H,1}}$, respectively, so we conclude that they are equal.  In particular, $F^L_{J,J',J_H}$ is independent of the choice of $J_H$, and so we can simply write it as $F^L_{J,J'}$.

Now we can pick any $L' > L$ and a choice of $J_H$ as above.  Then the diagram
\[ \xymatrix{
\ech^L(M,\Gamma,\alpha,J) \ar[r]^-{\sim} \ar[d]_{i^{L,L'}_J} &
\ech^L(M',\Gamma',\alpha',J_H) \ar[r]^-{\Phi^L_{J_H,J'}} \ar[d]^{i^{L,L'}_{J_H}} &
\ech^L(M',\Gamma',\alpha',J') \ar[d]^{i^{L,L'}_{J'}} \\
\ech^{L'}(M,\Gamma,\alpha,J) \ar[r]^-{\sim} &
\ech^{L'}(M',\Gamma',\alpha',J_H) \ar[r]^-{\Phi^{L'}_{J_H,J'}} &
\ech^{L'}(M',\Gamma',\alpha',J')
} \]
commutes, again by applying Theorem \ref{thm:j-independence}, and the maps along each row are $F^L_{J,J'}$ and $F^{L'}_{J,J'}$ respectively, so we have $i^{L,L'}_{J'} \circ F^L_{J,J'} = F^{L'}_{J,J'} \circ i^{L,L'}_{J}$.

Suppose that $J''$ is another generic tailored almost complex structure on $\R \times M^*$, and fix a corresponding $J''_H$ on $\R \times (M')^*$.  Then the $H$-compatibility of the embedding data ensures that the following diagram commutes for any $L$:
\[ \xymatrix{
\ech^L(M,\Gamma,\alpha,J) \ar[r]^-{\sim} \ar[d]_{\Phi^L_{J,J''}} &
\ech^L(M',\Gamma',\alpha',J_H) \ar[r]^-{\sim} \ar[d]^{\Phi^L_{J_H,J''_H}} &
\ech^L(Y_n,\alpha_n) \ar@{=}[d] \\
\ech^L(M,\Gamma,\alpha,J'') \ar[r]^-{\sim} &
\ech^L(M',\Gamma',\alpha',J''_H) \ar[r]^-{\sim} &
\ech^L(Y_n,\alpha_n)
} \]
where $(Y_n,\alpha_n)$ is the closed manifold constructed from the embedding data for some large $n$ as in Section \ref{sec:j-independence}, and in each row the first arrow labeled ``$\sim$'' is the canonical identification of ECH generators while the second such arrow is the isomorphism of \cite[Theorem 1.3]{ht2} used in Section \ref{sec:j-independence}.  Indeed, the rightmost square commutes by the definition of $\Phi^L_{J_H,J''_H}$, since $(Y_n,\alpha_n)$ is built from embedding data for $(M',\Gamma',\alpha')$.  Similarly, the outermost square commutes because the embedding data of $(M,\Gamma,\alpha)$ also leads to the use of $(Y_n,\alpha_n)$ to construct $\Phi^L_{J,J''}$ from the composition of ``$\sim$'' isomorphisms along each row.  (We ensure this by extending $J$ from $\R\times M$ across the embedding handles in a way which agrees with $J_H$ on $\R \times H$.)  Since every map in the diagram is an isomorphism, it follows that the leftmost square commutes as well.

Using the commutativity of the left square, we now construct another commutative diagram
\[ \xymatrix{
\ech^L(M,\Gamma,\alpha,J) \ar[r]^-{\sim} \ar[d]_{\Phi^L_{J,J''}} &
\ech^L(M',\Gamma',\alpha',J_H) \ar[r]^{\Phi^L_{J_H,J'}} \ar[d]_{\Phi^L_{J_H,J''_H}} &
\ech^L(M',\Gamma',\alpha',J') \ar@{=}[d]
\\
\ech^L(M,\Gamma,\alpha,J'') \ar[r]^-{\sim} &
\ech^L(M',\Gamma',\alpha',J''_H) \ar[r]^{\Phi^L_{J''_H,J'}} &
\ech^L(M',\Gamma',\alpha',J')
} \]
whose right square commutes by Theorem \ref{thm:j-independence}.  The compositions of the maps along each row are $F^L_{J,J'}$ and $F^L_{J'',J'}$ respectively, so we conclude that $F^L_{J'',J'} \circ \Phi^L_{J,J''} = F^L_{J,J'}$.

Now suppose that $J'''$ is another generic tailored almost complex structure on $\R \times (M')^*$.  It is immediate from the definition and from Theorem \ref{thm:j-independence} that $F^L_{J,J'''} = \Phi^L_{J',J'''} \circ F^L_{J,J'}$ for each $L$, since we have the identity $\Phi^L_{J_H,J'''} = \Phi^L_{J',J'''} \circ \Phi^L_{J_H,J'}$ for any $J_H$.  Combining this with the above information, we see that for any $J,J''$ on $\R \times M^*$ and $J',J'''$ on $\R \times (M')^*$, there is a commutative diagram
\[ \xymatrix{
\ech^L(M,\Gamma,\alpha,J) \ar[r]^-{F^L_{J,J'}} \ar[d]_{\Phi^L_{J,J''}} &
\ech^L(M',\Gamma',\alpha',J') \ar[d]^{\Phi^L_{J',J'''}} \\
\ech^L(M,\Gamma,\alpha,J'') \ar[r]^-{F^L_{J'',J'''}} &
\ech^L(M',\Gamma',\alpha',J''')
} \]
and so there is a well-defined isomorphism $F^L: \ech^L(M,\Gamma,\alpha) \isomto \ech^L(M',\Gamma',\alpha')$.

Since each map of the form $F^L_{J,J'}$ or $\Phi^L_{J,J'}$ is intertwined with the various $i^{L,L'}_{J}$ and $i^{L,L'}_{J'}$ maps, it follows that $F^{L'} \circ i^{L,L'} = i^{L,L'} \circ F^L$ as well.  Thus we have an isomorphism of directed systems
\[ F:
\left(\{\ech^L(M,\Gamma,\alpha)\}_L, \{i^{L,L'}\}_{L,L'}\right)
\isomto
\left(\{\ech^L(M',\Gamma',\alpha')\}_L, \{i^{L,L'}\}_{L,L'}\right) \]
and hence taking direct limits gives an isomorphism
\[ F: \ech(M,\Gamma,\alpha) \isomto \ech(M',\Gamma',\alpha'). \]
Moreover, each $F^L_{J,J'}$ sends the class $[\emptyset]$ to $[\emptyset]$, so we have $F([\emptyset]) = [\emptyset]$ as well and we conclude that $F(c(\alpha)) = c(\alpha')$.
\end{proof}

It is straightforward to show that given an isotopy of contact forms supported on $\mathrm{int}(M)$, the isomorphism of Theorem \ref{thm:1-handle-invariance} also behaves naturally with respect to the isomorphism of Theorem \ref{thm:alpha-independence-natural}, and hence it defines a natural isomorphism
\[ \ech(M,\Gamma,\xi,\alpha|_{\partial M}) \isomto \ech(M',\Gamma',\xi',\alpha'|_{\partial M'}) \]
carrying $c(\xi)$ to $c(\xi')$ as expected.

For any closed contact manifold $(Y,\xi)$ and collection of $n$ points in $Y$, we can define the sutured contact manifold $Y(n)$ by removing the interior of a Darboux ball centered at each of the $n$ points and then turning the resulting manifold into a sutured contact manifold $(M,\Gamma,\xi)$ as in \cite[Proposition 4.6]{cghh}. In what follows, use $\xi$ to denote the restriction of $\xi$ to $Y(n)$ as well\footnote{For the purposes of computing ECH, our notation should keep track of a contact form on the boundary as well as the collection of $n$ points, cf. Theorem \ref{thm:alpha-independence-natural}, but up to isomorphism this does not matter.}.  Then we have the following immediate corollary, which was originally \cite[Theorem 1.8(2)]{cghh}.

\begin{corollary}
Given any two contact manifolds $(Y_1,\xi_1)$ and $(Y_2, \xi_2)$ we have an isomorphism
\[ \ech((Y_1 \# Y_2)(1), \xi_1\#\xi_2) \cong \ech(Y_1(1),\xi_1) \otimes \ech(Y_2(1),\xi_2) \]
if either we take coefficients in a field or one of the groups $\ech(Y_i(1),\xi_i)$ is free Abelian.  This isomorphism identifies $c(\xi_1 \# \xi_2)$ with $c(\xi_1) \otimes c(\xi_2)$.
\end{corollary}

\begin{proof}
We can build $(Y_1 \# Y_2)(1)$ by gluing a contact $1$-handle to the disjoint union $Y_1(1) \sqcup Y_2(1)$ with one foot on the boundary of each component.  The ECH chain complex of the disjoint union is canonically identified as
\[ \ecc(Y_1(1),\alpha_1, J_1) \otimes \ecc(Y_2(1),\alpha_2, J_2) \]
for any choice of contact form and almost complex structure on each $Y_i(1)$, with contact class represented by $\emptyset \otimes \emptyset$, so the result follows from Theorem \ref{thm:1-handle-invariance} and the K\"{u}nneth theorem.
\end{proof}

It is clear that $Y(n+1)$ can be obtained by taking $Y(n) \sqcup S^3(2)$, where $S^3$ has the standard tight contact structure, and attaching a contact $1$-handle with one foot on the boundary of each component.  Thus in order to compute $\ech(Y(n+1), \xi)$ from $\ech(Y(n),\xi)$ it suffices to determine the homology of $S^3(2)$.

\begin{lemma}
\label{lem:s3-2}
Suppose that $S^3$ is equipped with its tight contact structure $\xi_{\mathrm{std}}$.  Then there is an isomorphism $\ech(S^3(2), \xi_{\mathrm{std}}) \cong \zz^2$ identifying the contact class with $(1,0)$.
\end{lemma}

\begin{proof}
We can realize $S^3(2)$ topologically by taking the disk $D^2_{r,\theta} \times [-2,2]_t$, where $D^2$ has radius 2, with contact form $\alpha = dt + \beta$ for some rotationally symmetric Liouville form $\beta$ on $D^2$, and removing the region $\{r \leq 1\} \times [-1,1]_t$.  Now there are two boundary components, but the innermost one is ``concave'' in the sense of \cite{cghh}, so we must glue on a piece to make it convex as in \cite[Proposition 4.6]{cghh}.  We can modify that procedure slightly for convenience: while they glue on a solid torus with corners of the form $([0,1]_\tau \times[-1,1]_t \ssm \mathrm{int}(D\cup N_\epsilon)) \times S^1$ (which in their notation is $\overline{M''\ssm M}$; see \cite[Figure 2]{cghh}) and then round corners, we can rotate the given $[0,1]_\tau \times [-1,1]_t \ssm \mathrm{int}(D\cup N_\epsilon)$ about the $\tau=1$ portion of its boundary (call the axis of rotation $L$), thus gluing in a thickened $2$-sphere with corners instead.  The result is a sutured contact manifold $(S^3(2),\alpha)$ with rotational symmetry about the pair of line segments $\tilde{L}$ obtained by extending $L$ to $\partial S^3(2)$ via the Reeb flow.

By construction, $(S^3(2),\alpha)$ has a single closed hyperbolic Reeb orbit $h$, and the complement of $h$ is fibered by Reeb flowlines from $R_-(\Gamma)$ to $R_+(\Gamma)$, any of which is isotopic to one of the flowlines which foliate the neighborhood $U(\Gamma)$ of $\Gamma$.  Letting $M^*$ denote its completion and $J$ a tailored almost complex structure on $\R\times M^*$, it follows that
\[ \ecc(S^3(2),\alpha,J) = \zz \emptyset \oplus \zz h. \]
The differential can only count ECH index 1 $J$-holomorphic curves from $h$ to $\emptyset$, so suppose we have such a curve $\mathcal{C}$; its projection $\pi(\mathcal{C})$ to $M^*$ is a 2-chain bounded by $h$.  Since $S^3(2) \ssm \tilde{L}$ retracts onto $h$, and hence $h$ is homologically nontrivial there, we observe that $\pi(\mathcal{C})$ cannot avoid the pair of Reeb flow lines $\gamma_1,\gamma_2$ through $\tilde{L}$ in $M^*$, hence $\mathcal{C} \cdot (\R\times(\gamma_1 \cup \gamma_2)) > 0$ by intersection positivity.  Now given any flow line $\gamma_0$ through a point of $U(\Gamma)$, we can take the segments of $\gamma_0,\gamma_1,\gamma_2$ with $-t_0 \leq t \leq t_0$ for some positive $t_0$ and then connect the endpoints of these segments by paths contained in the regions $\{t=\pm t_0\}$ to get a closed cycle.  For $t_0$ sufficiently large we now repeat the arguments of Lemma \ref{lem:tau-stretching} to show that $\mathcal{C} \cdot (\R\times \gamma_0) = \mathcal{C} \cdot (\R\times(\gamma_1\cup\gamma_2)) > 0$, which is again impossible.

In particular, we must have $\partial h = 0$, so the differential on $\ecc(S^3(2),\alpha,J)$ is zero and the lemma follows.
\end{proof}

\begin{corollary}
\label{cor:remove-darboux-balls}
For any closed contact manifold $(Y,\xi)$ and $n \geq 1$, there is an isomorphism
\[ \ech(Y(n+1),\xi) \cong \ech(Y(n),\xi) \otimes_{\zz} \zz^2 \]
sending $c(\xi|_{Y(n+1)})$ to $c(\xi|_{Y(n)}) \otimes (1,0)$, and so $\ech(Y(n+1),\xi) \cong \ech(Y(1),\xi) \otimes_{\zz} \zz^{2^n}$.
\end{corollary}

%% file: stabilize.tex

Let $(M,\Gamma,\alpha)$ be a sutured contact manifold and $E$ be a choice of embedding data.  The latter consists of a set of contact $1$-handles $\mathcal{H}$ attached to $(M,\Gamma,\alpha)$ to produce a new sutured contact manifold $(M',\Gamma',\alpha')$ with connected suture $\Gamma'$, and a diffeomorphism $f:R_+(\Gamma')\isomto R_-(\Gamma')$ (see Definition \ref{def:embedding-data}). In this section, we prove that various isomorphisms and maps defined in Sections \ref{sec:j-independence} and \ref{sec:alpha-independence} are independent of the choice of embedding data. To do this, we will use Lemma \ref{lem:monopoles-114}, Proposition \ref{prop:comp-62}, and Proposition \ref{prop:comp-63}. In what follows, we emphasize the dependence of various groups, isomorphisms, and maps defined in Sections \ref{sec:j-independence}--\ref{sec:handle-invariance} on a choice of embedding data via a subscript, as in $\ech^L_E(M,\Gamma,\alpha)$.

To apply Lemma \ref{lem:monopoles-114}, Proposition \ref{prop:comp-62}, and Proposition \ref{prop:comp-63}, we need the upcoming lemma. To set the stage, given a sutured contact manifold $(M,\Gamma,\alpha)$ with connected suture $\Gamma$ and contact form $\alpha$ restricting to $U(\Gamma)\simeq(-1,0]_\tau\times[-1,1]\times\Gamma$ as $Cdt+Ke^\tau d\theta$, let $T\geq 0$ and denote by $M^T$ the sutured contact manifold $M\cup_{\{0\}\times[-1,1]\times\Gamma}[0,T]_\tau\times[-1,1]_t\times \Gamma$ with contact form that restricts to $M$ as $\alpha$ and to $[0,T]_\tau\times[-1,1]_t\times \Gamma$ as $Cdt+Ke^\tau d\theta$. A diffeomorphism $f:R_+(\Gamma)\to R_-(\Gamma)$ as in Lemma \ref{lem:glue-pm-filtered} can be extended to $R_+(\Gamma)\cup_{\{0\}\times\{1\}\times\Gamma}[0,T]\times\{1\}\times\Gamma$ so as to restrict to $[0,T]\times\{1\}\times\Gamma$ as the identity map on the $[0,T]$ and $\Gamma$ factors. Having fixed $T>0$, we can choose a set of handles to be used with both $(M,\Gamma,\alpha)$ and $(M^T,\Gamma,\alpha^T)$. To be more explicit, we may choose $A>\frac{(Ke^T)^2}{2}$ in the proof of Theorem \ref{thm:1-handle}, even though the contact forms on these handles for $M$ and $M^T$ would necessarily be different.
\begin{lemma}
\label{lem:tau-neck}
Let $(M,\Gamma,\alpha)$ be a sutured contact manifold with connected suture, and $E$ be a choice of embedding data. For any $L>0$ and $T\geq 0$, there exists a canonical isomorphism 
\[\mathbb{I}^{L,T}_E:\ech^L_E(M,\Gamma,\alpha)\to\ech^L_E(M^T,\Gamma,\alpha^T),\]
satisfying $\mathbb{I}^{L',T}_E\circ i^{L,L'}=i^{L,L'}\circ\mathbb{I}^{L,T}_E$. Hence, in the direct limit as $L\to\infty$, we have canonical isomorphisms 
\[\mathbb{I}^{T}_E:\ech_E(M,\Gamma,\alpha)\to\ech_E(M^T,\Gamma,\alpha^T).\]
Moreover, given a path $\{\alpha_s\}_{s\in[0,1]}$ of contact forms agreeing with $\alpha$ on a neighborhood of $\partial M$, the canonical isomorphisms $\mathbb{I}^{T}_E$ commute with the maps $F_{\alpha_s}$ of Section \ref{sec:alpha-independence}.
\end{lemma}
\begin{proof}
Let $L>0$ be fixed, and $n>0$ be sufficiently large with respect to $L$.  Use $E$, as in Proposition \ref{prop:embed-filtered}, to embed $(M,\alpha)$ and $(M^T,\alpha^T)$ in closed contact manifolds $(Y_n,\alpha_n)$ and $(\ynt,\ant)$, respectively. Note that in each embedding, we can use a solid torus $V_n$ with radius $\frac{n+1}{\pi}$. With the preceding understood, fix a path of strictly increasing functions $h_s:[-1,0]\to[-1,sT]$ such that
\begin{itemize}\leftskip-0.35in
\item For each $s\in[0,1]$, $h_s$ is a diffeomorphism with $h_0={\rm Id}$,
\item For each $s\in[0,1]$, $h_s(\tau)=\tau$ for $\tau\in[-1,-\frac{1}{\kappa})$ and $h_s(\tau)=\tau+sT$ for $\tau\in(-\frac{1}{2\kappa},0]$, where $\kappa\gg1$.
\end{itemize}
Following the proof of Proposition \ref{prop:j-independence-n-iso}, consider the smooth family of diffeomorphisms
\[\varphi_s:Y_n\to {Y_n}^{sT},\]
which restrict to $(Y_n\ssm [-1,0]\times S^1_t\times\Gamma)\cup V_n$ as identity and are defined by $(h_s,{\rm Id}_{S^1},{\rm Id}_\Gamma)$ on $[-1,0]\times S^1_t\times\Gamma$. Note that $(M^\ast,\alpha^\ast)=((M^T)^\ast,(\alpha^T)^\ast)$, and hence a tailored almost complex structure $J$ on $\R\times M^\ast$ is also a tailored almost complex structure on $\R\times(M^T)^\ast$. Pull back via $\varphi_s$ the contact form on ${Y_n}^{sT}$ and the almost complex structure for $\R\times {Y_n}^{sT}$ so as to get a smooth 1-parameter family of contact forms on $Y_n$ and almost complex structures on $\R\times Y_n$. Applying Lemma \ref{lem:j-independence-homotopy} to this smooth family as in the proof of Proposition \ref{prop:j-independence-n-iso}, we get a canonical isomorphism
\[\mathbb{I}^{L,T}_{n;E}:\ech^L(Y_n,\alpha_n,J_n)\to\ech^L(\ynt,\ant,\jnt),\]
induced by the chain map sending an ECH generator $\Theta\in \ecc^L(Y_n,\alpha_n,J_n)$ to its image $\Theta\in \ecc^L(\ynt,\ant,\jnt)$. As a result, we have a commutative diagram
\[ \xymatrixcolsep{5pc}\xymatrix{
\ech^L(M,\Gamma,\alpha,J) \ar[r]^{\tilde{\Phi}^{L,J}_{n;E}} \ar[d]_-{\mathbb{I}^{L,T}} &
\ech^L(Y_n,\alpha_n,J_n)\ar[d]^-{\mathbb{I}^{L,T}_{n;E}}\\
\ech^L(M^T,\Gamma,\alpha^T,J) \ar[r]^{(\tilde{\Phi}^T)^{L,J}_{n;E}} &
\ech^L(\ynt,\ant,\jnt),
} \]
where the canonical isomorphism $\mathbb{I}^{L,T}$ is induced by the chain map sending an ECH generator $\Theta\in \ecc^L(M,\Gamma,\alpha,J)$ to its image $\Theta\in\ecc^L(M^T,\Gamma,\alpha^T,J)$, while $\tilde{\Phi}^{L,J}_{n;E}$ and $(\tilde{\Phi}^T)^{L,J}_{n;E}$ are the isomorphisms defined in Section \ref{sec:j-independence} for $M$ and $M^T$, respectively. Note that the isomorphism $\mathbb{I}^{L,T}$ satisfies $\mathbb{I}^{L',T}\circ i^{L,L'}=i^{L,L'}\circ\mathbb{I}^{L,T}$. Meanwhile, for a pair of $\ech^L$-generic tailored almost complex structures $J,J'$ on $\R\times M^\ast$, we have a commutative diagram
\begin{equation}
\label{eq:tau-neck}
\xymatrix{
\ech^L(Y_n,\alpha_n,J_n) \ar[r]^{} \ar[d]_-{\mathbb{I}^{L,T}_{n;E}} &
\ech^L(Y_n,\alpha_n,J'_n)\ar[d]^-{\mathbb{I}^{L,T}_{n;E}}\\
\ech^L(\ynt,\ant,\jnt) \ar[r]^{} &
\ech^L(\ynt,\ant,\jpnt),
}
\end{equation}
where the horizontal arrows are the canonical isomorphisms of \cite[Theorem 1.3]{ht2}. The commutativity of the above diagram is due to \cite[Lemma 3.4(a)]{ht2}. To see this, note that the admissible deformation, in the sense of \cite[Definition 3.3]{ht2}, from $(\alpha_n,J_n)$ to $(\alpha_n,J'_n)$ is homotopic through admissible deformations defined via the smooth family $\{\varphi_s\}_{s\in[0,1]}$ to the concatenation of the admissible deformations from $(\alpha_n,J_n)$ to $(\ant,\jnt)$, from $(\ant,\jnt)$ to $(\ant,\jpnt)$, and from $(\ant,\jpnt)$ to $(a_n,J'_n)$. Consequently, we have a commutative diagram
\[\xymatrixcolsep{5pc}\xymatrix{
\ech^L(M,\Gamma,\alpha,J) \ar[r]^{\Phi^L_{J,J';E}} \ar[d]_-{\mathbb{I}^{L,T}} &
\ech^L(M,\Gamma,\alpha,J') \ar[d]^-{\mathbb{I}^{L,T}} \\
\ech^L(M^T,\Gamma,\alpha^T,J) \ar[r]^{(\Phi^T)^{L}_{J,J';E}} \ar[r] &
\ech^L(M^T,\Gamma,\alpha^T,J'),
}\]
hence a canonical isomorphism 
\[\mathbb{I}^{L,T}_E:\ech^L_E(M,\Gamma,\alpha)\to\ech^L_E(M^T,\Gamma,\alpha^T),\]
satisfying $\mathbb{I}^{L',T}_E\circ i^{L,L'}=i^{L,L'}\circ\mathbb{I}^{L,T}_E$. Consequently, in the direct limit as $L\to\infty$ we obtain canonical isomorphisms
\[\mathbb{I}^T_E:\ech_E(M,\Gamma,\alpha)\to\ech_E(M^T,\Gamma,\alpha^T).\]
Note that we also have $\mathbb{I}^{L,T}_E\circ P^{L,J}_E=(P^T)^{L,J}_E\circ\mathbb{I}^{L,T}$, where 
\[P^{L,J}_E:\ech^L(M,\Gamma,\alpha,J)\isomto\ech^L_E(M,\Gamma,\alpha),\]
and 
\[(P^T)^{L,J}_E:\ech^L(M^T,\Gamma,\alpha^T,J)\isomto\ech^L_E(M^T,\Gamma,\alpha^T)\]
are the canonical isomorphisms \eqref{eq:P-L-J} defined at the end of Section \ref{sec:j-independence}.

As for the last claim of the lemma, it follows analogously to Proposition \ref{prop:alpha-ind-psi-n} that the canonical isomorphisms $\mathbb{I}^T_E$ commute with maps induced by paths of contact forms agreeing with $\alpha$ on a neighborhood of $\partial M$. In this regard, note that we may choose $\kappa\gg1$ so that the path of contact forms remains constant where the diffeomorphisms $\varphi_s$ are not the identity.
\end{proof}

Now we address the issue of independence of the choice of embedding data. We will do this in two steps. First we consider what happens when we change just the gluing diffeomorphism part of the embedding data. 
\begin{proposition}
\label{prop:gluing-map-doesn't-matter}
Given $L>0$, for any two choices $E$ and $E'$ of embedding data which have the same sets of handles but potentially different diffeomorphisms $f$ and $f'$, there is a canonical isomorphism
\[ \Psi^L_{E,E'}: \ech^L_E(M,\Gamma,\alpha) \isomto \ech^L_{E'}(M,\Gamma,\alpha) \]
for every $\alpha$.  It satisfies the following properties:
\begin{enumerate}\leftskip-0.25in
\item \label{item:psi-transitive} Given a third $E''$ with the same set of handles as $E$ and $E'$, we have
\[ \Psi^L_{E,E''} = \Psi^L_{E',E''} \circ \Psi^L_{E,E'}. \]
In particular, $\Psi^L_{E,E}$ is the identity on $\ech^L_E(M,\Gamma,\alpha)$.
\item For $L < L'$ we have $\Psi^{L'}_{E,E'} \circ i^{L,L'}_E = i^{L,L'}_{E'} \circ \Psi^{L}_{E,E'}$, and thus in the direct limit as $L\to\infty$, we get a canonical isomorphism 
\[\Psi_{E,E'}: \ech_E(M,\Gamma,\alpha) \isomto \ech_{E'}(M,\Gamma,\alpha).\]
\item \label{item:psi-isotopy} Given any path $\{\alpha_s\}_{s\in[0,1]}$ of contact forms agreeing with $\alpha$ on a neighborhood of $\partial M$, the diagram
\[ \xymatrix{
\ech_E(M,\Gamma,\alpha_0) \ar[r]^{F_{\alpha_s}} \ar[d]_{\Psi_{E,E'}} &
\ech_E(M,\Gamma,\alpha_1) \ar[d]^{\Psi_{E,E'}} \\
\ech_{E'}(M,\Gamma,\alpha_0) \ar[r]^{F_{\alpha_s}} &
\ech_{E'}(M,\Gamma,\alpha_1)
} \]
commutes.
\end{enumerate}
\end{proposition}

\begin{proof}
Let $E=(\calh,f)$ and $E'=(\calh,f')$ be two choices of embedding data with the same handle set $\calh$, and fix $L>0$ and $n$ sufficiently large. Let $(Y_n,\alpha_n)$ and $(Y'_n,\alpha'_n)$ be the compact contact $3$-manifolds constructed using the embedding data $E$ and $E'$, respectively, as in Section \ref{sec:j-independence}. Then for an $\textit{ECH}^L$-generic tailored almost complex structure $J$, we have canonical isomorphisms
\begin{align*}
\Phi^{L,J}_{n,E} &: \ech^L(M,\Gamma,\alpha,J)\to\ech^L(Y_n,\alpha_n) \\
\Phi^{L,J}_{n,E'} &: \ech^L(M,\Gamma,\alpha,J)\to\ech^L(Y'_n,\alpha'_n)
\end{align*}
and given another $\textit{ECH}^L$-generic tailored almost complex structure $J'$ we have 
\[ \Phi^L_{J,J';E}=(\Phi^{L,J'}_{n,E})^{-1}\circ \Phi^{L,J}_{n,E}:\ech^L(M,\Gamma,\alpha,J)\to \ech^L(M,\Gamma,\alpha,J'), \]
and likewise $\Phi^L_{J,J';E'}$, as defined in Section \ref{sec:j-independence}. We would like to show that for any pair of $\textit{ECH}^L$-generic tailored almost complex structures $J^0$ and $J^1$, the maps $\Phi^L_{J^0,J^1;E}$ and $\Phi^L_{J^0,J^1;E'}$ are the same. It would then follow that there exists an isomorphism 
\[\Psi^L_{E,E'}: \ech_E^L(M,\Gamma,\alpha) \isomto \ech_{E'}^L(M,\Gamma,\alpha), \]
as claimed.

To see this, note that in light of Lemma \ref{lem:tau-neck}, we can replace $(M,\Gamma,\alpha)$ with $(M^T,\Gamma,\alpha^T)$. Then the maps $\Phi^L_{J^0,J^1;E}$ and $\Phi^L_{J^0,J^1;E'}$ are the same if and only if the maps
\begin{eqnarray*}
(\Phi^T)^L_{J^0,J^1;E}&:& \ech^L(M^T,\Gamma,\alpha^T,J^0)\to \ech^L(M^T,\Gamma,\alpha^T,J^1),\\
(\Phi^T)^L_{J^0,J^1;E'}&:& \ech^L(M^T,\Gamma,\alpha^T,J^0)\to \ech^L(M^T,\Gamma,\alpha^T,J^1)
\end{eqnarray*}
are the same. Now, assume without loss of generality\footnote{See \cite[Proposition 2.4]{taubes1} for a justification of this claim.  Note that we used a similar argument in the proof of Lemma \ref{lem:j-independence-homotopy}.} that $(\ant,J^{0,T}_n)$ and $(\ant,J^{1,T}_n)$ obey the conditions in (4-1) of \cite{taubes1} for $L$ on $\ynt$, and similarly for $(\anth,(J^{0,T}_n)')$ and $(\anth,(J^{1,T}_n)')$ on $(\ynt)'$, and set $T=n+1$. Then, the diagram 
\[ \xymatrix{
\ech^L(M^T,\Gamma,\alpha^T,J^i)\ar[r]^-{\sim}_-{\eqref{eq:canid}}\ar@{=}[d]&
\ech^L(\ynt,\ant,J^{i,T}_n)\ar[d]^-{\rotatebox{90}{$\sim$}}_-{\psi^{L,J^i}_{n;E,E'}}\\
\ech^L(M^T,\Gamma,\alpha^T,J^i)\ar[r]^-{\sim}_-{\eqref{eq:canid}}&
\ech^L(\ynth,\anth,(J^{i,T}_n)'),
}\]
commutes for both $i=0$ and $i=1$, where the vertical isomorphism $\psi^{L,J^i}_{n;E,E'}$ is induced by a map identifying ECH generators in the source with their images in the target. We claim that the latter is a chain map from $\ecc^L(\ynt,\ant,J^{i,T}_n)$ to $\ecc^L(\ynth,\anth,(J^{i,T}_n)')$ by Lemma \ref{lem:monopoles-115} and Proposition \ref{prop:comp-63}.

To elaborate on this claim, since the contact forms and almost complex structures obey the conditions in (4-1) of \cite{taubes1}, there are canonical isomorphisms between $\ech^L(\ynt,\ant,J^{i,T}_n)$ and $\HMfrom^*_{L}(\ynt,\ant,J^{i,T}_n,r)$ as well as between $\ech^L(\ynth,\anth,(J^{i,T}_n)')$ and $\HMfrom^*_{L}(\ynth,\anth,(J^{i,T}_n)',r)$ for sufficiently large $r$, induced by canonical identifications between Reeb orbit sets and moduli spaces of solutions to the 3-dimensional Seiberg-Witten equations \cite[Theorem 4.2]{taubes1}. These canonical identifications define chain maps between the corresponding complexes $\ecc^L$ and $\CMfrom^*_{L}$, due to the canonical identifications between moduli spaces of pseudo-holomorphic curves that define the differential on $\ecc^L$ and moduli spaces of solutions to the Seiberg--Witten equations that define the differential on $\CMfrom^*_{L}$ \cite[Theorem 4.3]{taubes1}. Meanwhile, the first bullet of Lemma \ref{lem:monopoles-115} supplies 1--1 correspondences between the sets of generators for the Seiberg--Witten Floer cochain complexes $\CMfrom^*_{L}(\ynt,\ant,J^{i,T}_n,r)$ and $\CMfrom^*_{L}(\ynth,\anth,(J^{i,T}_n)',r)$ for both $i=0$ and $i=1$, and the first bullet point of Proposition \ref{prop:comp-63} provides 1--1 correspondences between the moduli spaces of solutions to the Seiberg--Witten equations that define the differentials for $\CMfrom^*_{L}(\ynt,\ant,J^{i,T}_n,r)$ and $\CMfrom^*_{L}(\ynth,\anth,(J^{i,T}_n)',r)$.

From the above discussion and the fact that the maps $(\Phi^T)^L_{J^0,J^1;E}$ and $(\Phi^T)^L_{J^0,J^1;E'}$ are defined in terms of the isomorphisms of \eqref{eq:canid}, we see that $(\Phi^T)^L_{J^0,J^1;E}$ and $(\Phi^T)^L_{J^0,J^1;E'}$ are the same if and only if the diagram
\begin{equation}
\label{eq:J-path}
\xymatrix{
\ech^L(\ynt,\ant,J^{0,T}_n) \ar[r]^{\sim} \ar[d]_-{\rotatebox{90}{$\sim$}}^-{\psi^{L,J^0}_{n;E,E'}} &
\ech^L(\ynt,\ant,J^{1,T}_n) \ar[d]^-{\rotatebox{90}{$\sim$}}_-{\psi^{L,J^1}_{n;E,E'}}\\
\ech^L(\ynth,\anth,(J^{0,T}_n)') \ar[r]^{\sim} &
\ech^L(\ynth,\anth,(J^{1,T}_n)'),
} 
\end{equation}
commutes, where the horizontal isomorphisms are again the canonical isomorphisms of \cite[Theorem 1.3]{ht2}. The commutativity of the diagram in \eqref{eq:J-path} follows from the second bullet point of Proposition \ref{prop:comp-62} via Proposition \ref{prop:comp-64}. To be more explicit, note that the manifold $\ynt$ is a version of the manifold denoted in Appendix \ref{sec:appendix} by $Y_T$, whereas $M^\ast$ is denoted there by $Y_\infty$. In order to apply Proposition \ref{prop:comp-62}, we need the path joining $J^0$ and $J^1$ to be independent of $s$ when restricted to a neighborhood of $M^\ast\ssm\textit{int}(M)$. To reduce to this case, we make use of the following lemma:

\begin{lemma}
\label{lem:J-constant-at-infinity}
Let $(M,\Gamma,\alpha)$ be a sutured contact manifold with connected suture, $L>0$ be fixed, $J^0$ and $J^1$ be $\ech^L$-generic tailored almost complex structures connected by a path of tailored almost complex structures $\{J^s\}_{s\in[0,1]}$ on $\R\times M^\ast$, and $E$ be a choice of embedding data that gives a closed contact manifold $(Y_n,\alpha_n)$ into which $(M,\alpha)$ embeds. Suppose that both $(\alpha,J^0)$ and $(\alpha,J^1)$ obey the conditions in (4-1) of \cite{taubes1} for $L$ on $Y_n$ for sufficiently large $n>0$. Then, there exists another path $\{\tilde{J}^s\}_{s\in[0,1]}$ of tailored almost complex structures connecting $\tilde{J}^0:=J^0$ to an $\ech^L$-generic tailored almost complex structure $\tilde{J}^1$, where 
\begin{enumerate}\leftskip-0.25in
\item $\tilde{J}^1$ agrees with $J^1$ on the complement in $\R\times M$ of a neighborhood of $\R\times\partial M$ and $\tilde{J}^s$ agrees with $J^0$ on a neighborhood of $\R\times (M^\ast\ssm\textit{int}(M))$.
\item The following diagram commutes:
\[ \xymatrix{
\ech^L(Y_n,\alpha_n,J^0_n) \ar[r]^{\sim} \ar[dr]_-{\rotatebox{-30}{$\sim$}}
& \ech^L(Y_n, \alpha_n,J^1_n) \ar[d]^-{\rotatebox{90}{$\sim$}} \\
& \ech^L(Y_{n}, \alpha_{n},\tilde{J}^1_n)
} \]
where all three isomorphisms are the canonical isomorphisms of \cite[Theorem 1.3]{ht2} and are induced by chain maps sending an ECH generator $\Theta$ to $\Theta$.
\end{enumerate}
\end{lemma}
\begin{proof}
To begin, define two smooth functions $\sigma:M^\ast\to[0,1]$ and $\chi:M^\ast\to[0,1]$ which have support in a neighborhood of $M^\ast \ssm \textit{int}(M)$ and the complement in $M$ of a neighborhood of $\partial M$, respectively, and are equal to $1$ in a slight deformation retract of their support. The restriction of these functions to a neighborhood of $\partial M$ depend only on $t$ and $\tau$, respectively, and admit graphs as seen in Figure \ref{fig:cutoff}.

\begin{figure}[ht]
\labellist
\small \hair 2pt
\pinlabel $\sigma$ [l] at 85 97
\pinlabel $t$ [b] at 166 -1
\pinlabel $\chi$ [l] at 277 97
\pinlabel $\tau$ [b] at 311 -1
\pinlabel $-1$ [t] at 6 -1
\pinlabel $\epsilon-1$ [t] at 34 -1
\pinlabel $1-\epsilon$ [t] at 129 -1
\pinlabel $1$ [t] at 154 -1
\pinlabel $-1$ [t] at 222 -1
\tiny
\pinlabel $1$ [r] at 82 69
\pinlabel $1$ [r] at 274 69
\endlabellist
\centering
\includegraphics{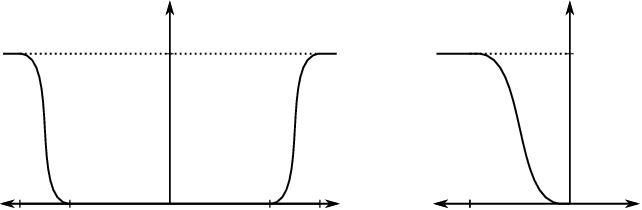}
\caption{The restriction of the functions $\sigma,\chi: M^\ast \to [0,1]$ to a neighborhood of $\partial M$.}
\label{fig:cutoff}
\end{figure}
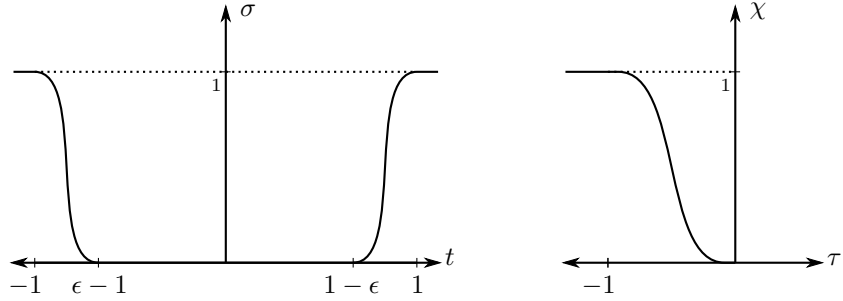

Now define a homotopy of paths of tailored almost complex structures by
\[\{{J^s}_\lambda:=J^{(1-\chi(1-\sigma))(1-\lambda)s+\chi(1-\sigma)s}\}_{\lambda\in[0,1]}.\]
The starting path of this homotopy is $\{{J^s}_0=J^s\}_{s\in[0,1]}$, and the ending path $\{\tilde{J}^{s}:={J^s}_1\}_{s\in[0,1]}$ consists of tailored almost complex structures $J^{\chi(1-\sigma)s}$ which restrict to $J^0$ in a neighborhood of $\R\times (M^\ast\ssm\textit{int}(M))$ and to $J^s$ on the complement in $\R\times M$ of a neighborhood of $\R\times\partial M$. The path of end points of this homotopy, namely $\{{J^1}_\lambda=J^{1-\lambda(1-\chi(1-\sigma))}\}_{\lambda\in[0,1]}$, connects $J^1$ to $\tilde{J}^1$ through generic tailored almost complex structures that restrict to $J^1$ on the complement in $\R\times M$ of a neighborhood of $\R\times\partial M$, and it restricts on a neighborhood of $\R\times (M^\ast\ssm\textit{int}(M))$ to the path $\{J^{1-\lambda}\}_{\lambda\in[0,1]}$. It follows from Lemmas \ref{lem:tau-stretching} and \ref{lem:t-stretching} that each ${J^1}_\lambda$ is $\ech^L$-generic for $n>0$ sufficiently large. This gives us a commuting diagram of paths
\[ \xymatrix{
J^0 \ar[r]^{J^s}\ar[dr]_-{\tilde{J}^s}
& J^1 \ar[d]^-{{J^1}_\lambda} \\
& \tilde{J}^1.
} \]
Note that since $(\alpha_n,J^1_n)$ is assumed to obey the conditions in (4-1) of \cite{taubes1} for $L$ for sufficiently large $n>0$, so does $(\alpha_n,\tilde{J}^1_n)$. In fact, the path $\{{J^1}_\lambda\}_{\lambda\in[0,1]}$ of $\textit{ECH}^L$-generic tailored almost complex structures yields a path of pairs $\{(\alpha_n,{J^1}_{\lambda,n})\}_{\lambda\in[0,1]}$ obey the conditions in (4-1) of \cite{taubes1} for $L$ on $Y_n$, and by \cite[Lemma 3.4(d)]{ht2}, the isomorphism
\[\ech^L(Y_n, \alpha_n,J^1_n)\isomto \ech^L(Y_{n}, \alpha_{n},\tilde{J}^1_n),\]
is induced by the chain map sending an ECH generator $\Theta$ to $\Theta$. 

Next, we define a homotopy between the concatenation of paths $\{J^s\}_{s\in[0,1]}\ast\{J^1_\lambda\}_{\lambda\in[0,1]}$ and the path $\{\tilde{J}^s\}_{s\in[0,1]}$ as follows:
\[ {\mathbb{J}^s}_r = \begin{cases}
{J^{\frac{2s}{r+1}}}_r & 0\leq s\leq \frac{r+1}{2} \\
{J^1}_{2s-1} & \frac{r+1}{2}\leq s\leq 1.
\end{cases} \]
Note that ${\mathbb{J}^s}_0$ is the path $\{J^s\}_{s\in[0,1]}\ast\{J^1_\lambda\}_{\lambda\in[0,1]}$, while ${\mathbb{J}^s}_1=\tilde{J}^s$. With this homotopy in hand, \cite[Lemma 3.4(a)]{ht2} and \cite[Lemma 3.4(b)]{ht2} imply that we have a commutative diagram
\[ \xymatrix{
\ech^L(Y_n,\alpha_n,J^0_n) \ar[r]^{\sim} \ar[dr]_-{\rotatebox{-30}{$\sim$}}
& \ech^L(Y_n, \alpha_n,J^1_n) \ar[d]^-{\rotatebox{90}{$\sim$}} \\
& \ech^L(Y_{n}, \alpha_{n},\tilde{J}^1_n)
} \]
where the vertical arrow is the canonical map sending an ECH generator $\Theta$ to $\Theta$. 
\end{proof}
\noindent With Lemma \ref{lem:J-constant-at-infinity} in hand, we may assume that the path joining $J^0$ to $J^1$ is constant when restricted to $\R\times (M^\ast\ssm\textit{int}(M))$. Hence, the second bullet of Proposition \ref{prop:comp-62} applies via Proposition \ref{prop:comp-64} to give us the commutativity of the diagram in \eqref{eq:J-path} for all sufficiently large $T>0$. We can employ this proposition due to the canonical isomorphism of Theorem \ref{thm:taubes} between embedded contact homology and Seiberg--Witten Floer cohomology of closed contact 3-manifolds induced by a canonical isomorphism between the respective chain complexes under the assumption that the pair of a contact form and almost complex structure satisfies the conditions depicted in (4-1) of \cite{taubes1}. Given $L>0$, we can choose $r>\kappa_{2\pi L}$ and $T>c_\ast$ sufficiently large where $\kappa_{2\pi L}>0$ is a constant depending only on $L$ and $c_\ast>0$ is a constant possibly depending on $r$ such that there is a canonical 1--1 correspondence between moduli spaces of solutions to the equations \eqref{eq:sw} and \eqref{eq:instanton-sw} for embedding data $\ynt$ and $(\ynt)'$. Therefore, the isomorphism $\Psi^L_{E,E'}$ exists and satisfies property \eqref{item:psi-transitive}. Since the isomorphism $\Psi^L_{E,E'}$ commutes with the inclusion induced maps $i^{L,L'}$, the direct limit as $L\to\infty$ yields the isomorphism $\Psi_{E,E'}$.

Property \eqref{item:psi-isotopy} follows by a similar argument. Given a path $\{\alpha_s\}$ of contact forms on $(M,\Gamma)$ agreeing with $\alpha$ on a neighborhood of $\partial M$, we define $X^n_{10}$ and $(X')^n_{10}$ to be the Liouville cobordisms $([R,2R]\times \ynt, \lambda^{10}_n)$ and $([R,2R]\times \ynth, (\lambda')^{10}_n)$ constructed in Section~\ref{ssec:alpha-independence} for the respective sets of embedding data $E$ and $E'$ from $(M^T,\Gamma,\alpha_s^T)$, and we wish to show that the diagram
\[ \xymatrix{
\ech^{e^{2R}L}(\ynt, e^{2R}\antz) \ar[rr]^-{\Phi^{e^{2R}L}(X^n_{10})} \ar[d]_-{\rotatebox{90}{$\sim$}} &&
\ech^{e^{2R}L}(\ynt, e^{R}\anto) \ar[d]^-{\rotatebox{90}{$\sim$}} \\
\ech^{e^{2R}L}(\ynth, e^{2R}\antzp) \ar[rr]^-{\Phi^{e^{2R}L}((X')^n_{10})} &&
\ech^{e^{2R}L}(\ynth, e^{R}\antop)
} \]
commutes. As above, assume without loss of generality that $(\antz,\jntz)$ and $(\anto,\jnto)$ obey the conditions in (4-1) of \cite{taubes1} for $e^{2R}L$ on $\ynt$, and $J^0$ is connected to $J^1$ by a path of tailored almost complex structures that is constant in a neighborhood of $\R\times (M^\ast\ssm\textit{int}(M))$ via an analog of Lemma \ref{lem:J-constant-at-infinity} for a path $\{\alpha_s\}_{s\in[0,1]}$ of contact forms on $(M,\Gamma)$ agreeing with $\alpha$ on a neighborhood of $\partial M$. Note that this implies that $(\antzp,\jntzp)$ and $(\antop,\jntop)$ obey the conditions in (4-1) of \cite{taubes1} for $e^{2R}L$ on $\ynth$. Then the vertical isomorphisms in the above diagram identify ECH generators in the source with their images in the target by Lemma \ref{lem:monopoles-114} and Proposition \ref{prop:comp-63}, and the commutativity of the diagram follows from the second bullet of Proposition \ref{prop:comp-62} via Proposition \ref{prop:comp-64}. 
\end{proof}

What Proposition \ref{prop:gluing-map-doesn't-matter} shows is that if we fix a collection $\calh$ of handles so that $M' = M \cup \calh$ has connected suture, then for each $L>0$ we have a transitive system of isomorphisms
\[  \left(\{\ech^L_E(M,\Gamma,\alpha)\}_{E}, \{\Psi^L_{E,E'}\}_{E,E'}\right) \]
where $E$ and $E'$ range over all embedding data for $(M,\Gamma,\alpha)$ with precisely the set $\calh$ of handles and any diffeomorphism $R_+(\Gamma') \isomto R_-(\Gamma')$.  This produces a canonical group
\[ \ech^L_\calh(M,\Gamma,\alpha), \]
and a canonical group $\ech_\calh(M,\Gamma,\alpha)$ in the direct limit as $L\to\infty$, associated to the collection $\calh$ of handles.  Moreover, since the maps $F_{\alpha_s}$ are intertwined with the isomorphisms $\Psi_{E,E'}$, we have canonical isomorphisms
\[ F_{\alpha_s,\calh}: \ech_\calh(M,\Gamma,\alpha_0) \to \ech_\calh(M,\Gamma,\alpha_1) \]
depending only on the path $\alpha_s$ and on $\calh$. Note also that $P^{L,J}_{E'}=\Psi^L_{E,E'}\circ P^{L,J}_E$, since $P^{L,J}_E=P^{L,J'}_E\circ\Phi^L_{J,J';E}$ by definition and $\Phi^L_{J,J';E} = \Phi^L_{J,J';E'}$.

Hence, we have canonical isomorphisms
\[P^{L,J}_\calh:\ech^L(M,\Gamma,\alpha,J)\isomto\ech^L_\calh(M,\Gamma,\alpha).\]

\begin{definition}
Let $(M,\Gamma,\alpha)$ be a sutured contact manifold, and let $\calh$ be a collection of contact $1$-handles attached to $(M,\Gamma,\alpha)$ along disjoint closed arcs of $\Gamma$, so that the resulting $(M',\Gamma',\alpha')$ has connected suture $\Gamma'$.  We will refer to $\calh$ as \emph{handle data} for $(M,\Gamma,\alpha)$.
\end{definition}

\begin{definition}
Given two choices of handle data $\calh, \calh'$ for $(M,\Gamma,\alpha)$, we say that $\calh < \calh'$ if the set $\calh$ is a proper subset of the set $\calh'$.  We write $\calp_{(M,\Gamma,\alpha)}$ to denote the partially ordered set of handle data associated to $(M,\Gamma,\alpha)$.
\end{definition}

We remark that this partial order depends not only on where the handles attached but also the handles themselves, up to a notion of equivalence.  More precisely, a handle depends not only on the arcs of $\Gamma$ to which it is attached but also on its contact form; two handles $H_1,H_2$ which are attached along the same arcs of $\Gamma$ may be said to be equivalent if there is some diffeomorphism $f:(H_1,\partial M\cap H_1)\to (H_2,\partial M\cap H_2)$ which pulls back the contact form on $H_2$ to the one on $H_1$.  If $\calh$ and $\calh'$ differ by replacing one handle with an inequivalent handle attached to the same region, then $\calh$ and $\calh'$ are not directly comparable.

\begin{lemma}
\label{lem:poset-connected}
The poset $\calp_{(M,\Gamma,\alpha)}$ is connected.  In other words, given two choices $\calh$ and $\calh'$ of handle data, there is a finite sequence
\[ \calh = \calh_0, \calh_1, \calh_2, \dots, \calh_n = \calh' \]
such that either $\calh_i < \calh_{i+1}$ or $\calh_i > \calh_{i+1}$ for $0 \leq i < n$.
\end{lemma}

\begin{proof}
We first replace $\calh$ with handle data $\tilde{\calh}$ in which every handle $H$ of $\calh$ has been replaced by a ``smaller'' handle $\tilde{H}$, each end of which is attached to $\Gamma$ along an arbitrarily small sub-arc of the corresponding end of $H$.  Figure \ref{fig:stabilization} shows how this can be done in a small neighborhood of $H$ by a series of stabilizations and destabilizations, i.e.\ adding and removing pairs of handles while preserving the connectedness of the resulting suture $\Gamma'$.

\begin{figure}[ht]
\labellist
\small \hair 2pt
\pinlabel $H$ at 25 44
\pinlabel $\tilde{H}$ at 329 55
\pinlabel $<$ at 63 55
\pinlabel $>$ at 139 55
\pinlabel $<$ at 215 55
\pinlabel $>$ at 291 55
\endlabellist
\centering
\includegraphics{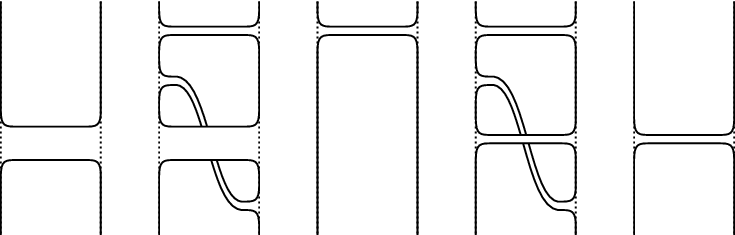}
\caption{A series of choices of handle data which result in replacing the handle $H$ with the smaller $\tilde{H}$, in which each adjacent pair is comparable.  We view the handles (and nearby points of $U(\Gamma)$) by identifying each picture locally as $R_+(\Gamma')\times[-1,1]_t$ and projecting out the $[-1,1]_t$ factor.}
\label{fig:stabilization}
\end{figure}
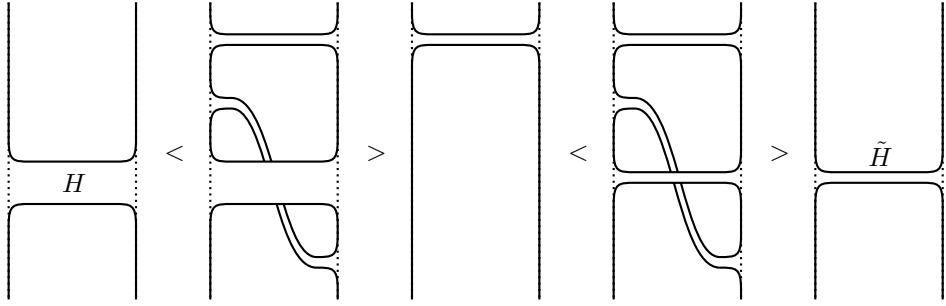

We can similarly replace $\calh'$ with $\tilde{\calh}'$, and moreover we can choose the new handles of $\tilde{\calh}$ and $\tilde{\calh}'$ to be attached along disjoint arcs since their attaching arcs can be any arbitrarily small sub-arcs of the attaching arcs for $\calh$ and $\calh'$.  In particular we can then attach all of the handles of $\tilde{\calh}$ and $\tilde{\calh}'$ to $(M,\Gamma,\alpha)$ simultaneously, and the resulting sutured contact manifold may not have connected suture but we can attach even more handles to produce $(M',\Gamma',\alpha')$ with $\Gamma'$ connected.  The corresponding handle data $\calh_0$ satisfies $\tilde{\calh} < \calh_0$ and $\tilde{\calh}' < \calh_0$, so $\tilde{\calh}$ and $\tilde{\calh}'$ are in the same connected component of $\calp_{(M,\Gamma,\alpha)}$, hence so are $\calh$ and $\calh'$.
\end{proof}

\begin{proposition}
\label{prop:poset-functor}
For all $\calh,\calh' \in \calp_{(M,\Gamma,\alpha)}$ with $\calh \leq \calh'$, and for all $L>0$, there is a canonical isomorphism
\[ \Psi^L_{\calh,\calh'}: \ech^L_\calh(M,\Gamma,\alpha) \to \ech^L_{\calh'}(M,\Gamma,\alpha) \]
satisfying $\Psi^L_{\calh,\calh} = \mathrm{Id}$ and $\Psi^L_{\calh,\calh''} = \Psi^L_{\calh',\calh''}\circ\Psi^L_{\calh,\calh'}$ whenever $\calh \leq \calh' \leq \calh''$. In the direct limit as $L\to\infty$, we obtain a canonical isomorphism
\[ \Psi_{\calh,\calh'}: \ech_\calh(M,\Gamma,\alpha) \to \ech_{\calh'}(M,\Gamma,\alpha). \]
Given a path $\{\alpha_s\}_{s\in[0,1]}$ of contact forms agreeing with $\alpha$ on a neighborhood of $\partial M$, the diagram
\[ \xymatrix{
\ech_\calh(M,\Gamma,\alpha_0) \ar[rr]^{F_{\alpha_s,\calh}} \ar[d]_{\Psi_{\calh,\calh'}} &&
\ech_\calh(M,\Gamma,\alpha_1) \ar[d]^{\Psi_{\calh,\calh'}} \\
\ech_{\calh'}(M,\Gamma,\alpha_0) \ar[rr]^{F_{\alpha_s,\calh'}} &&
\ech_{\calh'}(M,\Gamma,\alpha_1) \\
} \]
commutes.
\end{proposition}

\begin{proof}
Let $(M,\Gamma,\alpha)$ be a sutured contact manifold, not necessarily with connected suture, and $\calh,\calh'\in\calp_{(M,\Gamma,\alpha)}$ with $\calh \leq \calh'$. Denote by $(M',\Gamma',\alpha')$ and $(M'',\Gamma'',\alpha'')$ the sutured contact manifolds obtained from $(M,\Gamma,\alpha)$ via attaching the handles provided by the handle data $\calh$ and $\calh'$, respectively. Note that Theorem \ref{thm:1-handle-invariance} provides canonical isomorphisms 
\[\ech^L_\calh(M,\Gamma,\alpha)\underset{F^L_\calh}{\isomto}\ech^L_\emptyset(M',\Gamma',\alpha')\] 
and 
\[\ech^L_{\calh'\ssm\calh}(M',\Gamma',\alpha')\underset{(F')^{L}_{\calh'\ssm\calh}}{\isomto}\ech^L_\emptyset(M'',\Gamma'',\alpha'')\underset{(F^L_{\calh'})^{-1}}{\isomto}\ech^L_{\calh'}(M,\Gamma,\alpha)\] that commute with the maps induced by paths $\{\alpha_s\}_{s\in[0,1]}$ of contact forms agreeing with $\alpha$ on a neighborhood of $\partial M$. To see this, in the case of the top isomorphism, recall that Theorem \ref{thm:1-handle-invariance} assumes that the embedding data $(\calh,f)$ and $(\emptyset,f')$ for $(M,\Gamma,\alpha)$ and $(M',\Gamma',\alpha')$, respectively, are such that $f'=f$. Meanwhile, given two choices of embedding data for $(M,\Gamma,\alpha)$ with identical handle data but different gluing diffeomorphisms, and a choice of $\ech^L$-generic tailored almost complex structure $J$ on $\R\times M^\ast$, the corresponding isomorphisms from $\ech^L(M,\Gamma,\alpha,J)$ to itself and from $\ech^L(M',\Gamma',\alpha',J)$ to itself as in the proof of Proposition \ref{prop:gluing-map-doesn't-matter} are both the identity. Therefore, the isomorphism of Theorem \ref{thm:1-handle-invariance} defines a canonical isomorphism $F^L_\calh:\ech^L_\calh(M,\Gamma,\alpha)\isomto\ech^L_\emptyset(M',\Gamma',\alpha')$ depending only on the handle data.
\begin{pt}
\label{part1}
This part defines the map $\Psi^L_{\calh,\calh'}$. Recall from Section \ref{sec:j-independence} that we have canonical isomorphisms
\[P^{L,J}_\calh:\ech^L(M,\Gamma,\alpha,J)\to \ech^L_\calh(M,\Gamma,\alpha),\]
such that $P^{L,J^0}_\calh=P^{L,J^1}_\calh\circ\Phi^L_{J^0,J^1;\calh}$. We define $\Psi^L_{\calh,\calh'}=P^{L,J}_{\calh'}\circ (P^{L,J}_\calh)^{-1}$ and claim that the right hand side is independent of the choice of $J$. Note that for $L'>L$, there is an inclusion induced map $i_\calh^{L,L'}:\ech^L_\calh(M,\Gamma,\alpha)\to\ech^{L'}_\calh(M,\Gamma,\alpha)$ with which the isomorphism $\Psi^L_{\calh,\calh'}$ commutes, since $P^{L',J}_\calh\circ i^{L,L'}=i_\calh^{L,L'}\circ P^{L,J}_\calh$ by definition. Meanwhile,
\[\Psi^L_{\calh,\calh'}=(F^L_{\calh'})^{-1}\circ (F')^{L}_{\calh'\ssm\calh}\circ(\Psi')^{L}_{\emptyset,\calh'\ssm\calh}\circ F^L_\calh\]
where 
\[(\Psi')^{L}_{\emptyset,\calh'\ssm\calh}:\ech^L_\emptyset(M',\Gamma',\alpha') \to \ech^L_{\calh'\ssm\calh}(M',\Gamma',\alpha').\] Therefore, it suffices to prove that $(\Psi')^{L}_{\emptyset,\calh'\ssm\calh}$ is independent of the choice of tailored almost complex structure on $\R\times (M')^\ast$. In order to prove this, we need to see that 
\[(\Phi')^{L}_{J^0,J^1;\calh'\ssm\calh}\circ ((\Phi')^{L}_{J^0,J^1;\emptyset})^{-1}={\rm Id}\]
for two $\ech^L$-generic tailored almost complex structures $J^0,J^1$ on $\R\times (M')^\ast$. As before, Lemma \ref{lem:tau-neck} can be used to replace $(M',\Gamma',\alpha')$ with $((M')^T,\Gamma',(\alpha')^T)$. In other words, we would like to prove that 
\[(\Phi^{'T})^{L}_{J^0,J^1;\calh'\ssm\calh}\circ ((\Phi^{'T})^{L}_{J^0,J^1;\emptyset})^{-1}={\rm Id}.\]
which is equivalent to showing that the following diagram commutes: 
\[ \xymatrix{
\ech^L(\ynt,\ant,J^{0,T}_n) \ar[r]^{\sim}\ar[d]_{(\tilde{\Phi}^{'T})^{L,J^0}_{n;E'\ssm E}\circ((\tilde{\Phi}^{'T})^{L,J^0}_{n;\emptyset})^{-1}}  &
\ech^L(\ynt,\ant,J^{1,T}_n) \ar[d]^{(\tilde{\Phi}^{'T})^{L,J^1}_{n;E'\ssm E}\circ((\tilde{\Phi}^{'T})^{L,J^1}_{n;\emptyset})^{-1}}\\ 
\ech^L((\ynt)',(\ant)',(J^{0,T}_n)') \ar[r]^{\sim} &
\ech^L((\ynt)',(\ant)',(J^{1,T}_n)')
} \]
where $\emptyset=(\emptyset,f)$ and $E'\ssm E=(\calh'\ssm\calh,f')$ with $f'$ extending $f$, and the horizontal isomorphisms are the canonical isomorphisms of \cite[Theorem 1.3]{ht2}. Assume without loss of generality that $(\ant,J^0_n)$ and $(\ant,(J^1_n)')$ obey the conditions in (4-1) of \cite{taubes1} for $L$ on $\ynt$, and similarly for $(\anth,(J^0_n)')$ and $(\anth,(J^1_n)')$ on $(\ynt)'$. By Lemma \ref{lem:J-constant-at-infinity}, we may also assume that the path joining $J^0$ to $J^1$ is constant when restricted to $\R\times ((M')^\ast\ssm\textit{int}(M'))$. Then setting $T=n+1$, the vertical isomorphisms identify ECH generators in the source with their images in the target by Lemma \ref{lem:monopoles-114} and Proposition \ref{prop:comp-63}, and the commutativity of the diagram follows from the second bullet of Proposition \ref{prop:comp-62} for all sufficiently large $T>0$. 
\end{pt}
\begin{pt}
\label{part2}
This part concerns the cobordism maps induced by a path $\{\alpha_s\}_{s\in[0,1]}$ of contact forms agreeing with $\alpha$ on a neighborhood of $\partial M$. Similar to Part \ref{part1}, it suffices to prove that the following diagram commutes:
\begin{equation}
\label{eq:handlecobcomdiag}
\xymatrix{
\ech_\emptyset(M',\Gamma',\alpha'_0) \ar[rr]^{F_{\alpha_s,\emptyset}} \ar[d]_{(\Psi')_{\emptyset,\calh'\ssm\calh}} &&
\ech_\emptyset(M',\Gamma',\alpha'_1) \ar[d]^{(\Psi')_{\emptyset,\calh'\ssm\calh}} \\
\ech_{\calh'\ssm\calh}(M',\Gamma',\alpha'_0) \ar[rr]^{F_{\alpha_s,\calh'\ssm\calh}} &&
\ech_{\calh'\ssm\calh}(M',\Gamma',\alpha'_1) \\
} 
\end{equation}
What with Lemma \ref{lem:tau-neck}, the commutativity of the diagram in \eqref{eq:handlecobcomdiag} is equivalent to the commutativity of 
\begin{equation}
\label{eq:neckhandlecomdiag}
\xymatrix{
\ech_\emptyset((M')^T,\Gamma',(\alpha'_0)^T) \ar[rr]^{F_{\alpha_s,\emptyset}} \ar[d]_{(\Psi^{'T})_{\emptyset,\calh'\ssm\calh}} &&
\ech_\emptyset((M')^T,\Gamma',(\alpha'_1)^T) \ar[d]^{(\Psi^{'T})_{\emptyset,\calh'\ssm\calh}} \\
\ech_{\calh'\ssm\calh}((M')^T,\Gamma',(\alpha'_0)^T) \ar[rr]^{F_{\alpha_s,\calh'\ssm\calh}} &&
\ech_{\calh'\ssm\calh}((M')^T,\Gamma',(\alpha'_1)^T).
} 
\end{equation}
Recalling the definition of cobordism maps from Section \ref{ssec:alpha-independence}, the commutativity of the diagram in \eqref{eq:neckhandlecomdiag} is equivalent to the commutativity of 
\[ \xymatrix{
\ech^{e^{2R}L}(\ynt, e^{2R}\antz) \ar[rr]^-{\Phi^{e^{2R}L}(X^n_{10})} \ar[d]_-{\rotatebox{90}{$\sim$}} &&
\ech^{e^{2R}L}(\ynt, e^{R}\anto) \ar[d]^-{\rotatebox{90}{$\sim$}} \\
\ech^{e^{2R}L}(\ynth, e^{2R}\antzp) \ar[rr]^-{\Phi^{e^{2R}L}((X')^n_{10})} &&
\ech^{e^{2R}L}(\ynth, e^{R}\antop).
} \]

Once again, assume without loss of generality that $(\antz,\jntz)$ and $(\anto,\jnto)$ obey the conditions (4-1) in \cite{taubes1} for $e^{2R}L$ on $\ynt$, and $J^0$ is connected to $J^1$ by a path of tailored almost complex structures that is constant in a neighborhood of $\R\times ((M')^\ast\ssm\textit{int}(M'))$ via an analog of Lemma \ref{lem:J-constant-at-infinity} for a path $\{\alpha_s\}_{s\in[0,1]}$ of contact forms on $(M,\Gamma)$ agreeing with $\alpha$ on a neighborhood of $\partial M$. Note that this implies that $(\antzp,\jntzp)$ and $(\antop,\jntop)$ obey the conditions (4-1) in \cite{taubes1} for $e^{2R}L$ on $\ynth$. Then setting $T=n+1$, the vertical isomorphisms identify ECH generators in the source with their images in the target by Lemma \ref{lem:monopoles-114} and Proposition \ref{prop:comp-63}, and the commutativity of the diagram follows from the second bullet of Proposition \ref{prop:comp-62}.
\end{pt}
This completes the proof.
\end{proof}
Finally, we prove the naturality of sutured ECH as we intended to, completing the proofs of Theorems \ref{thm:filtered-sutured-ech} and \ref{thm:intro-alpha-naturality}.
\begin{theorem}
\label{thm:naturality}
Let $(M,\Gamma,\alpha)$ be a sutured contact 3-manifold. Then we have canonically defined groups $\{\ech^L(M,\Gamma,\alpha)\}_L$ which admit inclusion induced maps
\[i^{L,L'}:\ech^L(M,\Gamma,\alpha)\to\ech^{L'}(M,\Gamma,\alpha)\] so that in the direct limit as $L\to\infty$ we obtain canonically defined groups $\ech(M,\Gamma,\alpha)$. Furthermore, for a path $\{\alpha_s\}_{s\in[0,1]}$ of contact forms agreeing with $\alpha$ in a neighborhood of $\partial M$, we have isomorphisms $F_{\alpha_s}:\ech(M,\Gamma,\alpha_0)\to\ech(M,\Gamma,\alpha_1)$ that depend only on the homotopy class of the path $\{\alpha_s\}_{s\in[0,1]}$ relative to its boundary. As a result, we have a canonically defined group $\ech(M,\Gamma,\xi,\alpha|_{\partial M})$ where $\xi=\ker(\alpha)$, and canonical isomorphisms
\[ F_{\xi_s}: \ech(M,\Gamma,\xi_0,\alpha|_{\partial M}) \to \ech(M,\Gamma,\xi_1,\alpha|_{\partial M}) \]
associated to any path $\{\xi_s\}_{s\in[0,1]} \subset \Xi(M,\Gamma,\alpha)$. Moreover, there is a contact class $c(\xi) \in \ech(M,\Gamma,\xi,\alpha|_{\partial M})$ satisfying $F_{\xi_s}(c(\xi_0)) = c(\xi_1)$ for all paths $\xi_s$. 
\end{theorem}

\begin{proof}
By Lemma \ref{lem:poset-connected}, the partially ordered set $\calp_{(M,\Gamma,\alpha)}$ of handle data associated to $(M,\Gamma,\alpha)$ is connected. When $\calh \leq \calh'$, Proposition \ref{prop:poset-functor} gives us the canonical isomorphism $\Psi^L_{\calh,\calh'} = P^{L,J}_{\calh'} \circ (P^{L,J}_{\calh})^{-1}$, which is independent of $J$. We extend this to arbitrary pairs as follows: when $\calh \leq \calh'$, we define $\Psi^L_{\calh',\calh} = (\Psi^L_{\calh,\calh'})^{-1} = P^{L,J}_{\calh} \circ (P^{L,J}_{\calh'})^{-1}$. Now given $\calh$ and $\calh'$ which may not be comparable, we have a sequence $\calh = \calh_0, \calh_1, \calh_2, \dots, \calh_n = \calh'$ such that either $\calh_i < \calh_{i+1}$ or $\calh_i > \calh_{i+1}$ for each $i=0,1,\dots,n-1$, and we define
\[ \Psi^L_{\calh,\calh'} = \Psi^L_{\calh_{n-1},\calh'} \circ \Psi^L_{\calh_{n-2},\calh_{n-1}} \circ \dots \circ \Psi^L_{\calh,\calh_1}. \]
But then for any fixed $J$ we have
\[\Psi^L_{\calh,\calh'} = \big(P^{L,J}_{\calh'} \circ (P^{L,J}_{\calh_{n-1}})^{-1}\big) \circ \big(P^{L,J}_{\calh_{n-1}} \circ (P^{L,J}_{\calh_{n-2}})^{-1}\big) \circ \dots \circ \big(P^{L,J}_{\calh_1} \circ (P^{L,J}_{\calh})^{-1}\big)=P^{L,J}_{\calh'} \circ (P^{L,J}_{\calh})^{-1}.\]
These isomorphisms are evidently independent of the path from $\calh$ to $\calh'$ in $\calp_{(M,\Gamma,\alpha)}$, and they satisfy $\Psi^L_{\calh,\calh''} = \Psi^L_{\calh',\calh''} \circ \Psi^L_{\calh,\calh'}$ for any $\calh,\calh',\calh''$. As a result, 
\[ \big(\{\ech^L_\calh(M,\Gamma,\alpha)\}_{\calh\in\calp_{(M,\Gamma,\alpha)}}, \{\Psi^L_{\calh,\calh'}\}_{\calh,\calh'}\big) \]
yields a canonically defined group $\ech^L(M,\Gamma,\alpha)$.

The fact that there are inclusion induced maps $i^{L,L'}:\ech^L(M,\Gamma,\alpha)\to \ech^{L'}(M,\Gamma,\alpha)$ also follows from the proof of Proposition \ref{prop:poset-functor}. To be more precise, for arbitrary $\calh$ and $\calh'$ we have seen that $\Psi^L_{\calh,\calh'} = P^{L,J}_{\calh'} \circ (P^{L,J}_{\calh})^{-1}$, and that the maps $i^{L,L'}_{\calh}$ and $i^{L,L'}_{\calh'}$ are intertwined with the isomorphisms $P^{L,J}_{\calh}$ and $P^{L,J}_{\calh'}$ by definition, so $\Psi^{L'}_{\calh,\calh'} \circ i^{L,L'}_{\calh} = i^{L,L'}_{\calh'} \circ \Psi^{L}_{\calh,\calh'}$.  Therefore there are canonical maps $i^{L,L'}: \ech^L(M,\Gamma,\alpha) \to \ech^{L'}(M,\Gamma,\alpha)$, and the direct limit
\[\ech(M,\Gamma,\alpha):=\lim_{L\to\infty}\ech^L(M,\Gamma,\alpha)\]
is canonically defined.

As for the penultimate claim of the theorem, this follows from the last claim of Proposition~\ref{prop:poset-functor}. To elaborate, when $\calh \leq \calh'$, the isomorphisms $\Psi_{\calh,\calh'}$ of that proposition commute with the isomorphisms of Section \ref{ssec:alpha-naturality} defined by a path $\{\alpha_s\}_{s\in[0,1]}$ of contact forms agreeing with $\alpha$ in a neighborhood of $\partial M$. The latter isomorphisms depend only on the homotopy class of the path $\{\alpha_s\}_{s\in[0,1]}$ relative to its boundary. For arbitrary $\calh,\calh'$, the isomorphism $\Psi_{\calh,\calh'}$ is a composition of such maps and their inverses, hence it has the same properties.  As a result, there exists a canonical isomorphism $F_{\alpha_s}:\ech(M,\Gamma,\alpha_0)\to\ech(M,\Gamma,\alpha_1)$. These isomorphisms compose naturally and therefore result in a canonically defined group $\ech(M,\Gamma,\xi,\alpha|_{\partial M})$ where $\xi=\ker(\alpha)$, and canonical isomorphisms $F_{\xi_s}$ as claimed, exactly as in  Lemmas \ref{lem:f-alpha-independent} and \ref{lem:f-xi-independent} respectively.

Finally, the assertion about the contact class follows in two steps. First, for each $L>0$, there is a well-defined contact class in $\ech^L_\calh(M,\Gamma,\alpha)$ due to Proposition \ref{prop:contact-class-j} and the fact that the isomorphisms $\Psi^L_{J;E,E'}$ send $[\emptyset]$ to $[\emptyset]$. Second, Proposition \ref{prop:contact-class-j} implies that the maps $P^{L,J}_\calh$ send $[\emptyset]$ to this contact class, and hence the maps $\Psi^L_{\calh,\calh'}$ send contact class to contact class. As a result, there is a well-defined contact class in each $\ech^L(M,\Gamma,\xi,\alpha|_{\partial M})$, which yields in the direct limit as $L\to\infty$ a well-defined contact class $c(\xi)\in \ech^L(M,\Gamma,\xi,\alpha|_{\partial M})$ satisfying $F_{\xi_s}(c(\xi_0)) = c(\xi_1)$ for all paths $\xi_s$.
\end{proof}

%% file: a_setup.tex
Use $T$ in what follows to denote either a number greater than 16 or infinity. Suppose that $Y_T$ is an oriented 3-manifold of the following sort: $Y_T$ is a compact manifold without boundary when $T$ is finite, and it is a non-compact manifold without boundary if $T$ is infinity. There exists an open set $N\subset Y_T$ with compact closure whose complement is the union of three sets: the closure of $R_+(\Gamma)\times[1,T)$ where $R_+(\Gamma)$ is a connected oriented surface with one boundary component $\Gamma$, the closure of $R_-(\Gamma)\times(-T,-1]$ where $R_-(\Gamma)$ is also a connected oriented surface with one boundary component $\Gamma$, and the closure of $[0,T)\times(-T,T)\times\Gamma$ where $\partial R_+(\Gamma)\times[1,T)$ is identified with $\{0\}\times[1,T)\times\Gamma$, and $\partial R_-(\Gamma)\times(-T,-1]$ is identified with $\{0\}\times(-T,-1]\times\Gamma$.  See Figure~\ref{fig:a-yt}: here the components of $N$ are the interior of the sutured manifold $M$, as well as the interior of the solid torus pictured in the center of the left figure (drawn as a cylinder, with the front and back faces glued together) in the case where $T$ is finite.

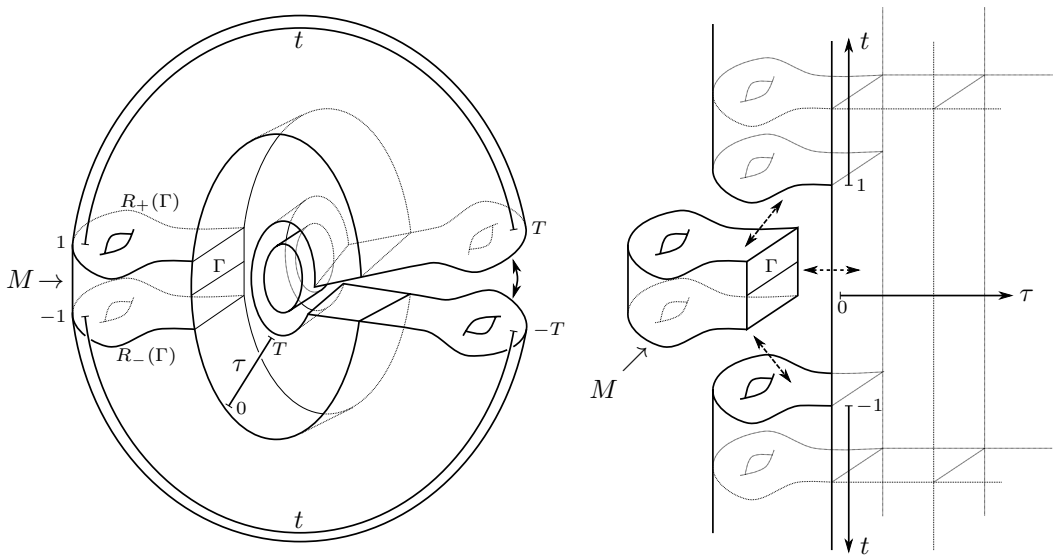
\begin{figure}[ht]
\labellist
\small \hair 2pt
\pinlabel $M\!\rightarrow$ [r] at 105 136
\pinlabel $t$ [t] at 213 255
\pinlabel $t$ [b] at 213 16
\pinlabel $\tau$ [b] at 185 90
\pinlabel $\nearrow$ [r] at 380 98
\pinlabel $M$ [r] at 366 84
\pinlabel $t$ [l] at 475 248
\pinlabel $t$ [l] at 475 10
\pinlabel $\tau$ [b] at 556 123
\tiny
\pinlabel $R_+(\Gamma)$ [b] at 143 165
\pinlabel $R_-(\Gamma)$ [t] at 141 106
\pinlabel $\Gamma$ [l] at 170 142
\pinlabel $1$ [r] at 106 153
\pinlabel $T$ [l] at 320 160
\pinlabel $-1$ [r] at 106 118
\pinlabel $-T$ [l] at 321 112
\pinlabel $0$ [l] at 181 75
\pinlabel $T$ [l] at 198 103
\pinlabel $\Gamma$ [l] at 430 142
\pinlabel $1$ [l] at 473 180
\pinlabel $-1$ [l] at 473 76
\pinlabel $0$ [t] at 469 128
\endlabellist
\includegraphics[scale=0.8]{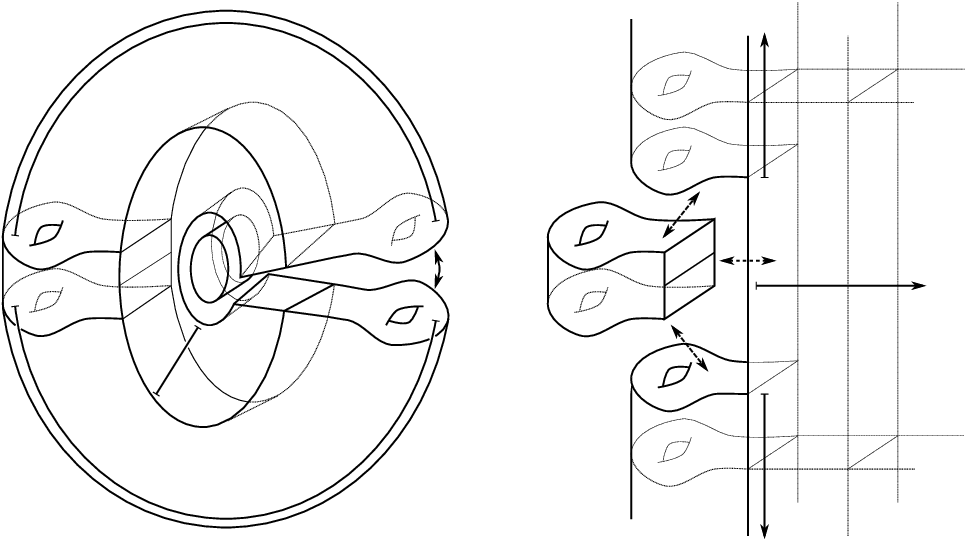}
\caption{The manifolds $Y_T$ for $T < \infty$ (left) and $Y_\infty$ (right).}
\label{fig:a-yt}
\end{figure}

To set the notation, use $t$ to denote the Euclidean coordinate for the $[1,T)$ factor of $R_+(\Gamma)\times[1,T)$, for the $(-T,-1]$ factor of $R_-(\Gamma)\times(-T,-1]$, and for the $(-T,T)$ factor of $[0,T)\times(-T,T)\times\Gamma$. Use $\tau$ to denote the Euclidean coordinate on the $[0,T)$ factor of $[0,T)\times(-T,T)\times\Gamma$. A chosen $\R/2\pi\Z$ is denoted by $\theta$.

Suppose that $Y_T$ is equipped with a contact 1-form denoted by $a$. This 1-form is assumed to have the following properties: $a$ restricts to $R_+(\Gamma)\times[1,T)$ so as to have the form $a=dt+b_+$ where $b_+$ is a 1-form on $R_+(\Gamma)$ such that $db_+$ is an area form on $R_+(\Gamma)$. Likewise, $a$ restricts to $R_-(\Gamma)\times(-T,-1]$ so as to have the form $a=dt+b_-$ where $b_-$ is a 1-form on $R_-(\Gamma)$ such that $db_-$ is an area form on $R_-(\Gamma)$. Meanwhile, $a$ restricts to $[0,T)\times(-T,T)\times\Gamma$ as $dt+e^\tau d\theta$. 

Fix a Riemannian metric on $Y_T$ such that $a$ has unit norm and the kernel of $a$ is orthogonal to the kernel of $da$. The metric on the $[0,T)\times(-T,T)\times\Gamma$ part of $Y_T$ can be chosen so that an orthonormal basis for $T^\ast Y_T$ is given by 
\begin{equation}
\label{eq:1.1}
\hat{e}^1=\frac{1}{2}e^\tau d\tau,\hspace{0.25in}\hat{e}^2=d\theta,\hspace{0.25in}\hat{e}^3=dt+e^\tau d\theta,
\end{equation}
Note that $a=\hat{e}^3$ and $da=2\hat{e}^1\wedge\hat{e}^2$. As a result, the curvature tensor is covariantly constant, and there is no worry about unruly Riemannian curvature bounds as $T$ increases. Note also that the norm of the covariant derivatives of $a$ and $da$ are constant.

In general, let $\rmu$ be a smooth 3-manifold equipped with a contact form $a$ and a Riemannian metric whereby $a$ has unit norm and $2\ast a=da$. Suppose also that there is a uniform bound on the norm of the curvature of the metric. Having fixed such a metric, a $\spinc$ structure on $\rmu$ is a lift of the oriented orthonormal frame bundle of $\rmu$ to a principal $\spinc(3)$-bundle. Since $\spinc(3)$ is the Lie group $U(2)$, the $\spinc$ structure has an associated Hermitian $\C^2$-bundle $\mathbb{S}$, called the \emph{spinor bundle}, and Clifford multiplication by $a$ defines a splitting of $\mathbb{S}$ into $\pm i$ eigenbundles:
\begin{equation}
\label{eq:splitting}
E\oplus (E\otimes K^{-1}),
\end{equation}
where $E$ is the $+i$ eigenbundle and $E\otimes K^{-1}$ is the $-i$ eigenbundle of Clifford multiplication by $a$, and $K^{-1}$ is isomorphic to the oriented 2-plane bundle that is the kernel of $a$ and oriented by $da$. With this splitting in mind, any given section of $\mathbb{S}$ can be written as a pair $(\alpha,\beta)$ where $\alpha$ is a section of $E$ and $\beta$ is a section of $E\otimes K^{-1}$. The $\spinc$ structure with $E$ isomorphic to the product complex line bundle $\underline{\C}$ is called the \emph{canonical ${\textit Spin^c}$ structure}. The spinor bundle associated to the canonical $\spinc$ structure is denoted by $\mathbb{S}_{{\rm I}}$. 
Let $\Ai$ denote the product connection on $\underline{\C}$ and $1_\C$ denote the constant section given by the complex number $1\in\C$. Use the splitting of $\si$ as $\underline{\C}\oplus K^{-1}$ to view the pair $\psii=(1_\C,0)$ as a section of $\si$. 

With $\Ai$ understood, a connection on $K^{-1}$ and the metric's Levi--Civita connection define a Dirac operator on sections of $\si$. In particular, there is a unique connection on $K^{-1}$ such that the resulting Dirac operator, denoted $\textup{D}_{\uptheta_0}$, annihilates the section $\psii$. This connection is called the \emph{canonical connection} on $K^{-1}$. If $E$ is any given bundle, and $A$ is a connection on $E$, then $A$, the canonical connection on $K^{-1}$, and the metric's Levi--Civita connection define a Dirac operator on $\mathbb{S}$. This is denoted by $\textup{D}_{A}$. By way of notation, the covariant derivative defined by $A$ on sections of $E$ is denoted by $\nabla_A$. The covariant derivative defined by $A$ and the canonical connection on $K^{-1}$ on sections of $E\otimes K^{-1}$ is also denoted by $\nabla_A$. This notation is also used to denote the covariant derivatives that is defined by these connections and the Levi--Civita connection on sections of the tensor product of these bundles with powers of $T^\ast Y_T$ and $TY_T$.

Let $A_0$ denote in what follows a second connection on $K^{-1}$, one whose curvature $2$-form is zero on $Y_T \ssm N$.  This is important. It is also important to choose $A_0$ to be independent of $T$ on $N$. (The inclusion-induced homomorphism from $H_2(Y_T \ssm N; \Z)$ to $H_2(Y_T;\Z)$ is zero, so a connection with these properties can always be found.) The curvature 2-form of $A_0$ is denoted in what follows by $F_{A_0}$.

Throughout the rest of the appendix, $c_0>100$ denotes a number that is independent of the relevant solution to \eqref{eq:sw} or \eqref{eq:instanton-sw}, as well as $T$ and $r$. Its value is allowed to increase between successive appearances.

%% file: a_monopoles.tex
Fix a $\spinc$ structure on $\rmu$ and let $\mathbb{S}$ denote the associated spinor bundle with the splitting as in (\ref{eq:splitting}). Fix $r\geq 1$. A pair $(A,\psi)$ of a smooth connection $A$ on $E$ and a smooth section $\psi$ of $\mathbb{S}$ is said to obey the Seiberg--Witten equations if it satisfies
\begin{eqnarray}
\label{eq:sw}
\nonumber\ast F_{A}&=&r(\psi^{\dagger} \tau \psi-ia) - \frac{1}{2}*F_{A_0}\\
\textup{D}_{A}\psi&=&0.
\end{eqnarray}
where $F_{A}\in \Omega^2(\rmu,i \mathbb{R})$ denotes the curvature form of the connection $A$, $\psi^{\dagger}\tau\psi$ denotes the section of $iT^{\ast}\mathrm{U}$ which is the metric dual of the homomorphism $\psi^{\dagger}\mathfrak{cl(\cdot)}\psi:T^{\ast}\mathrm{U}\rightarrow i\mathbb{R}$ with $\mathfrak{cl}$ being the Clifford multiplication, and $\textup{D}_{A}$ is the Dirac operator associated to $A$, which is defined by
\[\Gamma(\mathbb{S})\stackrel{\nabla_{A}}\longrightarrow \Gamma(\mathrm{T^*U\otimes \mathbb{S}})\stackrel{\mathfrak{cl}}{\longrightarrow}\Gamma(\mathbb{S}).\]
The linearized version of the equations in (\ref{eq:sw}) defines a differential operator for a pair $(\fra,\eta)$ of ${L^2}_1$ sections of $iT^\ast \rmu\oplus\mathbb{S}$. These equations with an extra gauge fixing equation can be extended to give an elliptic operator $\mathfrak{L}$, from the space of ${L^2}_1$ sections of $i(T^\ast \rmu\oplus \underline{\R})\oplus\mathbb{S}$ to the space of $L^2$ sections of the same bundle. Here and in what follows $\underline{\R}$ denotes the product real line bundle over $\rmu$. Let $\frh=(\fra,\phi,\eta)$ denote an ${L^2}_1$ section of the bundle $i(T^\ast \rmu\oplus \underline{\R})\oplus\mathbb{S}$. Then 
\begin{equation}
\label{eq:linearized}
\mathfrak{L}_{(A,\psi)}(\fra,\phi,\eta)=\left(\begin{array}{c}\ast\mathrm{d\fra}-\mathrm{d}\phi-r^{1/2}(\psi^{\dagger}\tau\eta+\eta^{\dagger}\tau\psi)\\ 
\ast d \ast\fra+r^{1/2}(\psi^{\dagger}\eta-\eta^{\dagger}\psi)\\
\textup{D}_{A}\eta+r^{1/2}(\mathfrak{cl}(\fra)\psi+\phi\psi)\end{array}\right).
\end{equation}
The version of the operator $\mathfrak{L}$ defined by the pair $(\Ai,\psii)$ is denoted in what follows by $\Lif$.
\subsubsection{Exponential decay}
\label{sssec:exponential_decay}
A Bochner--Weitzenb\"ock formula for $\mathfrak{L}^2$ can be found in \cite[Section 5e]{taubes-weinstein1}. The version for $\Lif$ leads in particular to the following assertion:
\begin{quote}
There exists $\kappa>100$ with the following significance: Fix $r>\kappa$ and suppose that $\frh$ is a compactly supported section of $i(T^\ast \rmu\oplus\underline{\R})\oplus\spb$ on $\rmu$. Then 
\begin{equation}
\label{eq:monopoles-15}
{||\Lif\frh||_2}^2\geq (1-\kappa^{-1})({||\nabla_{\Ai}\frh||_2}^2+r{||\frh||_2}^2).
\end{equation}
Here $||\cdot||_2$ denotes the $L^2$-norm.
\end{quote}
To set the stage for the first application of \eqref{eq:monopoles-15}, let $\pzee$ denote a linear operator on sections of $i(T^\ast \rmu\oplus \underline{\R})\oplus\mathbb{S}$ such that if $\frh$ has compact support in $\rmu$, then
\begin{equation}
\label{eq:monopoles-16}
{||\pzee\frh||_2}^2\leq \frac{1}{2}({||\nabla_{\Ai}\frh||_2}^2+r{||\frh||_2}^2).
\end{equation}
By way of a relevant example, suppose that $(A,\psi)$ is a pair of a connection $A$ on $E=\underline{\C}$ and a section $\psi$ of $\spb=\si$. Suppose further that there exists an isomorphism between $E$ and $\underline{\C}$ on $\rmu$ such that we can write $A=\Ai+\aA$ and $\psi=\psii+\eta$ where 
\begin{equation}
\label{eq:monopoles-17}
r^{-1/2}|\aA|+|\eta|\leq {c_0}^{-1},
\end{equation}
for some constant $c_0>1$ which is suitably large, but can be taken to be independent of $r$ and $T$. Then this isomorphism between $E$ and $\underline{\C}$ writes $\LLf_{(A,\psi)}=\Lif+\pzee$ with $\pzee$ obeying \eqref{eq:monopoles-16}. 

In any event, suppose that $\pzee$ is a linear operator on sections of $i(T^\ast \rmu\oplus \underline{\R})\oplus\mathbb{S}$ that obeys \eqref{eq:monopoles-16}. Suppose now that $\frh$ is a section of $i(T^\ast \rmu\oplus \underline{\R})\oplus\mathbb{S}$ with no constraint on its support that obeys the equation
\begin{equation}
\label{eq:monopoles-18}
\Lif\frh+\pzee\frh=0
\end{equation}
on $\rmu$. The upcoming Lemma \ref{lem:monopoles-11} makes an assertion to the effect that $\frh$ is small at distances greater than $\mathcal{O}(r^{-1/2})$ from the \emph{frontier} of $\rmu$. The notion of the frontier is made precise using the following definitions. Fix an open set $V\subset \rmu$. Fix $\rho>0$ and for each integer $k\in\{0,1,\dots\}$, let $V_k$ denote the set of points in $V$ with distance $k\rho$ or more from each point in $\rmu\ssm V$. Assume in any event that $V_k\ssm V_{k+1}$ has compact closure in $\rmu$ for all $k\in\{0,1,\dots\}$. For each integer $k\in\{0,1,\dots\}$, let $||\frh||_{\ast,k}$ denote the positive square root of the number 
\begin{equation}
\label{eq:monopoles-19}
\int_{V_k\ssm V_{k+1}}(|\nabla_{\Ai}\frh|^2+r|\frh|^2),
\end{equation}
this being the measure of the size of $\frh$ on $V_k\ssm V_{k+1}$.
\begin{lemma}
\label{lem:monopoles-11}
There exists $\kappa>1$ with the following significance: Fix $r>\kappa$ and suppose that $\pzee$ obeys \eqref{eq:monopoles-16}. Fix an open set $V\subset \rmu$ and $\rho>\kappa r^{-1/2}$. For each integer $k\in\{0,1,\dots\}$, define $V_k$ as above and assume that $V_k\ssm V_{k+1}$ has compact closure in $\rmu$. If $\frh$ obeys \eqref{eq:monopoles-18}, then
\[||\frh||_{\ast,k}\leq \kappa ||\frh||_{\ast,0}e^{-\sqrt{r}k\rho/\kappa},\]
for each integer $k\in\{0,1,\dots\}$.
\end{lemma}
\begin{proof}
Fix a non-negative smooth function $\chi$ on $\R$ that is equal to $0$ on $(-\infty,\frac{1}{8}]$ and equal to $1$ on $[\frac{7}{8},\infty)$, and with derivative bounded by a constant $c_0>1$. Let ${\textsc d}$ denote the function on $V$ that gives the distance to $\rmu\ssm V$. For each non-negative integer $k$, let  $\chi_k=\chi(2\frac{{\textsc d}}{\rho}-k)$. This function is equal to $0$ where $\frac{1}{2}k\rho+\frac{1}{16}\rho\geq{\textsc d}$ and it is equal to $1$ where $\frac{1}{2}k\rho+\frac{7}{16}\rho\leq{\textsc d}$. For any non-negative integer $k$, set $\varphi_k=\chi_k-\chi_{k+1}$. Note that $\varphi_k$ is zero if $\frac{1}{2}k\rho+\frac{1}{16}\rho\geq{\textsc d}$ or $\frac{1}{2}k\rho+\frac{15}{16}\rho\leq{\textsc d}$. Therefore, $\varphi_k$ is supported in $V_k\ssm V_{k+1}$. Moreover $\sum_{k=1}^\infty\varphi_k=1$ on $V_1$. 

Let $\frh$ obey \eqref{eq:monopoles-18}. For any non-negative integer $k$ set $\frh_k=\varphi_k\frh$. These obey an equation that has the schematic form 
\[\Lif\frh_k+\pzee\frh_k=\sigma_\LLf(d\varphi_k)\frh+\sigma_\pzee(d\varphi_k)\frh,\]
with $\sigma_\LLf$ and $\sigma_\pzee$ being the principal symbols of the operators $\mathfrak{L}$ and $\pzee$, respectively. The assumption in \eqref{eq:monopoles-16} implies that $|\sigma_\pzee|\leq c_0$. Since $\varphi_k\neq0$ only where $\varphi_{k-1}$, $\varphi_k$, and $\varphi_{k+1}$ are non-zero, the preceding equation for $\frh_k$ can be written schematically as 
\[\Lif\frh_k+\pzee\frh_k=\rho^{-1}(z_{k-1}\frh_{k-1}+z_k\frh_k+z_{k+1}\frh_{k+1})\]
with $z_{k-1}$, $z_k$, and $z_{k+1}$ being zeroth order operators whose norms are bounded by $c_0$. Let $x_k$ denote the $L^2$-norm of $\frh_k$. Take the $L^2$-norm of both sides of the last equation and then use \eqref{eq:monopoles-15} and \eqref{eq:monopoles-16} to see that the collection $\{x_k\}_{k=0,1,\dots}$ obeys 
\begin{equation}
\label{eq:monopoles-110}
x_k\leq c_0 r^{-1/2}\rho^{-1}(x_{k+1}+x_{k-1}).
\end{equation}
Granted that $r^{1/2}\rho>c_0$, \eqref{eq:monopoles-110} implies a `second order' difference equation that reads
\begin{equation}
\label{eq:monopoles-111}
-(x_{k+1}+x_{k-1}-2x_k)+{c_0}^{-1}r^{1/2}\rho x_k\leq 0.
\end{equation}
This difference equation implies in turn that $x_k$ obeys
\begin{equation}
\label{eq:monopoles-112}
x_k\leq c_0 x_0e^{-\sqrt{r}k\rho/c_0}.
\end{equation}
Then \eqref{eq:monopoles-112} together with \eqref{eq:monopoles-15} and \eqref{eq:monopoles-16} imply what is asserted by Lemma \ref{lem:monopoles-11}.
\end{proof}
\subsubsection{Fredholm property of $\LLf$}
\label{sssec:fredholm}
It is a standard result when $\rmu$ is a compact manifold that any given $(A, \psi)$ version
of the operator $\LLf$ from \eqref{eq:linearized} defines an unbounded, self-adjoint operator on the Hilbert space $L^2(\rmu;i(T^\ast \rmu\oplus\underline{\R})\oplus\spb)$ with pure point spectrum that has no accumulation points. In the case when $\rmu$ is not compact, this is not going to be the case in general. Even so, in the case when $\rmu$ is $Y_\infty$, the operator $\LLf$ will still be self-adjoint and Fredholm if $(A,\psi)$ is suitably constrained on $Y_\infty\ssm N$ and if $\spb=\si$ on $Y_\infty\ssm N$. These constraints are defined in the next paragraph. Pairs that satisfy the constraints are said to be \emph{admissible}. 

A pair $(A,\psi)$ of connection on the $Y_\infty$ version of $E$ and section of the $Y_\infty$ version of $\mathbb{S}$ is said to be admissible if there exists an isomorphism on $Y_\infty\ssm N$ from $\underline{\C}$ to $E$ to be denoted by $g$ with the following properties: both $g^\ast\psi-\psii$ and $\nabla_{\Ai}(g^\ast\psi)$ are square integrable on $Y_\infty\ssm N$ as is both $g^\ast A-\Ai$ and $\nabla(g^\ast A-\Ai)$. All solutions to the $T=\infty$ version of \eqref{eq:sw} are henceforth assumed without further comment to be admissible. 

By way of an example, the pair $(\Ai,\psii)$ is an admissible solution to \eqref{eq:sw} in the case when $\spb$ is the spinor bundle associated to the canonical $\spinc$ structure. By way of notation, the definition of the kernel of $\LLf$ that is used in the forthcoming lemma, and henceforth, for $Y_\infty$ has $\frh$ being in the kernel if and only if $\LLf\frh=0$ and $\frh$ is square integrable on $Y_\infty$. 

\begin{lemma}
\label{lem:monopoles-12}
Take $\rmu$ to be $Y_\infty$. There exists $\kappa>1$ with the following significance: If $r>\kappa$ and if $(A,\psi)$ is admissible, then the corresponding version of the operator $\LLf$ defines an unbounded, self-adjoint operator on $L^2$ with dense domain ${L^2}_1$. As a self-adjoint operator, $\LLf$ has only pure point spectrum in the interval $[-\kappa^{-1}r^{1/2},\kappa^{-1}r^{1/2}]$ with no accumulation points in this interval. In particular, if $\LLf$ has trivial kernel, then there exists $c>1$ such that 
\[{||\LLf(\cdot)||_2}^2\geq c^{-2}({||\nabla_A(\cdot)||_2}^2+r{||\cdot||_2}^2),\]
on the domain of $\LLf$.
\end{lemma}

\begin{proof}
The operator $\LLf$ defines a unbounded operator from the space of ${L^2}_1$ sections to the Sobolev space of $L^2$ sections of $i(T^*\rmu \oplus \underline{\R}) \oplus \spb$. If $(A,\psi)$ is admissible, then there is an isomorphism between the bundles $E$ and $\underline{\C}$ on $Y_\infty\ssm N$ that writes $A$ as $\uptheta_0+\aA$ and $\psi$ as $\psii+\eta$ with $(\aA,\eta)$ and its covariant derivatives being in $L^2$ on $Y_\infty\ssm N$. This implies that the isomorphism writes $\LLf$ as $\Lif+\pzee$ with $|\pzee|$ being in $L^p$ for $p\in\{2,\dots,6\}$. This follows from the Sobolev inequalities in dimension 3 because $|\pzee| \leq c_0(|\aA|+|\eta|)$. Keep in mind in this regard that ${L^2}_1$ functions in dimension 3 are in $L^p$ for the indicated range of $p$ (this is the Sobolev inequalities); and since both $\aA$ and $\eta$ are in ${L^2}_1$, their norms are ${L^2}_1$ functions.  Keeping the preceding in mind, suppose that $R>1$ and that $\frh$ is an ${L^2}_1$ section of $i(T^*\rmu\oplus\underline{R})\oplus\spb$ over $Y_\infty$ with compact support on the part of $Y_\infty \ssm N$ where the distance to $N$ is greater than $R$.  Then, the $L^2$ norm of $\pzee\frh$ is bounded by the product of the $L^4$ norms of $\pzee$ and $\frh$.  A dimension-3 Sobolev inequality says that the $L^4$ norm of $\frh$ is bounded by $c_0$ times its ${L^2}_1$ norm. Meanwhile, since the $L^4$ norm of $\pzee$ on $Y_\infty \ssm N$ is finite, the $L^4$ norm of $\pzee$ where the distance to $N$ is greater than $R$ will be less than any positive number specified in advance if $R$ is sufficiently large.  As a consequence, there exists $R>1$ so that \eqref{eq:monopoles-16} holds for all ${L^2}_1$ sections of $i(T^*\rmu\oplus\underline{\R})\oplus\spb)$ on $Y_\infty$ with compact support where the distance to $N$ on $Y_\infty \ssm N$ is greater than $R$. This implies that
\begin{equation}
\label{eq:monopoles-113}
{||\LLf\frh||_2}^2 \geq \frac{1}{4}({||\nabla_{\uptheta_0}\frh||_2}^2 + r{||\frh||_2}^2)
\end{equation}
if $\frh$ has compact support where the distance to $N$ is greater than $R$.

Now suppose that $c>100$ and that $\frh$ is an ${L^2}_1$ section of $i(T^\ast\rmu\oplus\underline{\R})\oplus\spb$ on $Y_\infty$ with the property that $||\LLf\frh||_2 \leq c^{-1}r^{1/2}||\frh||_2$. Let $\chi_R$ denote for the moment a non-negative function that is equal to $1$ where the distance to $N$ is less than $R$ and equal to $0$ where the distance is greater than $R+1$. Such a function can and should be taken so that $|d\chi_R| \leq c_0$.  The inequality in \eqref{eq:monopoles-113} with $\frh$ replaced by $(1-\chi_R)\frh$ leads to 
\begin{itemize}\leftskip-0.25in
\item ${||\nabla_{\uptheta_0}((1-\chi_R)\frh)||_2}^2 + r{||(1-\chi_R)\frh||_2}^2 \leq c_0c^{-1}r{||\chi_R\frh||_2}^2$.
\item ${||\nabla_A((1-\chi_R)\frh)||_2}^2 + r{||(1-\chi_R)\frh||_2}^2 \leq c_0c^{-1}r{||\chi_R\frh||_2}^2$.
\end{itemize}
\begin{equation}
\label{eq:monopoles-new-112}
\end{equation}
if $c>c_0$. To explain how these bounds come about, note that the second bullet's inequality follows from the first bullet's inequality (and vice-versa) using the aforementioned dimension 3 Sobolev inequality concerning $L^4$ norms. This is because of the assumption that $\aA=A-\uptheta_0$ is an $L^2_1$ function where the distance to $N$ is large. A trick is used to derive the top bullet's inequality: Write $\LLf((1-\chi_R)\frh)$ as $(1-\chi_R)\LLf\frh - \sigma_\LLf(d\chi_R)\frh$ with $\sigma_\LLf$ denoting here the principal symbol of $\LLf$. The $L^2$ norm of $(1-\chi_R)\LLf\frh$ is no greater than $c^{-1}r^{1/2}$ times that of $\frh$ which is, in turn, no greater than the sum of the products of $c^{-1}r^{1/2}$ times the $L^2$ norms of $(1-\chi_R)\frh$ and $\chi_R\frh$.  Also, write $d\chi_R$ as $\chi_Rd\chi_R + (1-\chi_R)d\chi_R$ so as to deal with the $L^2$ norm of the $\sigma_\LLf(d\chi_R)\frh$ term. Meanwhile, the Bochner-Weitzenb\"ock formula for $\LLf$ and \eqref{eq:monopoles-new-112} lead to a bound of the form
\begin{equation}
\label{eq:monopoles-new-113}
{||\LLf(\chi_R\frh)||_2}^2 + r{||\chi_R\frh||_2}^2 \geq c_0^{-1}{||\nabla_A(\chi_R\frh)||_2}^2;
\end{equation}
and then, assuming that $c>c_0$ (and after writing $\LLf(\chi_R\frh)$ as $\chi_R\LLf\frh + \sigma_\LLf(d\chi_R)\frh$ and with an appeal to \eqref{eq:monopoles-new-112}):
\begin{equation}
\label{eq:monopoles-new-114}
{||\LLf\frh||_2}^2 + r{||\chi_R\frh||_2}^2 \geq c_0^{-1}{||\nabla_A(\chi_R\frh)||_2}^2.
\end{equation}
Given that $||\LLf\frh||_2 \leq c^{-1}r^{1/2}||\frh||_2$, this bound (and \eqref{eq:monopoles-new-112} again) tell us that
\begin{equation}
\label{eq:monopoles-new-115}
{||\nabla_A(\chi_R\frh)||_2}^2 \leq c_0r{||\chi_R\frh||_2}^2
\end{equation}
when $c > c_0^{-1}$.

Now suppose that $c>100$ and that $\calv$ is a linear subspace of ${L^2}_1$ sections of $i(T^*\rmu\oplus\underline{\R})\oplus\spb$ on $Y_\infty$ with each section obeying $||\LLf\frh||_2 \leq c^{-1}r^{1/2}||\frh||_2$.  The argument that follows proves that the unit sphere in $\calv$ (measured by the $L^2$ norm) is compact.  To prove this, let $\{\frh\}_{i=1,2,\dots}$ denote a sequence in $\calv$ with each $\frh_i$ having $L^2$ norm equal to 1.  Since the $\frh=\frh_i$ version of the left hand side of \eqref{eq:monopoles-new-115} is bounded by $c_0r$ (assuming $c>c_0$), the sequence $\{\chi_R\frh_i\}$ is uniformly bounded in the ${L^2}_1$ topology.  Keeping in mind that the sections in the latter sequence are supported in a fixed compact subset of $Y_\infty$, the Rellich lemma says that the sequence $\{\chi_R\frh_i\}$ converges strongly in the $L^2$ topology. (The Rellich lemma says that the tautological ${L^2}_1$ to $L^2$ map on compact sets is a compact mapping.)  Given the $L^2$ convergence of $\{\chi_R\frh_i\}$ then the $\frh=\frh_i-\frh_k$ versions of \eqref{eq:monopoles-new-115} imply that $\{\chi_R\frh_i\}$ converges strongly in the ${L^2}_1$ topology also.  Meanwhile, the $\frh=\frh_i-\frh_k$ versions of \eqref{eq:monopoles-new-112} imply that $\{(1-\chi_R)\frh_i\}$ also converges strongly in the ${L^2}_1$ topology.  Thus, $\{\frh_i\}$ converges in the ${L^2}_1$ topology.  This implies that $\calv$ is sequentially compact (if $c>c_0$).  What with this last conclusion, standard spectral theory for essentially self-adjoint operators (see for example Kato's book \cite{kato-book}) leads directly to all of the conclusions of Lemma~\ref{lem:monopoles-12}.
\end{proof}
\subsubsection{The behavior of solutions}
\label{sssec:behavior}
Suppose that $V\subset \rmu$ is an open subset. The metric on $\rmu$ may not be complete, which would be the case if $\rmu$ is a subset of a larger manifold. If the metric is not complete, assume that any geodesic arc in $\rmu$ that starts at a point in $V$ can be continued in $\rmu$ for distance at least $1$. 

Suppose that $(A,\psi)$ obeys \eqref{eq:sw} on $\rmu$ with $|\psi|$ being bounded. Let $c$ denote an upper bound for its norm. In the case of $Y_T$ for $T<\infty$ and in the case of $Y_\infty$ if $(A,\psi)$ is admissible, a maximum principle argument can be used to prove that $|\psi| \leq c_0$ (see the proof of Lemma 2.2 in \cite{taubes-weinstein1}).  Let $m$ denote for the moment the distance in $V$ to the locus in $\rmu$ where $|\alpha| \leq 1 - c_0^{-1}$. The following bounds are obeyed on $V$ if $c>c_0$:
\begin{equation}
\label{eq:monopoles-new-116}
\begin{split}
|\alpha| &\leq 1+c_0cr^{-1} \\
r(1-|\alpha|^2) + r^2|\beta|^2 + |\nabla_A\alpha|^2 + r|\nabla_A\beta|^2 &\leq c_0c(re^{-r^{1/2}m/c_0}+1).
\end{split}
\end{equation}
The proof copies that of Lemma 3.6 in \cite{taubes-weinstein2}.  (The only tools used in the proof are the maximum principle and the Bochner-Weitzenb\"ock formula for the Dirac operator. See also the proofs of similar assertions in \cite[Sections 2 and 3]{taubes4} and \cite[Section 6]{taubes-weinstein1}.)  Uniform a priori bounds for the norms of higher derivatives of $(\alpha,\beta)$ can be obtained as well. See for example Lemma 1.7 in \cite{taubes-weinstein2}.

The bounds in \eqref{eq:monopoles-new-116} can be refined on $V$ if the curvature 2-form $F_{A_0}$ from \eqref{eq:sw} is zero on $\rmu$. As explained directly, this is done by invoking Lemma~\ref{lem:monopoles-11}.  To set the stage, suppose that $p\in V$ and $\varepsilon \in (0,1)$ has been specified, and that the distance $m(p)$ from $p$ to the $|\alpha| \leq 1-c_0^{-1}$ locus is greater than $r^{-1/2}c_0(1+c)|\ln \varepsilon|$.  In this event, \eqref{eq:monopoles-new-116} guarantees that
\begin{equation}
\label{eq:monopoles-new-117}
|1-|\alpha|| \leq \frac{1}{100}\varepsilon \mathrm{\ \ \ and\ \ \ } |\beta| \leq r^{-1/2}\frac{1}{100}\varepsilon \mathrm{\ \ \ and\ \ \ } |\nabla_A\psi| \leq c_0c
\end{equation}
at the point $p$ if $r > c_0c\varepsilon^{-2}$.  The bound for $|\alpha|$ in particular has the following consequence: Let $V_{\varepsilon,r} \subset V$ denote the subspace where this bound $m(p) \geq r^{-1/2}c_0(1+c)|\ln\varepsilon|$ holds.  Because $\alpha$ is non-zero on $V_{\varepsilon,r}$, the section $u=\alpha/|\alpha|$ defines an isomorphism between $E$ on $V_{\varepsilon,r}$ and the product $\C$ bundle. Moreover, this isomorphism identifies $\alpha$ with a \emph{real}-valued function that can be written as $1-z$ where $z$ has norm less than $\frac{1}{100}\varepsilon$.  Meanwhile, this identification sends $A$ to a connection on the product bundle that will be written as $\Ai+\aA$.  Because $z$ is real (and has norm at most $\frac{1}{100}\varepsilon$) and $\aA$ is $i\R$-valued, the last inequality in \eqref{eq:monopoles-new-117} implies that $|dz| \leq 1$ and $|\aA| \leq 1.02$. The identification sends $\beta$ to a section of $K^{-1}$ that is denoted below by $\beta_\diamond$.

Let $\frh$ denote $((r^{-1/2}\aA, 0), (z,\beta_\diamond))$, thought of as a section of $i(T^*\rmu \oplus \R) \oplus \spb$. Viewed in this light, $\frh$ on $V$ obeys an equation that has the schematic form of \eqref{eq:monopoles-18} with $\pzee$ being a first order differential operator whose coefficients are either bounded a priori (independent of $T$ and $(A,\psi)$) or are linear functions of $\aA$, $z$, and $\beta_\diamond$. To explain, the $iT^*\rmu$ and $\spb$ summands of \eqref{eq:monopoles-18} are a rewriting of \eqref{eq:sw} in the case when $F_{A_0}$ is absent. The corresponding components of $\pzee$ are zeroth order endomorphisms with pointwise norm bounded by $r^{1/2}|z| + r^{1/2}|\beta|$. These in turn are bounded by $c_0r^{1/2}\varepsilon$ by virtue of \eqref{eq:monopoles-new-117}.  To describe the $i\R$ summand of \eqref{eq:monopoles-18}, note first that the $i\R$ summand of $\Lif\frh$ (with $\frh$ as above) is simply $r^{-1/2}*d*\aA$ because if $\eta=(-z,\beta_\diamond)$ and $z$ is real, then $\psii^\dagger\eta - \eta^\dagger\psii = 0$.  Meanwhile, the Bochner-Weitzenb\"ock formula for ${\rmD_A}^2\psi$ can be used to write the $i\R$-valued function $*d*\aA$ as $-e_\diamond(\eta)$ with $e_\diamond$ being a first order differential operator whose symbol has norm bounded by $c_0(1+|\aA|)$ and whose zeroth order part has norm bounded by $c_0$.  (The equation $*d*\aA + e_\diamond(\eta)=0$ is obtained by writing the imaginary part of the identity $\psii^\dagger{\rmD_A}^2\psii=0$ in terms of $\aA$, $z$, and $\beta_\diamond$.)

It follows from the preceding description of $\pzee$ that $|\pzee(\frh)| \leq c_0(r^{1/2}\varepsilon|\frh| + r^{-1/2}|\nabla\frh|)$ on the set $V_{\varepsilon,r}$. (Remember that $r \geq c_0c\varepsilon^{-2}$ is required.) This implies the following: If $\varepsilon$ is less than $c_0^{-1}$, then $\pzee$ obeys \eqref{eq:monopoles-16} on $V_{\varepsilon,r}$.  Supposing therefore this bound, then Lemma~\ref{lem:monopoles-11} can be invoked to obtain:

\begin{lemma}
\label{lem:monopoles-14}
There exists $\kappa>100$ such that the following is true when $r>\kappa$: Suppose that $\rmu$ and $V$ are as described above and that $(A,\psi)$ obeys the $F_{A_0}=0$ version of \eqref{eq:sw} on $\rmu$.  Let $B\subset V$ denote a ball of radius $\rho\geq \kappa r^{-1/2}$ and suppose that $|\alpha| \geq 1-\kappa^{-1}$ on $B$. There is an isomorphism between $E$ and the product $\C$ bundle on $B$ that identifies $(A,\psi)$ with a pair $(\Ai+\aA,\psii+\eta)$ obeying
\[r^{-1/2}|\aA|+|\eta|+r^{-1}|\nabla\aA|+r^{-1/2}|\nabla_{\Ai}\eta|\leq \kappa e^{-\rho\sqrt{r}/\kappa},\]
on the concentric radius $\frac{1}{2}\rho$ ball.
\end{lemma}
\begin{proof}
The ${L^2}_1$ bound for the norms of $\aA$ and $\eta$ that come via Lemma \ref{lem:monopoles-11} can be parlayed into $C^1$ bounds using standard elliptic regularity inside $B$. 
\end{proof}
\subsubsection{The zero locus of $\alpha$}
\label{sssec:zero_locus}
Suppose that $(A,\psi=(\alpha,\beta))$ obeys \eqref{eq:sw} on $\rmu$. Various things can be said about the zero locus of $\alpha$ if there is an a priori bound on an `energy', $\en$, which is defined as follows:
\begin{equation}
\label{eq:energy}
\en=r\int_{\rmu}|1-|\alpha|^2|.
\end{equation}
If $\rmu$ is a compact manifold, then $\en$ is necessarily finite. If $\rmu$ is the $T=\infty$ version of $Y_T$, then it follows from Lemma \ref{lem:monopoles-14} that $\en$ is finite if $(A,\psi)$ is an admissible solution to \eqref{eq:sw}. 

Suppose as before that $V\subset \rmu$ is an open set such that any geodesic arc in $\rmu$ that starts in $V$ can be continued in $\rmu$ for distance at least $1$. 
\begin{lemma}
\label{lem:monopoles-15}
Given $\cale>1$, there exists $\kappa_\cale>1$ with the following significance: Fix $r>\kappa_\cale$ and suppose that $(A,\psi=(\alpha,\beta))$ is a solution to \eqref{eq:sw} on $\rmu$ with $\en\leq \cale$. There is a set of properly embedded curves in $V$ of total length at most $(1+\kappa_\cale^{-1})\frac{1}{2\pi}\en$ with the properties listed below. This list uses $\Theta$ to denote the set of disjoint embedded curves.
\begin{itemize}\leftskip-0.25in
\item Each $|\alpha|<1-\kappa_\cale^{-1}$ point in $V$ has distance $\leq\kappa_\cale r^{-1/2}$ from some point in a curve from $\Theta$.
\item Conversely, $|\alpha|<1-\kappa_\cale^{-1}$ on the curves from $\Theta$. 
\item The angle between the Reeb vector field and the tangent vector to any curve from $\Theta$ is at most $\kappa_\cale r^{-1/2}$ at each point on the curve.
\end{itemize}
\end{lemma}
\begin{proof}
The arguments differ little from those in \cite[Section 6.4]{taubes-weinstein1}.
\end{proof}

By way of a parenthetical remark, it is likely that the number $\kappa_\cale$ in Lemma \ref{lem:monopoles-15} can be assumed to be independent of $\cale$. 
\subsubsection{A compactness theorem for $Y_\infty$}
\label{sssec:compactness}
The uniform bounds in Lemmas \ref{lem:monopoles-14} and \ref{lem:monopoles-15} lead directly to the next two lemmas. What is denoted by $\en$ is the function in \eqref{eq:energy} with $\rmu=Y_\infty$.
\begin{lemma}
\label{lem:monopoles-16}
There exists $\kappa>1$, and given $\cale>1$, there exists $\kappa_\cale>1$ with the following significance: Fix $r>\kappa_\cale$. If $(A,\psi=(\alpha,\beta))$ is an admissible solution to \eqref{eq:sw} on $Y_\infty$ with $\en\leq\cale$, then
\[r(1-|\alpha|^2)+r^2|\beta|^2+|\nabla_A\alpha|^2+r|\nabla_A\beta|^2\leq\kappa re^{-\kappa^{-1}r^{1/2}\operatorname{dist}(\cdot,N)}\]
at points with distance $\kappa$ or more from $N$. 
\end{lemma}
\begin{proof}
This follows from Lemmas \ref{lem:monopoles-14} and \ref{lem:monopoles-15} because there are no compact Reeb orbits in $Y_\infty\ssm N$. In fact, the gradient of the function $\operatorname{dist}(\cdot,N)$ has constant, non-zero angle with the Reeb vector field on $Y_\infty\ssm N$. 
\end{proof}
The second lemma concerns the behavior of sequences of solutions to \eqref{eq:sw} on $Y_\infty$. To set the stage for the lemma, suppose that $g$ is a map from $Y_\infty$ to $S^1$. This map defines an automorphism of $E$ and thus a fiber preserving isometry of $\spb$.  In particular, it induces an action on the set of pairs $(A,\psi)$ in the following way: the map $g$ sends $\psi$ to $g\psi$ (so it acts by multiplication), and it sends $A$ to $A-g^{-1}dg$. The result of this action is denoted in what follows by $(g^\ast A,g^\ast\psi)$. 
\begin{lemma}
\label{lem:monopoles-17}
Given $\cale>1$, there exists $\kappa_\cale>1$ with the following significance: Fix $r>\kappa_\cale$. The space of admissible $\en\leq\cale$ solutions to \eqref{eq:sw} on $Y_\infty$ is compact modulo the action of the group $C^\infty(Y_\infty;S^1)$. This means the following: Let $\{(A_k,\psi_k)\}_{k=1,2,\dots}$ denote a sequence of admissible solutions to \eqref{eq:sw} on $Y_\infty$ with $\en\leq\cale$. There exists a solution $(A,\psi)$ to \eqref{eq:sw} on $Y_\infty$ with $\en\leq\cale$, a strictly increasing subsequence $\Lambda$ of positive integers, and a sequence of maps $\{g_k\}_{k\in\Lambda}$ from $Y_\infty$ to $S^1$ such that $\{(g_k^\ast A_k,g_k^\ast \psi_k)\}_{k\in\Lambda}$ converges in the $C^\infty$-topology to $(A,\psi)$. 
\end{lemma}
The proof of Lemma \ref{lem:monopoles-17} introduces the constant $c_0>1$ that depends only on $\cale$ and its value increases between successive appearances.
\begin{proof}
The proof has six parts. Everything but Part \ref{prt:lem26-1} is a very standard application of constructions from \cite{uhlenbeck} in the Abelian case.
\begin{prt}
\label{prt:lem26-1}
Let $\kappa_\ast$ denote the version of the number $\kappa_\cale$ that is given by Lemma \ref{lem:monopoles-16}. Fix $T>\kappa_\ast$ so that when $(A,\psi)$ is an admissible solution to \eqref{eq:sw} on $Y_\infty$ with $\en\leq\cale$, then $|\alpha|>1-\frac{1}{100}{c_0}^{-1}$ where the distance to $N$ is greater than $T$. There is an isomorphism $g$ on this part of $Y_\infty$ between the complex product line bundle $\underline{\C}$ and the given bundle $E$ that pulls $\psi$ back as a pair of complex valued function and a section of $K^{-1}$ with the complex valued function part being $|\alpha|$. Use $\beta'$ to denote the section of $K^{-1}$. Meanwhile, write $g^\ast A$ as $\Ai+\aA$. Note that $g^\ast\nabla_A\alpha$ can be written as 
\begin{equation}
\label{eq:monopoles-lem27-1}
\nabla|\alpha|+\aA|\alpha|.
\end{equation}
Therefore, since $\aA$ is $i\R$-valued and since $|\alpha|>\frac{1}{2}$, a pointwise bound on $|\nabla_A\alpha|$ (such as from Lemma \ref{lem:monopoles-16}) leads to pointwise bounds on both $\nabla|\alpha|$ and $\aA$. Note next that
\begin{equation}
\label{eq:monopoles-lem27-2}
g^*\nabla_A^{\dagger}\nabla^{}_A \alpha=\nabla^\dagger\nabla|\alpha|+|\aA|^2|\alpha|-2\langle \nabla|\alpha|,\aA\rangle+\nabla^\dagger\aA|\alpha|,
\end{equation}
with $\langle\,,\rangle$ denoting the metric pairing. Note in particular that the left most two terms are $R$-valued and the right most two terms are $i\R$-valued. Therefore, this equation can be used to obtain pointwise bounds on the Laplacian of $|\alpha|$ and the divergence of $\aA$ (this being $\nabla^\dagger\aA$) from pointwise bounds on $\nabla_A^{\dagger}\nabla^{}_A\alpha$. given the aforementioned pointwise bounds on $\nabla|\alpha|$ and $\aA$. Meanwhile, the $E$ summand of the Bochner--Weitzenb\"ock formula for ${\mathcal{D}_A}^2\psi$ equates this same $\nabla_A^{\dagger}\nabla^{}_A\alpha$ with a sum of terms that are each proportional to one of $|\alpha|$ and $\beta'$, or the $g^\ast A$-covariant derivatives of $\beta'$ (see, for example, Section 6 of \cite{taubes-weinstein1}.) This understood, then the bound in Lemma \ref{lem:monopoles-16} supplies an apriori bound for $\nabla_A^{\dagger}\nabla^{}_A\alpha$ and thus for $\nabla^\dagger\nabla|\alpha|$ and $\nabla^\dagger\aA$. More to the point, the $\R$ and $i\R$ parts of the $E$ summand of the identity ${\mathcal{D}_A}^2\psi=0$ (with $E$ identified as above with the product complex line bundle) and the $K^{-1}$ summand of this same equation make for an elliptic system of first order equations for the data set $(\aA,(|\alpha|,\beta'))$ that can be bootstrapped using essentially standard arguments to obtain apriori $C^k$ bounds on $(\aA,(|\alpha|,\beta'))$ for any positive integer $k$ where the distance to $N$ is greater than $T+1$.
\end{prt}
\begin{prt}
\label{prt:lem26-2}
If $\{(A_n,\psi_n)\}_{n\in\{1,2,\dots\}}$ is the sequence as in the statement of the lemma, then there is a corresponding sequence $\{g_n\}_{n\in\{1,2,\dots\}}$ with any given $g_n$ being the $(A,\psi)=(A_n,\psi_n)$ version of $g$ from Part \ref{prt:lem26-1}. By virtue of what is said in the preceding paragraph, the sequence of gauge equivalent pairs $(g_n^\ast A_n,g_n^\ast\psi_n)$ has a subsequence that converges in the $C^\infty$ topology where the distance to $N$ is greater than $T+1$. Let $\Lambda\subset\{1,2,\dots\}$ denote the integers that label this subsequence and let $(A_\diamond,\psi_\diamond)$ denote the limiting pair of $\{(g_n^\ast A_n,g_n^\ast\psi_n)\}_{n\in\Lambda}$.
\end{prt}
\begin{prt}
\label{prt:lem26-3}
Let $Z\subset Y_\infty$ denote for the moment a smooth, codimension zero, compact submanifold with boundary that contains the part of $Y_\infty$ where the distance to $N$ is less than $T+3$. Let $A^E$ denote a fixed connection on the bundle to $E$ over $Z$. Suppose again that $(A,\psi)$ is an admissible solution to \eqref{eq:sw} on $Y_\infty$ with $\en\leq\cale$. There is an automorphism of $(A,\psi)$ on $Z$ that pulls back $A$ to a connection that can be written as $A^E+\fra'$ such that $d\ast\fra'=0$, such that $\ast\fra'$ is zero on $\partial Z$, and such that ${L^2}_1$ norm of $\fra'$ on $Z$ is bounded by $c_0$. The construction of this automorphism amounts to little more than Hodge theory on manifolds with boundary. Denote this automorphism by $g'$. The equations in \eqref{eq:sw} when written for $(g'^{\ast}A',g'^{\ast}\psi')$ are elliptic with the extra condition that $d\ast\fra'=0$. Therefore, standard elliptic bootstrapping leads to a priori $C^k$ bounds on $(g'^{\ast}A',g'^{\ast}\psi')$ on the part of $Z$ where the distance to $N$ is less than $T+2$.
\end{prt}
\begin{prt}
\label{prt:lem26-4}
If $\{(A_n,\psi_n)\}_{n\in\Lambda}$ is the sequence from Part \ref{prt:lem26-2}, then there is a corresponding sequence $\{g'_n\}_{n\in\Lambda}$ with any given $g'_n$ being the $(A,\psi)=(A_n,\psi_n)$ version of $g'$ from Part \ref{prt:lem26-3}. By virtue of what is said in the preceding paragraph, the sequence of gauge equivalent pairs $\{({g'_n}^\ast A_n,{g'_n}^\ast\psi_n)\}_{n\in\Lambda}$ has a subsequence whose index set will be denoted by $\Lambda'$ that converges in the $C^\infty$ topology where the distance to $N$ is less than $T+2$. Let $(A'_\diamond,\psi'_\diamond)$ denote the limiting pair of this subsequence of pull-backs. For $n\in\Lambda'$, let $\alpha'_n$ denote $g'_n\alpha_n$.
\end{prt}
\begin{prt}
\label{prt:lem26-5}
Supposing that an integer $n$ is large, let $(A_n,\psi_n)$ denote a pair from the subsequence in Part \ref{prt:lem26-4}, and let $g_n$ denote the corresponding automorphism from Part \ref{prt:lem26-2} for the subsequence from Part \ref{prt:lem26-2}. Let $g'_n$ denote the corresponding automorphism from Part \ref{prt:lem26-4}. Let $u_n$ denote $g^{}_n(g'_n)^{-1}$ which is defined where the distance to $N$ is between $T+1$ and $T+2$. This is an isomorphism between the complex product line bundle and $E$ over this part of $Y_\infty$. It follows from the definitions that $|\alpha|=u^{}_n\alpha'_n$. Thus, the sequence $\{u_n\}_{n\in\Lambda'}$ converges in the $C^\infty$ topology to an isomorphism between the complex product line bundle and $E$. Let $u_\diamond$ denote the limit. This isomorphism is such that $u_\diamond^\ast A'_\diamond=A_\diamond$ and $u_\diamond^\ast\psi'_\diamond=\psi_\diamond$. Therefore, the pairs $(A_\diamond,\psi_\diamond)$ and $(A'_\diamond,\psi'_\diamond)$ define a pair of a connection and a section of a bundle over $Y_\infty$ to be denoted by $E_\diamond$ that is defined by a cocycle data that has $E_\diamond$ being $E$ where the distance to $N$ is less than $T+2$ and as $E_\diamond$ being the complex product line bundle where the distance to $N$ is greater than $T+1$. The map $u_\diamond$ is the gluing cocycle for this new bundle.
\end{prt}
\begin{prt}
\label{prt:lem26-6}
A construction using cut-off function that differs only in notation from \cite{uhlenbeck} in the non-Abelian case can be used to construct a sequence of isomorphisms labeled by elements in $\Lambda'$ between $E_\diamond$ and $E$. These isomorphisms (to be denoted by $\{f_n\}_{n\in\Lambda'}$ have the property that $f_n^{\ast}(A_n,\psi_n)$ converges to the limit pair of a connection and section on $E_\diamond$ from Part \ref{prt:lem26-5}. (Keep in mind that this limit pair is defined by $(A_\diamond,\psi_\diamond)$ where the distance to $N$ is greater than $T+1$ and defined by $(A'_\diamond,\psi'_\diamond)$ where the distance to $N$ is less than $T+2$.) The construction of $f_n$ uses the fact that $\{u_n{u_\diamond}^{-1}\}_{n\in\Lambda'}$ converges to $1$; as noted, it is an instance of what is done in \cite{uhlenbeck}. To construct $f_n$, first fix a smooth cut-off function, $\chi_\diamond$, that is equal to $1$ where the distance to $N$ is less that $T+\frac{3}{2}$ and equal to zero where the distance to $N$ is greater than $T+2$. With $\chi_\diamond$ in hand, write the large $n\in\Lambda'$ versions of $u_n{u_\diamond}^{-1}$ as $\exp(\varphi_n)$ with $\varphi_n$ being an $i\R$-valued function with small positive norm. These are such that the corresponding sequence $\{\varphi_n\}_{n\in\Lambda'}$ converges to zero in the $C^\infty$ topology. Set $g_n''=\exp(-\chi_\diamond\varphi_n)g_n$. It then follows that $g''_n(g'_n)^{-1}=u_\diamond$ where the distance to $N$ is between $T+1$ and $T+\frac{3}{2}$. This being the case, the pair $(g''_n,g'_n)$ defines an isomorphism from $E$ to $E_\diamond$. This is the isomorphism $f_n$. It follows from what is said in Parts \ref{prt:lem26-2} and \ref{prt:lem26-4} that $f_n$ has the desired properties. \qedhere
\end{prt}
\end{proof}
As explained later, these lemmas have implications for solutions to \eqref{eq:sw} on the finite $T$ version of $Y_T$. 

The compactness assertion in Lemma \ref{lem:monopoles-17} with an extra assumption implies that there are but a finite number of $C^\infty(Y_\infty;S^1)$-orbits of admissible solutions to \eqref{eq:sw} with a priori bound on $\en$. The extra assumption is that the operator in \eqref{eq:linearized} that is defined by an admissible solution has trivial kernel. This finiteness result is stated formally in the next lemma.
\begin{lemma}
\label{lem:monopoles-18}
Suppose that the $Y_\infty$ version of the operator $\LLf$ in \eqref{eq:linearized} defined by any admissible solution to \eqref{eq:sw} on $Y_\infty$ has trivial cokernel. Then there are at most a finite number of $C^\infty(Y_\infty;S^1)$ orbits of admissible solutions to \eqref{eq:sw} on $Y_\infty$ with given a priori bound on $\en$.
\end{lemma} 
A `slice' lemma is needed to prove Lemma \ref{lem:monopoles-18}. As in the case of a compact manifold, such a lemma is needed to account for the gauge invariance of the equations that are depicted in \eqref{eq:sw}. The point being that this gauge invariance implies that the equations have infinitely many solutions if they have just one. A weak slice lemma is proved that is sufficient for the present purposes. 

To set the stage for the slice lemma, note that the vanishing of the function
\begin{equation}
\label{eq:monopoles-117}
\ast d\ast\fra+r^{1/2}(\psi^\dagger\eta-\eta^\dagger\psi)
\end{equation}
is formally a slice constraint akin to the Coulomb gauge condition that is defined by the rule $\ast d\ast \fra=0$. Keep in mind that the expression in \eqref{eq:monopoles-117} is also the $i\underline{\R}$ component of the operator $\LLf$ in \eqref{eq:linearized}.

The upcoming weak slice condition requires only that \eqref{eq:monopoles-117} be small in a suitable sense. To make this notion precise, fix $\upvarepsilon>0$. The \emph{$\upvarepsilon$-slice} centered at $(A,\psi)$ consists of the set of admissible pairs of a connection on $E$ and a section of $\spb$ on $Y_\infty$ that obey the following constraint: Let a pair from this set be denoted by $(A',\psi')$. The section $\frh$ of $iT^\ast Y_\infty\oplus\spb$ given by the pair $(\fra=r^{-1/2}(A'-A),\eta=\psi'-\psi)$ obeys
\begin{equation}
\label{eq:monopoles-118}
||\ast d\ast\fra+r^{1/2}(\psi^\dagger\eta-\eta^\dagger\psi)||_2\leq\upvarepsilon||\frh||_{\bbh},
\end{equation}
where $||\frh||_\bbh$ denotes $({||\nabla_A\frh||_2}^2+r{||\frh||_2}^2)^{1/2}$.

The lemma that follows gives the promised weak slice assertion:
\begin{lemma}
\label{lem:monopoles-19}
Given $\cale>1$, there exists $\kappa_\cale>1$ with the following significance: Fix $r>\kappa_\cale$ and suppose that $(A,\psi)$ is an admissible solution to \eqref{eq:sw} on $Y_\infty$ with $\en\leq\cale$. Suppose that $\{\frc_k=(A_k,\psi_k)\}_{k=1,2,\dots}$ is a sequence of admissible solutions to \eqref{eq:sw} on $Y_\infty$ with $\en\leq\cale$ that converges to $(A,\psi)$ in the $C^\infty$-topology. Given $\upvarepsilon>0$, there is a corresponding sequence of gauge transformations $\{h_k\}_{k=1,2,\dots}$ such that all sufficiently large $k$ versions of $h_k\frc_k$ are in the $\upvarepsilon$-slice centered at $(A,\psi)$.
\end{lemma}
This lemma is proved momentarily.
\begin{proof}[Proof of Lemma \ref{lem:monopoles-18}]
If $(A',\psi')$ is an admissible solution to \eqref{eq:sw} on $Y_\infty$ and if it is in an $\upvarepsilon$-slice centered at $(A,\psi)$, then $\frh=(r^{-1/2}(A'-A),\psi'-\psi)$ when viewed as a section of $i(T^\ast Y_\infty\oplus\underline{\R})\oplus\spb$ obeys an equation that has the schematic form 
\begin{equation}
\label{eq:monopoles-119}
\LLf\frh+\pzee\frh=0
\end{equation}
with $\pzee\frh$ obeying
\begin{equation}
\label{eq:monopoles-120}
{||\frh||_2}^2\leq c_0(||\frh||_\bbh+\upvarepsilon)||\frh||_\bbh.
\end{equation}
Note that the $\upvarepsilon$-slice condition in \eqref{eq:monopoles-118} makes the $L^2$-norm of the $i\underline{\R}$ component of $\LLf\frh$ to be $\mathcal{O}(\upvarepsilon||\frh||_\bbh)$. The $L^2$-norms of the other components of $\LLf\frh$ are $\mathcal{O}({||\frh||_\bbh}^2)$ because $(A',\psi')$ obeys \eqref{eq:sw}.

With the preceding understood, suppose in addition that the $(A,\psi)$ version of the operator $\LLf$ on $Y_\infty$ has trivial kernel and let $c$ denote the corresponding number from Lemma \ref{lem:monopoles-12}. It follows from \eqref{eq:monopoles-118} and \eqref{eq:monopoles-120} that $\frh$ must be zero if $\upvarepsilon$ and $||\frh||_\bbh$ are both less than $c_0^{-1}c^{-1}$. This can be seen by writing \eqref{eq:monopoles-119} as $\LLf\frh=-\pzee\frh$ and taking the $L^2$-norms of both sides. Then Lemma \ref{lem:monopoles-18} follows from what was just said and Lemmas \ref{lem:monopoles-17} and \ref{lem:monopoles-19}.
\end{proof}

\begin{proof}[Proof of Lemma \ref{lem:monopoles-19}]
The proof has five steps.
\begin{step}
Fix $T>c_0$ for the moment. Fix a smooth codimension-0 submanifold in $Y_\infty$ with the boundary that is contained in the part of $Y_\infty$ with distance less than $T+10$ from $N$ and contains in its interior the part of $Y_\infty$ with distance less than $T+5$ from $N$. Denote this submanifold by $Z_T$. As explained momentarily, an appeal to the inverse function theorem using the operator $d^\dagger d+r|\psi|^2$ with Dirichlet boundary conditions on $Z_T$ can be used to prove the next assertion:
\begin{quote}
If $(A',\psi')$ is sufficiently close to $(A,\psi)$ in the $C^2$-topology on $Z_T$, then there is a smooth map from the interior of $Z_T$ to $i\R$ with the following properties: Denote this map by $u$ and let $h_0=e^{u}$, this being a map to $S^1$. Then
\begin{itemize}\leftskip-0.15in
\item With $\fra$ defined to be $r^{-1/2}(h_0^\ast A'-A)$ and $\eta$ defined to be $h_0^\ast\psi'-\psi$, the expression in \eqref{eq:monopoles-117} vanishes on the part of $Z_T$ where the distance to $N$ is less than $T+2$.
\item the ${L^2}_1$-norm of the pair $(\fra,\eta)$ is bounded by a constant multiple of the $C^2$-norm of the pair $(r^{-1/2}(A'-A),\psi'-\psi)$ with the constant being independent of $(A',\psi')$.
\end{itemize}
\begin{equation}
\label{eq:monopoles-121}
\end{equation}
\end{quote}
The argument for \eqref{eq:monopoles-121} differs little from the proof of the analogous assertion in the compact manifold case. Here is an outline of the argument: Suppose that $u$ is smooth on $\partial Z_T$ and write the $\fra_u=r^{-1/2}(A'-A-du)$ and $\eta_u=e^{iu}\psi'-\psi$ version of \eqref{eq:monopoles-117} schematically as
\begin{equation}
\label{eq:monopoles-122}
d^\dagger du+r|\psi|^2u+r^{1/2}\mathfrak{T}_0+r\mathfrak{R}(u)
\end{equation}
where $\mathfrak{T}_0$ is the $\fra_0=r^{-1/2}(A'-A)$ and $\eta_0=\psi'-\psi$ version of \eqref{eq:monopoles-117} and $\mathfrak{R}$ is a function of $u$ obeying $|\mathfrak{R}(u)|\leq c_0|u|^2+|u|(|\fra_0|+|\eta_0|)$.

Construct a smooth non-negative function on $Y_\infty$ to be denoted by $\chi_T$ that is equal to $1$ where the distance from $N$ is less than $T+3$ and equal to $0$ where the distance to $N$ is greater than $T+4$. This function can and should be constructed so that the norm of its derivative is bounded by $c_0$. What is needed is a solution to the equation
\begin{equation}
\label{eq:monopoles-123}
d^\dagger du+r|\psi|^2u+\chi_T(r^{1/2}\mathfrak{T}_0+r\mathfrak{R}(u))=0
\end{equation}
with $u$ being zero on $\partial Z_T$. A relatively straightforward application of the inverse function theorem proves that \eqref{eq:monopoles-123} has a solution if $(\fra_0,\eta_0)$ has sufficiently small $C^1$-norm. Note that the applications to come have no need to control the $r$ and $T$ dependence of the constant in the second bullet of \eqref{eq:monopoles-121} nor the $r$ and $T$ dependence of what is meant by `sufficiently small'.
\end{step}
\begin{step}
Now suppose that $(A',\psi')$ obeys \eqref{eq:sw} on the part of $Y_\infty$ with distance less than $T$ from $N$. Supposing that \eqref{eq:monopoles-121} can be invoked, let $\frh=(\fra,\eta)$. If $|\frh|$ is small, which will be the case if $(A',\psi')$ is sufficiently close to $(A,\psi)$, then $\frh$ viewed as a section over the part of $Y_\infty$ with distance less than $T$ from $N$ obeys an equation that has the schematic form 
\begin{equation}
\label{eq:monopoles-124}
\LLf\frh +r^{1/2}\frh\ast\frh=0,
\end{equation}
where $\LLf$ is defined as in \eqref{eq:linearized} by the pair $(A,\psi)$. If $|\frh|$ is small, then this equation can be written as $\LLf\frh+\pzee\frh=0$ with $\pzee$ obeying \eqref{eq:monopoles-16}. Thus, Lemma \ref{lem:monopoles-11} can be invoked to bound the $||\cdot||_\bbh$-norm of $\frh$ where the distance from $N$ is between $\frac{1}{4}T$ and $\frac{3}{4}T$ by 
\begin{equation}
\label{eq:monopoles-125}
c_0(||\frh||_{2,10}+||\frh||_{2,T})e^{-r^{1/2}T/c_0},
\end{equation}
where $||\frh||_{2,10}$ and $||\frh||_{2,T}$ are the respective $L^2$-norms of $\frh$ on the regions in $Y_\infty$ where the distance to $N$ is less than $10$ and where the distance to $N$ is between $T-1$ and $T$.
\end{step}
\begin{step}
As noted previously, there is a gauge for $(A,\psi)$ on the part of $Y_\infty$ with distance greater than $c_\cale$ from $N$ that writes the $E$ component $\alpha$ of $\psi$ as $1-z$ with $z$ being a small real number. This gauge writes $A$ as $\Ai+r^{1/2}\fra$  with $|\fra|\leq c_0r^{-1/2}|\nabla_A\alpha|$. As already noted, the pair $\frq=(\fra,\eta=\psi-\psii)$ obeys an equation of the form $\Lif\frq+\pzee\frq=0$ with $\pzee$ obeying \eqref{eq:monopoles-16}. There is an analogous gauge on this part of $Y_\infty$ for $(A',\psi')$ leading to a $\frq'$ obeying $\Lif\frq'+\pzee'\frq'=0$ with $\pzee'$ obeying \eqref{eq:monopoles-16}. Use $\frhi$ to denote the difference $\frq'-\frq$. This obeys an equation that has the form $\Lif\frhi+\pzee''\frhi=0$ with $\pzee''$ again obeying \eqref{eq:monopoles-16}. It follows as a consequence that $\frhi$ on the part of $Y_\infty$ where the distance to $N$ is between $\frac{1}{4}T$ and $\frac{3}{4}T$ has $||\cdot||_\bbh$-norm bounded by 
\begin{equation}
\label{eq:monopoles-126}
||\frhi||_{2,10}e^{-r^{1/2}T/c_0}.
\end{equation}
\end{step}
\begin{step}\label{lem:monopoles-19step4}
What is denoted by $\frhi$ can be written as $(r^{-1/2}(\hi^\ast A'-A),\hi^\ast\psi'-\psi)$ for a suitable gauge transformation $\hi$. It then follows that 
\begin{equation}
\label{eq:monopoles-127}
(r^{-1/2}(\hi^\ast A'-h_0^\ast A'),\hi^\ast\psi'-h_0^\ast\psi')
\end{equation}
restricted to the part of $Y_\infty$ where the distance to $N$ is between $\frac{1}{4}T$ and $\frac{3}{4}T$ is bounded by the sum of the expressions in \eqref{eq:monopoles-125} and \eqref{eq:monopoles-126}. This implies that $h_0^{-1}\hi$ can be written as $e^{v}$ where $|v|$ is bounded by the same expressions.

With the preceding understood, let $\chi$ denote a non-negative function on $Y_\infty$ that is equal to $1$ where the distance to $N$ is greater than $\frac{1}{2}T$ and equal to $0$ where the distance to $N$ is less than $\frac{1}{2}T-1$. This function can and should be constructed so that the norm of its derivative is bounded by $c_0 T$. Let $h$ denote the gauge transformation $h_0e^{\chi v}$. The latter is equal to $h_0$ where the distance to $N$ is less than $\frac{1}{2}T-1$ and it is equal to $\hi$ where the distance to $N$ is greater that $\frac{1}{2}T$. This gauge transformation is such that the pair $\frh=(\fra=r^{-1/2}(h^\ast A'-A),h^\ast\psi'-\psi)$ obeys 
\begin{equation}
\label{eq:monopoles-128}
||\ast d\ast\fra+r^{1/2}(\psi^\dagger\eta-\eta^\dagger\psi)||_2\leq c_0 r Te^{-r^{1/2}T/c_0}||\frh||_2.
\end{equation}
\end{step}
\begin{step}
The preceding analysis can be applied to any sufficiently large $k$ version of $(A_k,\psi_k)$ to produce a gauge transformation $h_k$ with the desired properties. \qedhere
\end{step}
\end{proof}
\subsubsection{Limits as $T\to\infty$}
\label{sssec:limit_at_infty}
Lemma \ref{lem:monopoles-16} has an analog on the large $T$ versions of $Y_T$. The precise statement will appear momentarily. By way of background, assume in the case $T<\infty$, the submanifold $N$ is a union of two parts, $M\cup N_T$ with the parameter $t$ from Section \ref{ssec:setup} having values in $[-1,1]$ on the boundary of the closure of $M$, and in the interval $[-T,T]$ on the boundary of the closure of $N_T$. Likewise, the parameter $\tau$ has value $0$ on the boundary of the closure of $M$ and it has value $T$ on the boundary of the closure of $N_T$. In comparison to the body of this article, $M$ denotes the interior of a sutured contact manifold with connected suture.  The case $T < \infty$ is shown in Figure~\ref{fig:a-m_union_nt}.

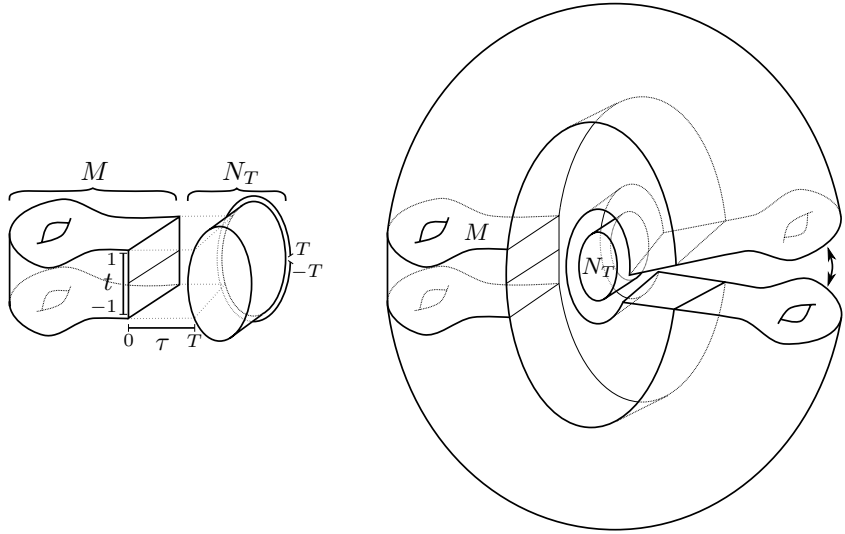
\begin{figure}[ht]
\labellist
\small \hair 2pt
\pinlabel $M$ [t] at 65 424
\pinlabel $N_T$ [t] at 134 424
\pinlabel $t$ [r] at 76 365
\pinlabel $\tau$ [t] at 97 341
\tiny
\pinlabel $1$ [r] at 78 376
\pinlabel $-1$ [r] at 78 353
\pinlabel $0$ [t] at 81 343
\pinlabel $T$ [t] at 112 343
\pinlabel $T$ [l] at 157 381
\pinlabel $-T$ [l] at 155 370
\footnotesize
\pinlabel $M$ at 245 389
\pinlabel $N_T$ at 302.5 372
\endlabellist
\includegraphics[scale=0.8]{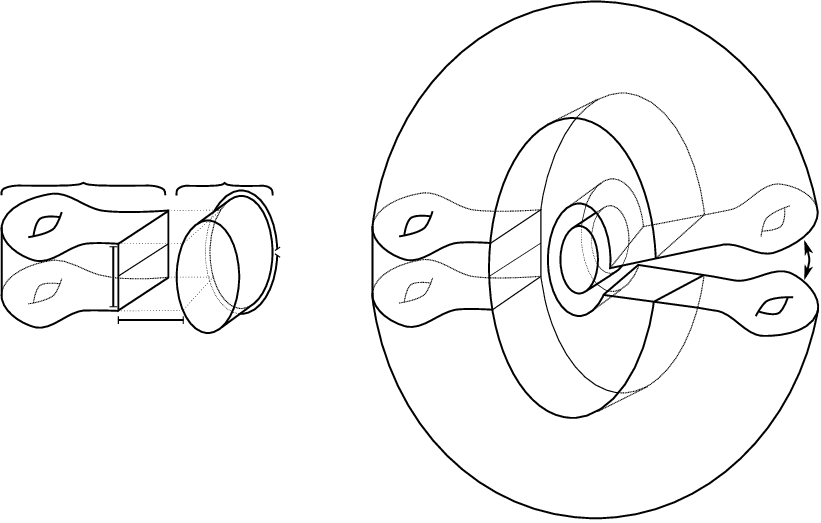}
\caption{The manifold $N = M \cup N_T$, viewed both separately and as part of $Y_T$.}
\label{fig:a-m_union_nt}
\end{figure}

For $T<\infty$, the Riemannian curvature tensor on $Y_T$ has $T$-independent upper bound, so do the covariant derivatives of $a$ and $da$. Let $L(T)$ denote the length of the shortest Reeb orbit contained in $N_T$. Note that $\{L(T)\}_{T>16}$ has no bounded subsequences for the version of $N_T$ that is depicted in Figures~\ref{fig:a-yt} and \ref{fig:a-m_union_nt}. 
\begin{lemma}
\label{lem:monopoles-110}
There exists $\kappa>1$, and given $\cale>1$, there exists $\kappa_\cale>1$ with the following significance: Fix $T>\kappa_\cale$. If $(A,\psi)$ is an admissible solution to \eqref{eq:sw} on $Y_T$ with $\en\leq\cale$, then
\[r(1-|\alpha|^2)+r^2|\beta|^2+|\nabla_A\alpha|^2+r|\nabla_A\beta|^2\leq\kappa re^{-\kappa^{-1}r^{1/2}\operatorname{dist}(\cdot,N)}\]
at points in $Y_T$ with distance $\kappa$ or more from $N=M\cup N_T$. 
\end{lemma}
\begin{proof}
This follows as a corollary of Lemmas \ref{lem:monopoles-14} and \ref{lem:monopoles-15}.
\end{proof}
Lemma \ref{lem:monopoles-110} suggests that it is possible to compare solutions to \eqref{eq:sw} on the large $T$ versions of $Y_T$ with solutions to \eqref{eq:sw} on $Y_\infty$. This is the content of the next lemma.
\begin{lemma}
\label{lem:monopoles-111}
There exists $\kappa>1$, and given $\cale>1$, there exists $\kappa_\cale>1$ with the following significance: Suppose that $\{T(k)\}_{k=1,2,\dots}$ is an increasing unbounded sequence of positive numbers each greater than $16$. Fix $r>\kappa_\cale$. For each positive integer $k$, let $\frc_k=(A_k,\psi_k)$ denote a solution to \eqref{eq:sw} on $Y_{T(k)}$ with energy $\en\leq\cale$. There exists
\begin{enumerate}\leftskip-0.25in
\item A solution $\frc=(A,\psi)$ to \eqref{eq:sw} on $Y_\infty$ with energy $\en\leq\cale$,
\item A subsequence $\Lambda$ of positive integers,
\item A sequence $\{g_k\}_{k\in\Lambda}$ with any given $k\in\Lambda$ version of $g_k$ mapping $Y_{T(k)}\ssm N_{T(k)}$ to $S^1$.
\item A sequence $\{\hat{u}_k\}_{k\in\Lambda}$ with any given $k\in\Lambda$ version of $\hat{u}_k$ mapping $N_{T(k)}$ to $S^1$.
\end{enumerate}
The data $\frc$, $\Lambda$, $\{g_k\}_{k\in \Lambda}$, and $\{\hat{u}_k\}_{k\in \Lambda}$ have the following properties:
\begin{itemize}\leftskip-0.25in
\item The sequence indexed by $\Lambda$ whose $k^{th}$ term is the $C^1$-norm of the pair $e^{r\operatorname{dist}(\cdot,M)/\kappa}(g_k\frc_k-\frc)$ on $Y_{T(k)}\ssm N_{T(k)}$ has limit zero as $k\to\infty$.
\item For each positive integer $m$, the sequence indexed by $\Lambda$ whose $k^{th}$ term is the $C^m$-norm of the pair $g_k\frc_k-\frc$ on $Y_{T(k)}\ssm N_{T(k)}$ has limit zero as $k\to\infty$.
\item For each positive integer $m$, the sequence indexed by $\Lambda$ whose $k^{th}$ term is the $C^m$-norm of $\hat{u}_k\frc_k-(\Ai,\psii)$ on $N_{T(k)}$ has limit zero as $k\to\infty$.
\end{itemize}
\end{lemma}
\begin{proof}
Lemma \ref{lem:monopoles-110} can be invoked with Lemma \ref{lem:monopoles-14} and standard elliptic regularity arguments to obtain uniform convergence in the manner dictated above on $Y_{T(k)}\ssm N_{T(k)}$ to a solution on $Y_\infty$ that is a priori admissible.
\end{proof}

Lemma \ref{lem:monopoles-111} leads to the following observation: Fix $\cale>1$ and then $r>c_\cale$ with $c_\cale$ denoting here and in what follows a number greater than $1$ and depends on $\cale$. Its precise value can be assumed to increase between appearances. Fix $\upvarepsilon\in (0,{c_\cale}^{-1}]$. Supposing that $T$ is sufficiently large given $r$ and $\upvarepsilon$, let $\frc_T=(A_T,\psi_T)$  denote a solution to \eqref{eq:sw} on $Y_T$ with $\en<\cale$. There exists
\begin{enumerate}\leftskip-0.25in
\item A $C^\infty(Y_\infty;S^1)$ equivalence class of admissible solutions to \eqref{eq:sw} on $Y_\infty$ with $\en<\cale$.
\item Given $\frc=(A,\psi)$ in this equivalence class, there exists a map $g_T:Y_T\ssm N_T\to S^1$.
\item A map $\hu_T:N_T\to S^1$.
\end{enumerate} 
\begin{equation}
\label{eq:monopoles-129}
\end{equation}
These are such that 
\begin{itemize}\leftskip-0.25in
\item $e^{r\operatorname{dist}(\cdot,M)/c_0}(g_T\frc_T-\frc)$ on $Y_T\ssm N_T$ has $C^1$-norm less than $\upvarepsilon$.
\item $\hu_T\frc_T-(\Ai,\psii)$ on $N_T$ has $C^1$-norm less than $\upvarepsilon e^{-rT/c_0}$.
\end{itemize}
\begin{equation}
\label{eq:monopoles-130}
\end{equation}
It follows from Lemma \ref{lem:monopoles-18} that the $C^\infty$ equivalence class in item (1) of \eqref{eq:monopoles-129} is unique if the version of the operator $\LLf$ that are defined by the $\en\leq \cale$ solutions to \eqref{eq:sw} on $Y_\infty$ for the given value of $r$ have trivial kernel.
\subsubsection{Existence of solutions with a bound on $\cale$} Suppose that $\cale>1$ has been specified and that $\frac{1}{2\pi} \cale$ cannot be written as a sum with positive integer coefficients of the lengths of the closed orbits of the Reeb vector field. Assume in addition that the other conditions in (4-1) of \cite{taubes1} can be met when what is denoted in \cite{taubes1} by $L$ is $\frac{1}{2\pi} \cale$. Now suppose that $r_\upvarepsilon$ is such that the conclusions of Lemmas \ref{lem:monopoles-15} and \ref{lem:monopoles-110} hold for $r>r_\upvarepsilon$. Note that these conclusions require only that $T\geq \kappa_\cale$ with $\kappa_\cale$ coming from Lemma \ref{lem:monopoles-110}. The analysis that proves Theorem 4.2 in \cite{taubes1} can be repeated given the conclusions of Lemmas \ref{lem:monopoles-15} and \ref{lem:monopoles-110}. The result is stated in the following lemma.
\begin{lemma}
\label{lem:monopoles-114}
Fix $\cale>1$ and suppose that the $L=\frac{1}{2\pi} \cale$ version of (4-1) in \cite{taubes1} holds. There exists $\kappa_\cale>1$ such that if $r\geq \kappa_\cale$, then the following is true:
\begin{itemize}\leftskip-0.25in
\item Supposing that $T\geq \kappa_\cale$ or that $T=\infty$, then the set of gauge equivalence classes of solutions to \eqref{eq:sw} on $Y_T$ with $\en\leq\cale$ are in 1--1 correspondence with the set of generators of the ECH chain complex with total length less than $\frac{1}{2\pi}\cale$.
\item Given $\upvarepsilon\in(0,1]$, there exists $\kappa_{\cale,\upvarepsilon}>\kappa_\cale$ with the following significance: Fix $T>\kappa_{\cale,\upvarepsilon}$ and suppose that $\frc_T$ and $\frc_\infty$ are respective $\en<\cale$ solutions to \eqref{eq:sw} on $Y_T$ and $Y_\infty$ from gauge equivalence classes that correspond to the same generator of the ECH chain complex. There exists a map $g_T:Y_T\ssm N_T\to S^1$ and a map $\hu_T:N_T\to S^1$ such that the conclusions of \eqref{eq:monopoles-130} hold.
\item Supposing that $T> \kappa_\cale$ or that $T=\infty$, let $\frc_T$ denote a solution to \eqref{eq:sw} on $Y_T$ with $\en<\cale$. The spectrum of the corresponding version of the operator $\LLf$ in \eqref{eq:linearized} has distance at least ${\kappa_\cale}^{-1}$ from $0$. 
\end{itemize}
\end{lemma}

\begin{proof}
The gluing theorem that constructs solutions to the equations in \eqref{eq:sw} from Reeb orbits differs only cosmetically from that in \cite[Section~3]{taubes2}.  The other assertions in the lemma are proved using the analysis in \cite[Section~2a]{taubes3} and \cite[Section~2]{taubes4}.  The analysis from these papers can be borrowed almost verbatim because of Lemmas~\ref{lem:monopoles-11} and \ref{lem:monopoles-110}.  The latter says in effect that the solutions to \eqref{eq:sw} can be approximated for all intents and purposes by the canonical pair $(\Ai,\psii)$ where the distance to the $Y_\infty$ version of $M$ is greater than a fixed number that depends solely on $\cale$.  Granted this, then Lemma~\ref{lem:monopoles-11} says in effect that equations of the form $\LLf\frh + \frf = 0$ on $Y_T$ or $Y_\infty$ with $\frf$ supported near $M$ become very small at large distances from $M$.
\end{proof}

By way of some perspective on Lemma~\ref{lem:monopoles-114}: A 1-1 correspondence between the respective gauge equivalence classes of $\en\leq\cale$ admissible solutions to \eqref{eq:sw} on $Y_\infty$ and any sufficiently large $T$ version of $Y_T$ can be constructed directly if the spectrum of the operator $\LLf$ for all $\en\leq 2\cale$ admissible solutions on $Y_\infty$ is uniformly bounded away from zero.  The lower bound for $T$ is determined by $\cale$ and this spectral gap, but it is independent of $r$.  This correspondence is asserted by the upcoming Lemma~\ref{lem:monopoles-115}.  (Lemma~\ref{lem:monopoles-12} with the analysis in \cite[Section~2]{taubes2}, \cite[Section~2a]{taubes3}, and \cite[Section~2]{taubes4} can be borrowed to prove that the desired spectral gap does in fact exist when $r$ is large if the contact 1-form and compatible metric are suitably generic on $M$.  As noted above, the borrowing from these references is possible because of what is said by Lemmas~\ref{lem:monopoles-11} and \ref{lem:monopoles-110}.)

\begin{lemma} \label{lem:monopoles-115}
Given $\cale > 1$, there exists $\kappa_\cale > 1$ with the following significance: Suppose that $r > \kappa_\cale$ and that the version of the operator $\LLf$ in \eqref{eq:linearized} for any admissible solution to \eqref{eq:sw} on $Y_\infty$ with $\en \leq 2\cale$ has trivial kernel.  Fix $c>1$ so that the spectrum of each such version of $\LLf$ has distance $c^{-1}$ or more from $0$.  There exists $\kappa_* > \kappa_\cale$ such that the assertions in the bullets that follow are true when $T > \kappa_*$.
\begin{itemize}\leftskip-0.25in
\item The set of gauge equivalence classes of solutions to \eqref{eq:sw} on $Y_T$ with $\en<\cale$ enjoys a bijective correspondence with the set of gauge equvalence classes of admissible solutions to \eqref{eq:sw} on $Y_\infty$.
\item Supposing that $\upvarepsilon \in (0,1]$, there exists $\kappa_{\cale,\upvarepsilon} > \kappa_\cale$ such that if $T > \kappa_{\cale,\upvarepsilon}$ and if $\frc_T$ and $\frc_\infty$ are respective $\en<\cale$ solutions to \eqref{eq:sw} on $Y_T$ and $Y_\infty$ from corresponding gauge equivalence classes, then there are maps $g_T: Y_T \ssm N_T \to S^1$ and $\hu_T \to S^1$ that make \eqref{eq:monopoles-130} hold.
\item Let $\LLf$ denote the version of the operator in \eqref{eq:linearized} that is defined by a given $\en \leq \cale$ solution to \eqref{eq:sw} on $Y_T$. Then $\LLf$ has trivial kernel and the distance in $\R$ from the spectrum of $\LLf$ to $0$ is greater than $\frac{1}{2}c^{-1}$.
\end{itemize}
\end{lemma}

\begin{proof}
It is assumed that the $Y_\infty$ version of the operator in \eqref{eq:linearized} has trivial kernel when it is defined by an admissible solution to \eqref{eq:sw} with $\en \leq 2\cale$. Granted that this is so, then Lemma~\ref{lem:monopoles-18} asserts that the space of $C^\infty(Y_\infty;S^1)$ equivalence classes of admissible solutions to \eqref{eq:sw} on $Y_\infty$ is a discrete set, and that there are at most a finite number of such equivalence classes with $\en \leq \frac{3}{2}\cale$.

In what follows, $c_\cale$ is used to denote numbers that are greater than $1$ and are independent of any particular $T$ or solution to \eqref{eq:sw} on $Y_T$ for the given value of $r$. It can depend on $\cale$ however. Lemma~\ref{lem:monopoles-111} has the following implications: Given $\upvarepsilon>0$ and $T$ sufficiently large (with a lower bound depending only on $\cale$ and $\upvarepsilon$, but not $r$), let $\frc_T$ denote a solution to \eqref{eq:sw} on $Y_T$ with $\en \leq \cale$. There is an admissible solution to \eqref{eq:sw} on $Y_\infty$ with $\en \leq \cale$ to be denoted by $\frc$ and maps $g_T: Y_T \ssm N_T \to S^1$ and $\hu_T: N_T \to S^1$ such that \eqref{eq:monopoles-130} holds.  The assignment of $\frc$'s gauge equivalence class to the solution $\frc_T$ defines a map from the space of gauge equivalence classes of solutions to \eqref{eq:sw} on $Y_T$ with $\en \leq \cale$ to the space of gauge equivalence classes of admissible solutions to \eqref{eq:sw} on $Y_\infty$ with $\en \leq \cale$.  This map is well defined for the following two reasons: First, there are but finitely many gauge equivalence classes of admissible solutions on $Y_\infty$ with $\en \leq \cale$ to choose from. Second, because of Lemma~\ref{lem:monopoles-110}, only one gauge equivalence class on $Y_\infty$ will obey the condition in the first bullet of \eqref{eq:monopoles-130} for any given $\frc_T$ if $\upvarepsilon < c_\cale^{-1}$.

Denote the map just defined by $\T$.  It does not depend on $\upvarepsilon$ for the following reason: If $\frc$ is an $\en \leq \cale$ admissible solution to \eqref{eq:sw} on $Y_\infty$ that meets the conditions in the first bullet of \eqref{eq:monopoles-130} for a given $\upvarepsilon < c_\cale^{-1}$, and if $\frc'$ meets the conditions for a given $\upvarepsilon' < \upvarepsilon$, then $\frc'$ meets the conditions for $\upvarepsilon$. This implies that $\frc'$ and $\frc$ must be gauge equivalent because (as noted in the preceding paragraph) there is only one gauge equivalence class on $Y_\infty$ that obeys the condition in the first bullet of \eqref{eq:monopoles-130} for any given solution $\frc_T$.

With $\T$ in hand, consider (out of turn) the third bullet of the lemma. To this end, suppose that $T$ is greater than $c_\cale^{-1}$ (so as to use the map $\T$), and suppose that $\frc_T$ is a solution to \eqref{eq:sw} on $Y_T$ with $\en \leq \cale$. Let $\LLf_T$ denote the $\frc_T$ version of the operator in \eqref{eq:linearized} and let $\frh$ denote an element in its domain with $L^2$ norm equal to 1. What follows is a consequence of Lemma~\ref{lem:monopoles-11} and Lemma~\ref{lem:monopoles-110}: Given $\delta > 0$, there exists $R \geq 1$ which depends only on $\delta$ and $\cale$ such that $\frh$ has $L^2$ norm less than $\delta$ where the distance to $M$ is greater than $R$. Let $\chi_R$ denote a smooth function on $Y_T$ mapping to $[0,1]$ that is zero where the distance to $Y_T$ is greater than $2R$ and 1 where the distance is less than $R$. This function can be chosen so that $|d\chi_R| \leq 10R^{-1}$. These bounds imply that
\begin{equation} \label{eq:monopoles-new-129}
||\LLf_T(\chi_R \frh)||_2 \leq ||\LLf_T \frh||_2 + c_\cale\delta^2R^{-1}.
\end{equation}
Now let $\frc = \T(\frc_T)$ and let $\LLf$ denote $\frc$'s version of the operator in \eqref{eq:linearized}. It follows from \eqref{eq:monopoles-130} that there is a map from $Y_\infty$ to $S^1$ (to be denoted by $g$) such that
\begin{equation} \label{eq:monopoles-new-130}
||\LLf(g\frh)||_2 \leq ||\LLf_T(\chi_R \frh)||_2 + c_R \upvarepsilon^2,
\end{equation}
with $c_R$ denoting a number that depends only on $\cale$ and $R$. The claim made by Lemma~\ref{lem:monopoles-115}'s third bullet follows directly from these last two inequalities.

The argument that $\T$ is 1-1 onto its image differs little from the argument that proves Lemma~\ref{lem:monopoles-18} when $T$ is large. Here is how it goes: Assume that $\T$ isn't a bijection to generate nonsense with a sequence $\{(A_i,\psi_i)\}_{i\in\{1,2,\dots\}}$ and $\{(A'_i,\psi'_i)\}_{i\in\{1,2,\dots\}}$ with $i$th member on $Y_{T(i)}$ such that $\{T(i)\}_{i\in\{1,2,\dots\}}$ diverges.  Assume that the primed and unprimed pairs are solutions to \eqref{eq:sw} with $\en \leq \cale$ and such that $\T$ maps them to the same equivalence class.  Having just established the third bullet of Lemma~\ref{lem:monopoles-115}, it is enough to prove an $\upvarepsilon$ slice lemma to play the role of Lemma~\ref{lem:monopoles-19} so as to generate nonsense using the argument for the proof of Lemma~\ref{lem:monopoles-18}.  The statement of the analogous slice theorem is the same except that in this case, the $i$th elements of the sequences $\{(A_i,\psi_i)\}_{i\in\{1,2,\dots\}}$ and $\{(A'_i,\psi'_i)\}_{i\in\{1,2,\dots\}}$ are on the corresponding $Y_{T(i)}$. The proof of this analog copies almost verbatim the proof of Lemma~\ref{lem:monopoles-19}. In fact, one need only replace ``\emph{the part of $Y_\infty$ with distance greater than $c_*$ from $N$}'' with ``\emph{the part of $Y_T$ with distance greater than $c_\cale$ from $M$}'' and replace ``\emph{$Y_\infty$}'' with ``\emph{the versions of $Y_{T(i)}$ with $T(i) \gg T$}''.  Otherwise, there is no essential change.

The argument that $\T$ is onto uses a cut-off function $\chi$ on $Y_\infty$ that is defined as in Step 4 of the proof of Lemma~\ref{lem:monopoles-19} for a $T > c_\cale$. Write an $\en < \cale$ solution $(A,\psi)$ on $Y_\infty$ on the part of $Y_\infty$ with distance greater than $c_\cale$ from $M$ as $A=\Ai + \aA$ and $\psi = (1-z)\psii + \eta$ with $z$ being real with norm less than $\frac{1}{100}$. Setting $(A_T,\psi_T) = (A,\psi)$ near $M$ and setting $(A_T,\psi_T) = (\Ai+(1-\chi\beta)\aA, (1-(1-\chi)z)\psii + (1-\chi)\eta)$ otherwise defines a pair (to be denoted by $\frc_T$) on $Y_T$ that comes close to solving \eqref{eq:sw}.  The error is bounded pointwise by $c_\cale r^{1/2} e^{-\sqrt{r}T/c_0}$ with the error having components given by $r^{-1/2}(*F_{A_T}-r\psi_T^\dagger \tau\psi_T)$ and $\mathcal{D}_{A_T}\psi$.  That this is so follows from Lemma~\ref{lem:monopoles-16}. One can then look for $\frh$ obeying an equation that has the form
\begin{equation} \label{eq:monopoles-new-131}
\LLf_T\frh + r^{1/2}\frh \ast \frh + \mathrm{error} = 0
\end{equation}
with $\LLf_T$ being the $\frc_T$ version of the operator in \eqref{eq:linearized} on $Y_T$ and with $\frh*\frh$ denoting a section whose components are quadratic functions of the entries of $\frh$. This is to say that $\frh*\frh$ is the image of $\frh\otimes\frh$ via a certain (canonical) vector bundle homomorphism from $\otimes_2 (i(T^*Y_T\oplus\R)\oplus\spb)$ to $i(T^*Y_T\oplus\R)\oplus\spb$. In particular, the norm of $\frh*\frh$ is bounded by $c_0|\frh|^2$ and that of its $A$-covariant derivative by $c_0|\nabla_A\frh||\frh|$.  (The precise form of $\frh*\frh$, except for it being quadratic in the components of $\frh$, is of no consequence in what follows.)  Now, the spectrum of $\LLf_T$ has distance greater than $\frac{1}{2}c^{-1}$ from $0$. This can be proved using the same argument that proves the third bullet point of Lemma~\ref{lem:monopoles-115}. Granted that this is so, then \eqref{eq:monopoles-new-131} is equivalent to the fixed point equation on the domain of $\LLf_T$:
\begin{equation} \label{eq:monopoles-new-132}
\frh = -\LLf_T^{-1}(r^{1/2}\frh \ast \frh + \mathrm{error}).
\end{equation}
This equation can be solved with the solution $\frh$ being small using the contraction mapping theorem when $T \geq c_\cale$ and $r > c_\cale$.
\end{proof}

%% file: a_instantons.tex

Throughout the rest of the appendix, we assume that $Y_T$ is as described in Section \ref{ssec:monopoles}, where $T$ is either a large number greater than $16$ or it is $\infty$. Of interest here are solutions to the Seiberg--Witten equations on $\R\times Y_T$ with the properties listed in Section \ref{sssec:instanton-sw}. Let $a_+$ and $a_-$ be contact 1-forms on $Y_T$ of the sort described in Section \ref{ssec:setup}. suppose that these 1-forms agree on $Y_T\ssm N$, and in the case when $N=M\cup N_T$ with $M$ and $N_T$ as in Section \ref{sssec:limit_at_infty}, assume further that they agree on $N_T$. This last assumption is only for convenience and can be violated if needs be with only notational consequences.

Introduce by way of notation $s$ to denote the Euclidean coordinate on the $\R$ factor of $\R \times Y_T$. Assume that there exists $s_0>1$ and a 1-form to be denoted by $a$ on $\R\times Y_T$ that is equal to $a_-$ where $s\leq -s_0$, equal to $a_+$ where $s\geq s_0$. Require that $a$ be independent of $s$ on $\R\times(Y_T\ssm N)$ and on $\R\times N_T$ in the case when $N=M\cup N_T$. We also assume that the 2-form $e^{2s}(ds\wedge(2a+\frac{\partial}{\partial s}a)+da)$ is a symplectic form on $\R\times Y_T$. Here $d$ denotes the exterior derivative along the $Y_T$ factor of $\R\times Y_T$. We use $\omega$ to denote the 2-form $ds\wedge(a+\frac{1}{2}\frac{\partial}{\partial s}a)+\frac{1}{2}da$.

To define a metric, let $\frg_-$ denote a metric that is defined by $a_-$ on $Y_T$ and let $\frg_+$ denote one that is defined by $a_+$ on $Y_T$ in the manner that is described in Section \ref{ssec:setup}. We shall assume that these two metrics agree on $Y_T\ssm N$ and on $N_T$ when $N=M\cup N_T$. The metric on $\R\times Y_T$ is taken to be the product metric $ds^2+\frg_-$ on the $s\leq -s_0$ part, to be the product metric $ds^2+\frg_+$ on the $s\geq s_0$ part, and to be the interpolating product metric on the part $\R\times (Y_T\ssm N)$. The metric on the rest of $\R\times Y_T$ is of the form $ds^2+\frg$ with $\frg$ being an $s\in\R$ dependent metric on $Y_T$ whose Hodge-star is chosen so that $\ast da=2a+\frac{\partial}{\partial s}a$ and is such that $\ast da$ has length $2$. The 2-form $\omega$ with this metric is self-dual and has norm $\sqrt{2}$. The metric should also be chosen so that the norms of the covariant derivatives of $a$, the curvature tensor, and the covariant derivatives of the curvature tensor are bounded by $T$-independent constants.
\subsubsection{The Seiberg--Witten equations}
\label{sssec:instanton-sw}
Supposing that $\spb^+$ is the associated self-dual spinor bundle for a $\Sc$ structure on $\R\times Y_T$, there is a splitting as in \eqref{eq:splitting} with the summands being the respective $+i$ and $-i$ eigenbundles for Clifford multiplication by $\frac{1}{2}\omega$, the $E$ summand being the $+i$ eigenbundle. The canonical $\Sc$ structure has $\spb^+$ splitting as in \eqref{eq:splitting} with $E=\underline{\C}$.

Fix $r>1$.  Also, fix a smooth, $s\in [-1,1]$ dependent $i\R$-valued $2$-form on $M$ with the following properties: Its norm and that of its first and second order covariant derivatives should be bounded by 1; it must also have compact support in an $s$-independent open set in the interior of $M$; and it must vanish when $s$ is near $1$ or $-1$. Use $\frp$ to denote this $s$-dependent family of 2-forms. This family of $2$-forms can (and will) be viewed equivalently as a single $2$-form on $\R\times M$ with support in a compact subset of $(-1,1)\times M$ that annihilates all vectors tangent to the $\R$ factor. This latter incarnation is also denoted by $\frp$. By way of a look ahead, the 2-form $\frp$ plays a role only in Proposition~\ref{prop:comp-64}, which is at the very end of this Appendix.

A pair $(A,\psi)$ of connection $A$ on $E$ and section $\psi$ of $\spb^+$ obeys the Seiberg-Witten equations on $\R\times Y_T$ when
\begin{equation}
\label{eq:instanton-sw}
F_A^+ = \frac{1}{2}r((\psi^\dagger \tau\psi) - i\omega) - \frac{1}{2}(F_{A_0}^++\frp^+)\mathrm{\ \ \ and\ \ \ } \mathcal{D}_A \psi=0
\end{equation}
where the notation is as follows: First, $F_A^+$ is the self-dual part of the curvature 2-form of the connection $A$. Likewise, $F_{A_0}^+$ is the self-dual part of $F_{A_0}$ (viewed as a 2-form on $\R \times Y_T$ that annihilates the vector field $\frac{\partial}{\partial t}$), and $\frp^+$ is the self-dual part of $\frp$ (viewed similarly).  Meanwhile, $\mathcal{D}_A$ denotes here the Dirac operator on $\R\times Y_T$ that is defined as follows: Having specified the Riemannian metric on $\R\times Y_T$, the definition of the Dirac operator on $\R\times Y_T$ requires only the choice of a Hermitian connection on the line bundle $\textit{det}(\spb)$. In the case of $\mathcal{D}_A$, such a connection is defined by the connection $A$ on $E$ and a certain canonical connection on the bundle $K^{-1}$ that appears in the $\R\times Y_T$ analog of \eqref{eq:splitting}. This canonical connection is defined with the help of a chosen unit section of the product line bundle $\underline{\C}$. This section is chosen once and for all, and once chosen, it is denoted by $\psii$. Let $\Ai$ again denote the product connection on $\underline{\C}$. There is a unique connection on $K^{-1}$ such that the latter with $\Ai$ define a connection on the line bundle $\textit{det}(\spb)$, giving a Dirac operator $\mathcal{D}_{\Ai}$ that annihilates $\psii$. This connection on $K^{-1}$ is used to define $\mathcal{D}_A$ also.
\subsubsection{The Bochner--Weitzenb\"ock formula}
\label{sssec:instanton-bw}

The formal linearization of \eqref{eq:instanton-sw} with the extra gauge fixing equation defines a first order differential operator mapping sections of the bundle $iT^*(\R\times Y_T) \oplus \spb^+$ to sections of the bundle $i(\Lambda^+\oplus \R)\oplus \spb^-$ with $\Lambda^+$ denoting the bundle of self-dual 2-forms and with $\spb^-$ denoting the bundle of anti-self dual spinors. This operator is denoted by $\LL$ and it is defined by the rule whereby the respective $i\Lambda^+$, $i\R$, and $\spb^-$ summands of $\LL(\fra,\eta)$ are
\begin{equation}
\label{eq:instanton-linearize}
\LL(\fra,\eta) = \left(\begin{array}{c}
(d\fra)^+ - r^{1/2}\left(\psi^\dagger \tau\eta + \eta^\dagger\tau\psi\right) \\
*d*\fra + r^{1/2}\left(\psi^\dagger \eta - \eta^\dagger \psi\right) \\
\mathcal{D}_A\eta + r^{1/2}\clm(\fra)\psi
\end{array}\right).
\end{equation}
The notation here uses $d$ to denote the exterior derivative on $\R\times Y_T$. We use $\Li$ in what follows to denote the version of \eqref{eq:instanton-linearize} that is defined by the canonical pair $(\Ai,\psii)$ from Section \ref{ssec:setup}.  Supposing that $\rmu\subset \R\times Y_T$ is an open set, there is the analog of \eqref{eq:monopoles-15} that reads
\begin{quote}
There exists a number $\kappa > 100$ with the following significance: Fix $r>\kappa$ and suppose that $\frh$ is a section of $iT^*\rmu\oplus\spb^+$ on $\rmu$ with compact support. Then
\begin{equation}
\label{eq:instanton-li-l21-bound}
||\Li\frh||_2{}^2 \geq (1-\kappa^{-1})\left(||\nabla_{\Ai}\frh||_2{}^2 + r||\frh||_2{}^2\right).
\end{equation}
\end{quote}
Here, $||\cdot||_2$ denotes the $L^2$ norm on $\rmu$. The notation with \eqref{eq:instanton-li-l21-bound} also has $\nabla_A$ denoting the covariant derivative that is defined by a given connection $A$ (in this case $\Ai$) and the Levi-Civita connection.
\subsubsection{Exponential decay}
\label{sssec:instanton-decay}
There is an almost verbatim analog of Lemma \ref{lem:monopoles-11} for the operator $\LL$ with almost word for word the same proof.  To set the background, suppose that $\pzee$ is a first order operator mapping sections of $iT^*\rmu \oplus \spb^+$ to $i(\Lambda^+\oplus \R) \oplus \spb^-$. Of interest in Lemma \ref{lem:instanton-21} are sections $\frh$ that obey the equation
\begin{equation}
\label{eq:instanton-lipluse}
\Li\frh + \pzee\frh = 0.
\end{equation}
The lemma uses the following notation: Suppose that $V\subset \rmu$ is an open set and that $\rho$ is a positive number. For each integer $k \in \{0,1,\dots\}$, let $V_k \subset V$ denote the set of points with distance $k\rho$ or more from each point in $\rmu\ssm V$.  Given a section $\frh$ of $iT^*\rmu\oplus \spb^+$, define $||\frh||_{*,k}$ as in \eqref{eq:monopoles-19}.

\begin{lemma}
\label{lem:instanton-21}
There exists $\kappa>1$ with the following significance: Suppose that $\rmu$ is an open set in $\R\times Y_T$ and that $\pzee$ is a first order differential operator defined on $\rmu$ mapping sections of $iT^*\rmu\oplus\spb^+$ to sections of $i(\Lambda^+\oplus\R)\oplus \spb^-$ that obeys
\[ ||\pzee\frh||_2{}^2 \leq \frac{1}{2}\left(||\nabla_{\Ai}\frh||_2{}^2 + r||\frh||_2{}^2\right) \]
when $\frh$ has compact support in $\rmu$. Fix $r>\kappa$ and suppose that $\frh$ obeys \eqref{eq:instanton-lipluse}. Fix $V\subset \rmu$ with compact closure. Then $||\frh||_{*,k} \leq \kappa||\frh||_{*,0} e^{-\sqrt{r}k\rho/\kappa}$.
\end{lemma}
As noted above, the proof of this lemma differs little from the proof of Lemma \ref{lem:monopoles-11}.

There is an analog of this lemma that holds on the large $|s|$ parts of $\R\times Y_T$. To set the stage, suppose that $(A_+,\psi_+)$ is a pair of connection on $Y_T$ and section of $\spb$ over $Y_T$. The bundle $\spb$ can be identified in a canonical fashion with both $\spb^+$ and $\spb^-$ by using the vector field $\frac{\partial}{\partial s}$ on $\R\times Y_T$ to view the principal $SO(4)$-bundle of orthonormal frames on $\R\times Y_T$ as an associated bundle to the pull-back via the projection to $Y_T$ of the latter's orthonormal frame bundle.  Suppose now that the $(A_+,\psi_+)$ version of the operator $\LLf$ is such that $||\LLf\frh||_{Y,2} > z||\frh||_{Y,2}$ for all ${L^2}_1$ sections $\frh$ of $i(T^*Y_T\oplus \R)\oplus \spb$ with $z$ being a fixed positive number and $||\cdot||_{Y,2}$ indicating the $L^2$-norm on $Y_T$.  Note that the Bochner--Weitzenb\"ock formula for $\LLf$ in this case implies that
\begin{equation}
\label{eq:instanton-bw25}
||\LLf\frh||_{Y,2} > c_z^{-1} ||\nabla_A \frh||_{Y,2},
\end{equation}
with $c_z > 1$ being a number that depends on $z$, the norm of the curvature of $A$ and the norm of the covariant derivative $\nabla_{A_+}\psi_+$.  The next lemma views the pair $(A_+,\psi_+)$ as a pair of connection over $\R\times Y_T$ and section of $\spb^+$ over $\R\times Y_T$.

\begin{lemma}
\label{lem:instanton-22}
Given $(A_+,\psi_+)$ as just described, there exists $\kappa>1$ with the following significance: Let $\pzee$ denote a first order differential operator mapping sections of $iT^*(\R\times Y_T)\oplus \spb^+$ to $i(\Lambda^+\oplus\R)\oplus \spb^-$ with the property that $||\pzee\frh||_2 \leq \kappa^{-1}\left(||\nabla_A \frh||_2 + ||\frh||_2\right)$ for all sections $\frh$ with compact support.  Fix $s_0\in\R$ and suppose that $\frh$ is an $L^2$ section of $iT^*(\R\times Y_T) \oplus \spb^+$ on the $s>s_0$ or $s<-s_0$ part of $\R\times Y_T$ that obeys $\LL\frh + \pzee\frh=0$ with $\LL$ being defined by \eqref{eq:instanton-linearize} using the pair $(A_+,\psi_+)$.  Let $||\frh||_{\Delta,2}$ denote the $L^2$ norm of $\frh$ on the part of $\R\times Y_T$ where $s\in[s_0,s_0+1]$ or $s\in[-s_0-1,-s_0]$ as the case may be.  For $s>s_0+2$ or $s<-s_0-2$,
\[ ||\nabla_{A_+} \frh|_s||_{Y,2} + ||\frh|_s||_{Y,2} \leq \kappa e^{-|s|/\kappa} ||\frh||_{\Delta,2}. \]
\end{lemma}
The proof of this lemma differs little from the proof of Lemma \ref{lem:instanton-21}.
\subsubsection{Fredholm property of $\LL$}
\label{sssec:instanton-fredholm}
Lemmas \ref{lem:instanton-21} and \ref{lem:instanton-22} have implications for the operator $\LL$ when $(A,\psi)$ is a pair of connection on $E$ and $\psi$ is a section of $\spb^+$ over $\R\times Y_T$ with asymptotics as described below. To set the notation, $(A_+,\psi_+)$ and $(A_-,\psi_-)$ denote admissible pairs of connections and sections of $\spb$ on $Y_T$.

\begin{itemize}\leftskip-0.25in
\item There is a gauge transformation where $s\geq1$ that changes $(A,\psi)$ to $(A_+ + \ahat_+, \psi_++\sigma_+)$ with $\frh_+ = (r^{-1/2}\ahat_+,\sigma_+)$ obeying $\displaystyle \int_{[1,\infty)\times Y_T} \left(|\nabla_{A_+} \frh_+|^2 + |\frh_+|^2\right) <\infty$.
\item There is a gauge transformation where $s\leq-1$ that changes $(A,\psi)$ to $(A_-+\ahat_-, \psi_-+\sigma_-)$ with $\frh_- = (r^{-1/2}\ahat_-,\sigma_-)$ obeying $\displaystyle \int_{(-\infty,-1]\times Y_T} \left(|\nabla_{A_-} \frh_-|^2 + |\frh_-|^2\right) <\infty$.
\item There is a gauge transformation on $\R\times (Y_T\ssm N)$ that changes $(A,\psi)$ to $(\Ai+\ahat_0,\psii+\sigma_0)$ with $\frh_0 = (r^{-1/2}\ahat_0,\sigma_0)$ obeying $\displaystyle \int_{[-2,2]\times(Y_T\ssm N)} \left(|\nabla_{\Ai} \frh_0|^2+|\frh_0|^2\right) < \infty$.
\end{itemize}
\begin{equation}
\label{cond:instanton-26}
\end{equation}

The next lemma assumes that Condition \ref{cond:instanton-26} is obeyed. The lemma uses $\LL^\dagger$ to denote the formal $L^2$-adjoint of the operator $\LL$.

\begin{lemma}
\label{lem:instanton-fredholm}
Suppose that $(A,\psi)$ is as just described with $(A_+,\psi_+)$ and $(A_-,\psi_-)$ being admissible pairs on $Y_T$ with respective versions of the operator $\LLf$ in \eqref{eq:linearized} that have trivial $L^2$ kernel. The operator $\LL$ defined by $(A,\psi)$ defines a Fredholm operator from the space of ${L^2}_1$ sections of $iT^*(\R\times Y_T) \oplus \spb^+$ to the space of $L^2$ sections of $i(\Lambda^+ \oplus \R) \oplus \spb^-$.  The cokernel of $\LL$ consists of ${L^2}_1$ sections of $i(\Lambda^+ \oplus \R) \oplus \spb^-$ that are in the kernel of the operator $\LL^\dagger$.
\end{lemma}

\begin{proof}
Lemma \ref{lem:instanton-22} says that the semi-norm $\frh \to ||\LL\frh||_2$ is bounded from below by $c_0^{-1}r^{1/2}||\frh||_2$ when $\frh$ is supported where $t \gg 1$ on $Y_\infty$, and the lack of a kernel for the $|s|\to\infty$ versions of $\LL$ imply that this semi-norm is also bounded from below by a multiple of $||\frh||_2$ when $\frh$ has support where $|s|\gg 1$. These bounds are used with suitable cut-off functions to prove that the kernel is finite dimensional. They are also used with the Rellich lemma (for a compact domain) to prove that the cokernel is also finite dimensional. The Rellich lemma tells us that sections that are in the cokernel are smooth. One can then use cut-off functions and integration by parts to prove that elements in the cokernel are annihilated by the $L^2$ adjoint of $\LL$ and vice-versa.
\end{proof}
\subsubsection{The behavior of solutions as $|s|\to\infty$ or as $t\to\infty$}
\label{sssec:instanton-behavior}
A solution to \eqref{eq:instanton-sw} is said below to be an \emph{instanton} solution when Condition \ref{cond:instanton-26} is met with $(A_+,\psi_+)$ and $(A_-,\psi_-)$ being admissible solutions to the respective $a_+$ and $a_-$ versions of \eqref{eq:sw} on $Y_T$. I assume in what follows that both the $(A_+,\psi_+)$ and $(A_-,\psi_-)$ versions of the operator $\LLf$ in \eqref{eq:linearized} have trivial kernel.

The first step to analyzing solutions is to prove an assertion to the effect that $(A,\psi)$ decays exponentially fast to $(A_+,\psi_+)$ on the respective $s\gg 1$ and $s\ll -1$ parts of $\R\times Y_T$, and to $(\Ai,\psii)$ on the $t\gg 1$ part in the case of $\R\times Y_\infty$. The following lemma summarizes the story. The lemma uses $\Delta$ to denote the distance on $\R \times Y_\infty$ to $\R\times M$.

\begin{lemma}
\label{lem:instanton-24}
Suppose that $(A,\psi)$ is an instanton solution to \eqref{eq:instanton-sw} in $\R\times Y_T$ as described in Lemma~\ref{lem:instanton-fredholm}.
\begin{itemize}\leftskip-0.25in
\item The gauge transformations on $[1,\infty)\times Y_T$ and on $(-\infty,-1]\times Y_T$ from the respective first two bullets of \eqref{cond:instanton-26} can be chosen so that the corresponding elements $\frh_+$ and $\frh_-$ obey $\displaystyle \lim_{s\to\pm\infty} e^{|s|/\kappa} |\frh_\pm| = 0$, with $\kappa>1$ depending on $(A_+,\psi_+)$ and on $(A_-,\psi_-)$.

\item In the case of $Y_\infty$, the gauge transformation in the third bullet of \eqref{cond:instanton-26} can be chosen on $\R\times (Y_\infty\ssm M)$ so that $\frh_0$ obeys $\displaystyle \lim_{\Delta\to\infty} e^{\sqrt{r}\Delta/\kappa}|\frh_0| = 0$ with $\kappa\geq 1$ being independent of $(A_+,\psi_+)$ and $(A_-,\psi_-)$.

\end{itemize}
\end{lemma}

\begin{proof}
The proof has eight parts. The second bullet is treated first in Parts 1--6.
\begin{pt3}
This part makes some observations about \eqref{cond:instanton-26} that are used (sometimes implicitly) in the subsequent proof. The first point to make is that neither $\frh_+$, $\frh_-$ nor $\frh_0$ are gauge invariant. This is to say that any changing $(A,\psi)$ to $(A-g^{-1}dg,g\psi)$ with $g$ being a map to $S^1$ will change $\frh_+$, $\frh_-$ and/or $\frh_0$. Likewise, changing $(A_+,\psi_+)$ and/or $(A_-,\psi_-)$ by a gauge transformation will change $\frh_+$ and/or $\frh_-$. Nonetheless, certain gauge invariant functions and differential forms that are constructed from $(A,\psi)$ (or $(A_+,\psi_+)$ or $(A_-,\psi_-)$) can be bounded in terms of $\frh_+$, $\frh_-$, $\frh_0$ and their derivatives, The examples that follow are directly relevant:

\noindent\textsc{The curvature:} The curvature $F_A$ where $\Delta>1$ is $r^{1/2}$ times the exterior derivative of the 1-form component of $\frh_0$. Thus, its norm is bounded by $r^{1/2}|\nabla_{{\Ai}}\frh_0|$. Likewise, the differences $F_A-F_{A_+}$ where $s>1$ and $F_A-F_{A_-}$ where $s<-1$ are (respectively) $r^{1/2}$ times the exterior derivatives of the 1-form components of $\frh_+$ and $\frh_-$. Therefore, the norms of these differences are bounded by $r^{1/2}|\nabla_{A_+}\frh_+|$ and $r^{1/2}|\nabla_{A_-}\frh_-|$. 

\noindent\textsc{The norm of} $\psi$: The norm of $|\psi|-1$ where $\Delta>1$ is bounded by the norm of the $\spb^+$ component of $|\frh_0|$. By the same token, the norm of $|\psi|-|\psi_+|$ where $s>1$ and the norm of $|\psi|-|\psi_-|$ where $s<-1$ are bounded respectively by the norms of $\frh_+$ and $\frh_-$.

To state the next example, write the spinor bundle $\spb^+$ as $E\oplus E\otimes K^{-1}$ (analogous to \eqref{eq:splitting}) and then write $\psi$ accordingly as $(\alpha,\beta)$.

\noindent\textsc{The norms of $\nabla_A\alpha$ and $\nabla_A\beta$:} The functions $\nabla_A\alpha$ and $\nabla_A\beta$ where $\Delta>1$ are bounded respectively by $|\nabla_{\Ai}\frh_0|+r^{1/2}|\frh_0|^2$ and $|\nabla_{\Ai}\frh_0|+c_0|\frh_0|+r^{1/2}|\frh_0|^2$. By the same token, the norms of $|\nabla_A\alpha|-|\nabla_{A_+}\alpha_+|$ and $|\nabla_A\beta|-|\nabla_{A_+}\beta_+|$ where $s>1$ are bounded by $|\nabla_{A_+}\frh_+|+r^{1/2}|\frh_+|^2$ and $|\nabla_{A_+}\frh_+|+c_0|\frh_0|+r^{1/2}|\frh_+|^2$; and there are analogous inequalities with $+$ changed to $-$ where $s<-1$.

The examples with \eqref{cond:instanton-26} imply that the functions $|F_A|$, $|\alpha|-1$, $|\beta|$, $|\nabla_A\alpha|$, and $\nabla_A\beta$ have finite $L^2$ norm on any domain of the form $I\times Y_\infty$ with $I\subset\R$ being an interval of length $1$ (and that these $L^2$ norms are bounded independently of $I$). The examples given above and \eqref{cond:instanton-26} also imply that the difference between these functions and their $(A_+,\psi_+)$ analogs have finite $L^2$ norm on the $s>1$ part of $\R\times Y_\infty$ (and likewise for $\R\times Y_T$ with $T<\infty$). And, the difference between these functions and their $(A_-,\psi_-)$ analogs have finite $L^2$ norm on the $s<-1$ part of $\R\times Y_\infty$ (and likewise for $\R\times Y_T$ with $T<\infty$). (A dimension $4$ Sobolev inequality is used here to bound the $L^4$ norms of $|\frh_\pm|$ and $|\frh_0|$ by $c_0$ times their $L^2_1$ norms.)
\end{pt3}
\begin{pt3}
The argument for the second bullet starts by writing the spinor bundle $\spb^+$ as $E\oplus E\otimes K^{-1}$ and then writing $\psi$ with respect to this splitting as $(\alpha,\beta)$. Using this notation, the respective $E$ and $E\otimes K^{-1}$ parts of the identity $\diracop_A^\dagger\diracop_A\psi=0$ can be written as 
\begin{itemize}\leftskip-0.35in
\item ${\nabla_A}^{\dagger}\nabla_A\alpha+r(|\alpha|^2-1+|\beta|^2)\alpha+c_0\alpha+c_1\nabla_A\beta+c_2\beta=0$.
\item ${\nabla_A}^{\dagger}\nabla_A\beta+r(|\alpha|^2+1+|\beta|^2)\beta+c_3\nabla_A\beta+c_4\beta+c_5\nabla_A\alpha \,+\,c_6\alpha=0$.
\end{itemize}
\begin{equation}
\label{cond:instanton-27}
\end{equation}
With regards to the notation, what is denoted by $c_0(F_{A_0}+\frp)$ is an endomorphism of $E$ that is linear in the curvature of $A_0$ and the perturbing form $\frp$. Of particular importance is that this vanishes where $A_0$ is flat and $\frp$ is zero on $Y_T\smallsetminus N$ for example. Meanwhile, $\{c_k\}_{k=1,\dots,6}$ are $(A,\psi)$-independent and $r$-independent (and $s$-independent) sections of vector bundles over $\R\times Y_\infty$ that are bounded with bounded derivatives.  (The endomorphism $c_6$ is also linear in the curvature of $A_0$ and $\frp$.)

The equations in \eqref{cond:instanton-27} are used first to prove the following when $r>c_0$:
\begin{itemize}\leftskip-0.35in
\item $|\alpha|\leq 1+c_0r^{-1}$.
\item $|\beta|\leq c_0 r^{-1/2}$.
\end{itemize}
\begin{equation}
\label{cond:instanton-28}
\end{equation}
To prove \eqref{cond:instanton-28}, fix for the moment a number greater than $1$ to be denoted by $c$ and then let $\pzq=|\alpha|^2+c^{-1}r|\beta|^2$. If $c\geq c_0$ and if $r$ is sufficiently large (depending on $c$ but not on $A$, $\alpha$, or $\beta$) then (after some manipulations) the identities in \eqref{cond:instanton-27} lead to the inequality 
\begin{equation}
\label{eq:instanton-29}
\frac{1}{2}d^{\dagger}d\pzq+r(\pzq-1)\pzq-c_0\pzq\leq 0.
\end{equation}
On a compact manifold, this inequality would imply directly that $\pzq\leq 1+c_0r^{-1}$ via the maximum principle. This same bound is implied here also, but \eqref{cond:instanton-26} is needed to deduce it. To say more, fix for the moment a ball $B\subset\R\times Y_\infty$ of radius equal to $\frac{1}{100}$ times the injectivity radius of $Y_\infty$ (which is greater than ${c_0}^{-1}$); and let $\upchi_B$ denote a smooth non-negative function with compact support on $B$ that is equal to $1$ on the concentric ball of half the radius. Let $p$ denote the center of $B$ and let $G_p$ denote the Green's function for the operator $d^\dagger d$ on $B$ with Dirichlet boundary conditions. Note that $0\leq G_p\leq c_0\frac{1}{\dis(p,\cdot)^2}$.

Multiply both sides of \eqref{eq:instanton-29} by $\upchi_B G_p$ (which is non-negative) and then integrate over the part of $B$ where $\pzq>1$. Integration by parts leads to a bound saying that
\begin{equation}
\label{eq:instanton-210}
(\pzq(p)-1)_+\leq c_0\int_{B}\frac{1}{\dis(p,\cdot)^2}(\pzq-1)_+
\end{equation}
with $(\pzq-1)_+$ denoting the maximum of $(\pzq-1)$ and $0$. The right hand side of \eqref{eq:instanton-210} is bounded by virtue of the inequalities in Part 1, \eqref{cond:instanton-26} and what is called Hardy's inequality \cite{hlp}. This tells us immediately that $|\alpha|\leq c_0$ and $|\beta|\leq c_0 r^{-1/2}$.

Now, given $\upvarepsilon>0$, the right hand side of \eqref{eq:instanton-210} is less than $\upvarepsilon$ if $\Delta(p)$ is sufficiently large. Note also that what is meant by `sufficiently large' does not depend on the $s$ coordinate of $p$. This is to say that given only the value $\upvarepsilon>0$, there exists $\mathrm{R}>16$ such that if  $\Delta(p)>\mathrm{R}$, then the right hand side of \eqref{eq:instanton-210} is less than $\upvarepsilon$ (irrespective of the value of the $s$-coordinate of $p$). That this is so follows from $(A_\pm,\psi_\pm)$ versions of Lemma \ref{lem:monopoles-16}. (Lemma \ref{lem:monopoles-16} requires $\en$ in \eqref{eq:energy} be finite, which it is for admissible solutions such as $(A_\pm,\psi_\pm)$ by virtue of Lemma \ref{lem:monopoles-14}.) Of particular interest is the case when $\upvarepsilon=r^{-1}$ in which case one learns this: There exists $\mathrm{R}$ such that if $p\in\R\times Y_\infty$ and $\Delta(p)\geq\mathrm{R}$, then $\pzq\leq 1+r^{-1}$.

Meanwhile, the right hand side of \eqref{eq:instanton-210} is less than $c_0r^{-1}$ if $|s|$ is sufficiently large because of the top line of \eqref{eq:monopoles-new-116} and the appearance of $r^2|\beta|^2$ in the second line of \eqref{eq:monopoles-new-116}.

To summarize the preceding two paragraphs: There exists $m<c_0$ such that the bound $\pzq\leq 1+mr^{-1}$ holds on the complement of a compact subset of $\R\times Y_\infty$. Thus, if $\pzq$ is ever larger than $1+mr^{-1}$, then it takes on its maximum in this same compact subset. With the preceding understood, suppose that $p$ is a point in this compact subset where $\pzq$ has a local or global maximum. Then, by virtue of \eqref{eq:instanton-29}, one has $r(\pzq-1)-c_0\leq 0$ at the point $p$ because $d^\dagger d\pzq\geq 0$ at a local maximum of $\pzq$. Thus $\pzq\leq 1+c_0r^{-1}$ at its maximum so $\pzq\leq 1+c_0r^{-1}$ on the whole of $\R\times Y_\infty$.
\end{pt3}
\begin{pt3}
The equations in \eqref{cond:instanton-27} have a second consequence which is this: Given $\upvarepsilon>0$, there exists $\mathrm{R}_\upvarepsilon>10$ such that the bounds that follow hold.
\begin{itemize}\leftskip-0.35in
\item $|1-|\alpha||<\upvarepsilon$ where $\Delta>\mathrm{R}_\upvarepsilon$.
\item $\displaystyle \int_{\{p\in I\times\R:\Delta(p)\geq \mathrm{R}_\upvarepsilon\} }(|\nabla_A\nabla_A\alpha|^2+|\nabla_A\nabla_A\beta|^2)<\upvarepsilon$ if $I\subset\R$ is an interval of length $1$.
\end{itemize}
\begin{equation}
\label{cond:instanton-211}
\end{equation}
The existence of $\mathrm{R}_\upvarepsilon$ for the top bullet in \eqref{cond:instanton-211} follows from \eqref{eq:instanton-210} using \eqref{cond:instanton-26} and Lemma \ref{lem:monopoles-16} with aforementioned Hardy's inequality. The existence of $\mathrm{R}_\upvarepsilon$ for the lower bullet is argued as follows: Write the two equations in \eqref{cond:instanton-27} as 
\begin{equation}
\label{eq:instanton-212}
{\nabla_A}^\dagger\nabla_A\alpha=\mathcal{C}_\alpha\quad\textrm{and}\quad {\nabla_A}^\dagger\nabla_A\beta=\mathcal{C}_\beta
\end{equation}
by moving terms from left hand side of \eqref{cond:instanton-27} to the right hand side. Fix a ball (to be denoted by B) as before, but where $\Delta>20$. Multiply both sides of both equations in \eqref{eq:instanton-212} by $\upchi_B$ and then integrate the square of the norm of what results over the ball. The $L^2$ norms of $\upchi_B\mathcal{C}_\alpha$ and $\upchi_B\mathcal{C}_\beta$ are bounded by $c_0$ times the $L^2$ norm over the ball of one of the following:
\begin{itemize}\leftskip-0.35in
\item $|\beta_+|$, $|\nabla_{A_+}\alpha_+|$, $|\nabla_{A_+}\beta_+|$, $|\frh_+|$, $r^{1/2}|\frh_+|^2|\nabla_{A_+}\frh_+|$,
\item $|\beta_-|$, $|\nabla_{A_-}\alpha_-|$, $|\nabla_{A_-}\beta_-|$, $|\frh_-|$, $r^{1/2}|\frh_-|^2|\nabla_{A_-}\frh_-|$,
\item $|\frh_0|$, $r^{1/2}|\frh_0|^2$, $|\nabla_{\Ai}\frh_0|$,
\end{itemize}
\begin{equation}
\label{cond:instanton-213}
\end{equation}
as the case may be. (Keep in mind what is said in Part 1.) Then because of \eqref{eq:instanton-212}, the $L^2$ norm over the ball of what appears in one of the bullets in \eqref{cond:instanton-213} times $c_0$ dominates the $L^2$ norms of $\upchi_B{\nabla_A}^\dagger\nabla_A\alpha$ and $\upchi_B{\nabla_A}^\dagger\nabla_A\beta$. Meanwhile some integration by parts and commuting of derivatives can be used to bound the $L^2$ norms of $\nabla_A\nabla_A\alpha$ and $\nabla_A\nabla_A\beta$ by those of $\upchi_B{\nabla_A}^\dagger\nabla_A\alpha$ and $\upchi_B{\nabla_A}^\dagger\nabla_A\beta$ plus $c_0$ times the $L^2$ norms of $\nabla_A\alpha$, $\nabla_A\beta$, and $F_A$ over the ball. The first two norms are also bounded by $c_0$ times the $L^2$ norms over the ball of what appears in one of the bullets in \eqref{cond:instanton-213}. Meanwhile, the $L^2$ norm of $F_A$ is bounded by these same $L^2$ norms (times $r^{1/2}$) plus (depending on the ball) $c_0$ times the $L^2$ norm over the ball of either $F_{A_+}$ or $F_{A_-}$. (Remember what is said in Part 1 about $F_A$.)

The just derived bound on the $L^2$ norm of $\nabla_A\nabla_A\alpha$ and $\nabla_A\nabla_A\beta$ over radius $\frac{1}{2}$ balls imply (with \eqref{cond:instanton-26} and Lemma \ref{lem:monopoles-16}) what is asserted by the lower bullet of \eqref{cond:instanton-211}. This is because the $\Delta>10$ part of $I\times Y_\infty$ has an open cover by balls of this sort with the property that at most $c_0$ elements in the cover contain a given point.
\end{pt3}
\begin{pt3}
Let $\mathrm{R}$ denote $\upvarepsilon=\frac{1}{100}$ version of $\mathrm{R}_\upvarepsilon$ from \eqref{cond:instanton-211}. Because $|\alpha|$ is greater than $1-\frac{1}{100}$ where $\Delta>\mathrm{R}$, there is an isomorphism between the line bundle $E$ and the product $\C$-bundle where $\Delta>\mathrm{R}$ that identifies $\alpha$ with a section of the product $\C$-bundle which can be written as $1-\pzz$ with $\pzz$ denoting an $\R$-valued function with norm less than $\frac{1}{100}$. (Note that $\pzz$ also has finite $L^2$ norm on the $\Delta>2\mathrm{R}$ part of $I\times Y_\infty$ because its norm at any given point is at most $c_0$ times that of either $\frh_+$, $\frh_-$ or $\frh_0$.) Use this same isomorphism to write the connection $A$ as $\Ai+r^{1/2}\fra$ with $\fra$ denoting an $i\R$ valued $1$-form. Using these identifications of $\alpha$ and $A$, then $\nabla_A\alpha$ appears as $-d\pzz+(1-\pzz)r^{1/2}\fra$. The key point now is this: Since $\pzz$ is real valued and $\fra$ is $i\R$ valued, the function $|\nabla_A\alpha|$ bounds ${c_0}^{-1}$ times both $|d\pzz|$ and $r^{1/2}|\fra|$. And, since \eqref{cond:instanton-26} and Lemma \ref{lem:monopoles-16} imply that $\nabla_A\alpha$ has square integrable norm where $\Delta>2\mathrm{R}$ on $I\times Y_\infty$, this is also the case for both $d\pzz$ and $\fra$.

It also follows from the lower bullet in \eqref{cond:instanton-211} that $|\nabla_A\alpha|$ is an ${L^2}_1$ function on the $\Delta>2\mathrm{R}$ part of $I\times Y_\infty$ whose ${L^2}_1$ norm is uniformly small where $\Delta$ is large (independent of the choice for the interval $I$). As a consequence of this and a dimension $4$ Sobolev inequality (the $L^4$ norm of an ${L^2}_1$ function on a ball is bounded by $c_0$ times the ${L^2}_1$ norm on the ball), both $|d\pzz|$ and $|\fra|$ have bounded $L^4$ norms on the $\Delta>2\mathrm{R}$ part of $I\times Y_\infty$; and these norms are uniformly small where $\Delta$ is large (independent of $I$). Because 
\begin{equation}
\label{eq:instanton-214}
\nabla_A\nabla_A\alpha=-\nabla d\pzz-(1-\pzz)r\fra\otimes\fra-r^{1/2}(d\pzz\otimes\fra+\fra\otimes d\pzz)+(1-\pzz)r^{1/2}\nabla\fra,
\end{equation}
these $L^4$ bounds imply in turn that both $\nabla d\pzz$ and $\nabla\fra$ have bounded $L^2$ norms on the $\Delta>2\mathrm{R}$ part of $I\times Y_\infty$ which are also uniformly small independent of $I$ where $\Delta$ is large. (Note in this regard that $\nabla d\pzz$ is $\R$ valued and $\nabla\fra$ is $i\R$ valued. As a consequence, the $L^2$ norm of $\nabla_A\nabla_A\alpha$ bounds the $L^2$ norms of both $\nabla d\pzz$ and $\nabla\fra$ when the $L^4$ norms of $d\pzz$ and $\fra$ are bounded.)

The isomorphism that writes $\alpha$ as $1-\pzz$ writes $\psi$ as $(1-\pzz,\beta_{\diamond})$ with $\beta_\diamond$ being the $E\otimes K^{-1}$ component. Let $\frh_\diamond$ denote the section of $iT^\ast(\R\times Y_\infty)\oplus\spb^+$ over the $\Delta>2\mathrm{R}$ part of $\R\times Y_\infty$ given by $(\fra,(-\pzz,\beta_\diamond))$. Since $\fra$ and $\nabla\fra$ have finite $L^2$ norms (and likewise $\pzz$ and $d\pzz$, and $\beta_\diamond$ and $\nabla_A\beta_\diamond$), this section $\frh_\diamond$ has finite ${L^2}_1$ norm over the $\Delta>2\mathrm{R}$ part of $I\times Y_\infty$ and its ${L^2}_1$ norm is uniformly small when $\Delta$ is large. This means that the isomorphism that writes $\alpha$ as $1-\pzz$ with real $\pzz$ (and thus $(A,\psi)$ as $(\Ai,\psii)+(r^{1/2}\fra,(-\pzz,\beta_\diamond))$) has the following property: Given $\upvarepsilon>0$, there exists $\mathrm{R}_\upvarepsilon>\mathrm{R}$ such that if $I\subset\R$ has length $1$, then
\begin{equation}
\label{eq:instanton-215}
\int\limits_{\{p\in I\times Y_\infty\;:\;\Delta(p)>\mathrm{R}_\upvarepsilon\}}(|\nabla_{\Ai}\frh_\diamond|^2+|\frh_\diamond|^2)\leq\upvarepsilon.
\end{equation}
This implies (among other things) that the aforementioned isomorphism that writes $\alpha$ as $1-\pzz$ can be used for the isomorphism in the third bullet of \eqref{cond:instanton-26}. (If it is used there, then $\frh$ plays the role of $\frh_\diamond$.)
\end{pt3}
\begin{pt3}
This part of the proof and the next prove that $\lim_{\Delta\to\infty}e^{\sqrt{r}\Delta/ c_0}|\frh_\diamond|=0$ which finishes the proof of the second bullet of Lemma \ref{lem:instanton-24}. The argument starts with the observation (about which more is said in Part 6) that $\frh_\diamond$ obeys a differential equation on the $\Delta>r\mathrm{R}$ part of $\R\times Y_\infty$ that has the form
\begin{equation}
\label{eq:instanton-216}
\Li\frh_\diamond+\pzee(\frh_\diamond)=0
\end{equation}
with $\Li$ being the $(\Ai,\psii)$ version of the operator that is defined in \eqref{eq:instanton-linearize}, and with $\pzee$ denoting a homomorphism from $iT^\ast(\R\times Y_T)\oplus\spb^+$ to $i(\Lambda^+\oplus\R)\oplus\spb^-$ that obeys this:
\begin{quote}
Given $\upvarepsilon>0$ and if $r$ is sufficiently large given $\upvarepsilon$, then there exists $\mathrm{L}_\upvarepsilon>\mathrm{R}$ such that $||\pzee(\frh)||_2\leq\upvarepsilon||\Li\frh||_2$ when $\frh$ has compact support on the $\Delta>\mathrm{L}_\upvarepsilon$ part of $\R\times Y_\infty$.
\begin{equation}
\label{cond:instanton-217}
\end{equation}
\end{quote}
Granted that $\frh_\diamond$ obeys \eqref{eq:instanton-216} with $\pzee$ as in \eqref{cond:instanton-217} (see Part 6 for why this is so), then Lemma \ref{lem:instanton-21} can be brought to bear to see that if $r\geq c_0$ and if $\mathrm{L}$ is sufficiently large then
\begin{equation}
\label{eq:instanton:218}
\int\limits_{\{p\in I\times Y_\infty\;:\;\mathrm{L}<\Delta(p)<\mathrm{L}+1\}}|\frh_\diamond|^2\leq e^{-\sqrt{r}\mathrm{L}/ c_0}
\end{equation}
when $I\subset\R$ is any length $1$ interval. This and \eqref{eq:instanton-li-l21-bound} imply in turn that 
\begin{equation}
\label{eq:instanton:219}
\int\limits_{\{p\in I\times Y_\infty\;:\;\mathrm{L}<\Delta(p)<\mathrm{L}+1\}}(|\nabla_{\Ai}\frh_\diamond|^2+|\frh_\diamond|^2)\leq e^{-\sqrt{r}\mathrm{L}/ c_0}
\end{equation}
for any length $1$ interval $I$. Given this last bound, and given the structure of $\pzee(\cdot)$ described in Part 6 (it is quadratic in the components of $\frh$), then standard elliptic bootstrapping arguments can be used to prove the pointwise bound $|\frh_\diamond|\leq c_\ast e^{-\sqrt{r}\Delta/c_0}$ with $c_\ast$ being constant.
\end{pt3}
\begin{pt3}
This part of the proof explains where \eqref{eq:instanton-216} and \eqref{cond:instanton-217} come from. The components in the $i\Lambda^+\oplus\spb^+$ summand of \eqref{eq:instanton-216} are a rewriting of the equations in \eqref{eq:instanton-sw} when $(A,\psi)$ is written as $(\Ai+r^{1/2}\fra,\psii+(-\pzz,\beta_\diamond))$. The $i\Lambda^+\oplus\spb^+$ part of $\pzee(\frh)$ is the image of $r^{1/2}\frh_\diamond\otimes\frh$ via a canonical homomorphism from $\otimes^2(iT^\ast(\R\times Y_T)\oplus\spb^+)$ to $i(\Lambda^+\oplus\R)\oplus\spb^-$. The image of $\frz\otimes\frw$ under this homomorphism is written as $\frz\#\frw$. Since it is linear separately with respect to $\frz$ and $\frw$, it obeys $|\frz\#\frw|\leq c_0|\frz||\frw|$ and it also obeys $|\nabi(\frz\#\frw)|\leq c_0(|\frw||\nabi\frz|+|\frz||\nabi\frw|)$.

To see about \eqref{cond:instanton-217} for this part of $\pzee(\frh)$, suppose that $\mathrm{L}$ is some large positive number and that $I\subset\R$ is an interval of length $1$. It follows from what was said in the preceding paragraph that the square of the $L^2$ norm of the $i\Lambda^+\oplus\spb^-$ part if $\pzee(\frh)$ on the $\Delta>\mathrm{L}$ part of $I\times Y_\infty$ obeys
\begin{equation}
\label{eq:instanton-220}
\int\limits_{\{p\in I\times Y_\infty\;:\;\Delta(p)>\mathrm{L}\}}|\pzee(\frh)|^2\leq c_0r(\int\limits_{\{p\in I\times Y_\infty\;:\;\Delta(p)>\mathrm{L}\}}|\frh_\diamond|^4)^{1/2}(\int\limits_{\{p\in I\times Y_\infty\;:\;\Delta(p)>\mathrm{L}\}}|\frh|^4)^{1/2}.
\end{equation}
Granted this, then the dimension $4$ Sobolev inequality bounding the $L^4$ norm by the ${L^2}_1$ norm can be used to bound the right hand side of \eqref{eq:instanton-220} by $c_0r$ times the product of the squares of the ${L^2}_1$ norms of $\frh_\diamond$ and $\frh$ on the $\Delta>\mathrm{L}$ part of $I\times Y_\infty$. Therefore, if $\upvarepsilon>0$ is given, and supposing that $\mathrm{L}$ is sufficiently large (depending on $\upvarepsilon$ but not on $\frh$ nor on $I$), then (because of \eqref{eq:instanton-215}) the right hand side of \eqref{eq:instanton-220} is at most $\upvarepsilon$ times the square of the ${L^2}_1$ norm of $\frh$ on the $\Delta>\mathrm{L}$ part of $I\times Y_\infty$. This has the following implication (by summing over the instances with $I$ having integer endpoints): If $\frh$ has compact support in the $\Delta>\mathrm{L}$ part of $\R\times Y_\infty$, then the $L^2$ norm of $\pzee(\frh)$ is at most $\upvarepsilon$ times the ${L^2}_1$ norm of $\frh$. This in turn is at most $c_0\upvarepsilon$ times the $L^2$ norm of $\Li\frh$ if $r>c_0$ (because of \eqref{eq:instanton-li-l21-bound}). This last observation implies that \eqref{eq:instanton-216} holds for the $i\Lambda^+\oplus\spb^-$ part of $\pzee(\frh)$.

To describe the $i\R$ summand of $\pzee(\frh)$, note first that the $i\R$ component of $\Li\frh_\diamond$ is $\ast d\ast\fra$ because $\pzz$ is real valued (look at the middle term in \eqref{eq:instanton-linearize}). Meanwhile, it follows from the top bullet of \eqref{cond:instanton-27} and from \eqref{eq:instanton-214} that $\ast d\ast\fra$ can be written as a sum of terms having the following form:
\begin{equation}
\label{eq:instanton-221}
\ast d\ast\fra = - \big(x_1(r^{-1/2}\nabi\beta_\diamond +\fra\otimes\beta_\diamond)+r^{-1/2}x_2\cdot\beta_\diamond+\pzq_1(d\pzz\otimes\fra)\big)
\end{equation}
with the notation as follows: $\{x_k\}_{k=1,2}$ and $\pzq_1$ denote certain canonical sections of vector bundles. In particular, these are bounded with bounded covariant derivatives. With \eqref{eq:instanton-221} understood, write a section $\frh$ of $iT^\ast(\R\times Y_\infty)\oplus\spb^+$ as $\frh=(\frb,(u_0,u_1))$. Now, define the $i\R$ summand of $\pzee(\frh)$ to be 
\begin{equation}
\label{eq:instanton-222}
-(x_1(r^{-1/2}\nabi u_1+\fra\otimes u_1)+r^{-1/2}x_2\cdot u_1+\pzq_1(d\pzz\otimes\frb)).
\end{equation}
The $i\R$ part of \eqref{cond:instanton-217} follows directly from this definition.

To see about \eqref{eq:instanton-216} for the $i\R$ part of $\pzee(\frh)$, suppose that $I\subset\R$ is an interval of length $1$ and that $\mathrm{L}$ is large. The key point is that the $L^2$ norm of what is written in \eqref{eq:instanton-222} on the $\Delta>\mathrm{L}$ part of $I\times Y_\infty$ is at most $c_0$ times the sum of 
\begin{itemize}\leftskip-0.35in
\item The product of the ${L^2}_1$ norm of $\frh$ on this domain and $r^{-1/2}$.
\item The product of the $L^4$ norms of $\fra$ and $\frh$ on this domain.
\item The product of the $L^4$ norms of $d\pzz$ and $\frh$ on this domain.
\end{itemize}
\begin{equation}
\label{cond:instanton-223}
\end{equation}
Now, given $\upvarepsilon>$, then what is described by the first bullet of \eqref{cond:instanton-223} is less than $\upvarepsilon$ times the ${L^2}_1$ norm of $\frh$ on the $\Delta>\mathrm{L}$ part of $I\times Y_\infty$ if $r>c_0\upvarepsilon^{-1}$. The argument used previously for $i\Lambda^+\oplus\spb^-$ component of $\pzee(\frh)$ can be repeated to see what is described by the middle bullet of \eqref{cond:instanton-223} is no greater than $\upvarepsilon$ times the ${L^2}_1$ norm of $\frh$ on the $\Delta>\mathrm{L}$ part of $I\times Y_\infty$ if $\mathrm{L}$ is sufficently large (`large' depending on $\upvarepsilon$ but not on $I$ or $\frh$). These same arguments can be used to draw the same conclusion with regards to what is described by the third bullet of \eqref{cond:instanton-223}. This is because the ${L^2}_1$ norm of $d\pzz$ on any given domain is bounded by $c_0$ times the sum of the $L^2$ norms of $|\nabla_A\nabla_A\alpha|$ and $|\nabla_A\alpha|$ and both are uniformly small on the $\Delta>\mathrm{L}$ part of $I\times Y_\infty$ when $\mathrm{L}$ is uniformly large (this follows from the second bullet of \eqref{cond:instanton-211} and from what is said in Part 1.)
\end{pt3}
\begin{pt3}
This part of the proof and Part 8 address the assertion in the first bullet of Lemma \ref{lem:instanton-24}. What follows considers only the case where $s\to\infty$. The other case is treated the same way but for some changes in the notation.

To start, suppose for the moment that there exists a gauge transformation as described by the top bullet of \eqref{cond:instanton-26} so that the resulting version of $\frh_+$ obeys 
\begin{equation}
\label{eq:instanton-224}
\ast d\ast(r^{-1/2}\ahat_+)+r^{1/2}(\psi^\dagger\upsigma_+-{\upsigma_+}^\dagger\psi)=0
\end{equation}
where $s\gg 1$. Assuming this, then the equations in \eqref{eq:instanton-sw} when written in terms of $\frh_+$ and the equation in \eqref{eq:instanton-224} can be written schematically (where $s\ll 1$) as
\begin{equation}
\label{eq:instanton-225}
\LL_+\frh_++r^{1/2}\frh_+\#\frh_+=0,
\end{equation}
where the notation is as follows: First, $\LL_+$ is the version of \eqref{eq:instanton-linearize} that is defined by viewing $(A_+,\psi_+)$ as an $s$-independent pair on $\R\times Y_T$. Second, what is denoted by $\frh_+\#\frh_+$ is the image of $\frh_+\otimes\frh_+$ under a certain canonical vector bundle homomorphism. (The notation in what follows uses $\frf_1\#\frf_2$ to denote the image of $\frf_1\otimes\frf_2$ under this homomorphism.) This vector bundle homomorphism obeys $|\frf_1\#\frf_2|\leq c_0|\frf_1||\frf_2|$; and there are also uniform bounds that hold for its covariant derivatives.

The key observation is the following: There exists $c>0$ such that if $\frh$ is an ${L^2}_1$ section of $iT^\ast(\R\times Y_T)\oplus\spb^+$, then 
\begin{equation}
\label{eq:instanton-226}
c\int\limits_{\R\times Y_T}(|\nabla_{A_+}\frh|^2+|\frh|^2)\leq \int\limits_{\R\times Y_T}|\LL_+\frh|^2.
\end{equation}
This follows from the assumption that the $(A_+,\psi_+)$ version of the operator $\LLf$ (which is depicted in \eqref{eq:linearized}) has trivial kernel and because the operator $\LL_+$ can be written as $\frac{\partial}{\partial s}+\LLf$. What follows directly is now a consequence of \eqref{eq:instanton-226}, the top bullet of \eqref{cond:instanton-26} and a dimension $4$ Sobolev inequality (the $L^4$ norm is bounded by $c_0$ times the ${L^2}_1$ norm): There exists $\mathrm{R}_\ast>1$ such that if $\mathrm{R}\geq\mathrm{R}_\ast$ and if $\frh$ is an ${L^2}_1$ section of $iT^\ast(\R\times Y_T)\oplus\spb^+$ with support where $s>\mathrm{R}$, then
\begin{equation}
\label{eq:instanton-227}
\frac{1}{2}c\int\limits_{\R\times Y_T}(|\nabla_{A_+}\frh|^2+|\frh|^2)\leq \int\limits_{\R\times Y_T}|\LL_+\frh_++r^{1/2}\frh_+\#\frh|^2.
\end{equation} 
Let $\en_k$ denote the integral of $|\nabla_{A_+}\frh_+|^2+|\frh_+|^2$ over the $s\geq \mathrm{R}+k$ part of $\R\times Y_T$. The number $\mathrm{R}$ is chosen in particular so that $\en_k\ll r^{-1}c$. This is made precise below.

With \eqref{eq:instanton-227} in mind, fix a positive integer $k$ for the moment, and let $\upchi$ now denote a smooth function of $s$ which is zero for $s\leq0$ and $1$ for $s\geq 1$. Having fixed $\upchi$ and supposing that $k$ is a positive integer, let $\upchi_k$ denote the function $s\mapsto\upchi(s-k-\mathrm{R})$. This is equal to $1$ where $s\geq \mathrm{R}+k+1$ and zero where $s\leq\mathrm{R}+k$. Set $\frh_k=\upchi_k\frh_+$. It is a consequence of \eqref{eq:instanton-225} that
\begin{equation}
\label{eq:instanton-228}
|\LL_+\frh_k+r^{1/2}\frh_+\#\frh|\leq c_0\theta_k(|\frh_+|+r^{1/2}|\frh_+|^2)
\end{equation}
where $\theta_k$ is the characteristic function of the set in $\R$ where $s$ is between $\mathrm{R}+k$ and $\mathrm{R}+k+1$. It now follows from \eqref{eq:instanton-228}, \eqref{eq:instanton-227}, and the aforementioned Sobolev inequality that if $\mathrm{R}$ is sufficiently large, then 
\begin{equation}
\label{eq:instanton-229}
\en_{k+1}\leq c_\ast\int\limits_{[k+\mathrm{R},k+1+\mathrm{R}]\times Y_T}(|\nabla_{A_+}\frh_+|^2+|\frh_+|^2)
\end{equation}
with $c_\ast>1$ being determined by $(A_+,\psi_+)$ (in particular, it is independent of $k$). Because the right hand side of this is $c_\ast(\en_k-\en_{k+1})$, this implies in return that
\begin{equation}
\label{eq:instanton-230}
\en_{k+1}\leq \upzeta\en_k,
\end{equation}
with $\upzeta$ being a $k$-independent number obeying $0<\upzeta <1$. In particular, by virtue of $\upzeta$ being less than $1$, the bound in \eqref{eq:instanton-230} implies that $\en_{k+1}\leq e^{-k/c_{\ast\ast}}\en_1$ with $c_{\ast\ast}>1$ being determined by $(A_+\psi_+)$. This is to say that the ${L^2}_1$ norm of $\frh_+$ on the $s\geq \mathrm{R}+k$ part of $\R\times Y_T$ has exponential decay to zero as $k\to\infty$. granted this fact, then standard elliptic regularity arguments can be used to prove the pointwise exponential decay that is claimed by the lemma. (Keep in mind that \eqref{eq:instanton-225} is an elliptic equation.)
\end{pt3}
\begin{pt3}
This part of the proof explains why there is a gauge transformation on the large $s$ part of $\R\times Y_T$ that changes $(A,\psi)$ to $(A+\ahat_+,\psi_++\upsigma_+)$ with $(\ahat_+,\upsigma_+)$ obeying \eqref{eq:instanton-224} and with $\frh_+=(r^{-1}\ahat_+,\upsigma_+)$ as described in the top bullet of \eqref{cond:instanton-26}. Note in this regard that one can assume at the outset that there is a gauge transformation that changes $(A,\psi)$ to $(A_++\ahat,\psi_++\upeta)$ with $\frh=(r^{-1}\ahat,\upeta)$ obeying
\begin{equation}
\label{eq:instanton-231}
\int\limits_{[1,\infty)\times Y_T}(|\nabla_{A_+}\frh|^2+|\frh|^2)<\infty.
\end{equation}
Now suppose that $\rmx$ is a real valued function on $[1,\infty)\times Y_T$. Then $e^{i\rmx}$ is a map from this domain of $\rmx$ to $S^1$. Viewed as a gauge transformation, it changes $(A_++\ahat,\psi_++\upeta)$ to $(A_++\ahat-id\rmx,\psi_++(e^{i\rmx}-1)\psi_++e^{i\rmx}\upeta)$. This understood, the plan for what follows is to find a function $\rmx$ so that $\ahat_+=\ahat-id\rmx$ and $\upsigma_+=(e^{i\rmx}-1)\psi_++e^{i\rmx}\upeta$ obey \eqref{eq:instanton-224} and so that the corresponding $\frh_+$ obeys the finite integral constraint in the top bullet of \eqref{cond:instanton-26}. Note in this regard that \eqref{eq:instanton-224} is obeyed if $\rmx$ obeys the equation
\begin{eqnarray}
\label{eq:instanton-232}
\nonumber d^\dagger d\rmx+2r|\psi_+|^2\rmx-i(\ast d\ast\ahat+r({\psi_+}^\dagger\upeta-\upeta^\dagger\psi_+))&&\\
\nonumber&&\hspace{-1.5in}-i(e^{i\rmx}-1-i\rmx)r|\psi_+|^2-ir((e^{i\rmx}-1){\psi_+}^\dagger\upeta-(e^{-i\rmx}-1)\upeta^\dagger\psi_+)=0.\\
\end{eqnarray}
This is a non-linear, inhomogeneous equation for $\rmx$. The term that is independent of $\rmx$ is $-i(\ast d\ast\ahat+r({\psi_+}^\dagger\upeta-\upeta^\dagger\psi_+))$ which is square integrable on $[1,\infty)\times Y_T$ because of the assumption that \eqref{eq:instanton-231} holds. The term linear in $\rmx$ is $d^\dagger d\rmx+2r|\psi_+|^2\rmx$. Supposing that $\mathrm{R}>1$ has been specified, then the operator that defines this term, $d^\dagger d+2r|\psi_+|^2$, is an isomorphism from the ${L^2}_2$ Sobolev space of functions on $[R,\infty)\times Y_T$ that vanish at $s=\mathrm{R}$ to the $L^2$ Sobolev space of functions on $[\mathrm{R},\infty)\times Y_T$. This is because $|\psi_+|$ is non-zero somewhere, and in the case $T=\infty$, its norm limits to $1$ as $\Delta\to\infty$. Meanwhile, the non-linear term in \eqref{eq:instanton-232}, which is $-i(e^{i\rmx}-1-i\rmx)r|\psi_+|^2-ir((e^{i\rmx}-1){\psi_+}^\dagger\upeta-(e^{-i\rmx}-1)\upeta^\dagger\psi_+)$, is pointwise bounded by $c_0r(|\rmx|^2+|\rmx||\upeta|)$; and its derivative is pointwise bounded by $c_0r((|\rmx|+|\upeta|)|d\rmx|+|\rmx||\nabla_{A_+}\upeta|)$.

Granted the preceding bounds, then a contraction mapping argument (or an inverse function theorem argument) can be used to prove that there exists $\mathrm{R}>1$ so that \eqref{eq:instanton-232} has a unique solution on the $s\geq\mathrm{R}$ part of $\R\times Y_T$ that vanishes at $\rmx=\mathrm{R}$ and with ${L^2}_2$ norm bounded by $c_0$ times the integral of $|\nabla_{A_+}\frh|^2+|\frh|^2$ over the $s\geq\mathrm{R}$ part of $\R\times Y_T$. The details of setting this up are straightforward and left to the reader. \qedhere
\end{pt3}
\end{proof}

%% file: a_instanton-bounds.tex

This section gives the analogs for $\R\times Y_T$ of the bounds that are supplied by \cite[Section 3]{taubes4}.

\subsubsection{The action functional}
\label{sssec:action-functional}

This subsection is a digression of sorts to define a functional on the space of admissible pairs of connection on $Y_T$ and section of $\spb$.  To this end, let $a$ denote for the moment a given $\R$-valued 1-form on $Y_T$. Supposing that $\frc = (A,\psi)$ and $\frc' = (A',\psi')$ are pairs of connection on $E$ and section of $\spb$ (which are admissible if $T=\infty$), define
\begin{equation}
\label{eq:bounds-31}
\fra(\frc,\frc') = -\frac{1}{2} \int_{Y_T} (A-A')\wedge *(B_A + B_{A'}+B_{A_0}) - \frac{i}{2}r\int_{Y_T}(A-A')\wedge\ast a + r\int_{Y_T} \langle \psi, \textup{D}_A\psi\rangle - r\int_{Y_T} \langle \psi', \textup{D}_{A'}\psi'\rangle.
\end{equation}
Here $B_A$ denotes the Hodge-star dual of the curvature form $F_A$, and likewise $B_{A'}$ and $B_{A_0}$ denote the Hodge stars of the curvature 2-forms of $A'$ and $A_0$.  (Remember that $A_0$ is a chosen connection on $Y_T$ whose curvature is zero on $Y_T \ssm N$.) Note that $\fra(\frc,\frc')=-\fra(\frc',\frc)$ and that if $\frc$, $\frc'$, and $\frc''$ are three admissible pairs on $Y_T$, then $\fra(\frc,\frc'')=\fra(\frc,\frc')+\fra(\frc',\frc'')$. In any event, it is traditional to keep $\frc'$ fixed and view $\fra$ as a function of the pair $\frc$. Viewed in this way, the function $\fra$ is called the \emph{action function}. 

Supposing that $T=\infty$, it is necessary to impose some conditions on $a$ and the Riemannian metric to ensure that $\fra(\cdot,\frc')$ is finite. In the cases of interest, it is sufficient to assume that $|a|$ is bounded and that there exists $c>1$ such that if $R>1$, then the part of $Y_\infty$ where the distance to $M$ is less than $R$ has volume at most $ce^{cR}$.

Now suppose that $\frd = (A,\psi)$ is an instanton solution to \eqref{eq:instanton-sw} on $\R\times Y_T$. Introduce by way of notation $\frc_+=(A_+,\psi_+)$ and $\frc_-=(A_-,\psi_-)$ to denote the $s\to\infty$ and $s\to-\infty$ limits of $(A,\psi)$, these being solutions on $Y_T$ to the versions of \eqref{eq:sw} that are defined by $a_+$ and $a_-$. Given $s\in\R$, use $\frd|_s$ to denote the restriction of $(A,\psi)$ to $\{s\}\in Y_T$, this being a pair of connection on $E$ over $Y_T$ and section of $\spb$ over $Y_T$. Use $\fra(\frd|_s,\frc_+)$ to denote the $\frc'=\frc_+$ version of \eqref{eq:bounds-31} with $a$ evaluated at the indicated value of the $s$ coordinate. By the same token, $\fra_-(\frc_-,\frc_+)$ denotes the $\frc'=\frc_+$ version of \eqref{eq:bounds-31} with $a$ being the contact form $a_-$.

To continue with notation, introduce $i\R$-valued 1-forms $E_A$ and $B_A$ along the $Y_T$ factor of $\R\times Y_T$ by writing the curvature 2-form $F_A$ as $ds\wedge E_A + *B_A$ with $*$ here and below denoting the Hodge star operator that is defined along the $Y_T$ factor of $\R\times Y_T$ at any given $s\in\R$ by the metric $\frg$. With $B_A$ understood, define $\calB$ by the rule
\begin{equation}
\label{eq:bounds-32}
\calB = B_A - r(\psi^\dagger \tau\psi - i\hat{a}) - \frac{1}{2}B_{A_0},
\end{equation}
with $\hat{a}$ denoting $a+ \frac{1}{2}\frac{\partial}{\partial s} a$. The 1-form $\calB$ is viewed below as an $i\R$-valued 1-form along the $Y_T$ factor of $\R\times Y_T$. (Note that if $a$ is independent of $s$, then $\calB=0$ is the leftmost equation in \eqref{eq:sw}.)

\begin{lemma}
\label{lem:bounds-31}
There exists $\kappa>1$ which works for any $T\in (16,\infty]$ with the following significance: Fix $r>\kappa$ and suppose that $\frd = (A,\psi)$ is an instanton solution to \eqref{eq:instanton-sw}.
\begin{itemize}\leftskip-0.25in
\item Supposing that $s\in (-\infty,-s_0]$, let $I=(-\infty,s]$. Then
\[ \fra_-(\frc_-\frc_+) - \fra(\frd|_s,\frc_+) = \frac{1}{2}\int_{I\times Y_T} \left(|E_A|^2 + |\calB|^2 + 2r\left(|\nabla_{A,s}\psi|^2 + |\textup{D}_A\psi|^2\right)\right). \]

\item Supposing that $s\in [s_0,\infty)$, let $I=[s,\infty)$. Then
\[ \fra(\frd|_s,\frc_+) = \frac{1}{2}\int_{I\times Y_T} \left(|E_A|^2 + |\calB|^2 + 2r\left(|\nabla_{A,s}\psi|^2 + |\textup{D}_A\psi|^2\right)\right). \]

\item Supposing that $s_- < s_+$ are in $\R$, let $I=[s_-,s_+]$. Then $\fra(\frd|_{s_-},\frc_+) - \fra(\frd|_{s_+},\frc_+)$ is equal to
\begin{multline}
\frac{1}{2}\int_{I\times Y_T} \left(|E_A|^2 + |\calB|^2 + 2r\left(|\nabla_{A,s}\psi|^2 + |\textup{D}_A\psi|^2\right)\right) \\ + ir\int_{I\times Y_T} \left((B_A-B_{A_+}) \wedge *\frac{\partial}{\partial s} a\right)
+ r\int_{I\times Y_T}\langle \psi,\mathcal{R}\psi\rangle - \frac{1}{4}\int_{I\times M} |\frp|^2
\end{multline}
with $\mathcal{R}$ being an endomorphism of $\spb$ that has compact support in $[-s_0,s_0]\times Y_T$, is zero where the metric $\frg$ is independent of $s$, and has norm bounded by $\kappa$.
\end{itemize}
\end{lemma}
Henceforth, let $\fra_s(\frd|_s)$ denote $\fra(\frd|_s,\frc_+)$ and $\fra_-(\frc_-)=\fra_-(\frc_-,\frc_+)$. Note in this regard that the rightmost term in \eqref{eq:bounds-31} is zero because $\frc_+$ is a solution to \eqref{eq:sw} on $Y_T$.
\begin{proof}
Fix numbers $s_- < s_+$ and let $I=[s_-,s_+]$. The rightmost equation in \eqref{eq:instanton-sw} can be written as
\begin{equation}
\label{eq:bounds-33}
\nabla_{A,s}\psi + \textup{D}_A\psi = 0
\end{equation}
with $\nabla_{A,s}$ denoting the covariant derivative along the $\R$ factor of $\R\times Y_T$ and with $\textup{D}_A$ denoting the operator that appears in \eqref{eq:sw}. Taking the $L^2$ norm of this on $I\times Y_T$ gives an equation that reads
\begin{equation}
\label{eq:bounds-34}
\int_{I\times Y_T} \left(|\nabla_{A,s}\psi|^2 + |\textup{D}_A\psi|^2\right) = -2\int_{I\times Y_T} \langle \nabla_{A,s}\psi, \textup{D}_A\psi\rangle.
\end{equation}
Integration by parts writes the right hand side of \eqref{eq:bounds-34} as
\begin{equation}
\label{eq:bounds-35}
\int_{\{s_-\}\times Y_T} \langle \psi,\textup{D}_A\psi\rangle - \int_{\{s_+\}\times Y_T} \langle \psi,\textup{D}_A\psi\rangle - \int_{I\times Y_T} \langle E_A,\psi^\dagger \tau\psi\rangle + \int_{I\times Y_T} \langle \psi,\mathcal{R}\psi\rangle.
\end{equation}
The notation here has $\langle\cdot,\cdot\rangle$ denoting either the inner product on $\spb$ or, in the rightmost integral, $(-1)$ times the Riemannian inner product.  Lemma \ref{lem:instanton-24} can be used to see that the integrals in \eqref{eq:bounds-35} are convergent when $T=\infty$. What is denoted by $\mathcal{R}$ in \eqref{eq:bounds-35} is an endomorphism of $\spb$ that is defined from the derivatives of $\frg$ along the $\R$ factor of $\R\times Y_T$. In particular, $\mathcal{R}$ is zero where $g$ is independent of $s\in\R$ and its norm is bounded by $c_0$.

Meanwhile, the leftmost equation in \eqref{eq:instanton-sw} says that $E_A + \calB = \frac{1}{2}\frp$, which implies in turn that
\begin{equation}
\label{eq:bounds-36}
\tfrac{1}{2}\left(|E_A|^2 + |\calB|^2\right) = -\langle E_A,B_A+\tfrac{1}{2}B_{A_0}\rangle + r\langle E_A,\psi^\dagger\tau\psi\rangle - i\langle E_A,r\hat{a}\rangle + \tfrac{1}{4}|\frp|^2.
\end{equation}
The notation here has the inner products on the right hand side being $(-1)$ times the Riemannian inner products. Integrate \eqref{eq:bounds-36} on $I\times Y_T$ and add the result to $r$ times the integral identity in \eqref{eq:bounds-34}. Use $\mathcal{I}$ to denote the resulting sum of integral identities.

Suppose first that $a$ and the metric $\frg$ are independent of $s$ on $I\times Y_T$. Keeping in mind that $E_A$ can be written as $\frac{\partial}{\partial s}(A-A_+)$, an integration by parts and \eqref{eq:bounds-35} leads from $\mathcal{I}$ to the identities in the first and second bullets of Lemma \ref{lem:bounds-31} because
\begin{equation}
\label{eq:bounds-37}
\fra(\frd|_{s=s_-}) - \fra(\frd|_{s=s_+}) = \frac{1}{2}\int_{I\times Y_T} \left(|E_A|^2 + |\calB|^2 + 2r\left(|\nabla_{A,s}\psi|^2 + |\textup{D}_A\psi|^2\right)\right).
\end{equation}
Take $s_+ < -s_0$ and then send $s_-$ to $-\infty$ to obtain the first bullet of Lemma \ref{lem:bounds-31}. Take $s_- > s_0$ and send $s_+$ to $\infty$ to obtain the second bullet of Lemma \ref{lem:bounds-31}.

The term with $\mathcal{R}$ in the third bullet of Lemma \ref{lem:bounds-31} comes from \eqref{eq:bounds-35}. The term with $\frac{\partial}{\partial s}a$ in the third bullet of the lemma comes from the $-2i\langle E_A,r\hat{a}\rangle$ term in \eqref{eq:bounds-36}. To say more about this, write $\hat{a}$ in \eqref{eq:bounds-36} as $a+\frac{1}{2}\frac{\partial}{\partial s}a$. Recalling that the metric $g$ on $Y$ is such that $*da=2a+\frac{\partial}{\partial s}a$, it follows that
\begin{equation}
\label{eq:bounds-38}
-ir\int_{\{s\}\times Y_T} \langle E_A,\hat{a}\rangle = -ir\int_{\{s\}\times Y_T} (E_A \wedge da).
\end{equation}
An integration by parts writes the right hand side of \eqref{eq:bounds-38} as
\begin{equation}
\label{eq:bounds-39}
-ir\int_{\{s\}\times Y_T} \left(\frac{\partial}{\partial s} B_A \wedge \ast a\right).
\end{equation}
This uses the identity $dE_A = \frac{\partial}{\partial s} B_A$. Lemma \ref{lem:instanton-24} can be used to prove that this equation holds when $T=\infty$. A second integration by parts writes \eqref{eq:bounds-39} as
\begin{eqnarray}
\label{eq:bounds-310}
\nonumber-ir\int_{\{s_+\}\times Y_T} \left((B_A-B_{A_+})\wedge \ast a\right) &+& ir\int_{\{s_-\}\times Y_T} \left((B_A-B_{A_+})\wedge \ast a\right) \\ &+& ir\int_{I\times Y_T} \left((B_A-B_{A_+}) \wedge \ast\frac{\partial}{\partial s}a\right).
\end{eqnarray}
The boundary terms in this equation with the boundary terms in \eqref{eq:bounds-35} and those from the integral of $-\langle E_A,B_A+\frac{1}{2}B_{A_0}\rangle$ in \eqref{eq:bounds-36} give the identity in the third bullet of the lemma.
\end{proof}

\subsubsection{$L^2$ bounds on cylinders}
\label{sssec:l2-bounds}

Various local bounds for instanton solutions that are proved in Section 7.3 of \cite{ht2} are needed for what is to come. The first is the analog of Lemma 7.3 in \cite{ht2} which states the following:

\begin{lemma}
\label{lem:bounds-32}
There exists $\kappa>1$ which works for any $T\in (16,\infty]$ such that if $r>\kappa$ and $(A,\psi=(\alpha,\beta))$ is an instanton solution to \eqref{eq:instanton-sw}, then
\begin{itemize}\leftskip-0.25in
\item $|\alpha| \leq 1 + \kappa r^{-1}$.
\item $|\beta|^2 \leq \kappa r^{-1} \left(1-|\alpha|^2\right) + \kappa^2 r^{-2}$.
\end{itemize}
\end{lemma}
\noindent The proof is much the same as that of Lemma 3.1 in \cite{taubes4}. Note that these bounds are already obeyed as $|s|\to\infty$ and as $|t|\to\infty$ in the case of $Y_\infty$ by virtue of Lemma \ref{lem:instanton-24}.

To state the second of the required local bounds, suppose that $\frd = (A,\psi)$ is an instanton solution to \eqref{eq:instanton-sw}. Let $\mathcal{A}_\frd$ denote $\fra_-(\frc_-)$, this being the version of \eqref{eq:bounds-31} with $\frc=\frc_-$, $\frc'=\frc_+$, and $a=a_-$.

\begin{lemma}
\label{lem:bounds-33}
There exists $\kappa>1$ which works for any $T\in (16,\infty]$ with the following significance: Fix $r>\kappa$ and suppose that $\frd=(A,\psi)$ is an instanton solution to \eqref{eq:instanton-sw} with $\mathcal{A}_\frd \leq r^2$. Let $I$ denote a given interval in $\R$ of length $2$ and let $\rmu\subset Y_T$ denote a given open set with compact closure. Then
\[ \int_{I\times \rmu} \left(|F_A|^2 + r|\nabla_A\psi|^2\right) \leq \kappa r^2 (1+\mathrm{vol}(\rmu)). \]
\end{lemma}

\begin{proof}
Suppose that $\mathcal{A}_\frd \leq r^2$. Use this assumption with Lemmas \ref{lem:bounds-31} and \ref{lem:bounds-32} to see that
\begin{equation}
\label{eq:bounds-311}
\int_{I\times Y_T} \left(|E_A|^2 + |\calB|^2 + 2r(|\nabla_{A,s}\psi|^2 + |\textup{D}_A\psi|^2)\right) \leq c_0 r^2.
\end{equation}
Now let $\sigma_\rmu$ denote a smooth non-negative function on $Y_T$ with value $1$ on $\rmu$ and with compact support otherwise. Letting $\nabla_A^\perp \psi$ denote the covariant derivative of $\psi$ along the $Y_T$ factor of $I\times Y_T$, integrate by parts on each constant $s\in I$ slice of $I\times Y_T$ and use the Bochner-Weitzenb\"ock formula for the corresponding version of the operator ${\textup{D}_A}^2$ to write the integral of $\sigma_\rmu^2 |\mathcal{D}_A\psi|^2$ as an integral of $\sigma_\rmu^2|\nabla_A^\perp\psi|^2$ plus a curvature term plus a term that is bounded by $c_0$ times the integral of $(\sigma_\rmu|\nabla^{\otimes 2}\sigma_\rmu| + |\nabla \sigma_\rmu|^2)|\psi|^2$. The curvature term is bounded by the integral over $I\times Y_T$ of $c_0\sigma_\rmu^2(1+|B_A|)|\psi|^2$.  Meanwhile, by virtue of the triangle inequality, the integral of this term is no larger than $c_0r\sigma_\rmu^2 + \frac{1}{2000r}\sigma_\rmu^2|B_A|^2$ (because $|\psi|^2$ is bounded by $2$ when $r>c_0$ courtesy of Lemma~\ref{lem:bounds-32}).  Therefore, the integral of $\sigma_\rmu^22r|\textup{D}_A\psi|^2$ over $I\times Y_T$ is equal to that of $\sigma_\rmu^2 2r|\nabla_A^\perp\psi|^2$ plus an error term that is a priori bounded by $c_0r^2\operatorname{volume}(\rmu)$ plus the integral over $I\times Y_T$ of $\frac{1}{1000}\sigma_\rmu^2|B_A|^2$.  Meanwhile,
\begin{equation}
\label{eq:bounds-312}
|\calB|^2 \geq \frac{1}{2}|B_A|^2 - c_0 r^2,
\end{equation}
courtesy again of Lemma \ref{lem:bounds-32}. These bounds lead directly to the bound asserted by Lemma \ref{lem:bounds-33} when $\mathcal{A}_\frd \leq r^2$.
\end{proof}

\subsubsection{Pointwise bounds}
\label{sssec:pointwise-bounds}

The lemmas that follow in this subsection assert analogs of the various pointwise bounds for $F_A$ and $\nabla_A\psi$ that are proved in \cite[Sections 3a-c]{taubes4}. To set the stage for this, fix $T\in(16,\infty]$ for the moment and let $\frd = (A,\psi)$ denote an instanton solution to \eqref{eq:instanton-sw} on $\R\times Y_T$. The upcoming lemmas refer to a number to be denoted by $c_\frd$ that is defined so that if $\rmu\subset Y_T$ is an open set with compact closure and $I\subset \R$ is an interval of length $2$, then 
\begin{equation}
\label{eq:bounds-314}
\int_{I\times \rmu} \left(|F_A|^2 + r|\nabla_A\psi|^2\right) \leq c_\frd r^2 (1+\mathrm{vol}(\rmu)).
\end{equation}
Lemma \ref{lem:bounds-33} gives a priori bounds for $c_\frd$ when the condition $\mathcal{A}_\frd < r^2$ holds. Therefore, the bound in \eqref{eq:bounds-314} with Lemma \ref{lem:bounds-33} is an a priori bound that holds with the \emph{same} version of $c_\frd$ for any instanton solution $(A,\psi)$ on $\R\times Y_T$ (supposing that $\mathcal{A}_\frd<r^2$) and for any $T\in[16,\infty)$. The upcoming lemmas \ref{lem:bounds-34}--\ref{lem:bounds-37} give detailed bounds for $\alpha$, $\beta$, their covariant derivatives, and also $F_A$ that depend on an upper bound for $c_\frd$ in \eqref{eq:bounds-314}. Thus, if the condition $\mathcal{A}_\frd<r^2$ holds, then (by virtue of Lemma \ref{lem:bounds-33}) their detailed bounds are a priori bounds that hold for any instanton solution $(A,\psi)$ on $\R\times Y_T$ and for any $T\in[16,\infty)$. An a priori bound for $c_\frd$ is also needed to make effective use of the upcoming Proposition~\ref{prop:conv-41}.

The first of the bounds concerns the norms of $F_A$ and the components of $\nabla_A \psi$. The following lemma gives a priori bounds that are the analog of those in Lemmas 3.2 and 3.6 of \cite{taubes4}.

\begin{lemma}
\label{lem:bounds-34}
Given $c>1$, there exists $\kappa>1$ with the following significance: Suppose that $T\in (16,\infty]$, that $r \geq \kappa$ and that $\frd = (A,\psi=(\alpha,\beta))$ is an instanton solution to \eqref{eq:instanton-sw} on $\R\times Y_T$ such that \eqref{eq:bounds-314} holds with $c_\frd < c$. Then
\begin{itemize}\leftskip-0.25in
\item $|F_A| \leq \kappa r$.
\item $|\nabla_A \alpha|^2 \leq \kappa r$.
\item $|\nabla_A \beta|^2 \leq \kappa$.
\end{itemize}
In addition, supposing that $q$ is a positive integer, there exists a constant $\kappa_q \geq 1$ such that if $r>\kappa$, then $|\nabla_A^{\otimes q}\alpha|^2 + r|\nabla_A^{\otimes q}\beta|^2 \leq \kappa_q r^q$.
\end{lemma}

The proof of the bound for $|F_A|$ differs only in notation from the proof of Lemma 3.2 in \cite{taubes4}. The roles played in the latter proof by Lemmas 3.1 and 3.3 of \cite{taubes4} are played here by Lemma \ref{lem:bounds-32} and \eqref{eq:bounds-314}. The proof of the bounds for the components of the covariant derivatives of $\psi$ is almost word for word the same as the proof of Lemma 3.6 in \cite{taubes4}.

The next lemma refers to a function on $\R$ that is denoted by $\scm$ and it is defined by the rule whereby
\begin{equation}
\label{eq:bounds-315}
\scm(s) = \int_{[s,s+1]\times Y_T} \left|1-|\alpha|^2\right|.
\end{equation}
Lemma \ref{lem:instanton-24} guarantees that $\scm(\cdot)$ is finite. The following lemma is the analog here of Lemmas 3.8 and 3.9 in \cite{taubes4}.

\begin{lemma}
\label{lem:bounds-35}
Given $c>1$ and $\mathcal{K}\geq1$, there exists $\kappa>1$ with the following significance: Suppose that $T\in (16,\infty]$, that $r \geq \kappa$ and that $\frd = (A,\psi=(\alpha,\beta))$ is an instanton solution to \eqref{eq:instanton-sw} on $\R\times Y_T$ such that \eqref{eq:bounds-314} holds with $c_\frd < c$.   Fix a point $s_0\in\R$ and $R\geq 2$ and suppose that $\scm(\cdot)\leq \mathcal{K}$ on $[s_0-R-3, s_0+R+3]$. Let $X_*$ denote the subset of points where $1-|\alpha| \geq \kappa^{-1}$. The bounds that are stated below hold on the domain $[s_0-R,s_0+R] \times Y_T$:
\begin{itemize}\leftskip-0.25in
\item $|\nabla_A\alpha|^2 + r|\nabla_A\beta|^2 \leq \kappa r(1-|\alpha|^2) + \kappa^2$.
\item $(1-|\alpha|^2) \leq \kappa (r^{-1} + e^{-\sqrt{r}\operatorname{dist}(\cdot,X_*)/\kappa})$.
\item $|\nabla_A\alpha|^2 + r|\nabla_A\beta|^2 \leq \kappa(r^{-1} + re^{-\sqrt{r}\operatorname{dist}(\cdot,X_*)/\kappa})$.
\item $|\beta|^2 \leq \kappa(r^{-2} + r^{-1}e^{-\sqrt{r}\operatorname{dist}(\cdot, X_*)/\kappa})$.
\end{itemize}
In addition,
\begin{itemize}\leftskip-0.25in
\item $|\frac{\partial}{\partial s}A + B_A| \leq r(1-|\alpha|^2) + \kappa$.
\item $|\frac{\partial}{\partial s}A - B_A| \leq r(1-|\alpha|^2) + \kappa$.
\end{itemize}
\end{lemma}

\begin{proof}
The proof of the first bullet in the top set of four bullets differs only cosmetically from the proof of the first bullet in Lemma 3.8 of \cite{taubes4}.  The second bullet in Lemma 3.8 as stated in the published version \cite{taubes4} is not correct.  The argument in \cite{taubes4} does however prove the following:
\begin{equation}
\label{eq:bounds-new-316}
r(1-|\alpha|^2) + |\nabla_A\alpha|^2 + r|\nabla_A\beta|^2 \leq \kappa\left(1+re^{-\sqrt{r}\operatorname{dist}(\cdot,X_*)/\kappa}\right).
\end{equation}
This implies what is said by the second bullet from the top set of four. The argument for the $|\nabla_A\alpha|^2$ bound in the third bullet from the top set of four is given directly in seven steps. The eighth step of what follows derives the bounds for $|\beta|^2$ and $|\nabla_A\beta|^2$ that are asserted by the third and fourth bullets from the top set of four. Meanwhile, the proof of the last two bullets of the lemma (these are the bounds for $|\frac{\partial}{\partial s}A \pm B_A|$) differs only slightly from the proof of Lemma 3.9 in \cite{taubes4} (the latter proof does not use Lemma 3.8 in \cite{taubes4}).

\begin{sp6}
Let $\kappa_*$ denote the version of the number $\kappa$ that appears in the first and second bullets of the lemma. Then, the bound that is asserted by the third bullet holds (with $\kappa=c_0\kappa_*$) where the distance to $X_*$ is less than $100\kappa_*\sqrt{r}$ because of the first two bullets.
\end{sp6}

\begin{sp6}
According to the second bullet, $|\alpha| \geq \frac{99}{100}$ where the distance to $X_*$ is greater than $c_0\kappa_*\sqrt{r}$ (assuming $r>c_0$). As a consequence, the bundle $E$ on this part of $\R \times Y_T$ is isomorphic to the product $\C$ bundle via the isomorphism that takes $\alpha$ to a real number that can be written as $1-z$ with $z$ obeying $|z| \leq \kappa(r^{-1} + e^{-\sqrt{r}\operatorname{dist}(\cdot,X_*)/\kappa})$. This same isomorphism identifies $\beta$ with a section of $K^{-1}$ to be denoted by $\beta_\diamond$; and it identifies $A$ with a connection that is written as $\Ai + \aA$.

This same isomorphism identifies $\nabla_A\alpha$ with a section of $T^*(\R\times Y_T)$ that can be written as $-dz+\aA(1-z)$. Thus, bounds for $|dz|$ and $|\aA|$ lead to bounds for $|\nabla_A\alpha|$ (and vice versa because $z$ is real and $\aA$ is $i\R$-valued). Thus, $c_0r^{-1/2}(1+re^{-\sqrt{r}\operatorname{dist}(\cdot,X_*)/\kappa})$ bounds for the latter imply what is asserted by the third of the top four bullets in the lemma.
\end{sp6}

\begin{sp6}
The Riemannian metric on $\R\times Y_T$ with the self-dual 2-form $\omega$ (which is $ds\wedge (a+\frac{1}{2}\frac{\partial}{\partial s}a) + \frac{1}{2}*da$) defines an almost complex structure on $T(\R\times Y_T)$. With this almost complex structure understood, introduce by way of notation $\bar\partial_A\alpha$ to denote the $T^{0,1}$ part of $\nabla_A\alpha$. The rightmost equation in \eqref{eq:instanton-sw} when written in terms of $\alpha$ and $\beta$ identifies $\bar\partial_A\alpha$ with $x(\nabla_A\beta)$ where $x$ is a homomorphism with norm bounded by $c_0$.  Using the isomorphism of the preceding step, this identification takes the form
\begin{equation} \label{eq:bounds-new-317}
\bar\partial z + \aA{}^{0,1}(1-z) = x(\nabla_A\beta)
\end{equation}
with $\bar\partial z$ denoting the $T^{0,1}$ part of $dz$ and with $\aA{}^{0,1}$ denoting the analogous part of $\aA$. Because $\aA$ is $i\R$-valued, this last identity with \eqref{eq:bounds-new-316} leads to the bound
\begin{equation} \label{eq:bounds-new-318}
|\aA| \leq c_0\left(|\bar\partial z| + r^{-1/2}\left(1+re^{-\sqrt{r}\operatorname{dist}(\cdot,X_*)/\kappa}\right)\right).
\end{equation}
This implies that a $c_0r^{-1/2}(1+re^{-\sqrt{r}\operatorname{dist}(\cdot,X_*)/\kappa})$ bound on $|dz|$ leads to a corresponding $c_0r^{-1/2}(1+re^{-\sqrt{r}\operatorname{dist}(\cdot,X_*)/\kappa})$ bound on $|\aA|$.

Steps 4-7 derive the desired $c_0r^{-1/2}(1+re^{-\sqrt{r}\operatorname{dist}(\cdot,X_*)/\kappa})$ bound for $|dz|$.
\end{sp6}

\begin{sp6}
Since $\alpha$ obeys the first bullet of \eqref{cond:instanton-27}, the function $z$ obeys an equation that can be written schematically as 
\begin{equation} \label{eq:bounds-new-319}
d^\dagger dz + 2rz = -|\aA|^2(1-z) + \tfrac{1}{4}r(1-|\alpha|^2+|\beta|^2+z)z + x_0(1-z) + x_1(\nabla_A\beta)+x_2(\beta)
\end{equation}
with $x_0$, $x_1$, and $x_2$ each denoting a homomorphism whose norm is bounded by $c_0$. (The derivation invokes \eqref{eq:instanton-214} with $\fra = r^{-1/2}\aA$.)
\end{sp6}

\begin{sp6} \label{step:bounds-35-5}
With \eqref{eq:bounds-new-319} in hand, fix a point $p\in \R\times Y_T$ with distance $c_0\kappa r^{1/2}$ or more from the boundary of $X_*$. Let $\rho$ denote this distance and let $B$ denote the ball of radius equal to the minimum of $\frac{1}{2}\rho$ and $1$ centered at $p$.

Define the bump function $\chi_p$ by the rule $\chi_p(\cdot) = \chi(4\rho^{-1}\operatorname{dist}(\cdot,p))$.  This function is equal to $1$ where the distance to $p$ is less than $\frac{1}{16}\rho$ and it is equal to zero where the distance to $p$ is greater than $\frac{1}{4}\rho$. Let $z_p = \chi_p z$.  This function has compact support in $B$; and it obeys an equation that has the schematic form
\begin{equation} \label{eq:bounds-new-320}
d^\dagger dz_p + 2rz_p = -2\langle d\chi_p,dz\rangle + d^\dagger d\chi_p z + \chi_p\mathpzc{h}
\end{equation}
where $\mathpzc{h}$ is shorthand for what appears on the right hand side of \eqref{eq:bounds-new-319}.
\end{sp6}

\begin{sp6}
Let $q$ denote for the moment a point in $B$ with distance less than $\frac{1}{32}\rho$ from $p$.  Introduce now $G_q$ to denote the Dirichlet Green's function for the operator $d^\dagger d + \frac{1}{2}r$ on $B$ with pole at the point $q$. This Green's function is zero on the boundary of $B$, positive inside $B$ and smooth except at $q$. Moreover, it obeys:
\begin{equation} \label{eq:bounds-new-321}
G_q \leq c_0 \frac{1}{\operatorname{dist}(\cdot,q)^2}e^{-\sqrt{r}\operatorname{dist}(\cdot,q)/c_0} \mathrm{\ \ \ and\ \ \ } |(dG)_q| \leq \frac{1}{\operatorname{dist}(\cdot,q)^3}(1+\sqrt{r}\operatorname{dist}(\cdot,q))e^{-\sqrt{r}\operatorname{dist}(\cdot,q)/c_0}.
\end{equation}
\end{sp6}

\begin{sp6} \label{step:bounds-35-7}
The Green's function $G_q$ can be used to write $z_p$ at any point $q$ with distance less than $\frac{1}{32}\rho$ from $p$ as
\begin{equation} \label{eq:bounds-new-322}
z_p|_q = \int_B G_q(-2\langle dz, d\chi_p\rangle + d^\dagger d\chi_p z + \chi_p \mathpzc{h}).
\end{equation}
The exterior derivative of this identity gives an identity for $dz|_q$ that has the same form but with $(dG)_q$ replacing $G_q$. The desired bound $|dz|_p| \leq c_0r^{-1/2}(1+re^{-\sqrt{r}\operatorname{dist}(p,X_*)/c_0})$ follows from the $q=p$ version of the latter identity using the right hand inequality in \eqref{eq:bounds-new-321}, the inequality in \eqref{eq:bounds-new-316} and the bound $|\beta| \leq c_0r^{-1/2}$. (These are used to bound the term in the integrand with $\mathpzc{h}$ and the terms in the integrand with derivatives of $\chi_p$. Keep in mind in this regard that derivatives of $d\chi_p$ are zero where the distance to $p$ is less than $\frac{1}{32}\rho$.)
\end{sp6}

\begin{sp6}
This final step of the proof derives the bounds for $|\beta|^2$ and for $|\nabla_A\beta|^2$ that are asserted by the third and fourth bullets of Lemma~\ref{lem:bounds-35}. To start: The $|\beta|^2$ bound follows directly from \eqref{eq:bounds-new-316} and the second bullet of Lemma~\ref{lem:bounds-32}.  Meanwhile, the proof of the $|\nabla_A\beta|^2$ bound uses the Green's function in \eqref{eq:bounds-new-322} in conjuction with the equation in the second bullet of \eqref{cond:instanton-27}.

To say more, note first that $|\nabla_A\beta|$ is the same as $|\nabla_{\Ai+\aA}\beta_\diamond|$ at points where the distance to $X_*$ is greater than $c_0r^{-1/2}$. Let $p$ denote such a point and let $B$ denote the ball from Step~\ref{step:bounds-35-5}. If the radius of $B$ is less than $c_0^{-1}$, then the bundle $K^{-1}$ on $B$ is isomorphic to the product $\C$ bundle. In particular, one can and should fix an isomorphism that identifies the canonical connection on $K^{-1}$ with a connection on the product bundle that can be written as $\Ai+\Gamma$ with $|\Gamma|$ and $|\nabla\Gamma|$ both less than $c_0$. Having done this, then writing \eqref{cond:instanton-27} using this isomorphism gives an equation for $\beta_\diamond$ that has the schematic form
\begin{equation} \label{eq:bounds-new-323}
d^\dagger d\beta_\diamond + r\beta_\diamond = c_6 + \mathcal{R}
\end{equation}
where $c_6$ is the same as its namesake in the lower bullet of \eqref{cond:instanton-27} after accounting for the isomorphism between $K^{-1}$ and the product bundle; and where $\mathcal{R}$ obeys the pointwise bound $|\mathcal{R}| \leq c_0r^{-1/2}(1+re^{-\sqrt{r}\operatorname{dist}(p,X_*)/c_0})$. (The $c_6$ term is part of the $c_6\alpha$ term in \eqref{cond:instanton-27} and it appears when $\alpha$ is written as $(1-z)$.  The $-c_6z$ part is incorporated into the $\mathcal{R}$ term in \eqref{eq:bounds-new-323}.)  The only contribution to $\mathcal{R}$ that does not directly obey this bound by virtue of what has been proved previously is a term that comes from $(\nabla_{\Ai+\aA})^\dagger \nabla_{\Ai+\aA}\beta_\diamond$, this being $(d^\dagger\aA)\beta_\diamond$.  But, then $d^\dagger\aA$ can be identified as done in \eqref{eq:instanton-221} which leads directly to the asserted bound for the absolute value of $(d^\dagger\aA)\beta_\diamond$.

With \eqref{eq:bounds-new-323} understood, then the argument used to bound $dz$ in Step~\ref{step:bounds-35-7} can be used with only notational changes and one additional observation to bound the norm of $\nabla_{\rm I}\beta_\diamond$ by $c_0r^{-1}(1+re^{-\sqrt{r}\operatorname{dist}(p,X_*)/c_0})$.  The additional observation is that the solution $w$ to the equation $d^\dagger dw+\frac{1}{4}rw = c_6$ with Dirichlet boundary conditions on the ball $B$ is such that its norm and that of $dw$ are both bounded by $c_0r^{-1}$.  (This is because $c_6$ is smooth.)  This bound on $|\nabla_{\rm I}\beta_\diamond|$ with the previous bounds on $\aA$ and $\beta$ lead to the desired bound $|\nabla_{A}\beta| \leq c_0r^{-1}(1+re^{-\sqrt{r}\operatorname{dist}(p,X_*)/c_0})$. \qedhere
\end{sp6}
\end{proof}

Lemma \ref{lem:bounds-35} is strengthened with the addition of the following lemma.

\begin{lemma}
\label{lem:bounds-36}
There exists $\kappa>1$ with the following significance: Suppose that $T\in (16,\infty]$, that $r > \kappa$ and that $\frd = (A,\psi=(\alpha,\beta))$ is an instanton solution to \eqref{eq:instanton-sw} on $\R\times Y_T$ such that \eqref{eq:bounds-314} holds with $c_\frd<c$.  Let $\rmu\subset \R\times Y_T$ denote an open set where $1-|\alpha|^2 \leq \kappa^{-1}$ and where $F_{A_0}$ and $\frp$ both vanish. Use $\partial \rmu$ to denote the boundary of the closure of $\rmu$. There exists an identification between $E|_\rmu$ with the product bundle $\rmu\times \C$ that writes $(A,\psi)$ on $\rmu$ as $(\Ai + r^{1/2}\fra, \psii+\eta)$ with $\frh = (\fra,\eta)$ obeying the bound $|\frh| \leq \kappa e^{-\sqrt{r}\operatorname{dist}(\cdot,\partial U)/\kappa}$.
\end{lemma}

\begin{proof}
Assume that $|\alpha| > 1-\frac{1}{100}$ on $\rmu$. The unit length element $\alpha/|\alpha|$ defines a bundle isomorphism from $\rmu \times \C$ to $E|_\rmu$ and the inverse isomorphism writes $\alpha$ as $(1-z)$ with $z$ being real with norm less than $\frac{1}{10}$. The bound for $(1-|\alpha|^2)$ in the second bullet of Lemma \ref{lem:bounds-35} implies that $|z|\leq c_0(r^{-2}+e^{-\sqrt{r}\operatorname{dist}(\cdot,\partial \rmu)/c_0})$. Meanwhile, the isomorphism writes $A$ as $\Ai+r^{1/2}\fra$ and the bound in the second bullet of Lemma \ref{lem:bounds-35} for $|\nabla_A\alpha|$ implies that $|\fra| \leq c_0(r^{-1}+e^{-\sqrt{r}\operatorname{dist}(\cdot,\partial \rmu)/c_0})$. Let $\ahat_0$ denote $r^{-1/2}\fra$. This $i\R$-valued 1-form obeys the equation in \eqref{eq:instanton-221}. This implies in turn that \eqref{cond:instanton-26} is obeyed when $\frh_\diamond$ is given by $(\fra, ((1-z),\beta_\diamond))$ where $\beta_\diamond$ corresponds to $\beta$ via the bundle isomorphism that writes $\alpha$ as $(1-z)$. Moreover, what is denoted by $\pzee$ in this particular version of \eqref{cond:instanton-26} obeys $||\pzee(\frh)||_2 \leq {c_0}^{-1} \frac{1}{100}||\Li\frh||_2$ when $\frh$ has compact support on the part of $\rmu$ with distance greater than $c_0r^{-1/2}$ from $\partial\rmu$ (assuming that $r>c_0$).  This is because bounds from Lemma~\ref{lem:bounds-35} say that
\begin{equation} \label{eq:bounds-new-324}
|\frh|^2 \leq c_0\left(r^{-1} + e^{-\sqrt{r}\operatorname{dist}(\cdot, \partial\rmu)/c_0}\right)
\end{equation}
and the homomorphism $\pzee$ has linear dependence on the components of $\frh$.  Granted the preceding, it follows that $\pzee$ obeys the bounds in Lemma~\ref{lem:instanton-21} when $r>c_0$ and when $\frh$ has compact support where the distance to $\partial\rmu$ is greater than $c_0r^{-1/2}$. With this understood, then what is said by Lemma~\ref{lem:bounds-36} follows from Lemma~\ref{lem:instanton-21}.
\end{proof}

An analog here for Lemma 3.10 in \cite{taubes4} is also needed. It follows directly.

\begin{lemma}
\label{lem:bounds-37}
Given $c>1$ and $\mathcal{K} \geq 1$, there exists $\kappa>1$ with the following significance: Suppose here that $r\geq \kappa$ and that $\frd = (A,\psi=(\alpha,\beta))$ is an instanton solution to \eqref{eq:instanton-sw} such that \eqref{eq:bounds-314} holds with $c_\frd \leq c$. Fix a point $s_0\in\R$ and $R\geq 2$ and suppose that $\scm(\cdot)\leq \mathcal{K}$ on $[s_0-R-4,s_0+R+4]$. Given ${\bf x} \in [s_0-R-1, s_0+R+1] \times Y_T$ and given a number $\rho \in (r^{-1/2},\kappa^{-1})$, let $\scm({\bf x},\rho)$ denote the integral of $r(1-|\alpha|^2)$ over the radius $\rho$ ball in $\R\times M$ centered at ${\bf x}$. Then
\begin{itemize}\leftskip-0.25in
\item If $\rho_1 > \rho_0$ are in $(r^{-1/2},\kappa^{-1})$, then $\scm({\bf x},\rho_1) \geq \kappa^{-1} \rho_1^2/\rho_0^2 \scm({\bf x},\rho_0)$.
\item $\scm({\bf x},\rho) \leq \kappa\mathcal{K}\rho^2$.
\item Suppose that $|\alpha|({\bf x}) \leq \frac{1}{2}$. If $\rho \in (r^{-1/2},\kappa^{-1})$, then $\scm({\bf x},\rho) \geq \kappa^{-1}\rho^2$.
\end{itemize}
\end{lemma}

The proof of this lemma is virtually identical to the proof of Lemma 3.10 in \cite{taubes4}.

%% file: a_local-convergence.tex

The propositions in this section are the analogs for $\R\times Y_T$ of Propositions 4.5 and 5.5 in \cite{taubes4}. These propositions refer to a connection on the bundle $E$ that is constructed from a pair $(A,\psi=(\alpha,\beta))$ as follows: fix a smooth, non-decreasing function $\frP:[0,\infty)\to[0,1]$ with $\frP(x)=x$ for $x<\frac{1}{2}$ and with $\frP(x)=1$ for $x>\frac{3}{4}$. Having chosen $\frP$, define $\hat{A}$ by the formula
\begin{equation}
\label{eq:flatconn}
\hat{A}=A-\frac{1}{2}\frP(|\alpha|^2)|\alpha|^{-2}(\bar{\alpha}\nabla_A\alpha-\alpha\nabla_A\bar{\alpha}).
\end{equation}
The curvature of $\hat{A}$ is given by the formula
\begin{equation}
\label{eq:curvature-flatconn}
F_{\hat{A}}=(1-\frP)F_A-\frP'\nabla_A\bar{\alpha}\wedge\nabla_A\alpha,
\end{equation}
with $\frP'$ denoting the function on $\R\times Y_T$ given by evaluating the derivative of $\frP$ on the function $|\alpha|^2$. Note that $F_{\hat{A}}=0$ where $\frP=1$ and thus $|\alpha|\geq\frac{\sqrt{3}}{2}$. It is also the case that $\nabla_{\hat{A}}\frac{\alpha}{|\alpha|}=0$ where $\frP=1$.

\begin{proposition}
\label{prop:conv-41}
Suppose that $\mathcal{K}>1$ and that each Reeb orbit with length at most $\frac{1}{2\pi}\mathcal{K}$ is non-degenerate. Given such $\mathcal{K}$ and given $c>1$, there exists $\kappa\geq1$, and given in addition $\delta>0$, there exists $\kappa_\delta\geq1$; these numbers $\kappa$ and $\kappa_\delta$ having the following significance: Suppose that $T\in(16,\infty]$, that $r>\kappa_\delta$ and that $\frd=(A,\psi=(\alpha,\beta))$ is an instanton solution to \eqref{eq:instanton-sw} such that \eqref{eq:bounds-314} holds with $c_{\frd}<c$. Let $\mathbb{I}\subset\R$ denote a connected subset of length at least $2\delta^{-1}+16$ such that $\mathrm{sup}_{s\in\mathbb{I}} \scm(s)\leq\mathcal{K}$. Let $I\subset\mathbb{I}$ denote the set of points with distance at least $7$ from any boundary of $\mathbb{I}$. 
\begin{itemize}\leftskip-0.25in
\item Each point in $I\times N$ where $|\alpha|\leq1-\delta$ has distance $\kappa r^{-1/2}$ or less from a point in $\alpha^{-1}(0)$.
\item There exists 
\begin{itemize}
\item[(a)] A positive integer $\rmn\leq\kappa$ and a cover of $I$ as $\cup_{1\leq k\leq \rmn}I_k$ by connected open  sets of length at least $2\delta^{-1}$. These are such that $I_k\cap I_{k'}=\emptyset$ is $|k-k'|>1$. If $|k-k'|=1$, then $I_k\cap I_{k'}$ has length between $\frac{1}{128}\delta^{-1}$ and $\frac{1}{64}\delta^{-1}$.  
\item[(b)] For each $k\in\{1,\dots,\rmn\}$, a set $\vartheta_k$ whose typical element is a pair $(C,m)$ where $m$ is a positive integer and $C\subset \mathbb{I}\times Y_T$ is a pseudoholomorphic subvariety defined on a neighborhood of $I_k\times Y_T$. These elements of $\vartheta_k$ are constrained that no two pairs share the same subvariety component so that $\sum_{(C,m)\in\vartheta_k}m\int_Cda\leq\kappa$. 
\end{itemize}
In addition, the collection $\{\vartheta_k\}_{k=1,\dots,\rmn}$ is such that
\begin{enumerate}\leftskip-0.3in
\item $\mathrm{sup}_{z\in\cup_{(C,m)\in\vartheta_k}\,\mathrm{and}\,s(z)\in I_k}\operatorname{dist}(z,\alpha^{-1}(0))+\mathrm{sup}_{z\in\alpha^{-1}(0)\,\mathrm{and}\,s(z)\in I_k}\operatorname{dist}(z,\cup_{(C,m)\in\vartheta_k}C)<\delta$.
\item Let $k\in\{1,\dots,\rmn\}$, let $I'\subset I_k$ denote an interval of length $1$, and let $\upsilon$ denote the restriction to $I'\times Y_T$ of a 2-form on $\mathbb{I}\times Y_T$ with $|\upsilon|$ on $\mathbb{I}\times Y_T$ bounded by 1 and with $|\nabla\upsilon|$ on $\mathbb{I}\times Y_T$ bounded by $\delta^{-1}$. Then 
\[\left|\frac{i}{2\pi}\int_{I'\times Y_T}\upsilon\wedge F_{\hat{A}}-\sum_{(C,m)\in\vartheta_k}m\int_C\upsilon\right|<\delta.\]
\end{enumerate}
\item Suppose that $\mathbb{I}$ is unbounded from above. Fix $\en_+\leq\mathcal{K}$ and assume with regards to Reeb orbits only that all with length at most $\frac{1}{2\pi}\en_+$ are non-degenerate. Assume in addition that $\lim_{s\to\infty}\en(\frd|_s)\leq\en_+$. Then the conclusions of the first two bullets hold with the constant $\kappa$ depending on $\mathcal{K}$ and $\en_+$, and with $\kappa_\delta$ depending on the latter and on $\delta$. Moreover, if $\mathbb{I}=\R$, and all Reeb orbits of length at most $\frac{1}{2\pi}\en_+$ are non-degenerate, then $\lim_{s\to-\infty}\en(\frd|_s)\leq\en_++\delta$.
\end{itemize}
\end{proposition}
The proof of this proposition differs only cosmetically from the proof of Proposition 4.5 in \cite{taubes4} because the arguments for Proposition 4.5 in \cite{taubes4} are essentially local in nature. Note that Lemmas \ref{lem:bounds-35}, \ref{lem:bounds-36}, and \ref{lem:bounds-37} play the same role in the proof of Proposition \ref{prop:conv-41} as the role played by Lemmas 3.8, 3.9, and 3.10 in the proof of Proposition 4.5 in \cite{taubes4}. See also the proof of Propositions 7.8 and 7.9 in \cite{ht2} for dealing with the part of $\R\times Y_T$ where the form $a$ depends on the $\R$-parameter. 

The next proposition plays the role of Proposition 5.5 in \cite{taubes4}.
\begin{proposition}
\label{prop:conv-42}
There exists $\kappa>1$ such that if $c>\kappa$, and if the norms of $\frp$ and its covariant derivatives are pointwise less than $c^{-1}$, and if
\[|\frac{\partial}{\partial s}a|+|\frac{\partial}{\partial s}da|\leq c^{-1}\;and\;s_0<\kappa^{-1}c,\]
then what comes next is true. Suppose that $\mathcal{K}>1$ and that each Reeb orbit with length at most $\frac{1}{2\pi}\mathcal{K}$ is non-degenerate. There exists $\kappa_\ast$, and given $\delta>0$, there exists $\kappa_\delta\geq 1$ which have the following significance: Suppose that $T\in (16,\infty]$, that $r>\kappa_\delta$ and that $T\in(16,\infty]$, and that $\frd=(A,\psi=(\alpha,\beta))$ is an instanton solution to \eqref{eq:instanton-sw} on $\R\times Y_T$ with $\mathcal{A}_\frd\leq\mathcal{K} r$. Assume that $\en_+=\lim_{s\to\infty}\en(\frd|_s)\leq\mathcal{K}$.
\begin{itemize}\leftskip-0.25in
\item Let $\mathfrak{c}_-=\lim_{s\to-\infty}(A,\psi)|_s$. Then $\mathfrak{c}_-$ is a solution to \eqref{eq:sw} with $\en(\mathfrak{c}_-)\leq\en_++\delta$.
\item Each point in $\R\times Y_T$ where $|\alpha|\leq1-\delta$ has distance $\kappa_\ast r^{-1/2}$ or less from a point in $\alpha^{-1}(0)$. 
\item Moreover, there exists 
\begin{itemize}
\item[(a)] A positive integer $\rmn\leq\kappa$ and a cover of $\R$ as $\cup_{1\leq k\leq \rmn}I_k$ by connected open  sets of length at least $2\delta^{-1}$. These are such that $I_k\cap I_{k'}=\emptyset$ is $|k-k'|>1$. If $|k-k'|=1$, then $I_k\cap I_{k'}$ has length between $\frac{1}{128}\delta^{-1}$ and $\frac{1}{64}\delta^{-1}$.  
\item[(b)] For each $k\in\{1,\dots,\rmn\}$, a set $\vartheta_k$ whose typical element is a pair $(C,m)$ where $m$ is a positive integer and $C\subset \R\times Y_T$ is a pseudoholomorphic subvariety. These elements of $\vartheta_k$ are constrained so that $\sum_{(C,m)\in\vartheta_k}m\int_Cda\leq\kappa$. 
\end{itemize}
In addition, these sets $\{\vartheta_k\}_{k=1,\dots,\rmn}$ are such that
\begin{enumerate}\leftskip-0.3in
\item $\mathrm{sup}_{z\in\cup_{(C,m)\in\vartheta_k}\,\mathrm{and}\,s(z)\in I_k}\operatorname{dist}(z,\alpha^{-1}(0))+\mathrm{sup}_{z\in\alpha^{-1}(0)\,\mathrm{and}\,s(z)\in I_k}\operatorname{dist}(z,\cup_{(C,m)\in\vartheta_k}C)<\delta$.
\item Let $k\in\{1,\dots,\rmn\}$, let $I'\subset I_k$ denote an interval of length $1$, and let $\upsilon$ denote the restriction to $I'\times Y_T$ of a 2-form on $\R\times Y_T$ with $|\upsilon|$ on $\R\times Y_T$ bounded by $1$ and with $|\nabla\upsilon|$ on $\R\times Y_T$ bounded by $\delta^{-1}$. Then 
\[\left|\frac{i}{2\pi}\int_{I'\times Y_T}\upsilon\wedge F_{\hat{A}}-\sum_{(C,m)\in\vartheta_k}m\int_C\upsilon\right|\leq\delta.\]
\end{enumerate}
\item The version of the function $\scm$ that is defined by $\frd$ is bounded by $\kappa_\ast$.
\end{itemize}
\end{proposition}
\begin{proof}
This proposition differs from the previous one in two ways, the first being that the subvarieties that comprise each set from the collection $\{\vartheta_k\}_{k=1,\dots, \rmn}$ are defined on the whole of $\R\times Y_T$. This modification is straightforward to prove and the proof differs only in notation from the proof of this assertion in Proposition 5.5 of \cite{taubes4}. The substantive difference is the lack in Proposition \ref{prop:conv-42} of the assumption in Proposition \ref{prop:conv-41} of an a priori bound for the function $s\to \scm(s)$. Proposition 5.1 in \cite{taubes4} states the equivalent of Proposition \ref{prop:conv-41} without the assumed bound on $\scm$. The proof that follows directly is much like the proof of Proposition 5.1 in \cite{taubes4}. Note that if $T$ is finite and if the numbers $\kappa$ and $\kappa_\delta$ are allowed to depend on $T$, then the arguments in \cite{ht2} for Proposition 7.1 can be used almost verbatim to prove Proposition \ref{prop:conv-42}.

Proposition \ref{prop:conv-42} follows from Proposition \ref{prop:conv-41} given the assertion in the final bullet of Proposition \ref{prop:conv-42} to the effect that there is an a priori bound for the function $\scm$. The six parts that follow prove that this is the case. Supposing that $z>1$ and that $\mathcal{A}_\frd<zr$, the convention in the proof has $c_z$ denoting a number that is greater than $100$ that depends only on $z$. Its value can be assumed to increase between successive appearances. Note an important point: Given $z>1$, then the bound $\mathcal{A}_\frd<zr$ implies that $\mathcal{A}_\frd<r^2$ when $r\geq z$. This implies in turn that Lemma \ref{lem:bounds-33} can be invoked assuming $r>z$ which implies that \eqref{eq:bounds-314} holds with $c_\frd\leq c_z$. Granted that, then Lemmas \ref{lem:bounds-34}--\ref{lem:bounds-37} and Proposition \ref{prop:conv-41} can be invoked if their versions of $c$ are also bounded by $c_z$. As a consequence, the various versions of $\kappa$ from Lemmas \ref{lem:bounds-34}--\ref{lem:bounds-37} are bounded by $c_z$; and $\kappa$ and $\kappa_\delta$ from Proposition \ref{prop:conv-41} depend now only on $\mathcal{K}$ and $z$, and also $\delta$ in the case of $\kappa_\delta$.
\begin{pt2}
\label{part1:prop42}
Let $\scL$ denote the function on $\R$ given by the rule 
\begin{equation}
\label{eq:functionL}
s\to\scL(s)=r\int_{\{s\}\times Y_T}(1-|\alpha|^2+|\beta|^2)-\int_{\{s\}\times Y_T}i(B_{A_0}+\frp).
\end{equation}
Remember from Section~\ref{ssec:setup} that $A_0$ is a chosen connection on the bundle $K^{-1}$ over $Y_T$ with zero curvature $2$-form on $Y_T \ssm N$.  It is assumed to be $T$-independent on $N$.  Meanwhile, $\frp$ is the perturbation form that appears in \eqref{eq:instanton-sw}.  It is independent of $T$ and has compact support in $(-1,1) \times M$.

Let $\uscL$ denote the function 
\begin{equation}
\label{eq:conv-44}
s\to\uscL(s)=\int_s^{s+1}\scL(x)dx.
\end{equation}
As explained below, the function $\uscL$ can serve as a proxy for $\scm$ because $\scL$ obeys a bound of the form 
\begin{equation}
\label{eq:conv-45}
\scL(s)>(1-\frac{1}{100}{c_z}^{-1})\en(\frd|_s)-c_z
\end{equation}
when $r\geq c_z$. Keep in mind that $\en(\frd|_s)$ is the integral of the function $|1-|\alpha|^2|$ over $\{s\}\times Y_T$ and $\scm(s)$ is the integral of the function $s\to \en(\frd|_s)$ over $[s,s+1]$. 

If $T$ is finite, then $\scL(s)>\en(\frd|_s)-c_T$ with $c_T$ being a number that depends on $T$. This follows from Lemma \ref{lem:bounds-32} because the rightmost term in \eqref{eq:functionL} is bounded by $c_0$ and the leftmost term in \eqref{eq:functionL} would be greater than the integral of $|1-|\alpha|^2|$ over $\{s\}\times Y_T$ were it not for the contribution from the $|\alpha|>1$ part of $\{s\}\times Y_T$.

To prove \eqref{eq:conv-45}, let $u$ denote the maximum of zero and $|\alpha|^2-1$ and let $\scu$ denote the function defined by the rule
\begin{equation}
\label{eq:functionU}
s\mapsto\scu(s)=r\int_{\{s\}\times Y_T}u.
\end{equation}
It follows from the definitions and Lemma \ref{lem:bounds-32} that 
\begin{equation}
\label{eq:conv-47}
\scL(s)\geq\en(\frd|_s)-2\scu(s)-c_z
\end{equation}
and so \eqref{eq:conv-45} follows with a proof that $\scu\leq\frac{1}{100}{c_z}^{-1}\en+c_z$ when $r\geq c_z$. The four steps that follow prove a stronger assertion, this being that $\scu\leq (\ln{r})^8(r^{-1}\en+r^{-1/4})$ when $r\geq c_z$.
\begin{sp2}
\label{step1part1:prop42}
Let $\kappa_0$ denote for the moment the version of $\kappa$ that appears in Lemma \ref{lem:bounds-36}. Let $U$ denote the part of $\R\times Y_T$ where $|\alpha|$ is greater than $1-2{\kappa_0}^{-1}$ but less than $1-\frac{1}{2}{\kappa_0}^{-1}$. Supposing that $I\subset\R$ is a bounded interval, let $U_I$ denote the part of $U$ in $I\times Y_T$ and $\nu_I$ denote the volume of $U_I$. It follows from the bound in the second bullet of Lemma \ref{lem:bounds-34} that there exists a positive integer to be denoted by $\rmn_I$ such that the following is true:
\begin{itemize}\leftskip-0.25in
\item $U_I$ has a cover by $\rmn_I$ balls of radius $r^{-1/2}$ with centers in $U_I$.
\item $\nu_I\geq {c_z}^{-1}\rmn_Ir^{-2}$.
\end{itemize}
\begin{equation}
\label{eq:conv-48}
\end{equation}
A cover that obeys these conditions is obtained by taking the centers of the balls to be the points in a maximal subset of $U_I$ with the property that the distance between any two points in the subset is no less than $\frac{1}{2}r^{-1/2}$. Let $\scq_I$ denote the integral of $r(1-|\alpha|^2)^2$ over $U_I$. The second bullet above implies that 
\begin{equation}
\label{eq:conv-49}
\scq_I\geq {c_z}^{-1}r\nu_I.
\end{equation}
\end{sp2}
\begin{sp2}
\label{step2part1:prop42}
Fix $s\in\R$ and for $k\in\Z$, let $I_k$ denote $[s+kr^{-1/2},s+(k+1)r^{-1/2}]$, this being an interval of length $r^{-1/2}$. For $k\in\Z$, let $Z_k$ denote the part of $\{s\}\times Y_T$ where the closest point in $U$ comes from $I_k\times Y_T$ and let $\sch_k$ denote the integral of $ru$ on $Z_k$. Thus,
\begin{equation}
\label{eq:conv-410}
\scu(s)=\sum_{k\in\Z}\sch_k.
\end{equation}
To see about the size of $\sch_k$, fix a cover of $U_{I_k}$ of the sort that is described above and denote it by $\mathfrak{U}_{I_k}$. Supposing that $L\geq |k|$, then the subset of $\{s\}\times Y_T$ with distance between $Lr^{-1/2}$ and $(L+1)r^{-1/2}$ from any given ball from the cover $\mathfrak{U}_{I_k}$ has volume at most $c_0L^3r^{-3/2}$. It follows as a consequence of Lemma \ref{lem:bounds-36} that 
\begin{equation}
\label{eq:conv-411}
\sch_k\leq c_z r^{-3/2} \rmn_{I_k}e^{-|k|/c_z}.
\end{equation} 
What with \eqref{eq:conv-49} and the second bullet of \eqref{eq:conv-48}, it follows that 
\begin{equation}
\label{eq:conv-412}
\sch_k\leq c_z r^{1/2}\scq_{I_k}e^{-|k|/{c_z}}.
\end{equation}
This bound will be used when $|k|$ is large but not when $k$ is (relatively) small.
\end{sp2}
\begin{sp2}
\label{step3part1:prop42}
To obtain a bound when $|k|$ is small, first use \eqref{eq:conv-48} to see that the volume of the subset of $\{s\}\times Y_T$ with distance at most $Lr^{-1/2}$ from $U_k$ is at most $c_0L^3r^{-3/2}\rmn_{I_k}$. It then follows from \eqref{eq:conv-48} and \eqref{eq:conv-49} that this volume is at most $c_0L^3r^{-1/2}\scq_{I_k}$ What with the top bullet of Lemma \ref{lem:bounds-32}, it follows that this part of $U_k$ contributes at most $c_zL^3r^{-1/2}\scq_{I_k}$ to $\sch_k$. Meanwhile, it follows from Lemma \ref{lem:bounds-36} that the remaining part of $U_k$ contributes at most $c_zr^{1/2}\scq_{I_k}e^{-L/c_z}$ to $\sch_k$. Take $L$ to equal $c_z100\ln{r}$ and use the preceding pair of bounds to see that 
\begin{equation}
\label{eq:conv-413}
\sch_k\leq c_z(\ln{r})^3r^{-1/2}\scq_{I_k}.
\end{equation}
\end{sp2}
\begin{sp2}
\label{step4part1:prop42}
It follows from the definitions that 
\begin{equation}
\label{eq:conv-414}
\scq_{I_k}\leq r\int_{I_k\times Y_T}(1-|\alpha|^2)^2.
\end{equation}
To bound the right hand side of \eqref{eq:conv-414}, it proves convenient to define a non-negative function $\rmx$ on $\R$ by the rule whereby $\rmx^2$ at any $s'\in\R$ is 
\begin{equation}
\label{eq:conv-415}
\rmx^2(s')=\int_{\{s'\}\times Y_T} (1-|\alpha|^2)^2.
\end{equation}
The derivative of $\rmx(\cdot)$ is bounded by $c_0(\int_{\{s'\}\times Y_T}|\nabla_{A,s}\alpha|^2)^{1/2}$. This understood, it follows from Lemma \ref{lem:bounds-31} and the first bullet of Lemma \ref{lem:bounds-34} that 
\begin{equation}
\label{eq:conv-4.16}
\rmx(s')\leq\rmx(s)+c_zr^{1/2}|s-s'|^{1/2}.
\end{equation}
Use this bound with $s'\in I_k$ to bound the right hand side of \eqref{eq:conv-414} by 
\begin{equation}
\label{eq:conv-417}
2r^{-1/2}\int_{\{s\}\times Y_T}(1-|\alpha|^2)^2+c_z|k|.
\end{equation}
This in turn is no greater than $2r^{-3/2}\en(\frd|_s)+c_z|k|$ and so $\scq_{I_k}\leq c_z(r^{-1/2}\en(\frd|_s)+r|k|)$.

Use the preceding bound for $\scq_{I_k}$ in \eqref{eq:conv-412} and \eqref{eq:conv-413} to obtain respective bounds 
\begin{equation}
\label{eq:conv-418}
\sch_k\leq c_z(r^{-1}\en(\frd|_s)+r^{1/2}|k|)e^{-|k|/c_z}\; {\rm and}\; \sch_k\leq c_z(\ln{r})^3(r^{-1}\en(\frd|_s)+r^{-1/2}|k|).
\end{equation}
Use the leftmost bound in \eqref{eq:conv-415} when summing over the values of $|k|>(\ln{r})^2$; and use the rightmost bound in \eqref{eq:conv-413} when summing over the values of $|k|\leq (\ln{r})^2$ to see that
\begin{equation}
\label{eq:conv-419}
\scu(s)\leq c_z (\ln{r})^7(r^{-1}\en(\frd|_s)+c_z r^{-1/2}).
\end{equation}
This inequality leads directly to the desired bound.
\end{sp2}
\end{pt2}
\begin{pt2}
\label{part2:prop42}
As noted in Part \ref{part1:prop42}, an upper bound for the function $\uscL$ leads to an upper bound for the function $\scm$. An upper bound for $\uscL$ is seen in this part of the proof to follow from an upper bound for another function on $\R$, this denoted by $\upze$. To define the latter, first define the function $\pze$ on $\R$ by the rule 
\begin{equation}
\label{eq:conv-420}
s\mapsto \pze(s)=i\int_{\{s\}\times Y_T}B_A\wedge\ast a.
\end{equation}
The function $\upze$ is defined from $\pze$ by the rule
\begin{equation}
\label{eq:conv-421}
s\mapsto \upze(s)=\int_{s}^{s+1}\pze(x)dx.
\end{equation}
To see that a bound for $\upze$ leads to one for $\uscL$, suppose first that $s>s_0$ or that $s<-s_0$ so that the 1-form $a$ is independent of the factor $\R$ in $\R\times Y_T$. Differentiate \eqref{eq:functionL} and use the curvature equation \eqref{eq:instanton-sw} to see that
\begin{equation}
\label{eq:conv-422}
\frac{d}{ds}\pze=-2\pze+2\scL.
\end{equation}
If the interval $[s,s+t]$ is disjoint from $[-s_0,s_0]$, then integrating the equation gives
\begin{equation}
\label{eq:conv-423}
\pze(s+t)=e^{-2t}\pze(s)+e^{-2(s+t)}\int_s^{s+t}e^{2x}\scL(x)dx.
\end{equation}
Since $\scL\geq-c_z$ by virtue of \eqref{eq:conv-45}, this equation implies that
\begin{equation}
\label{eq:conv-424}
\pze(s+t)\geq e^{-2t}\pze(s)-e^{-2t}c_z,
\end{equation}
and it implies when $\tau=s+t$ that $\pze(\tau)\geq -c_0$.

Integrating \eqref{eq:conv-422} leads to the bounds 
\begin{equation}
\label{eq:conv-425}
\pze(s)\leq c_0+c_0\upze(s+1)\;\;\;\;\;{\rm and}\;\;\;\;\;\uscL(s)\leq c_0+c_0(\upze(s)+\upze(s+2))
\end{equation}
if the interval $[s,s+3]$ is disjoint from $[-s_0,s_0]$. The leftmost inequality in \eqref{eq:conv-425} implies that a bound for $\upze$ leads to a bound for $\uscL$ when $s\geq s_0$ or $s\leq -s_0-3$. 

To see about bounding $\uscL$ using $\upze$ for values of $s$ near $[-s_0,s_0]$, note first that \eqref{eq:conv-422} for $s$ near this interval is replaced by the equation
\begin{equation}
\label{eq:conv-426}
\frac{d}{ds}\pze=-2\pze+2\scL-i\int_{\{s\}\times Y_T}B_A\wedge\ast\frac{\partial}{\partial s}a.
\end{equation}
Let $I$ denote the interval $[-s_0-10,s_0+10]$. Integrating this equation over $I$ leads to an inequality asserting that 
\begin{equation}
\label{eq:conv-427}
\pze(s_0+10)-\pze(-s_0-10)+2\int_I\pze(x)dx=2\int_I\scL(x)dx-i\int_{I\times Y_T}B_A\wedge\ast\frac{\partial}{\partial s}a.
\end{equation}
The rightmost integral on the right hand side of \eqref{eq:conv-427} is bounded from below courtesy of Lemma \ref{lem:bounds-31} by $c_0z$. Let $I'$ denote the interval $[-s-13,s+13]$. Given the bound $\scL\geq -c_z$ from \eqref{eq:conv-45}, this $c_0z$ bound and the leftmost bound in \eqref{eq:conv-425} lead to a bound of the form
\begin{equation}
\label{eq:conv-428}
sup_{s\in I}\uscL(s)\leq c_z+c_z sup_{s'\in I'}\upze(s').
\end{equation}
The latter bound with that in \eqref{eq:conv-425} proves that a bound on $\upze$ leads to a bound on $\uscL$. 
\end{pt2}
\begin{pt2}
\label{part3:prop42}
The arguments in \cite[Section 5b]{taubes4} to obtain a bound for $\upze$ must be modified to obtain bounds that don't depend on $T$. The main issue concerns the way the action function in \eqref{eq:bounds-31} is used. To say more, suppose that $\frd=(A,\psi)$ is an instanton solution to \eqref{eq:instanton-sw} and that $u:Y_T\to S^1$ is a gauge transformation, then set $\frd^u=(A-u^{-1}du,u\psi)$. Define $\calB$ as done in \eqref{eq:bounds-32} and let $\sco$ denote the function on $\R$ given by the rule whereby $\sco$ is the integral of $|\calB|^2+r|\mathcal{D}_A\psi|^2$ on $\{s\}\times Y_T$. The proof of Proposition 5.1 in \cite{taubes4} invokes an inequality for $\fra_s(\frd^u|_s)$ that reads
\begin{equation}
\label{eq:conv-429}
\fra_s(\frd^u|_s)\leq c_\ast(1+\scL+r^{1/2}\sco^{1/2}+\sco+r^{2/3}\scL^{4/3})+\frq_u-\frac{1}{2}r\pze,
\end{equation}
with $c_\ast$ being a constant greater than $1$ and $\frq_u$ being the function on $\R$ that is defined as follows: Use $\frp(A^u)$ for the moment to denote the $L^2$ orthogonal projection of the $i\R$ valued $1$-form $A-u^{-1}du-A_+$ to the space of harmonic $1$-forms on $Y_T$. What is denoted by $\frq_u$ is the integral on $\{s\}\times Y_T$ of $-i\frp(A^u)\wedge\ast B_{A_+}$. It seems that all of the steps in \cite[Section 5]{taubes4} that lead to \eqref{eq:conv-429} need to be modified to get an analogous formula with $c_\ast$ being independent of $T$ when $T$ is very large. 

The statement of the analog here of \eqref{eq:conv-429} requires a preliminary digression to talk about $s\to\infty$ limit of a given instanton solution to \eqref{eq:instanton-sw}. Letting $\frd$ denote the instanton in question, let $\frc_+=(A_+,\psi_+)$ denote $\lim_{s\to\infty}\frd|_s$. Keep in mind that $\frc_+$ is a solution to \eqref{eq:sw} on $Y_T$. Suppose that $\cale>1$ and that $\en(\frc_+)\leq\cale$. It follows from Lemma \ref{lem:monopoles-110} that there exists $c_\cale$ with the following significance: if $r>c_\cale$ and $T>c_\cale$, then there exists $R_0\in(1,c_0)$ and a smooth map $u_+:Y_T\ssm M\to S^1$ such that the $E\subset\spb$ component of $u_+\psi_+$ on the part of $Y_T$ where $\operatorname{dist}(\cdot,M)\geq R_0$ has the form $1-z$ with $z$ real and $|z|$ less than $\frac{1}{1000}$. As explained in the next two paragraphs, the map $u_+$ in turn can be used to view the first Chern class of $c_1(det(\spb))$ as an element of $H^2_c(M)$, the cohomology of $M$ with compact support. The coefficient here and subsequently are in $\Z$.

To obtain an element in $H^2_c(M)$ from $u_+$, let $\sigma$ denote for the moment a smooth function that is equal to $1$ where the distance to $M$ is less than $R_0$ and equal to $0$ where the distance to $M$ is greater than $R_0+1$. Write the connection $A_+-{u_+}^{-1}du_+$ as $\Ai+\sca_+$ where $\operatorname{dist}(\cdot,M)\geq R_0$ and define a new connection on $E$ by setting it equal to $\uptheta_0+\sigma\sca_+$ where the distance to $M$ is greater than $R_0$ and equal to $A_+$ where the distance to $M$ is less than or equal to $R_0$. Denote this connection by ${A'}_+$. Its curvature $2$-form has compact support where the distance to $M$ is less than $R_0-1$ and so $i$ times this curvature $2$-form defines and element in the 2-dimensional compactly supported cohomology of the $\operatorname{dist}(\cdot,M)\leq R_0+1$ part of $Y_T$. This part of $Y_T$ deformation retracts onto $M$, so $i$ times the curvature of ${A'}_+$ defines a class in $H^2_c(M)$. This is the desired cohomology class. It is denoted in what follows by $\hatom$. 

The topological significance is as follows: The class $c_1(det(\spb))$ pulls back as zero to $Y_T\ssm M$ and so it comes from $H^2(Y_T,Y_T\ssm M)$ via the canonical homomorphism to $H^2(Y_T)$. This group $H^2(Y_T,Y_T\ssm M)$ is isomorphic to $H^2_c(M)$. Meanwhile, the kernel of the homomorphism from $H^2(Y_T,Y_T\ssm M)$ to $H^2(Y_T)$ is $H^1(Y_T\ssm M)/H^1(Y_T)$. The map $u_+$ defines an element in this quotient group and thus an element of $H^2(Y_T,Y_T\ssm M)\cong H^2_c(M)$ that maps to $c_1(det(\spb))$. This class in $H^2_c(M)$ is $2\pi\hatom$. 
\begin{lemma}
\label{lem:conv-43}
Given $\cale>0$, there exists $\kappa_\cale>1$ with the following significance: At most $\kappa_\cale$ classes can be obtained in the manner just described from all $\en\leq\cale$ solutions to \eqref{eq:sw} on all $T\in(16,\infty]$ versions of $Y_T$. In particular, if $\hatom$ is a class that is obtained in this way, then it can be represented by a smooth closed $2$-form with compact support in $M$ with norm bounded by $\kappa_\cale$. 
\end{lemma}
\begin{proof}
It follows from Lemma \ref{lem:monopoles-110} that $R_0$ can be chosen in $[1,c_0]$ so that $\sca_+$ obeys $|\sca_+|\leq c_0 e^{-r/c_0}$ on the support of $d\sigma$. This implies in turn that the curvature $2$-form of ${A'}_+$ obeys $|F_{{A'}_+}|\leq c_0|F_{A_+}|+c_0e^{-r/c_0}$. It follows as a consequence that the $L^1$ norm of the curvature $2$-form of ${A'}_+$ is bounded by $c_0\cale$. This has the following implication: let $\upupsilon$ denote a closed $1$-form on the $\operatorname{dist}(\cdot,M)<R_0+2$ part of $Y_T$. Then the integral of $\upupsilon\wedge F_{{A'}_+}$ is bounded by $c_0\cale\sup{|\upupsilon|}$. Granted this bound, it follows using Poincar\'e--Lefschetz duality that the class $\hatom$ lies in a compact, $r$ and $T$ independent subset of $H^2_c(M)$. 
\end{proof}
Fix a closed $2$-form with compact support in $M$ and norm less than $c_\cale$ that represents $\hatom$. This class is denoted by $\upomega_0$. Note that the curvature of ${A'}_+$ has sup norm that is $\mathcal{O}(r)$ so it cannot be used for $\upomega_0$. 

The next lemma states the desired analog of \eqref{eq:conv-429}.
\begin{lemma}
\label{lem:conv-44}
There exists $\kappa\geq 1$ and given $\frE<\infty$, there exists $\kappa_\frE>1$; these numbers $\kappa$ and $\kappa_\frE$ having the following significance: Fix $T\in (\kappa_\frE,\infty]$ and $r>\kappa_\frE$. Let $\frd=(A,\psi)$ denote an instanton solution to \eqref{eq:instanton-sw} with $\mathcal{A}_\frd<r^2$. Suppose that $\tau\in\R$ and that $\pze(\cdot)\leq \frE$ on the interval $[\tau,\infty)$. If $s\in[\tau+100,\infty)$, then $\fra(\frd_u|_s)$ obeys
\[\fra_s(\frd|_s,\frc_+)\leq \kappa(1+\scL+r^{1/2}\sco^{1/2}+\sco+r^{2/3}\scL^{4/3}+r^{2/3}\frE^{4/3}+r\cale)+\frq-\frac{1}{2}r\pze,\]
with $\frq(s)=-i\int_{\{s\}\times Y_T}(A-A_+)\wedge\upomega_0$.
\end{lemma}
The proof of \eqref{eq:conv-429} in \cite{taubes4} cannot be used verbatim to get $T$ independent bounds on \eqref{eq:conv-429}'s number $c_\ast$ because the arguments in \cite{taubes4} use the fact that the 3-manifold in question, in this case $Y_T$, has a given finite volume; and these arguments also use a priori bounds for the Green's function of the operator $d+d^\dagger$. The volume issue must be dealt with in any event, and it is not likely that there is a $T$ independent bound for the relevant Green's function on $Y_T$. As explained below, the assumption on $\pze$ in Lemma \ref{lem:conv-44} circumvents both of these issues because this assumption implies that $B_A$ for large $r$ (given $\frE$) is concentrated for the most part very near to $M$. In particular, the proof uses the following:
\begin{quote}
Suppose that $\frE<\infty$ and that $\pze(\cdot)\leq \frE$ on an interval of the form $[\tau,\infty)$. If $r$ is large with lower bound depending only on $\frE$, if $T$ is sufficiently large with a lower bound depending on $\frE$, and if $s\in[\tau+32,\infty)$, then $|\alpha|_s|>1-\frac{1}{1000}r^{-1}$ where $\operatorname{dist}(\cdot, M)>c_0$.
\begin{equation}
\label{eq:conv-430}
\end{equation}
\end{quote}
This follows from Proposition \ref{prop:conv-41}, Lemmas \ref{lem:bounds-35} and \ref{lem:bounds-36}, and \eqref{eq:conv-47}, and \eqref{eq:conv-428} because a pseudoholomorphic curve that extends out of $M$ must have area that is $\mathcal{O}(T)$. Note in this regard that Proposition \ref{prop:conv-41} and Lemmas \ref{lem:bounds-35} and \ref{lem:bounds-36} can be invoked with their versions of the number $c$ bounded a priori via Lemma \ref{lem:bounds-33} because of the assumption that $\mathcal{A}_\frd<r^2$ in Lemma \ref{lem:conv-44}.
\begin{proof}[Proof of Lemma \ref{lem:conv-44}]
The proof has six steps.
\begin{sp3}
\label{step1:lemma44}
The rightmost integral in the $\frc'=\frc_+$ version of \eqref{eq:bounds-31} is zero because $\frc_+$ obeys \eqref{eq:sw}. The next to rightmost integral is that of $r\langle\psi,\mathcal{D}_A\psi\rangle$. To bound this integral, invoke \eqref{eq:conv-430} to find $R\in[16,c_0)$ so that $|\alpha|>1-\frac{1}{1000}r^{-1}$ where both $s>\tau+32$ and the distance to $M$ is greater than $R$. Let $\sigma$ denote a smooth function on $Y_T$ that is equal to $1$ where the distance to $M$ is less than $R+5$ and equal to $0$ where the distance to $M$ is greater than $R+6$. This function can be assumed to have derivative bounded by $c_0$ and to be independent of $T$. The absolute value of the integral of $r\sigma\langle\psi,\mathcal{D}_A\psi\rangle$ is no greater than $c_0 r^{1/2}\sco^{1/2}$. This is because $\sco$ is greater than the integral of $r|\mathcal{D}_A\psi|^2$.

An $c_z\frE e^{-\sqrt{r}/c_z}$ bound for the norm of the integral of $r(1-\sigma)\langle\psi,\mathcal{D}_A\psi\rangle$ is obtained using a strategy much like that used in Part \ref{part1:prop42}. As in Step \ref{step2part1:prop42} of Part \ref{part1:prop42}, let $I_k$ for $k\in\Z$ denote $[s+kr^{-1/2}, s+(k+1)r^{-1/2}]$, this being an interval in $\R$ of length $r^{-1/2}$. For $k\in\Z$, let $Z_k$ denote the part of $\{s\}\times Y_T$ where $(1-\sigma)>0$ and where the closest point to $U$ comes from $I_k\times Y_T$. Let $\scp_k$ denote the integral of $|\textup{D}_A\psi|$ on $Z_k$. Thus,
\begin{equation}
\label{eq:conv-431}
|\int_{\{s\}\times Y_T}(1-\sigma)\langle\psi,\textup{D}_A\psi\rangle|\leq c_0\sum_{k\in\Z}\scp_k.
\end{equation}
Let $\scq_{I_k}$ denote the integral of $r(1-|\alpha|^2)$ over the part of $U$ in $I_k\times Y_T$. The arguments in Step \ref{step2part1:prop42} of Part \ref{part1:prop42} for \eqref{eq:conv-412} can be repeated almost verbatim to see that
\begin{equation}
\label{eq:conv-432}
\scp_k\leq c_z\scq_{I_k}e^{-|k|/c_z}.
\end{equation}
Step \ref{step4part1:prop42} of Part \ref{part1:prop42} proves that $\scq_{I_k}\leq c_z(r^{-1/2}\en(\frd|_s)+rk)$, and so \eqref{eq:conv-432} implies that
\begin{equation}
\label{eq:conv-433}
\scp_k\leq c_z(r^{-1/2}\en(\frd|_s)+r|k|)e^{-|k|/c_z}.
\end{equation}
The preceding bound is useful when $|k|\geq r^{1/2}$. When $|k|\leq r^{1/2}$, then the distance from $s$ to $I_k$ is less than $2$ but the distance from $Z_k$ to $U_k$ is greater than $c_0$. It follows in this case from Lemma \ref{lem:bounds-36} that
\begin{equation}
\label{eq:conv-434}
\scp_k\leq c_z e^{-\sqrt{r}/c_z}(\en(\frd|_s)+|k|).
\end{equation}
The bounds in \eqref{eq:conv-433} and \eqref{eq:conv-434} lead directly to a $c_z\frE e^{-\sqrt{r}/c_z}$ bound for the right hand side of \eqref{eq:conv-431}.
\end{sp3}
\begin{sp3}
\label{step2:lemma44}
The integral of $-\frac{i}{2}r B_A\wedge\ast a$ that appears in \eqref{eq:bounds-31} is by definition $-\frac{1}{2}r\pze(\frd|_s)$ and the integral of $\frac{i}{2}r B_{A_+}\wedge\ast a$ that appears in the $\frc'=\frc_+$ version of \eqref{eq:bounds-31} is $\frac{1}{2}r\en(\frc_+)$ in the case that $a=a_+$; but it is bounded in any case by $c_0r\cale$.
\end{sp3}
\begin{sp3}
\label{step3:lemma44}
This step and the next make preliminary observations that are used subsequently to analyze the remaining term in the formula for $\fra_s(\frd|_s,\frc_+)$, this being 
\begin{equation}
\label{eq:conv-435}
-\frac{1}{2}\int_{\{s\}\times Y_T}(A-A_+)\wedge\ast(B_A+B_{A_+}).
\end{equation}
The latter is denoted in what follows by $\frcs(A,A_+)$.

Keep in mind that there exists $R_0\in (1,c_0)$ and a map from $Y_T\ssm M$ to $S^1$, this denoted by $u_+$ such that the $E$ summand component in $\spb$ of $u_+\psi_+$ can be written on the $\operatorname{dist}(\cdot,M)>R_0$ part of $Y_T$ as $1-z_+$ with $z_+$ being real and having norm less than $\frac{1}{1000}$. This is relevant for $\frd|_s$ when $s\in[\tau+32,\infty)$ for the following reason: it follows from \eqref{eq:conv-430} that there exists a map $u_s$ from $Y_T\ssm M$ to $S^1$ such that $u_+$ and $u_s$ are homotopic, and there exists $R'=R_0+c_0$ such that the $E$ component of $u_s\psi|_s$ can be written on the part of $Y_T$ where $\operatorname{dist}(\cdot,M)>R'$ as $1-z_s$ with $z_s$ and $|z_s|$ at most $\frac{1}{1000}$. Moreover, when $T=\infty$, the map $u=u_s{u_+}^{-1}$ obeys $\lim_{\operatorname{dist}(\cdot,M)\to\infty}(|u-1|+|du|)=0$.
\end{sp3}
\begin{sp3}
\label{step4:lemma44}
Fix a smooth connection on the $Y_T$ version of $E$, to be denoted by $A_{00}$, with two salient properties: The first property is that $iF_{A_{00}}=\upomega_0$; the second property is that $A_{00}-{u_+}^{-1}du_+=\Ai$ on $Y_T\ssm M$. For the purposes of this proof, let $\frc_0$ denote $(A_{00},0)$, this being a pair of connection on $E$ and section of $\spb$.

With $A_{00}$ in hand, note that
\begin{equation}
\label{eq:conv-436}
\frcs(A,A_+)=-\frac{1}{2}\int_{\{s\}\times Y_T}(A-A_{00})\wedge\ast (B_A+B_{A_{00}})+\frac{1}{2}\int_{\{s\}\times Y_T}(A_+-A_{00})\wedge\ast (B_{A_+}+B_{A_{00}})
\end{equation}
because if $A$, $A'$, and $A''$ are any three connections on $E$, then $\frcs(A,A')-\frcs(A'',A')$ is equal to $\frcs(A,A')$. The leftmost term on the right hand side of \eqref{eq:conv-436} is rewritten as the sum of two integrals:
\begin{equation}
\label{eq:conv-437}
-\frac{1}{2}\int_{\{s\}\times Y_T}(A-A_{00})\wedge\ast (B_A-B_{A_{00}})-\int_{\{s\}\times Y_T}(A-A_{00})\wedge\ast B_{A_{00}}.
\end{equation}
The rightmost term on the right hand side of \eqref{eq:conv-436} can likewise be written as 
\begin{equation}
\label{eq:conv-438}
\frac{1}{2}\int_{\{s\}\times Y_T}(A_+-A_{00})\wedge\ast (B_{A_+}-B_{A_{00}})+\int_{\{s\}\times Y_T}(A_+-A_{00})\wedge\ast B_{A_{00}}.
\end{equation}
The sum of \eqref{eq:conv-437} and \eqref{eq:conv-438} can be written as
\begin{equation}
\label{eq:conv-439}
\frcs_\#(A,A_{00})-\frcs_\#(A_+,A_{00})+\frq
\end{equation}
with $\frq$ as defined in the statement of Lemma \ref{lem:conv-44}, with $\frcs_\#(A,A_{00})$ denoting the leftmost term in \eqref{eq:conv-437} and with $\frcs_\#(A_+,A_{00})$ denoting $-1$ times the leftmost term in \eqref{eq:conv-438}.

A bound for the absolute value of $\frcs_\#(A,A_{00})$ is derived below. The $s\to\infty$ limit of the latter bound is a bound for the absolute value of $\frcs_\#(A_+,A_{00})$. 
\end{sp3}
\begin{sp3}
\label{step5:lemma44}
This step starts the story for the leftmost version of $\frcs_\#$ in \eqref{eq:conv-439}, this being the leftmost term in \eqref{eq:conv-437}. To this end, write $A-{u_s}^{-1}du_s$ on the $\operatorname{dist}(\cdot, M)>R'$ part of $Y_T$ as $\Ai+\sca_\infty$. The $1$-form $\sca_\infty$ is bounded by $c_0|\nabla_A\alpha|$. More to the point, Lemma \ref{lem:bounds-36} can be employed to bound $\sca_\infty$ where the distance to $M$ is greater than $c_0$. In particular, the bounds from Lemma \ref{lem:bounds-36} can be used to find $R_\diamond=R'+c_0$ such that
\begin{equation}
\label{eq:conv-440}
\int_{\operatorname{dist}(\cdot,M)\geq R_\diamond}|\sca_\infty||d\sca_\infty|\leq c_0\frE e^{-\sqrt{r}/c_z}.
\end{equation}
The detailed arguments that lead to \eqref{eq:conv-440} are much like those in Step \ref{step1:lemma44}. Since $A_{00}$ on $Y_T\ssm M$ is such that $A_{00}-{u_+}^{-1}du_+=\Ai$, it follows that $A-A_{00}=\sca_\infty-u^{-1}du$ with $u=u_s {u_+}^{-1}$. Because $u$ is null-homotopic, it can be written as $d\frf_\infty$ with $\frf_\infty$ being and $i\R$-valued function of $Y_T\ssm M$. In the case when $T=\infty$, it follows from what was said at the end of Step \ref{step3:lemma44} that $\frf_\infty$ can be chosen so that $\lim_{\operatorname{dist}(\cdot,M)\to\infty}(|\frf_\infty|+|d\frf_\infty|)=0$.

Now let $\Omega\subset Y_T$ denote a compact, codimension-$0$ submanifold with boundary that contains $M$ and is homeomorphic to $M$. We will assume that the function $\operatorname{dist}(\cdot, M)$ on $\partial\Omega$ is greater that $R_\diamond +10$ but less than $R_\diamond +11$, and that $\partial\Omega$ is homeomorphic to $\partial M$. The manifold $\Omega$ can be constructed in a straightforward way using what is said in Section \ref{ssec:setup}. Hodge theory on manifolds with boundary can be used to construct a smooth $i\R$-valued function on $\Omega$ to be denoted by $\frf$ such that $A$ on $\Omega$ can be written as $A_{00}-d\frf+\sca_M$ with $\ast\sca_M$ being a closed form that pulls back as $0$ to $\partial\Omega$. In turn, the $i\R$-valued $1$-form $\sca_M$ can be written as $\hat{\sca}_M+\upnu$ with $\upnu$ obeying $d\upnu=0$, $d\ast\upnu=0$ and $\ast\upnu$ pulling back as $0$ to $\partial\Omega$; and with $\hat{\sca}_M$ obeying
\begin{equation}
\label{eq:conv-441}
\int_{\{p\in\Omega\,:\,\operatorname{dist}(p,\partial M)\geq1\}}|\hat{\sca}_M||d\hat{\sca}_M|\leq c_0(\sco+r^{2/3}\en^{4/3}).
\end{equation}
The Green's function for the operator $\ast d+d\ast$ can be used to obtain the bounds in \eqref{eq:conv-441}.

The $1$-form $\upnu$ defines a class in $H^1(\Omega;i\R)$. As such, $\upnu$ can be written as $\mathrm{u}^{-1}d\mathrm{u}+\upnu_\ast$ with $\mathrm{u}$ being a harmonic map from $\Omega$ to $S^1$ and with $\upnu_\ast$ being a harmonic $i\R$-valued $1$-form that obeys $|\upnu_\ast|\leq c_0$. Let $u_\Omega$ denote $e^\frf \mathrm{u}$, this being a map from $\Omega$ to $S^1$. 

Both of the $i\R$-valued $1$-forms $\hat{\sca}_M+\upnu$ and $\sca_\infty-u^{-1}du$ are defined on the part of $\Omega$ where $\operatorname{dist}(\cdot,M)\in(R_\diamond,R_\diamond+10)$. They differ here according to the rule 
\begin{equation}
\label{eq:conv-4.42}
\sca_\infty-u^{-1}du=\hat{\sca}_M+\upnu-d\frh
\end{equation}
with $\frh=\frf-\frf_\infty$ being an $i\R$-valued function. Meanwhile, $\hat{\sca}_M+\upnu_\ast$ and $\sca_\infty$ differ where $\operatorname{dist}(\cdot,M)\in(R_\diamond,R_\diamond+10)$ according to the rule $\sca_\infty=\hat{\sca}_M+\upnu_\ast+h^{-1}dh$ with $h=u {u_\Omega}^{-1}$. being a map to $S^1$. Note that this requires that the norm $|dh|$ obey $|dh|\leq |\hat{\sca}_M|+|\upnu_\ast|+|\sca_\infty|$.
\end{sp3}
\begin{sp3}
\label{step6:lemma44}
To see about the norm of $\frcs_\#(A,A_{00})$, let $\sigma_\diamond$ denote a smooth non-negative function that is equal to $1$ where the distance to $M$ is less than $R_\diamond +5$ and equal to zero where the distance to $M$ is greater than $R_\diamond+6$. Use this function to write $\frcs_\#(A,A_{00})$ as
\begin{equation}
\label{eq:conv-443}
-\int_{\{s\}\times Y_T}\sigma_\diamond(A-A_{00})\wedge\ast (B_A-B_{A_{00}})-\int_{\{s\}\times Y_T}(1-\sigma_\diamond)(A-A_{00})\wedge\ast(B_A-B_{A_{00}}).
\end{equation} 
It follows from this and what is said in Step \ref{step5:lemma44}, and from \eqref{eq:conv-440} and \eqref{eq:conv-441} that
\begin{eqnarray}
\label{eq:conv-444}
\nonumber|\frcs_\#(A,A_{00})|\leq c_0(\sco+r^{2/3}\en^{4/3})+c_z\frE e^{-\sqrt{r}/c_z}&-&\frac{1}{2}\int_{Y_T}\sigma_\diamond {u_\Omega}^{-1}du_\Omega\wedge\ast(B_A-B_{A_{00}})\\
\nonumber &-&\frac{1}{2}\int_{Y_T}(1-\sigma_\diamond)u^{-1}du\wedge\ast B_A.\\
\end{eqnarray}
Keeping in mind that $\ast(B_A-B_{A_{00}})$ is exact, and that $A_{00}$ is the product connection on $Y_T\ssm M$, integration by part writes the contribution from two integrals on the left hand side of \eqref{eq:conv-444} as 
\begin{equation}
\label{eq:conv-445}
-\frac{1}{2}\int_{Y_T}d\sigma_\diamond\wedge h^{-1}dh\wedge (A-A_{00}).
\end{equation}
Since the map $u=u_s{u_+}^{-1}$ is homotopic to the identity, the integral in \eqref{eq:conv-445} is the same as 
\begin{equation}
\label{eq:conv-446}
-\frac{1}{2}\int_{Y_T}d\sigma_\diamond\wedge h^{-1}dh\wedge\sca_\infty.
\end{equation}
What is said in Step \ref{step5:lemma44} about $\sca_M$ and $dh$ leads to a $c_0(\sco+r^{2/3}\en^{4/3})+c_z\frE e^{-\sqrt{r}/c_z}$ bound on the integral in \eqref{eq:conv-446}, and thus a $c_0(\sco+r^{2/3}\en^{4/3})+c_z\frE e^{-\sqrt{r}/c_z}$ bound for $|\frcs_\#(A,A_0)|$.
\end{sp3}
\end{proof}
\end{pt2}
\begin{pt2}
\label{part4:prop42}
If $T$ is finite and if $a$ and the metric are independent of the coordinate $s$ (and if $\frp=0$ in \eqref{eq:instanton-sw}), then \cite[Section 5b]{taubes4} uses \eqref{eq:conv-429} to derive an inequality for $\upze(s-1)$ that reads
\begin{equation}
\label{eq:conv-447}
r\upze(s-1)\leq -c_\ast\fra(u\frc_+)+c_\ast(r+\fra(\frc_-)-\fra(\frc_+))+c_\ast r^{2/3}\operatorname{sup}_{x\geq s}\upze(x)^{4/3}+c_\ast \operatorname{sup}_{[s,s+1]}|\frq_u|,
\end{equation}
with each occurrence of $c_\ast$ denoting a constant greater than $1$. This part explains how Lemma \ref{lem:conv-44} can be used in lieu of \eqref{eq:conv-429} to obtain a similar inequality with $c_\ast$ being independent of $T$ in the case when $a$ and the metric are independent of $s$. This part also explains how the upcoming analog of \eqref{eq:conv-447} leads to the desired bound on $\operatorname{sup}_{s\in\R}\upze(s)$. The case when $a$ or the metric depends on $s$ requires additional modifications (even when $T$ is finite) and so it is deferred for the moment.

The analog of \eqref{eq:conv-447} is derived using the assumption that $\en(\frc_+)\leq\cale$ for some given number $\cale$. The analog then holds if $r$ and $T$ are greater than $c_\cale$. Keeping in mind that $\en(\frc_+)\leq\cale$, suppose that $\frE>100\cale$ has been chosen. Let $\tau$ denote the largest value of $s\in\R$ such that $\operatorname{sup}_{s\geq\tau}\upze(s)\leq\frE$. Lemma \ref{lem:conv-44} can be invoked assuming that $r>\kappa_\frE$. To see what Lemma \ref{lem:conv-44} implies, it proves useful to introduce 
\begin{equation}
\label{eq:conv-448}
\usco(s)=\int_{[s,s+1]\times Y_T}(|E_A|^2+|\calB|^2+2r(|\nabla_{A,s}\psi|^2+|\textup{D}_A\psi|^2)).
\end{equation}
It follows from \eqref{eq:conv-424} and \eqref{eq:conv-425} that if $s\geq\tau+200$, then Lemma \ref{lem:conv-44} leads to the bound
\begin{equation}
\label{eq:conv-449}
r\upze(s-1)\leq-2\fra(\frd|_s,\frc_+)+c_0(r+\usco(s)+r^{2/3}\frE^{4/3}+r\cale+\operatorname{sup}_{[s,s+1]}|\frq|).
\end{equation}

Now the fact that the function $s\to \fra(\frd|_s,\frc_+)$ is \emph{decreasing} when $a$ and the metric are independent of $s$, and in particular what is said in Lemma \ref{lem:bounds-31} in this case can be invoked to deduce the following two crucial inequalities:
\begin{equation}
\label{eq:conv-450}
0\leq \fra(\frd|_s,\frc_+)\;\mathrm{and}\;\usco(s)\leq\fra(\frc_-,\frc_+)=\mathcal{A}_\frd.
\end{equation}
Supposing that $\mathcal{A}_\frd\leq zr$, then \eqref{eq:conv-449} and \eqref{eq:conv-450} imply that 
\begin{equation}
\label{eq:conv-451}
\upze(s-1)\leq c_0(z+\cale)+c_0r^{-1/2}\frE^{4/3}+c_0 r^{-1}\operatorname{sup}_{[s,s+1]}|\frq|
\end{equation}
where $r\geq\kappa_\frE$ and $s\in[\tau+100,\infty)$. This is the analog of \eqref{eq:conv-447}.

To see how this leads to a bound for $\operatorname{sup}_{s\in\R}\upze(s)$, invoke \eqref{eq:conv-451} for $s=\tau+200$. Since $\upze(\tau)=\frE$, the inequality in \eqref{eq:conv-424} implies that $\upze(s-1)\geq {c_0}^{-1}\frE$ and so \eqref{eq:conv-451} implies that 
\begin{equation}
\label{eq:conv-452}
\frE\leq c_0(\cale+z)+c_0 r^{-1/3}\frE^{4/3}+c_0 r^{-1}\operatorname{sup}_{[s,s+1]}|\frq|.
\end{equation}
Now there are two cases to consider, these distinguished by whether or not the first chern class of $\spb$ is torsion in $H^2_c(M)$. Supposing that this class is torsion, then the $2$-form $\upomega_0$ can be taken to be equal to $0$ and so the $\frq$ term in \eqref{eq:conv-452} is absent. In this case, \eqref{eq:conv-452} leads to 
\begin{equation}
\label{eq:conv-453}
\frE\leq c_0(\cale+z)+c_0 r^{-1/3}\frE^{4/3}
\end{equation}
which implies that
\begin{equation}
\label{eq:conv-454}
\mathrm{If}\;\;\frE\leq {c_0}^{-1},\;\mathrm{then}\;\;\frE\leq c_0(\cale+z).
\end{equation}
It follows as a consequence that $\upze(s)\leq c_0(\cale+z)$ for all $s\in\R$ if $r$ is greater than a constant that depends only on $\cale$. 

Some extra steps must be used to draw this same conclusion when $det(\spb)$ is not torsion in $H^2_c(M)$. These extra steps employ \eqref{eq:conv-430} again to localize the analysis to the part of $Y_T$ near $M$. Except for this and one other item discussed directly, the arguments differ little from what is done in \cite[Section 5d]{taubes4}. (More is said about this in Step \ref{step2part6:prop42} of Part \ref{part6:prop42}.) The additional item concerns the freedom to change $\upomega_0$ by adding the differential of a $1$-form with compact support in $M$. In particular, suppose that $\upupsilon_0$ is such a $1$-form and let $\upomega_0'=\upomega_0+d\upupsilon_0$. The differential forms $\upomega_0$ and $\upomega_0'$ represent the same class in $H^2_c(M)$. Let $\frq'$ denote the $\upomega_0'$ version of $\frq$. Since $\upupsilon_0$ has compact support on $M$, integration by parts writes
\begin{equation}
\label{eq:conv-455}
\frq'-\frq=\int_{\{s\}\times Y_T}(B_A-B_{A_+})\wedge\ast\upupsilon_0.
\end{equation}
Note that it follows in any event from Lemma \ref{lem:bounds-34} that $|\frq'-\frq|\leq c_z \operatorname{sup}_M|\upupsilon_0|r$.
 \end{pt2}
\begin{pt2}
\label{part5:prop42}
This part of the proof derives certain bounds that are used in Part \ref{part6:prop42} to deal with the case when $a$ and/or the metric depend on the coordinate $s$ for the $\R$ factor in $\R\times Y_T$ (and/or when the perturbation term $\frp$ is present in \eqref{eq:instanton-sw}). As is shown in Part \ref{part6:prop42}, an inequality that is analogous to the one in \eqref{eq:conv-452} holds in this case provided that the norm of $\frp$ and its covariant derivatives is less than $c_0^{-1}$, and provided that
\begin{equation}
\label{eq:conv-456}
|\frac{\partial}{\partial s}a|+|\frac{\partial}{\partial s}da|\leq c^{-1}\;and\;c^{-1}s_0<c_0^{-1}
\end{equation}
with $c$ being greater than $1$ but less than $c_0$. By way of a reminder, $s_0$ is defined by the requirement that $\frac{\partial}{\partial s}a\neq 0$ only on the $[-s_0,s_0]\times Y_T$ part of $\R\times Y_T$. This implies that \eqref{eq:conv-456} can be satisfied if $|a_+-a_-|\leq c^{-1}{c_0}^{-1}$.

The derivation in Part \ref{part4:prop42} makes fundamental use of the inequalities that are asserted by Lemma \ref{lem:bounds-31} in the case when $a$ and the metric are independent of $s$. This application of Lemma \ref{lem:bounds-31} is summarized by \eqref{eq:conv-450}. But, a verbatim repetition of the derivation in Part \ref{part4:prop42} runs afoul of the inequality in the third bullet of the lemma if either $a$ or the metric depends on $s$. This is so even when $T$ is finite. The offending terms are
\begin{equation}
\label{eq:conv-457}
ir\int_{\{s\}\times Y_T}((B_A-B_{A_+})\wedge\ast\frac{\partial}{\partial s}a)+r\int_{I\times Y_T}\langle\psi,\mathcal{R}\psi\rangle - \frac{1}{4}\int_{I\times M} |\frp|^2.
\end{equation}
These terms in \eqref{eq:conv-457} cause problems because the arguments in Part \ref{part4:prop42} explicitly use the fact that the function given by the rule $s\mapsto\fra(\frd|_s,\frc_+)$ is decreasing when $a$ and the metric are independent of $s$ and when $\frp$ is absent.  But noting that the term with $|\frp|^2$ is bounded from below in any event by $-c_0$ times the sup norm of $|\frp|$, the previous arguments could be repeated with only cosmetic changes to accommodate the presence of $\frp$.  More needs to be said about the other two terms in \eqref{eq:conv-457}.

The second term in \eqref{eq:conv-457} is bounded by $c_0r$ because $\mathcal{R}$ has compact support on $M$. Denote the leftmost term in \eqref{eq:conv-457} by $r\scq$. To deal with $r\scq$, it is useful to introduce the function $\frP(|\alpha|^2)$ from \eqref{eq:flatconn}. Let
\begin{equation}
\label{eq:conv-458}
\frb=\frac{1}{2}\frP(|\alpha|^2)|\alpha|^{-2}(\overline{\alpha}\nabla_A\alpha-\alpha\nabla_A\overline{\alpha}).
\end{equation}
Noting that
\begin{equation}
\label{eq:conv-459}
d\frb=\frP F_A+\frP'\nabla_A\overline{\alpha}\wedge\nabla_A\alpha,
\end{equation}
write $F_A=(1-\frP)F_A+d\frb-\frP'\nabla_A\overline{\alpha}\wedge\nabla_A\alpha$ and then write the offending term 
\begin{eqnarray}
\label{eq:conv-460}
\nonumber r\scq&=&\frac{i}{2}r\int_{\{s\}\times Y_T}((1-\frP)B_A\wedge\ast\frac{\partial}{\partial s}a)+\frac{i}{2}r\int_{\{s\}\times Y_T}(\frb\wedge\frac{\partial}{\partial s}da)\\
\nonumber&&\hspace{0.5in}+\frac{i}{2}r\int_{\{s\}\times Y_T}(\frP'\nabla_A\overline{\alpha}\wedge\nabla_A\alpha\wedge\frac{\partial}{\partial s}a)-\frac{i}{2}r\int_{\{s\}\times Y_T}(B_{A_+}\wedge\ast\frac{\partial}{\partial s}a).\\
\end{eqnarray}
The norm of the rightmost term in \eqref{eq:conv-460} is bounded by $c_0r$. To deal with the remaining terms, we need an independent bound for the integral of $|\nabla_A\psi|^2$ on sets of the form $I\times U$ with $I\subset\R$ being an interval of length $1$ and with $U\subset Y_T$ being an open subset with compact closure. Of particular interest is when each point in $U$ has distance at most $4$ from $M$. Such a bound comes from the Bochner--Weitzenb\"ock formula using the rightmost equation in \eqref{eq:sw}. What with the bounds in Lemma \ref{lem:bounds-32}, integrating this formula leads to a bound of the form
\begin{equation}
\label{eq:conv-461}
\int_{I\times Y_T}{\sigma_\diamond}^2(|\nabla_{A,s}\psi|^2+|\nabla_A\psi|^2)\leq c_0(1+r\int_{I\times Y_T}{\sigma_\diamond}^2|1-|\alpha|^2|).
\end{equation}

To exploit this bound, use the formula in \eqref{eq:conv-460} now to obtain the bound
\begin{equation}
\label{eq:conv-462}
r|\int_I\scq(s)ds|\leq c_0c^{-1}(1+r\int_{I\times Y_T}{\sigma_\diamond}^2(1-\frP)|B_A|+r^2\int_{I\times Y_T}{\sigma_\diamond}^2|1-|\alpha|^2|).
\end{equation}
write $I$ as $[s_-,s_+]$. Since $|B_A|\leq |\calB|+c_0r|1-|\alpha|^2|+c_0$, this last inequality can be used to replace the third bullet of Lemma \ref{lem:bounds-31} by
\begin{eqnarray}
\label{eq:conv-463}
\nonumber \fra_{s_-}(\frd|_{s_-},\frc_+)-\fra_{s_+}(\frd|_{s_+},\frc_+)&\geq&\frac{1}{8}\int_{I\times Y_T}(|E_A|^2+|\calB|^2+2r(|\nabla_{A,s}\psi|^2+|\textup{D}_A\psi|^2))\\
\nonumber&&\hspace{1in}-c_0r^2c^{-1}\int_{I\times Y_T}{\sigma_\diamond}^2|1-|\alpha|^2|-c_0r.\\
\end{eqnarray}
Now use what is said after \eqref{eq:conv-47} with the definitions in \eqref{eq:functionL} and \eqref{eq:conv-44} replace this by
\begin{eqnarray}
\label{eq:conv-464}
\nonumber \fra_{s_-}(\frd|_{s_-},\frc_+)-\fra_{s_+}(\frd|_{s_+},\frc_+)&\geq&\frac{1}{2}\int_{I\times Y_T}(|E_A|^2+|\calB|^2+2r(|\nabla_{A,s}\psi|^2+|\textup{D}_A\psi|^2))\\
\nonumber&&\hspace{2in}-c_0r(c^{-1}\uscL+1),\\
\end{eqnarray}
and then use \eqref{eq:conv-422} and \eqref{eq:conv-425} to go from \eqref{eq:conv-462} to the bound
\begin{eqnarray}
\label{eq:conv-465}
\nonumber \fra_{s_-}(\frd|_{s_-},\frc_+)-\fra_{s_+}(\frd|_{s_+},\frc_+)&\geq&\frac{1}{2}\int_{I\times Y_T}(|E_A|^2+|\calB|^2+2r(|\nabla_{A,s}\psi|^2+|\textup{D}_A\psi|^2))\\
\nonumber&&\hspace{1.5in}-c_0r(c^{-1}\operatorname{sup}_{s\in I}\upze(s)+1).\\
\end{eqnarray}
Let $x$ now denote the smaller of ${c_0}^{-1}r$ and $\operatorname{sup}_{s\in[-s_0-100,s_0+100]}\upze(s)$. The bounds in \eqref{eq:conv-463} and \eqref{eq:conv-465} for $s\in[-s_0-100,s_0+100]$ imply two useful inequalities:
\begin{eqnarray}
\label{eq:conv-466}
\nonumber\fra_{-s_0-100}(\frd|_{-s_0-100},\frc_+)-\fra_s(\frd|_s,\frc_+)&\geq&-\kappa_0 c^{-1}rx,\\
\fra_s(\frd|_s,\frc_+)-\fra_{s_0+100}(\frd|_{s_0100},\frc_+)&\geq&-\kappa_0 c^{-1}rx,
\end{eqnarray}
with $\kappa_0$ being a positive number that depends only on $s_0$. Taking $s=s_-$ in the top inequality and $s=s_+$ in the bottom inequality, these inequalities with those in the first two bullets of Lemma \ref{lem:bounds-31} imply in turn that 
\begin{equation}
\label{eq:conv-467}
\fra_{s_-}(\frd|_{s_-},\frc_+)-\fra_{s_+}(\frd|_{s_+},\frc_+)\leq r(z+\kappa_0c^{-1}x),
\end{equation}
when $\mathcal{A}_\frd\leq rz$. This bound and \eqref{eq:conv-465} lead in turn to a bound of the form
\begin{equation}
\label{eq:conv-468}
\usco(s)\leq\kappa_0r(z+c^{-1}x+1)\;\mathrm{when}\;s\in[-s_0-10,s_0+10],
\end{equation}
with $\kappa_0$ denoting here and in what follows a number that depends only on $s_0$. In particular, it does not depend on $T$.
\end{pt2}
\begin{pt2}
\label{part6:prop42}
This part of the proof derives an analog of the inequality in \eqref{eq:conv-452} in the case when $a$ and/or the metric depend on the $s$ coordinate. It uses this analog to draw the analogous conclusion: There is an $r$ and $T$ independent bound on $\upze$. The arguments in this case are presented in four steps.
\begin{sp4}
\label{step1part6:prop42}
Fix $m>c_0$. Let $\tau_1$ denote the largest $s\in\R$ such that $\fra_s(\frd_s,\frc_+)\geq mr$. Suppose first that $\tau_1\geq s_0$. Given $\frE>100\cale$, either $\upze(s)\leq\frE$ on $[\tau_1,\infty)$ or not. Suppose for the sake of argument that $\upze(s)>\frE$ on this interval for some $s$. Let $\tau$ denote the largest value of $s$ with this property. Since $\tau>s_0$, the second bullet of Lemma \ref{lem:bounds-31} can be invoked to replace \eqref{eq:conv-450} by
\begin{equation}
\label{eq:conv-469}
0\leq \fra(\frd|_s,\frc_+)\;\;\;\mathrm{and}\;\;\;\usco(s)\leq\fra_{\tau_1}(\frd|_{\tau_1},\frc_+)\leq r(m+c_0\cale).
\end{equation}
This inequality and Lemma \ref{lem:conv-44} lead to the analog below of \eqref{eq:conv-452}:
\begin{equation}
\label{eq:conv-470}
\frE\leq c_0(m+\cale)+c_0r^{-1/3}\frE^{4/3}+c_0r^{-1}\operatorname{sup}_{[s,s+1]}|\frq|.
\end{equation}
In the case when $c_1(det(\spb))$ is a torsion class in $H^2_c(M)$, then there is no $|\frq|$ term and this inequality leads to the following analog of \eqref{eq:conv-453}:
\begin{equation}
\label{eq:conv-471}
\mathrm{If}\;\frE\leq {c_0}^{-1}r,\;\mathrm{then}\;\frE\leq c_0(\cale+m).
\end{equation}
Arguments that employ \eqref{eq:conv-430} but otherwise differ little from what is done in \cite[Section 5d]{taubes4} can be used to prove that \eqref{eq:conv-470} also holds when $c_1(det(\spb))$ is not torsion. More is said about this in the next step.
\end{sp4}
\begin{sp4}
\label{step2part6:prop42}
To see about the size of the number $\tau_1$, suppose for the sake of argument that $\tau_1\geq s_0+100$. It then follows from \eqref{eq:conv-424} that $\upze(s)\leq c_0(\cale+m)$ on $[\tau_1-100,\infty)$. This implies in turn that $\scm(s)\leq c_0(\cale+m)$ on $[\tau_1-50,\infty)$. The latter bound can be used to invoke Lemma \ref{lem:bounds-35} on $[\tau_1-40,\infty)$. As explained directly, there is a tension between the bounds in Lemma \ref{lem:bounds-35} and the fact that $\fra_{\tau_1}(\frd|_{\tau_1},\frc_+)=mr$.

This tension plays out via a modified version of the bounds given in Lemma \ref{lem:conv-44}. In particular, the bounds in Lemma \ref{lem:bounds-35} for $|\nabla_A\psi|$ can be used to replace the bound on the contribution to $\fra$ from the integral $r\langle\psi,\textup{D}_A\psi\rangle$ that is derived in Step \ref{step1:lemma44} of the proof of Lemma \ref{lem:conv-44} by $c_0r^{1/2}\en^{1/2}+c_0\frE e^{-\sqrt{r}/c_z}$. Meanwhile, the bounds in Lemma \ref{lem:bounds-35} imply bounds for $|B_A|$ in terms of $r(1-|\alpha|^2)$; and these can be used in Step \ref{step5:lemma44} of the proof of Lemma \ref{lem:conv-44} to replace the right hand side of \eqref{eq:conv-441} by $c_0r^{2/3}\en^{4/3}$. These replacements then lead to the inequality
\begin{equation}
\label{eq:conv-472}
\fra_s(\frd|_s,\frc_+)\leq c_0(1+\pze+r^{1/2}\pze^{1/2}+r^{2/3}\pze^{4/3}+r\cale)-\frac{1}{2}r\pze+|\frq|
\end{equation}
which holds when $s\in[\tau_1-20,\infty)$. Since $\pze\leq c_0(\cale+m)$ and $\fra_s(\frd|_s,\frc_+)=mr$, the $s=\tau_1$ version of this inequality leads to the following conclusions: if $c_1(det(\spb))$ is a torsion class in $H^2_c(M)$ (so there is no $\frq$ term), then
\begin{equation}
\label{eq:conv-473}
m\leq c_0\;\;\mathrm{if}\;\; m\leq{c_0}^{-1}r.
\end{equation}
If $c_1(det(\spb))$ is not torsion, a modification of arguments from \cite[Section 5d]{taubes4} leads to the same conclusion. The next paragraphs say more about this.

To deal with $\frq$, let $\pzq$ denote a number greater than $100$ and let $\tau_2$ denote the largest number $s\in[\tau_1,\infty)$ where $|\frq|=\pzq$. For $s>\tau_2$, there is an a priori bound on $\fra_s(\frd|_s,\frc_+)$ and an a priori bound on $\upze$, first on $[\tau_2+100,\infty)$ and then, by invoking \eqref{eq:conv-424}, on the larger interval $[\tau_2-1--,\infty)$. The inequalities in Lemmas \ref{lem:bounds-35} and \ref{lem:bounds-36} are then used just as their analogs are used in \cite[Section 5d]{taubes4} to prove that $|\frq|\leq c_0\pzq$ for $s\geq \tau_2-100$ if $\pzq\geq c_0\frE$. This is done by using \eqref{eq:conv-455} and the bound $|\frq'-\frq|\leq c_0 \operatorname{sup}_M|\upupsilon_0|\frE$ with $\upupsilon_0$ chosen so that $\upomega_0+d\upupsilon_0$ has support very near a conveniently chosen embedded surface in $M$ that is dual to $c_1(det(\spb))$. Proposition \ref{prop:conv-41} is then invoked on $[\tau_2-100,\infty)$ to view $\frq$ as an approximation to the intersection number between a pseudo-holomorphic curve and this embedded surface. It follows as a consequence that the bound $\pzq$ is not realized on the interval $[\tau_1,\infty)$ unless $\pzq\leq c_0\frE$.
\end{sp4}
\begin{sp4}
\label{step3part6:prop42}
Turning \eqref{eq:conv-473} around, one can conclude that if $c_0(\frE+z)\leq m<{c_0}^{-1}r$ and if $\tau_1\in[s_0+10,\infty)$ is such that $\fra_{\tau_1}(\frd|_{\tau_1},\frc_+)=mr$, then it must be that $\tau_1\leq s_0+100$. This step uses the preceding conclusion to prove that $\operatorname{sup}_{s\in\R}\fra_s(\frd|_s,\frc_+)<(c_z+c_0\cale)r$. To do this, fix for the moment $s\in[-s_0-150, s_0+150]$. Since \eqref{eq:conv-47} bounds $\scL$ from below by $-c_z$, integrating both sides of \eqref{eq:conv-426} over the interval $[s,s+1]$ leads to an inequality for $\upze$ on the interval $[-s_0-150,s_0+150]$ that reads
\begin{equation}
\label{eq:conv-474}
\frac{d}{ds}\upze+2\upze\geq \scm-c_z-|\int_I\scq(s)ds|,
\end{equation}
with $\scq(s)$ being $\int_{\{s\}\times Y_T}((B_A-B_{A_0})\wedge\ast\frac{\partial}{\partial s}a)$. Use Lemma \ref{lem:bounds-34}'s bound for $|B_A|$ by $c_0r$ with \eqref{eq:conv-462} to bound the absolute value of the integral of $\scq$ in \eqref{eq:conv-474} by $c_0c^{-1}\scm$. Granted the latter bound, and granted that $c^{-1}<{c_0}^{-1}$, then \eqref{eq:conv-473} implies that
\begin{equation}
\label{eq:conv-475}
\frac{d}{ds}\upze+2\upze\geq {c_0}^{-1}\scm-c_z,
\end{equation}
for $s\in[-s_0-150, s_0+150]$. Integrating \eqref{eq:conv-475} leads in turn to the bound
\begin{equation}
\label{eq:conv-476}
\upze(s)\leq c_0\upze(-s_0-150)+c_z\leq c_0(\cale+m)+c_z,
\end{equation}
for $s\in[-s_0-150, s_0+150]$.

To bound in \eqref{eq:conv-476} implies that the number $x$ that appears in \eqref{eq:conv-466}--\eqref{eq:conv-468} is no greater than $c_z(\cale+m)$. With this understood, suppose that $s\in[-s_0-100,s_0+100]$ is such that $\fra_s(\frd^{u_+}|_s)=mr$. It then follows from \eqref{eq:conv-466} that
\begin{equation}
\label{eq:conv-477}
\fra_{-s_0-100}(\frd_{-s_0-100},\frc_+)\geq (1-\kappa_0c_0c^{-1})mr-\kappa_0c^{-1}c_0r\cale-\kappa_0c_zr.
\end{equation}
Supposing that $c>\kappa_0c_0$, this implies in turn that
\begin{equation}
\label{eq:conv-478}
\fra_{-s_0-100}(\frd|_{-s_0-100},\frc_+)\geq \frac{1}{2}mr-r\cale-\kappa_0 c_zr.
\end{equation}
The preceding with the first bullet of Lemma \ref{lem:bounds-31} leads to the conclusion that 
\begin{equation}
\label{eq:conv-479}
\fra_-(\frc_-,\frc_+)\geq \frac{1}{2}mr-r\cale-\kappa_0c_zr.
\end{equation}
Now, by assumption, $\mathcal{A}_\frd\leq rz$ which is to say that $\fra_-(\frc_-,\frc_+)\leq rz$. The latter bound with \eqref{eq:conv-479} leads to a bound on $m$ and thus to the conclusion that $\fra_s(\frd|_s,\frc_+)\leq c_0r(c_z+\cale)$ if $s\in[-s_0-100,\infty)$. The top bullet of Lemma \ref{lem:bounds-31} then implies the same for all $s\in\R$. 
\end{sp4}
\begin{sp4}
\label{step4part6:prop42}
Having determined that $\fra_s(\frd|_s,\frc_+)\leq c_0r(c_z+\cale)$ for all $s\in\R$, then the arguments that led to \eqref{eq:conv-469}--\eqref{eq:conv-471} can be repeated to obtain the following conclusion: Suppose that $\frE\geq 100\cale$ and that $\tau_1$ is the largest $s\in\R$ such that $\upze(s)=\frE$. If $r$ is greater than a number that depends only on $\frE$, then
\begin{equation}
\label{eq:conv-480}
\frE\leq c_z+c_0\cale\;\;\mathrm{if}\;\;\frE\leq {c_0}^{-1}r.
\end{equation}
This then gives a uniform bound for $\upze$ for all $s\in\R$. As explained previously, the latter bound is what is needed to complete the proof of Proposition \ref{prop:conv-42}. \qedhere
\end{sp4}
\end{pt2}
\end{proof}

%% file: a_instanton-convergence.tex

This section uses Proposition \ref{prop:conv-42} to prove the instanton analogue of Lemma \ref{lem:monopoles-111}.  To set the notation, fix $T \in (16,\infty]$ for the moment and suppose that $\frd$ is an instanton solution to \eqref{eq:instanton-sw} on $\R \times Y_T$.  Let $\frc_-$ and $\frc_+$ again denote the respective $s \to -\infty$ and $s \to \infty$ limits of $\frd|_s$.  The upcoming proposition uses $\ind$ to denote the spectral flow between the $\frc_-$ and $\frc_+$ versions of the operator that appears in \eqref{eq:linearized}.  Note in this regard that this spectral flow between $\LLf_{\frc_-}$ and $\LLf_{\frc_+}$ is defined when $T$ is finite as in \cite{taubes-weinstein1}.  When $T=\infty$, these operators are still Fredholm, and essentially self-adjoint (by virtue of Lemma~\ref{lem:monopoles-12}).  The spectral flow can be defined as in \cite{taubes-weinstein1} if there exists a 1-parameter family of such operators that interpolate between $\LLf_{\frc_-}$ and $\LLf_{\frc_+}$.  And, there exists such a family when there is an instanton $\frd$ whose respective $s \to \pm\infty$ limits are $\frc_-$ and $\frc_+$.  (The required 1-parameter family of operators exists when there is a smoothly varying family of admissible pairs $\{(A_s,\psi_s)\}_{s\in\R}$ of connection on $E$ and section of $\spb$ (over $Y_\infty$) that interpolate between $\frc_-$ and $\frc_+$ as $s\to\pm\infty$.  An instanton supplies just such a family because of Lemma~\ref{lem:instanton-24}.)

The upcoming proposition also reintroduces the positive number $s_0$ from the beginning of Section~\ref{ssec:instantons}: It is chosen so that the $s$-dependence of $a$ is confined to the $-s_0 < s < s_0$ part of $\R\times Y_T$.

\begin{proposition}
\label{prop:Tconv-51}
There exists $\kappa>1$, and given $\cale > 1$, there exists $\kappa_\cale > 1$; these numbers $\kappa$ and $\kappa_\cale$ have the following significance: Assume that $c>\kappa$ and that the supremum norm of $\frp$ and its covariant derivative are less than $c^{-1}$.  Also assume that
\[ \left|\frac{\partial}{\partial s} a\right| + \left|\frac{\partial}{\partial s} da\right| \leq c^{-1}
\mathrm{\ \ \ and\ \ \ }
s_0 < \kappa^{-1}c. \]	
Suppose that either $\{T(k)\}_{k=1,2,\dots}$ is an increasing, unbounded sequence of numbers (each greater than $16$), or that $T(k)=\infty$ for each positive integer $k$.  Fix $r > \kappa_\cale$.  For each integer $k \in \{1,2,\dots\}$, let $\frd_k = (A_k,\psi_k)$ denote an instanton solution to \eqref{eq:instanton-sw} on $\R \times Y_{T(k)}$ obeying $\displaystyle \lim_{s\to\pm\infty} \en(\frd_k|_s) \leq \cale$ and such that the absolute value of the spectral flow $\upiota_{\frd_k}$ is bounded by $\cale r$.  There exists
\begin{enumerate}\leftskip-0.25in
\item an instanton solution $\frd = (A,\psi)$ to \eqref{eq:instanton-sw} on $\R\times Y_\infty$ with $\sup_{s\in\R} \scm(s) \leq \kappa_\cale$,
\item a subsequence $\Lambda$ of positive integers,
\item a sequence $\{g_k\}_{k\in\Lambda}$ with any given $k\in\Lambda$ version of $g_k$ mapping $\R\times (Y_{T(k)}\ssm N_{T(k)})$ to $S^1$,
\item a sequence $\{\hu_k\}_{k\in\Lambda}$ with any given $k\in\Lambda$ version of $\hu_k$ mapping $\R\times N_{T(k)}$ to $S^1$.
\end{enumerate}
The data $\frd$, $\Lambda$, $\{g_k\}_{k\in\Lambda}$, and $\{\hu_k\}_{k\in\Lambda}$ have the properties listed in the subsequent bullets. The first two bullets use $I\subset\R$ to denote any given bounded interval.
\begin{itemize}\leftskip-0.25in
\item The sequence indexed by $\Lambda$ whose $k$th term is the $C^0$ norm of $e^{r\operatorname{dist}(\cdot,M)/\kappa}(g_k^\ast\frd_k -\frd)$ on $I \times (Y_{T(k)} \ssm N_{T(k)})$ is bounded and has limit zero as $k\to\infty$.
\item For each positive integer $m$, the sequence indexed by $\Lambda$ whose $k$th term is the $C^m$ norm of $(g_k^\ast\frd_k - \frd)$ on $I \times (Y_{T(k)} \ssm N_{T(k)})$ has limit zero as $k\to\infty$.
\item For each positive integer $m$, the sequence indexed by $\Lambda$ whose $k$th term is the $C^m$ norm of $(\hu_k^\ast\frd_k - (\Ai,\psii))$ on $\R\times N_{T(k)}$ has limit zero as $k\to\infty$.
\end{itemize}
\end{proposition}

With regards to the convergence on bounded parts of the $\R$ factor in the first and second bullets of the proposition, the convergence on the whole of $\R \times (Y_{T(k)} \ssm N_T)$ can't be claimed because there will almost surely be cases where the limit is a ``broken trajectory'' that looks like the gluing of widely separated instantons (with respect to the parameter $s$ for the $\R$ factor).  (The notion of a ``broken trajectory'' is discussed at length in \cite{kmbook}, see first the end of Chapter~2.1 in \cite{kmbook}, and then Chapter~16 in \cite{kmbook}.)

\begin{proof}[Proof of Proposition \ref{prop:Tconv-51}]
Suppose first that the sequence $\{\mathcal{A}_{\frd_k}\}_{k=1,2,\dots}$ is bounded by $rz$. Granted this bound, then (as asserted by the last bullet of Proposition \ref{prop:conv-42}) the following is a consequence: If $r$ is greater than a number that depends only on $z$ and $\cale$, then the sequence whose $k$th term is the $\frd_k$ version of $\sup_{s\in\R} \scm(s)$ is bounded. In addition (because of Lemma \ref{lem:bounds-33} assuming the lower bound for $r$ is greater than $z$), the conclusions of Lemmas \ref{lem:bounds-35} and \ref{lem:bounds-36} hold for each $(A_k,\psi_k)$ with $k$-independent versions of $\kappa$. By virtue of the bounds from these lemmas, there is a subsequence $\Lambda \subset \{1,2,\dots\}$ and a corresponding sequence of gauge transformations $\{g_k\}_{k\in\Lambda}$ and $\{\hu_k\}_{k\in\Lambda}$ that obey the assertions of the three bullets of the proposition.

With the preceding understood, it remains to prove that the assumptions about the $s\to\pm\infty$ limits of $\frd_k|_s$ lead to an a priori bound on $\{\mathcal{A}_{\frd_k}\}_{k=1,2,\dots}$.  The nine steps that follow derive such a bound.

\begin{sp5}
\label{step1:prop51}
Supposing that $\frd$ is an instanton solution to \eqref{eq:instanton-sw}, let $\frc_- = (A_-,\psi_-)$ and $\frc_+ = (A_+,\psi_+)$ denote the respective $s\to\pm\infty$ limits of $\frd|_s$. Assume that $\en(\frc_-)$ and $\en(\frc_+) \leq \cale$.  It follows from the formula in \eqref{eq:bounds-31} that
\begin{equation}
\label{eq:Tconv-51}
\fra_-(\frc_-,\frc_+) = -\frac{1}{2} \int_{Y_T} (A_- - A_+) \wedge \ast(B_{A_-} + B_{A_+}) + \fre,
\end{equation}
with $\fre$ obeying $|\fre| \leq c_0 r \cale$.  As explained in Part \ref{part3:prop42} of the proof of Proposition \ref{prop:conv-42}, the solution $\frc_+$ defines a map $u_+: Y_T \ssm M \to S^1$; its salient feature being that the $E$ summand of $u_+\psi_+$ can be written as $1-z_+$ with $z_+$ being real and having norm at most $\frac{1}{1000}$.  It follows from Lemma \ref{lem:monopoles-110} that $|z_+| \leq c_0 e^{-\sqrt{r}\operatorname{dist}(\cdot,M)/c_0}$.  Moreover, $A_+ - u_+^{-1}du_+$ can be written as $\Ai + \sca_+$ where the distance to $M$ is greater than $c_0$ with the norm $|\sca_+| \leq c_0 e^{-\sqrt{r}\operatorname{dist}(\cdot,M)/c_0}$ also.  Since $\en(\frc_-) \leq \cale$, there is a $\frc_-$ analogue of this map, the latter denoted by $u_-$.  There is a corresponding $z_-$ and $\sca_-$.  Let $u=u_- u_+^{-1}$. These bounds for $z_+$, $z_-$, $|\sca_+|$ and $|\sca_-|$ imply that the $i\R$-valued 1-form $A_-$ can be written where the distance to $M$ is greater than $c_0$ as
\begin{equation}
\label{eq:Tconv-52}
A_- = A_+ - u^{-1}du + \sca_- - \sca_+
\mathrm{\ \ \ with\ \ \ }
|\sca_- - \sca_+| \leq c_0 e^{-\sqrt{r}\operatorname{dist}(\cdot,M)/c_0}.
\end{equation}
In the case $T=\infty$, the map $u$ obeys
\begin{equation}
\label{eq:Tconv-53}
\lim_{\operatorname{dist}(\cdot,M)\to\infty} (|u-1| + |du|) = 0.
\end{equation}
This is a consequence of \eqref{cond:instanton-26} given the assumption that there is an instanton with $s\to-\infty$ limit equal to $\frc_-$ and with $s\to\infty$ limit equal to $\frc_+$.

Fix $R_\Diamond \in (1,c_0)$ so that the bounds in \eqref{eq:Tconv-52} hold where the distance to $M$ is greater than $R_\Diamond$ and so that the $E$ components of both $u_+\psi_+$ and $u_-\psi_-$ differ from 1 by at most $\frac{1}{1000}$ where the distance to $M$ is greater than $c_0$. Let $\sigma_\Diamond$ denote a smooth, non-negative function that is equal to 1 where the distance to $M$ is less than $R_\Diamond$ and equal to zero where the distance to $M$ is greater than $R_\Diamond+1$. It follows from the bounds in \eqref{eq:Tconv-52} and from those in Lemma \ref{lem:monopoles-110} that
\begin{align}
\label{eq:Tconv-54}
-\frac{1}{2} \int_{Y_T} (A_- - A_+) \wedge \ast(B_{A_-} + B_{A_+})
=& -\frac{1}{2}\int_{Y_T} \sigma_\Diamond (A_- - A_+) \wedge \ast(B_{A_-} + B_{A_+}) \notag\\
& + \frac{1}{2}\int_{Y_T} (1-\sigma_\Diamond)u^{-1}du \wedge \ast(B_{A_-} + B_{A_+}) + \fre',
\end{align}
with $\fre'$ being a term whose norm is at most $c_0 e^{-\sqrt{r}R_\Diamond/c_0}$. Moreover, an integration by parts (with \eqref{eq:Tconv-53} in the case $T=\infty$) writes the right hand side of \eqref{eq:Tconv-54} as
\begin{equation}
\label{eq:Tconv-55}
-\frac{1}{2}\int_{Y_T} \sigma_\Diamond (A_- - A_+) \wedge \ast(B_{A_-} + B_{A_+}) - \frac{1}{2}\int_{Y_T} d\sigma_\Diamond \wedge u^{-1}du \wedge (\sca_- + \sca_+) + \fre'',
\end{equation}
with $\fre''$ being a term whose norm is also bounded by $c_0 e^{-\sqrt{r}R_\Diamond/c_0}$. 
\end{sp5}

\begin{sp5}
\label{step2:prop51}
Let $\Omega \subset Y_T$ denote an open set in the $\operatorname{dist}(\cdot,M) < R_\Diamond + 10$ part of $Y_T$ with smooth boundary.  Assume that $\Omega$ contains the $\operatorname{dist}(\cdot,M) \leq R_\Diamond+8$ part of $Y_T$ so that $\partial\Omega$ has distance between $R_\Diamond+8$ and $R_\Diamond+10$ from $M$.  Use $\he$ in what follows to denote the normal vector to $\partial\Omega$. Hodge theory for manifolds with boundary can be used to write $A_- - A_+$ on $\Omega$ as
\begin{equation}
\label{eq:Tconv-56}
A_- - A_+ = \sca_\Omega + d\frf + \frv
\end{equation}
where $\sca_\Omega$ is coclosed and it annihilates $\he$ along $\partial\Omega$. In addition, $\sca_\Omega$ is $L^2$-orthogonal to all harmonic 1-forms on $\Omega$ that annihilate $\he$ on $\partial\Omega$. Meanwhile, $\frf$ is an $i\R$-valued smooth function on $\Omega$ and $\frv$ is an $i\R$-valued harmonic 1-form on $\Omega$ that annihilates $\he$ on $\partial \Omega$.  (The decomposition in \eqref{eq:Tconv-56} is a `Hodge' type decomposition: It is writing $A_--A_+$ as a sum of three terms which are pairwise $L^2$ orthogonal on $\Omega$ by virtue of the boundary constraints.  The first (which is $\sca_\Omega$) is annihilated by $d^\dagger$, the second (which is $d\frf$) is annihilated by $d$, and the third (which is $\frv$) is annihilated by both $d$ and $d^\dagger$.)

Keeping in mind that the support of $\sigma_\Diamond$ is inside $\Omega$, the Green's function for the first order operator $d + \ast d \ast$ can be used to bound the absolute value of the $\ahat_\Omega$ contribution to the leftmost integral in \eqref{eq:Tconv-55} by $c_0 r^{2/3} \cale^{4/3}$. Meanwhile, the sum $d\frf+\frv$ can be written as $-\textup{u}^{-1}d\textup{u} + \frv_\ast$ with $\textup{u}$ being a smooth map from $\Omega$ to $S^1$ and with $\frv_\ast$ being a harmonic 1-form on $\Omega$ with norm bounded by $c_0$.  The absolute value of the $\frv_\ast$ contribution to the leftmost integral in \eqref{eq:Tconv-55} is bounded by $c_0 \cale$.

The contribution from $\textup{u}^{-1}d\textup{u}$ to the left most integral in \eqref{eq:Tconv-55} can be dealt with as follows: Let $\upomega_{0-}$ and $\upomega_{0+}$ denote the respective $A_-$ and $A_+$ versions of the 2-form $\upomega_0$ that appears in Lemma \ref{lem:conv-44}.  Keeping in mind that $\upomega_{0-}$ and $\upomega_{0+}$ have compact support in $M$, the $\textup{u}^{-1}d\textup{u}$ contribution to this same integral is written as
\begin{equation}
\label{eq:Tconv-57}
-\frac{i}{2}\int_{Y_T} \textup{u}^{-1}d\textup{u} \wedge (\upomega_{0-}+\upomega_{0+}) + \frac{1}{2} \int_{Y_T} \sigma_\Diamond \textup{u}^{-1}d\textup{u} \wedge ((\ast B_{A_-} + i\upomega_{0-}) + (\ast B_{A_+} + i\upomega_{0+})).
\end{equation}
A second integration by parts writes the rightmost integral in \eqref{eq:Tconv-57} as
\begin{equation}
\label{eq:Tconv-58}
\frac{1}{2}\int_{Y_T} d\sigma_\Diamond \wedge \textup{u}^{-1}d\textup{u} \wedge ((A_- - A_{0-}) + (A_+ - A_{0+}))
\end{equation}
with $A_{0-}$ and $A_{0+}$ being connections on $E$ that can be written on $Y_T \ssm M$ as $A_{0-} = \Ai + u_-^{-1} du_-$ and $A_{0+} = \Ai + u_+^{-1}du_+$; and have respective curvatures given by $-i\upomega_{0-}$ and $-i\upomega_{0+}$. Since $\sca_+$ is defined by writing $A_+ = \Ai + u_+^{-1}du_+ + \sca_+$ and $\sca_-$ is defined by writing $A_- = \Ai + u_-^{-1}du_- + \sca_-$, it follows that \eqref{eq:Tconv-58} can also be written as
\begin{equation}
\label{eq:Tconv-59}
\frac{1}{2}\int_{Y_T} d\sigma_\Diamond \wedge \textup{u}^{-1}d\textup{u} \wedge (\sca_- + \sca_+).
\end{equation}
What with \eqref{eq:Tconv-57} and \eqref{eq:Tconv-59}, the expression in \eqref{eq:Tconv-55} can be written as
\begin{equation}
\label{eq:Tconv-510}
-\frac{i}{2}\int_{Y_T} \textup{u}^{-1}d\textup{u} \wedge (\upomega_{0-}+\upomega_{0+}) + \frac{1}{2} \int_{Y_T} d\sigma_\Diamond \wedge (\textup{u}^{-1}d\textup{u}-u^{-1}du) \wedge (\sca_- + \sca_+) + \fre'''
\end{equation}
where $|\fre'''| \leq c_0 r^{2/3} \cale^{4/3}$.
\end{sp5}

\begin{sp5}
\label{step3:prop51}
The absolute value of the $d\sigma_\Diamond \wedge (\textup{u}^{-1}d\textup{u} - u^{-1}du) \wedge (\sca_- + \sca_+)$ integral in \eqref{eq:Tconv-510} is also bounded by $c_0 r^{2/3} \cale^{4/3}$. This is because the definitions of $\textup{u}$ and $u$ are such that the $i\R$-valued 1-form $A_- - A_+$ on the support of $d\sigma$ is equal to $\sca_- - \sca_+ + u^{-1}du$ on the one hand, and equal to $\sca_\Omega + \frv_\ast + \textup{u}^{-1}d\textup{u}$ on the other. Thus,
\begin{equation}
\label{eq:Tconv-511}
|\textup{u}^{-1}d\textup{u} - u^{-1}du| \leq c_0(|\sca_\Omega| + |\frv_\ast| + |\sca_-| + |\sca_+|).
\end{equation}
Aforementioned bounds for the various terms on the right hand side of \eqref{eq:Tconv-511} lead to the $c_0 r^{2/3} \cale^{4/3}$ bound for the absolute value of the $d\sigma_\Diamond \wedge (\textup{u}^{-1}d\textup{u} - u^{-1}du) \wedge (\sca_- + \sca_+)$ integral in \eqref{eq:Tconv-510}.

If $\upomega_{0-} + \upomega_{0+}$ defines torsion classes in $H^2_c(M)$, then the leftmost term on the right hand side of \eqref{eq:Tconv-510} vanishes. In this special case, it follows from what has been said previously that $|\fra_-(\frc_-,\frc_+)| \leq c_0(r\cale + r^{2/3}\cale^{4/3})$.

Suppose henceforth that the sum $\upomega_{0-} + \upomega_{0+}$ is not torsion. The spectral flow information must be used in this case to obtain the desired bound on $|\fra_-(\frc_-,\frc_+)|$. This is done by introducing $\fra_-^{\frf}(\frc_-,\frc_+) = \fra_-(\frc_-,\frc_+) - 2\pi^2 \ind$ when $\frd$ is an instanton solution to \eqref{eq:instanton-sw} with $\frd|_s$ having $s\to-\infty$ limit equal to $\frc_-$ and $s\to\infty$ limit equal to $\frc_+$. The assumption that $|\upiota_{\frd_k}| \leq \cale r$ in the proposition implies that $|\mathcal{A}_{\frd_k}| \leq c_\cale r$ if the absolute value of $\fra_-^{\frf}(\frc_-,\frc_+)$ is bounded by $c_\cale r$. The proof that this is so is a bit tricky because the spectral flow for different values of $k$ is defined by operators on different Hilbert spaces.

The plan for proving the desired bounds on $\fra_-^{\frf}(\frc_-,\frc_+)$ is to repackage the definition of $\ind$ so that all of the spectral flows are for operators on the same Hilbert space, this being the space $L^2_1(Y_{2R}; T^\ast Y_{2R} \oplus \underline{\R} \oplus \spb)$ for a suitable value of $R$ bounded by $c_\cale$. This repackaging of the spectral flow occupies Steps \ref{step4:prop51}--\ref{step8:prop51}.  Step \ref{step9:prop51} uses this repackaging to complete the proof of Proposition \ref{prop:Tconv-51}.
\end{sp5}

\begin{sp5}
\label{step4:prop51}
Letting $\kappa_\cale$ denote the number supplied by Lemma \ref{lem:monopoles-110}, fix $R=2\kappa_\cale$. Suppose that $T>2R$ and that $(A,\psi)$ is a pair defined on $Y_T$ that is admissible if $T=\infty$ and in any event has the properties that are listed below.
\begin{itemize}\leftskip-0.25in
\item $|\beta| \leq e^{-\sqrt{r} R/c_0}$ where the distance to $M$ is greater than $\frac{1}{2} R$.
\item There is a gauge transformation $\hu: Y_T\ssm M \to S^1$ such that the following hold where the distance to $M$ is greater than $\frac{1}{2} R$:
\begin{enumerate}\leftskip-0.25in
\item $\hu \alpha = (1-z)$ with $z$ being a real number with norm bounded by $c_0 e^{-\sqrt{r} R/c_0}$.
\item $A = \Ai + \hu^{-1}d\hu + \hat{\sca}$ with $|\hat{\sca}| \leq c_0 e^{-\sqrt{r} R/c_0}$.
\end{enumerate}
\end{itemize}
\begin{equation} \label{eq:Tconv-512} \end{equation}
A $c_0$ lower bound on $r$ is chosen so that the $e^{-\sqrt{r} R/c_0}$ bound in \eqref{eq:Tconv-512} is much less than $1$. Note that Lemma \ref{lem:monopoles-110} guarantees that these conditions are met if $r>c_\cale$ and $T > c_\cale$ when $(A,\psi)$ is an $\en \leq \cale$ solution to \eqref{eq:sw}.

Let $\sigma_+$ denote a non-negative function on $Y_T$ that equals $1$ where the distance to $M$ is less than $R+4$ and equals zero where the distance to $M$ is greater than $R+5$.  The pair $(A,\psi)$ now defines a pair $(A_R, \psi_R)$ on $Y_{2R}$ by setting $(A_R, \psi_R) = (A,\psi)$ where the distance to $M$ is less than $R+4$ and setting $(A_R,\psi_R)$ to be $(\Ai + \sigma_+ \hat{\sca}, \psii + \sigma_+(\hu\psi - \psii))$ where the distance to $M$ is greater than $R+4$.

Let $\sigma_-$ denote another non-negative function on $Y_T$, this one equaling $1$ where the distance to $M$ is less than $R$ and equaling zero where the distance to $M$ is greater than $R+1$. The pair $(A,\psi)$ now defines a new pair on $Y_T$; this new pair is denoted by $(A_\ast, \psi_\ast)$. This is a pair of a connection on $\underline{\C}$ and section of the canonical spin bundle $\si$ that is defined by setting $(A_\ast,\psi_\ast) = (\Ai + (1-\sigma_-)\hat{\sca}, \psii + (1-\sigma_-)(\hu\psi - \psii))$.

Let $\LL$ denote the version of \eqref{eq:linearized} on $Y_T$ that is defined by $(A,\psi)$ and let $\LL_\ast$ denote the version of \eqref{eq:linearized} that is defined on $Y_T$ by the pair $(A_\ast, \psi_\ast)$. Meanwhile, let $\LL_R$ denote the version of the operator in \eqref{eq:linearized} on $Y_{2R}$ that is defined by $(A_R,\psi_R)$.
\end{sp5}

\begin{sp5}
\label{step5:prop51}
Let $\sigma$ denote yet a third nonnegative function; it is equal to $1$ where the distance to $M$ is less than $R+2$ and it is equal to zero where the distance to $M$ is greater than $R+3$. Let $\frS(d\sigma)$ denote the homomorphism defined by $\LL\sigma - \sigma\LL$, this being independent of the pair $(A,\psi)$; it depends only on the symbol of $\LL$. The homomorphism $\frS(d\sigma)$ has compact support where the distance to $M$ is between $R+2$ and $R+3$.

Let $(\eta_R,\eta_\infty)$ denote a pair with $\eta_R$ being a section on $Y_{2R}$ of $(T^\ast Y_{2R} \oplus \underline{\R}) \oplus \spb$ and $\eta_\infty$ being a section on $Y_T$ of the analogous $(T^\ast Y_T\oplus\underline{\R})\oplus\spb_{\rm I}$ version of this bundle. Suppose in addition that $\lambda$ is a real number and that $(\eta_R,\eta_\infty)$ obey the coupled equations
\begin{equation}
\label{eq:Tconv-513}
\LL_R \eta_R + \frS(d\sigma)(\eta_R-\eta_\infty) = \lambda \eta_R
\mathrm{\ \ \ and\ \ \ }
\LL_\ast \eta_\infty + \frS(d\sigma)(\eta_R - \eta_\infty) = \lambda \eta_\infty.
\end{equation}
Set $\eta = \sigma\eta_R + (1-\sigma)\eta_\infty$ to be viewed as a section on $Y_T$ of $(T^\ast Y_T \oplus \underline{\R}) \oplus \spb$. It follows from \eqref{eq:Tconv-513} that this section obeys $\LL\eta = \lambda\eta$. (Note that $\LL_R = \LL_\ast$ on the support of $d\sigma$. This is the reason for having three cut-off functions $\sigma$, $\sigma_-$, and $\sigma_+$ instead of just one.)

Conversely, if $r>c_0$ and $|\lambda| \leq c_0^{-1}r^{1/2}$ and $\eta$ is a section on $Y_T$ of $(T^\ast Y_T \oplus \underline{\R}) \oplus \spb$ that obeys $\LL\eta = \lambda\eta$, then there is a unique pair $(\eta_R,\eta_\infty)$ obeying \eqref{eq:Tconv-513} with the property that $\eta = \sigma\eta_R + (1-\sigma)\eta_\infty$. The proof that this is so is in the next paragraphs.

To start, let $\sigma_0$ denote a nonnegative function that equals $1$ where the distance to $M$ is less than $R+1$ and equal $0$ where the distance to $M$ is greater than $R+2$. With $\sigma_0$ in hand, write $\xi_R = \sigma_0 \eta + \sigma(1-\sigma_0)\eta$ and $\xi_\infty = (1-\sigma_0)\eta + \sigma(1-\sigma_0)\eta$.  View the former as being defined on $Y_{2R}$ and the latter as being defined on $Y_T$ but with the spinor bundle $\spb_{\rm I}$.  The definitions are designed so that $\eta = \sigma\xi_R + (1-\sigma)\xi_\infty$. The pair $(\xi_R,\xi_\infty)$ does not obey \eqref{eq:Tconv-513}, but if $r>c_0$ and $|\lambda| < c_0^{-1} r^{1/2}$, then Lemma \ref{lem:monopoles-11} can be used to see that the difference between the right and left hand sides of the two equations in the $(\xi_R,\xi_\infty)$ version of \eqref{eq:Tconv-513} has norm bounded by $e^{-\sqrt{r} R/c_0}$.

With the preceding in mind, suppose that $\phi$ is a section of $(T^\ast Y_T \oplus \underline{\R}) \oplus \si$.  If
\begin{equation}
\label{eq:Tconv-514}
\eta_R = \xi_R + (1-\sigma)\phi \mathrm{\ \ \ and\ \ \ } \eta_\infty = \xi_\infty - \sigma\phi,
\end{equation}
then $\sigma\eta_R + (1-\sigma)\eta_\infty$ is also equal to $\eta$. Meanwhile, $(\eta_R,\eta_\infty)$ will obey \eqref{eq:Tconv-513} if $\phi$ obeys
\begin{equation}
\label{eq:Tconv-515}
\LL_\ast\phi - \lambda\phi = -(\LL_R\xi_R + \frS(d\sigma)(\xi_R-\xi_\infty) - \lambda \xi_R) + (\LL_\infty\xi_\infty + \frS(d\sigma)(\xi_R-\xi_\infty) - \lambda\xi_\infty).
\end{equation}
Keep in mind that the right hand side of \eqref{eq:Tconv-515} is non-zero only where the distance to $M$ is between $R+1$ and $R+4$ and that its norm in any event is bounded by $e^{-\sqrt{r} R/c_0}$. Since $\LL_\ast$ differs from $\Li$ by a homomorphism with norm bounded by $e^{-\sqrt{r} R/c_0}$, it follows from Lemma \ref{lem:monopoles-11} that there is a unique solution $\phi$ to \eqref{eq:Tconv-515} if $r>c_0$ and $|\lambda| \leq c_0^{-1} r^{1/2}$.
\end{sp5}

\begin{sp5}
\label{step6:prop51}
As noted previously, the operators $\Li$ and $\LL_\ast$ differ when $r>c_0$ by a zeroth order operator whose norm is bounded by $c_0 e^{-\sqrt{r} R/c_0}$. It follows as a consequence of Lemma \ref{lem:monopoles-11} that the right hand equation in \eqref{eq:Tconv-513} can be solved for $\eta_\infty$ in terms of $\eta_R$ when $r$ is larger than $c_0$ and $|\lambda| < c_0 r^{1/2}$. In particular, it follows from Lemma \ref{lem:monopoles-11} that this solution obeys $||\eta_\infty||_2 \leq e^{-\sqrt{r} R/c_0}||\eta_R||_2$.   Keeping this fact in mind, the right hand equation in \eqref{eq:Tconv-513} can be solved for $\eta_\infty$ in terms of $\eta_R$ to write the left hand equation in \eqref{eq:Tconv-513} as
\begin{equation}
\label{eq:Tconv-516}
\LL_R\eta_R + \frS(d\sigma)\eta_R + \frT_{T,\lambda}(\eta_R) = \lambda \eta_R
\end{equation}
where $\frT_{T,\lambda}$ is the operator
\begin{equation}
\label{eq:Tconv-517}
\frT_{T,\lambda} = -\frS(d\sigma) \frac{1}{\LL_\ast - \lambda - \frS(d\sigma)} \frS(d\sigma).
\end{equation}
The operator $\frT_{T,\lambda}$ is a compact operator on $L^2(Y_{2R}; T^\ast Y_{2R} \oplus \underline{\R} \oplus \spb)$ that obeys
\begin{equation}
\label{eq:Tconv-518}
||\frT_{T,\lambda} \frb||_2 \leq c_0 r^{-1/2} ||\frb||_2
\end{equation}
for any $\frb \in L^2(Y_{2R}; T^\ast Y_{2R} \oplus \underline{\R} \oplus \spb)$.

Keep in mind $\frT_{T,\lambda}$ is an operator on $Y_R$ but that it nonetheless depends on the value of $T$. This is because $\LL_\ast - \lambda - \frS(d\sigma)$, whose inverse appears in the definition, is acting on $L^2_1(Y_T; T^\ast Y_T \oplus \underline{\R} \oplus \spb)$. This is why the subscript $T$ is used in the notation.
\end{sp5}

\begin{sp5}
\label{step7:prop51}
To summarize from the previous step: Suppose that $r>c_\cale$, that $R>c_\cale$, and that $T>c_0 R$. Fix $\lambda \in \R$ with norm bounded by $c_0^{-1} r^{1/2}$. Then the kernel of $\LL - \lambda$ on $L^2_1(Y_T; T^\ast Y_T \oplus \underline{\R} \oplus \spb)$ is in 1--1 correspondence with the kernel of $\LL_R + \frS(d\sigma) + \frT_{T,\lambda}$ on the space $L^2_1(Y_{2R}; T^\ast Y_{2R} \oplus \underline{\R} \oplus \spb)$. With the preceding fact in mind, introduce $\frQ$ to denote the operator on $L^2_1(Y_{2R}; T^\ast Y_{2R} \oplus \underline{\R} \oplus \spb)$ given by
\begin{equation}
\label{eq:Tconv-519}
\frQ = \LL_R + \frS(d\sigma) + \frT_{T,0}
\end{equation}
with $\frT_{T,0}$ denoting the $\lambda=0$ version of the operator in \eqref{eq:Tconv-516}. It follows from what was said prior to \eqref{eq:Tconv-519} that the kernel of $\frQ$ is in 1--1 correspondence with the kernel of $\LL$ acting on $L^2_1(Y_T; T^\ast Y_T \oplus \underline{\R} \oplus \spb)$.
\end{sp5}

\begin{sp5}
\label{step8:prop51}
This step argues that there is very little spectral flow between any given version of the operator depicted in \eqref{eq:Tconv-519} and the corresponding version of $\LL_R$. To see this, note first that if $\eta$ is an $L^2_1$ section of $T^\ast Y_{2R} \oplus \underline{\R} \oplus \spb$, then $||(\frQ - \LL_R)\eta||_2 \leq c_0||\eta||_2$ because of \eqref{eq:Tconv-518} and because $|d\sigma| \leq c_0$. With this bound in hand, then the framework in \cite{taubes-asymptotic} can be used to see that a given eigenvalue can move a distance at most $c_0$ as the parameter $\mu \in [0,1]$ is changed in the family of operators
\begin{equation}
\label{eq:Tconv-520}
\mu \to \LL_R + \mu(\frS(d\sigma) + \frT_{T,0}).
\end{equation}
By way of a reminder, the eigenvalues for the various $\mu \in [0,1]$ members of the family can be labeled consistently as functions of $\mu$ to define piecewise differentiable paths in $\R$. Let $\mu \to \lambda(\mu)$ denote such a path. The fact that $|\lambda(0) - \lambda(1)| \leq c_0$ implies that only eigenvalues that start with norm less than $c_0$ can cross zero and thus contribute to the spectral flow. Now if $m \in [1,c_0^{-1}r^{1/2}]$ and if $|\lambda| \leq m$ and if $\eta$ is an eigenvector with eigenvalue $\lambda$, then $|\eta|$ must have most of its support where $1-|\alpha| \geq c_0^{-1}$. This is a consequence of Lemma \ref{lem:monopoles-11}.  Given what is said by \eqref{eq:Tconv-512}, the latter fact can be used to see that the number of linearly independent eigenvectors with eigenvalue $\lambda$ with norm less than $m$ is bounded by $c_0 m(\ln r)^{c_0} \cale$. The arguments for this are much like those in the article \cite{taubes-weinstein2} that prove Proposition~3.1 in this same article.
\end{sp5}

\begin{sp5}
\label{step9:prop51}
In the context of Proposition \ref{prop:Tconv-51}, let $\frc_{R-} = (A_{R-}, \psi_{R-})$ and $\frc_{R+} = (A_{R+}, \psi_{R+})$ denote the respective versions of $(A_R,\psi_R)$ that are obtained using first $\frc = u_- \frc_-$ and then using $\frc = u_+ \frc_+$.  The analysis done in Steps \ref{step1:prop51} and \ref{step2:prop51} that lead to the depiction of $\fra_-(\frc_-,\frc_+)$ in \eqref{eq:Tconv-510} can be repeated to write
\begin{equation}
\label{eq:Tconv-521}
\fra_-(\frc_{R-},\frc_{R+}) = -\frac{i}{2} \int_{Y_T} \textup{u}^{-1}d\textup{u} \wedge (\upomega_{0-}+\upomega_{0+}) + \fre_R
\end{equation}
with $\fre_R$ obeying $|\fre_R| \leq c_0(r\cale + r^{2/3} \cale^{4/3})$. Given what is said in Step \ref{step3:prop51} about \eqref{eq:Tconv-510}, it follows that
\begin{equation}
\label{eq:Tconv-522}
|\fra_-(\frc_-,\frc_+) - \fra_-(\frc_{R-},\frc_{R+})| \leq c_0(r\cale + r^{2/3}\cale^{4/3}).
\end{equation}

Let $\upiota_R$ denote the spectral flow between the $\frc_{R-}$ version of the operator $\LL_R$ and the $\frc_{R+}$ version.  It follows from \eqref{eq:Tconv-522} and from what is said in Steps \ref{step5:prop51}--\ref{step7:prop51} that
\begin{equation}
\label{eq:Tconv-523}
|(\fra_-(\frc_-,\frc_+)-2\pi^2\ind) - (\fra_-(\frc_{R-},\frc_{R+})-2\pi^2 \upiota_R)| \leq c_0(r\cale + r^{2/3}\cale^{4/3})
\end{equation}
also. This inequality in \eqref{eq:Tconv-523} is used in the next step to bound $|\fra(\frc_-,\frc_+) - 2\pi^2\ind|$ given a bound for $|\fra_-(\frc_{R-},\frc_{R+}) - 2\pi^2 \upiota_R|$.

The advantage of using $\fra_-(\frc_{R-},\frc_{R+}) - 2\pi^2 \upiota_R$ as a proxy for $\fra_-(\frc_-,\frc_+) - 2\pi^2 \ind$ is that the manifold in the former case is $Y_{2R}$ with $R$ being a fixed $\mathcal{O}(c_0)$ number. This means that estimates that involve geometric properties are uniform.  In particular, the analysis that proves Proposition~1.9 in \cite{taubes-weinstein2} can be repeated with only cosmetic changes to see that $|\fra_-(\frc_{R-},\frc_{R+}) - 2\pi^2\upiota_R| \leq c_0r^{2/3}\cale^{4/3}(\ln r)^{c_0}$. Given the latter bound, it then follows from \eqref{eq:Tconv-523} that
\begin{equation}
\label{eq:Tconv-524}
|\fra_-(\frc_-,\frc_+) - 2\pi^2\ind| \leq c_0 (r\cale + r^{2/3} \cale^{4/3} (\ln r)^{c_0}).
\end{equation}
Apply the bound in \eqref{eq:Tconv-524} to the sequence in Proposition \ref{prop:Tconv-51} to obtain the desired a priori bound on $\{\mathcal{A}_{\frd_k}\}_{k=1,2,\dots}$. \qedhere
\end{sp5}
\end{proof}

%% file: a_instanton-comparison.tex
The propositions in this section assert compactness theorems for moduli spaces of instantons on $\R\times Y_\infty$ and $\R\times Y_T$ for $T<\infty$. This is done by invoking Proposition \ref{prop:Tconv-51}; but to invoke the latter, a number $\cale>1$ must be fixed in advance and then attention restricted to those instantons whose $s\to\pm\infty$ limits have $\en<\cale$. An integer $\fri$ must also be fixed a priori so as to consider only instantons with $|\upiota_\frd|\leq\fri$. The upcoming propositions consider only the cases $\ind=1$ and $\ind=0$. Note that Proposition \ref{prop:Tconv-51} requires a certain a priori bound for the pointwise norms of the $s$-derivatives of both $a$ and $da$, and it requires a bound for the pointwise norms of $\frp$ and its covariant derivatives.  These bounds are assumed (implicitly) to hold in this section so that appeals to Proposition~\ref{prop:Tconv-51} can be made.

So as to avoid talking about broken trajectories (they are defined as in Chapter~16 in \cite{kmbook}), the upcoming Proposition \ref{prop:comp-61} assumes implicitly that there are no instantons that interpolate between the $\en<c_\cale$ solutions on $Y_\infty$ with $\ind<\fri$ except the $s$-independent solution in the case when $a$ and the metric are independent of $s$. On a very much related note, Proposition \ref{prop:comp-61} also introduces the notation of an instanton being \emph{unobstructed}. This means that the operator depicted in \eqref{eq:instanton-linearize} has trivial cokernel. 

With regards to the statement of Proposition \ref{prop:comp-61}, keep in mind that when $a$ and the metric are independent of $s$ and when  $\frp$ is absent from \eqref{eq:instanton-sw}, there is an action of the semi-direct product of $\R$ and $C^\infty(\R\times Y_\infty;S^1)$ on the space of solutions, with the $\R$ action induced by the action of $\R$ as the group of translations along the $\R$ factor of $\R\times Y_\infty$.

\begin{proposition}
\label{prop:comp-61}
Given $\cale>1$ there exists $\kappa_\cale>1$ with the following significance: Fix $r>\kappa_\cale$ and assume that the operator $\LLf$ in \eqref{eq:linearized} that is defined by any $\en<\kappa_\cale$ solution of \eqref{eq:sw} on $Y_\infty$ with $a=a_-$ or with $a=a_+$ has trivial cokernel. Assume in addition that there are no non-constant solutions to either the $a=a_-$ or $a=a_+$, constant metric, and $\frp=0$ version of \eqref{eq:instanton-sw} on $\R\times Y_\infty$ with $\ind\leq 0$ and with $\lim_{s\to\pm\infty}\en(\frd|_s)\leq\kappa_\cale$.
\begin{itemize}\leftskip-0.25in
\item Suppose that $a$ and the metric are independent of the coordinate $s$ and that $\frp=0$ in \eqref{eq:instanton-sw}. The space of $\R\times C^\infty(\R\times Y_\infty;S^1)$ orbits of instanton solutions to \eqref{eq:instanton-sw} on $\R\times Y_\infty$ with $\ind=1$ and $\lim_{s\to\pm\infty}\en(\frd|_s)\leq\cale$ is compact. This space is finite if each of its instantons is unobstructed.
\item Suppose that $a$ and/or the metric depend on the coordinate $s$, and/or that $\frp$ in \eqref{eq:instanton-sw} is non-zero.  Assume that there are no instanton solutions to \eqref{eq:instanton-sw} on $\R\times Y_\infty$ with $\lim_{s\to\pm\infty} \en(\frd|_s) \leq \cale$ and $\ind < 0$.  Then, the space of $C^\infty(\R\times Y_\infty;S^1)$ orbits of instanton solutions to \eqref{eq:instanton-sw} on $\R\times Y_\infty$ with $\ind=0$ and $\lim_{s\to\pm\infty}\en(\frd|_s)\leq\cale$ is compact. This space is finite if each of its instantons is unobstructed.
\end{itemize}
\end{proposition}
\begin{proof}
Supposing that $r$ is greater than a number $r_\cale$ that depend only on $\cale$, then, Proposition \ref{prop:Tconv-51} with Lemma \ref{lem:bounds-35} can be used to see that any sequence of instanton solutions to \eqref{eq:instanton-sw} with $\ind=1$ in the case of the top bullet and with $\ind=0$ in the case of the lower bullet, and in either case with $\lim_{s\to\pm\infty}\en(\frd|_s)\leq\cale$ must converge to a broken trajectory with each constituent instanton having $s\to\pm\infty$ limit with $\en$ bounded by a number that is determined a priori only by $\cale$. (Numbers of this sort are denoted by $c_\cale$ in what follows. As before, the precise value can be assumed to increase between successive appearances.) The convergence here is analogous to that described in Chapter~16 of \cite{kmbook} which considers the case of a compact 3-manifold.

To elaborate slightly: Let $\{\frd_k\}_{k\in\Lambda}$ denote a sequence of instanton solutions that obey the constraints for either the first or the second bullet of Proposition \ref{prop:comp-61}. Now let $\{s_k\}_{k=1,2,\dots}$ any sequence of real numbers. Given such a sequence, let $\{{\frd'}_k\}_{k=1,2,\dots}$ denote the sequence of instantons that is defined by the rule whereby 
\[{\frd'}_k\big|_s=\frd_k\big|_{s-s_k}\;for\;each\;k\in\{1,2,\dots\}.\]
Thus, ${\frd'}_k$ is obtained from $\frd_k$ by translating a distance $s_k$ in the $\R$ factor. Proposition \ref{prop:Tconv-51} can be invoked for each such translated sequence. (Translation along the $\R$ factor does not change the a priori upper bound for the versions of the function $\scm$ on $\R$ from the last bullet of Proposition \ref{prop:conv-42}; nor does it change the bound on the various versions of $\mathcal{A}_\frd$.) The arguments in Chapter 16 of \cite{kmbook} can be repeated with only cosmetic changes to obtain a subsequence $\Lambda\subset\{1,2,\dots\}$, a positive integer $\rmn$ less than $c_\cale$ and a collection of $\rmn$ sets of translations (to be denoted by $\{\{s_{\upalpha,k}\}_{k\in\Lambda}\}_{\upalpha\in\{1,\dots,\rmn\}}$ that have the following properties:
\begin{itemize}\leftskip-0.25in
\item For any distinct $\upalpha,\upbeta\in\{1,\dots,\rmn\}$ with $\upalpha>\upbeta$, the sequence $\{s_{\upalpha,k}-s_{\upbeta,k}\}_{k\in\Lambda}$ is a positive, increasing and unbounded sequence.
\item Each $\{s_{\upalpha,k}\}_{k\in\Lambda}$ version of $\{{\frd'}_k\}_{k\in\Lambda}$ converges after termwise gauge transformations in the manner of the three bullets of Proposition \ref{prop:Tconv-51} to a non-constant instanton on $\R\times Y$. This solution is denoted by $\frd_{\ast\upalpha}$.
\item The finite set $\{\{s_{\upalpha,k}\}_{k\in\Lambda}\}_{\upalpha\in\{1,\dots,\rmn\}}$ is maximal in the following sense: If $\{s_k\}_{k\in\Lambda}$ is any sequence in $\R$ with $\{|s_k-s_{\upalpha,k}|\}_{k\in\Lambda}$ increasing and unbounded, then any Proposition \ref{prop:Tconv-51} limit of the corresponding $\{{\frd'}_k\}_{k\in\Lambda}$ is a constant solution on $\R\times Y$ (which is to say that it is gauge equivalent to the pull-back to $\R\times Y$ via the projection map to $Y$ of an admissible solution to \eqref{eq:sw}.)
\end{itemize}
\begin{equation}
\label{eq:instacomp-61}
\end{equation}
The collection $\{\frd_{\ast\upalpha}\}_{\upalpha\in\{1,\dots,\rmn\}}$ is the \emph{broken instanton} limit of $\{\frd_k\}_{k\in\Lambda}$. The assertion that the maximal number of non-trivial instanton limits is finite in the sense of the third bullet of \eqref{eq:instacomp-61} is ultimately a consequence of two properties of what is denoted by $\fra(\frc,\frc')$ in \eqref{eq:bounds-31}. The first (noted in the paragraph after \eqref{eq:bounds-31}) is that it is additive in the sense that $\fra(\frc,\frc')+\fra(\frc',\frc'')=\fra(\frc,\frc'')$. The second (noted by the first bullet of Lemma \ref{lem:bounds-31}) is that it is decreasing where $a$ is independent of the $s$ coordinate on the $\R$ factor in $\R\times Y$.

The first bullet in \eqref{eq:instacomp-61} and the third bullet in \eqref{eq:instacomp-61} have the following implications about the broken instanton $\{\frd_{\ast\upalpha}\}_{\upalpha\in\{1,\dots,\rmn\}}$:
\begin{itemize}\leftskip-0.25in
\item The $s\to-\infty$ limit of $\{\frd_{\ast1}\big|_s\}$ is the $k\to-\infty$ limit of $\lim_{s\to-\infty}\{\frd_k\big|_s\}_{k\in\Lambda}$.
\item The $s\to\infty$ limit of each $\upalpha\in\{1,\dots,\rmn-1\}$ version of $\{\frd_{\ast\upalpha}\big|_s\}$ is the same (up to a gauge transformation) as the corresponding $s\to-\infty$ limit of $\{\frd_{\ast\upalpha_1}\big|_s\}$.
\item The $s\to\infty$ limit of $\{\frd_{\ast\rmn}\big|_s\}$ is the $k\to\infty$ limit of $\lim_{s\to\infty}\{\frd_k\big|_s\}_{k\in\Lambda}$.
\end{itemize}
\begin{equation}
\label{eq:instacomp-62}
\end{equation}
These last bullets imply that the sum of the various $\upalpha\in\{1,\dots,\rmn\}$ versions of $\mathcal{A}_\frd$ is equal to the $k\to\infty$ limit of the various $\frd_k$ versions of $\mathcal{A}_\frd$ (remember that $\fra$ in \eqref{eq:bounds-31} is additive). This is analogous to what happens in the compact case as discussed in \cite{kmbook}.

Also as is the case in \cite{kmbook}, the respective values of $\ind$ for the constituent instantons from the broken instanton set $\{\frd_{\ast\upalpha}\}_{\upalpha\in\{1,\dots,\rmn\}}$ must add up to $1$ in the case of the top bullet of Proposition \ref{prop:comp-61}; and they must add up to $0$ in the case of the lower bullet. This is proved by a gluing argument which (given Lemmas \ref{lem:monopoles-12} and \ref{lem:monopoles-14}, and \eqref{eq:instacomp-62}) differs only cosmetically from the arguments in \cite[Chapter 19]{kmbook} that imply the analogous assertions were $Y_\infty$ compact. Since the assumptions preclude the existence of $\ind<1$ instantons when $a$ and the metric are constant, it follows that there is just one constituent broken instanton in the case of either the top or the lower bullet. Granted that this is so, then it follows directly that the space of $\R\times C^\infty(\R\times Y_\infty;S^1)$ orbits is compact.

The assertion that this moduli space is finite if each element is unobstructed follows from a local slice theorem that is the instanton equivalent of the slice assertion made by Lemma \ref{lem:monopoles-19}. What with Lemma \ref{lem:bounds-35} and the a priori bound $\scm$ bound from Proposition \ref{prop:conv-42}, the proof of this instanton version is not so different from the proof of Lemma \ref{lem:monopoles-19}. This being the case, the details are left to the reader.
\end{proof}

The next proposition gives conditions that lead to a 1--1 correspondence between gauge equivalence classes of instanton solutions to \eqref{eq:instanton-sw} on $\R\times Y_T$ and instanton solutions to these same equations on $\R\times Y_\infty$. 
\begin{proposition}
\label{prop:comp-62}
Given $\cale>1$ there exists $\kappa_\cale>1$ with the following significance: Fix $r>\kappa_\cale$ and assume that the operator $\LLf$ in \eqref{eq:linearized} that is defined by any $\en<\kappa_\cale$ solution of \eqref{eq:sw} on $Y_\infty$ with $a=a_-$ or with $a=a_+$ has trivial cokernel. Assume in addition that there are no non-constant solutions to either the $a=a_-$ or $a=a_+$, constant metric, and $\frp=0$ version of \eqref{eq:instanton-sw} on $\R\times Y_\infty$ with $\ind\leq 0$ and with $\lim_{s\to\pm\infty}\en(\frd|_s)\leq\kappa_\cale$.
\begin{itemize}\leftskip-0.25in
\item Suppose that $a$ and the metric are independent of the coordinate $s$ and that $\frp=0$ in \eqref{eq:instanton-sw}. Assume in addition that each instanton solution to \eqref{eq:instanton-sw} on $\R\times Y_\infty$ that has $\ind=1$ and $\lim_{s\to\pm\infty}\en(\frd|_s)\leq\cale$  is unobstructed. If $T$ is sufficiently large (given $r$), then the space of $\R\times C^\infty(\R\times Y_T;S^1)$ orbits of instanton solutions to \eqref{eq:instanton-sw} on $\R\times Y_T$ with $\ind=1$ and $\lim_{s\to\pm\infty}\en(\frd|_s)\leq\cale$ is in 1--1 correspondence with the space of $\R\times C^\infty(\R\times Y_\infty;S^1)$ orbits of instanton solutions to \eqref{eq:instanton-sw} on $\R\times Y_\infty$ with $\ind=1$ and $\lim_{s\to\pm\infty}\en(\frd|_s)\leq\cale$.

\item Suppose that $a$ and/or the metric depend on the coordinate $s$, and/or that $\frp$ in \eqref{eq:instanton-sw} is non-zero.  Assume that there are no instanton solutions to \eqref{eq:instanton-sw} on $\R\times Y_\infty$ with $\lim_{s\to\pm\infty} \en(\frd|_s) \leq \cale$ and $\ind < 0$; and assume that each instanton solution on $\R\times Y_\infty$ with $\lim_{s\to\pm\infty} \en(\frd|_s) \leq \cale$ and with $\ind=0$ is unobstructed.  If $T$ is sufficiently large (given $r$), then the space of $C^\infty(\R\times Y_T;S^1)$ orbits of instanton solutions to \eqref{eq:instanton-sw} on $\R\times Y_T$ with $\ind=0$ and $\lim_{s\to\pm\infty}\en(\frd|_s)\leq\cale$ is in 1--1 correspondence with the space of $C^\infty(\R\times Y_\infty;S^1)$ orbits of instanton solutions to \eqref{eq:instanton-sw} on $\R\times Y_\infty$ with $\ind=0$ and $\lim_{s\to\pm\infty}\en(\frd|_s)\leq\cale$.
\end{itemize}
In both cases the correspondence has the following property: For each $\upvarepsilon\in (0,1)$, there exists $\kappa_{\cale,\upvarepsilon}>1$ such that if $T>\kappa_{\cale,\upvarepsilon}$ and if $\frd_T$ and $\frd_\infty$ are respective instanton solutions to \eqref{eq:instanton-sw} from corresponding gauge equivalence classes of instanton solutions to \eqref{eq:instanton-sw} on $\R\times Y_T$ and $\R\times Y_\infty$, then there are maps $g_T:\R\times(Y_T\ssm N_T)\to S^1$ and $\hat{u}_T:\R\times N_T\to S^1$ such that 
\begin{enumerate}\leftskip-0.25in
\item[(a)] The $C^0$-norm of $e^{r \operatorname{dist}(\cdot,M)/c_0}(g_T\frd_T-\frd_\infty)$ on $\R\times (Y_T\ssm N_T)$ is less than $\upvarepsilon$.
\item[(b)] The $C^0$-norm of $\hat{u}_T\frd_T-(\uptheta_0,\psii)$ on $\R\times N_T$ is less than $\upvarepsilon e^{-2T/c_0}$.
\end{enumerate}
\end{proposition}
\begin{proof}
The proof is much like that of Lemma \ref{lem:monopoles-115} and left to the reader.
\end{proof}

The next proposition talks only about the $\ind=1$ case when $a$ and the metric are independent of the coordinate $s$ and when the perturbing 2-form $\frp$ in \eqref{eq:instanton-sw} is zero. We assume in this case that some finite $\cale$ has been chosen and that the contact form on $M$ and the associated almost complex structure are $L$-flat for $L\geq c_0\cale$. This is to say that the conditions in (4-1) of \cite{taubes1} are met. A second assumption is that the almost complex structure is suitably generic, which is to say that it is in the $Y_\infty$ analog of the residual set $\mathcal{J}_a$ that is the analog here of the set that is described in Section 2c of \cite{taubes1}. (Note that the conditions for membership in $\mathcal{J}_a$ need only hold for pseudo-holomorphic curves that lie entirely where the distance in $\R\times Y_\infty$ is less than $c_\cale$ because these are the only pseudo-holomorphic curves that are relevant.)
\begin{proposition}
\label{prop:comp-63}
There exists $\kappa>1$ and, given $\cale>1$, there exists $\kappa_\cale>16$ with the following significance: Assume that the contact form and almost complex structure obey the conditions in (4-1) of \cite{taubes1} for $L>\kappa\cale$ and that the almost complex structure is in the $Y_\infty$ analog of the residual set $\mathcal{J}_a$ from Section 2c of \cite{taubes1} with regards to pseudo-holomorphic subvarieties with distance less than $\kappa_\cale$ from $\R\times M$. Under these assumptions, there exists $\kappa_\ast$ such that if $r>\kappa_\ast$, then the following are true:
\begin{itemize}\leftskip-0.25in
\item Supposing that $T\geq\kappa_\ast$ or that $T=\infty$, the space of $\R\times C^\infty(\R\times Y_T;S^1)$ orbits of instanton solutions to \eqref{eq:instanton-sw} on $\R\times Y_T$ with $\ind=1$ and $\lim_{s\to\pm\infty}\en(\frd|_s)\leq\cale$ is in 1--1 correspondence with the space of $\R\times C^\infty(\R\times Y_\infty;S^1)$ orbits of instanton solutions to \eqref{eq:instanton-sw} on $\R\times Y_\infty$ with $\ind=1$ and $\lim_{s\to\pm\infty}\en(\frd|_s)\leq\cale$. Moreover, given $\upvarepsilon>0$, there exists $T_\upvarepsilon$ (independent of $r$) such that items (a) and (b) of Proposition \ref{prop:comp-62} hold when $T>T_\upvarepsilon$.
\item If $T$ is greater than $\kappa_\ast$, then the conclusions of Theorems 4.2 and 4.3 in \cite{taubes1} hold with their versions of $\kappa$ being independent of $T$.
\end{itemize}
\end{proposition}
 
The proofs of Theorems 4.2 and 4.3 in \cite{taubes1} can be repeated in this context with only minor changes because of the control given by Lemmas~\ref{lem:bounds-35} and \ref{lem:bounds-36}, and by Lemma~\ref{lem:instanton-22} on the $\R\times(Y_T\ssm M)$ part of $\R\times Y_T$.  (These lemmas are used to turn the question of whether or not there is a cokernel to an instanton version of $\LL$ into a question about operators on $\R\times Y_R$ for fixed $R \geq c_\cale$ (much like what was done in Steps 5--9 of the proof of Proposition \ref{prop:Tconv-51}). Meanwhile, Proposition \ref{prop:conv-42} and the arguments from \cite{taubes1} and \cite{taubes3} can be used (when $r \geq c_\cale$) to answer the operator question for $\R\times Y_R$ by writing the $\R\times Y_R$ operator using a $\bar\partial$ operator on certain $J$-holomorphic submanifolds.)

Proposition~\ref{prop:comp-63} requires that $a$ and the metric be independent of the coordinate $s$, and that $\frp$ be absent from \eqref{eq:instanton-sw}. When $a$ and the metric have $s$-dependence, then the 2-form $\frp$ in \eqref{eq:instanton-sw} can be used to guarantee that the conditions for the second bullet of Proposition~\ref{prop:comp-62} can be met. The following proposition makes a precise statement to this effect.

\begin{proposition}
\label{prop:comp-64}
Given $\cale > 1$, there exists $\kappa_\cale$ with the following significance: Fix $r > \kappa_\cale$ and assume that the operator $\LLf$ in \eqref{eq:linearized} that is defined by any $\en \leq \cale$ solution of \eqref{eq:sw} on $Y_\infty$ with $a=a_-$ or with $a=a_+$ has trivial cokernel.  Assume in addition that there are no non-constant solutions to either the $a=a_-$ or $a=a_+$, constant metric, and $\frp=0$ version of \eqref{eq:instanton-sw} on $\R\times Y_\infty$ with $\ind\leq 0$ and with $\lim_{s\to\pm\infty} \en(\frd|_s) \leq \cale$.  If the form $\frp$ in \eqref{eq:instanton-sw} is sufficiently generic (which means that it is chosen from a certain Baire set of smooth forms with compact support in $(-1,1)\times M$), then there are no instanton solutions to \eqref{eq:instanton-sw} on $\R \times Y_\infty$ with $\lim_{s\to\pm\infty} \en(\frd|_s) \leq \cale$ and $\ind < 0$; and all instanton solutions with $\lim_{s\to\pm\infty} \en(\frd|_s) \leq \cale$ and $\ind=0$ are unobstructed.  In particular, the assumptions for the second bullet of Proposition~\ref{prop:comp-62} are met with this choice of $\frp$; and therefore the conclusions of the second bullet of Proposition~\ref{prop:comp-62} also hold.
\end{proposition}
\begin{proof}
This is an application of the Smale-Sard theorem. The details of the argument are almost identical to those used in the proof of the regularity of Seiberg-Witten moduli spaces with non-zero spinor in the case of a compact 4-manifold.  See, e.g.\ \cite{morgan-book}.  One can also use the argument in Chapter 24.3 of \cite{kmbook}, keeping in mind that the assumptions of the proposition are such that no extra perturbation on the ends of $\R\times Y_T$ are needed (so the perturbation denoted by $\hat\frq$ in Equation (24.2) of \cite{kmbook} can be set equal to zero).
\end{proof}